\documentclass[openright, 11 pt, a4paper]{book}
\usepackage[T1]{fontenc}
\usepackage{amsthm,amsmath,amssymb,mathrsfs}
\usepackage{graphicx}
\usepackage{tikz-cd}
\usetikzlibrary{babel}
\usetikzlibrary{backgrounds}
\usepackage{mathtools}
\usepackage{quiver}
\usepackage[hcentering]{geometry}
\usepackage{fancyhdr}
\usepackage{mathtools}
\usepackage{graphicx}
\usepackage{xcolor}
\usepackage{microtype}
\usepackage[british]{babel} 
\usepackage{csquotes} 
\usepackage[style=ieee, sorting=nyt]{biblatex} 
\DeclareQuotePunctuation{}
\usepackage[
    colorlinks=true,
    allcolors=black
]{hyperref}

\usepackage{cleveref}

\fancypagestyle{plain}{
    \fancyhf{}
    \fancyfoot[LE,RO]{\thepage}
}

\theoremstyle{plain} 
\newtheorem{theorem}{Theorem}[section]
\newtheorem{lemma}[theorem]{Lemma}
\newtheorem{conjecture}[theorem]{Conjecture}
\newtheorem{proposition}[theorem]{Proposition}

\theoremstyle{definition}
\newtheorem{definition}[theorem]{Definition}
\newtheorem{example}[theorem]{Example}
\newtheorem{remark}[theorem]{Remark}
\newcommand{\FT}{\operatorname{FT}}

\let\hat\widehat

\let\tilde\widetilde

\title{Examples of rigid Twistor \(\mathcal{D}\)-modules}
\author{Patrik Hammer}

\date{}
\begin{document}
\frontmatter

\begin{titlepage}
    \centering
    
    {\Huge \bfseries Examples of rigid Twistor \(\mathcal{D}\)-modules \par}

    \vspace{2.0cm}

    {\Large \bfseries Dissertation \par}

    \vspace{1.5cm}

    {\Large zur Erlangung des akademischen Grades\par}
    \vspace{0.4cm}
    {\Large Dr.\ rer.\ nat.\par}

    \vspace{1.5cm}

    {\Large eingereicht an der\par}
    \vspace{0.4cm}
    {\Large Mathematisch-Naturwissenschaftlich-Technischen Fakultät\par}
    \vspace{0.4cm}
    {\Large der Universität Augsburg\par}

    \vspace{1.5cm}

    {\Large von\par}
    \vspace{0.4cm}
    {\Large \bfseries Patrik André Hammer \par}

    \vspace{1.5cm}

    {\Large Augsburg, Oktober 2026\par}

    \vspace{1.5cm}
    
    \includegraphics[width=5cm]{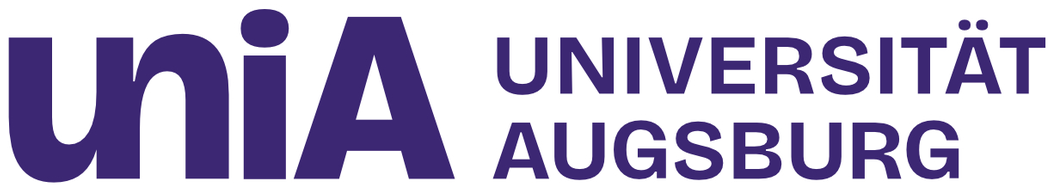}

\end{titlepage}

\newpage

\vspace*{\fill}

\begin{tabular}{@{}ll@{}}
    Erstgutachter: & \begin{tabular}[t]{@{}l@{}}
        \textbf{Prof. Dr. Marco Hien} \\
        Institut für Mathematik, Universität Augsburg \\
        86135 Augsburg, Deutschland
    \end{tabular} \\[2.5em] 
    \\
    Zweitgutachter: & \begin{tabular}[t]{@{}l@{}}
        \textbf{Prof. Dr. Claude Sabbah} \\
        Centre de mathématiques Laurent-Schwartz, École Polytechnique \\
        Cour Vaneau, 91120 Palaiseau, Frankreich
    \end{tabular}
\end{tabular}

\newpage

\section*{Abstract}
The well-established theory of \(\mathcal{D}\)-modules gives a framework to study linear differential equations from the perspective of algebraic geometry. A significant advancement in this theory has been the introduction of Twistor D-modules by Claude Sabbah and Takuro Mochizuki: These objects expose inherent, but hidden information of the underlying differential equations in an elegant way. We explain the foundational concepts of Twistor D-modules and give concrete examples of these objects. We introduce a notion of rigidity for Twistor \(\mathcal{D}\)-modules, which these examples fulfill. We conclude with a certain algebraicity conjecture about rigid Twistor \(\mathcal{D}\)-modules, which we prove in a special case. 

\section*{Zusammenfassung}
Die weit verbreitete Theorie der \(\mathcal{D}\)-Moduln erlaubt es, lineare Differentialgleichungen aus der Perspektive der algebraischen Geometrie zu studieren. Ein bedeutender Fortschritt in dieser Theorie war die Einführung von Twistor-D-Moduln durch Claude Sabbah und Takuro Mochizuki: Diese Objekte offenbaren inhärente, aber verborgene Informationen der zugrunde liegenden Differentialgleichungen auf eine elegante Weise. Wir erklären die grundlegenden Konzepte von Twistor-D-Moduln und geben konkrete Beispiele dieser Objekte. Wir führen einen Begriff von Rigidität für Twistor-$\mathcal{D}$-Moduln ein, welchen diese Beispiele erfüllen. Zuletzt formulieren wir eine gewisse Algebraizitätsvermutung über rigide Twistor-$\mathcal{D}$-Moduln, die wir in einem Spezialfall beweisen. 

\newpage\null\thispagestyle{empty}\newpage

\section*{Acknowledgments}

First and foremost, I am grateful to my supervisor Marco Hien for introducing me to this field of study, for his continuing support and for trusting me to complete this project. Thank you for all the important discussions and for being so supportive in general. 

I want to thank Claude Sabbah for inviting me to a three-month stay at École Polytechnique. This project would not have been possible without the numerous discussions we had during this time and on other occasions. Moreover, I want to express my gratitude towards Takuro Mochizuki, Peter Scholze, Christian Sevenheck and Claus Hertling for inspiring conversations. 

I also wish to thank my colleagues from the University of Augsburg, in particular Carsten Neumann, for countless discussions over the years, which brought many important insights. 

During the time of writing this thesis, I was encouraged by my parents, my girlfriend Julia and my friends, to all of whom I am deeply grateful. All of you are most important to me. 

I also want to thank my friend Wolfgang Utschick for sharing his experience in the academic world with me since my school days. 

Finally, I want to thank the Norbert-Janssen-Foundation, the German Academic Scholarship Foundation, the Max-Weber-Program and the Jacques Hadamard Society of Mathematics for various scholarships during my bachelor's, master's and doctoral studies. 

\newpage\null\thispagestyle{empty}\newpage

\tableofcontents

\mainmatter

\chapter{Introduction}

In this thesis, we study the theory of Twistor \(\mathcal{D}\)-modules, which was developed mainly by Claude Sabbah and Takuro Mochizuki. Before discussing Twistor \(\mathcal{D}\)-modules, we must first introduce the usual \(\mathcal{D}\)-modules: The theory of \(\mathcal{D}\)-modules allows one to study linear differential equations from the perspective of algebraic geometry. For example, consider the linear differential equation \(tf' - \alpha f = 0\), where \(\alpha\) is a fixed complex parameter, \(t\) is the coordinate and \(f\) is the unknown function, which should be a multi-valued holomorphic function on the punctured complex plane \(\mathbb{C}^*\). We have the sheaf of differential operators on \(\mathbb{C}\), denoted as \(\mathcal{D}_{\mathbb{C}}\): For each open \(U \subset \mathbb{C}\), the sections of \(\mathcal{D}_{\mathbb{C}}\) over \(U\) take the form 
\begin{align*}
\sum_{i = 0}^n c_i(t)\partial_t^i, 
\end{align*}

where \(\partial_t\) stands for differentiation with respect to \(t\) and the finitely many \(c_i\) are holomorphic functions on \(U\). The equation \(tf' - \alpha f = 0\) can be rewritten as \(\mathcal{P}f = 0\), where \(\mathcal{P} = t\partial_t - \alpha\) is a global section of \(\mathcal{D}_{\mathbb{C}}\). One can obtain such an operator \(\mathcal{P}\) for any linear differential equation on \(\mathbb{C}\), and then one can consider the \(\mathcal{D}_{\mathbb{C}}\)-module \(\mathcal{D}_{\mathbb{C}}/\mathcal{D}_{\mathbb{C}}\mathcal{P}\): It contains all information about the original differential equation - in particular the solutions - but it can be analyzed with the methods of algebraic geometry, since it is a sheaf of modules. Because the solutions to the differential equation are usually not algebraic, but analytic, one sometimes calls the study of \(\mathcal{D}\)-modules analytic geometry. For a detailed exposition of the theory of \(\mathcal{D}\)-modules, we refer to \cite{Hot18}. 

Some \(\mathcal{D}\)-modules can be equipped with an important filtration, which is called the "Hodge filtration". Originally, Hodge filtrations were first observed on the cohomology group of Kähler manifolds. It turns out that this notion can be generalized to give a filtration on \(\mathcal{D}\)-modules: This study was initiated by Phillip Griffiths (cf. \cite{Gri70}) in the form of "variations of Hodge structure" and brought into relationship with harmonic metrics by Carlos Simpson (cf. \cite{Sim88}). Variations of Hodge structure were generalised by Morihiko Saito in the form of "Hodge modules" (cf. \cite{Sai90}, \cite{Sai16}). 

While this theory is exceptionally strong, it has the drawback that Hodge filtrations only exist for \(\mathcal{D}\)-modules with certain properties: In particular, they can only have regular singularities, and they need to satisfy certain additional properties at their singularities. 

This leads to the generalisation of Hodge modules to "Twistor \(\mathcal{D}\)-modules". This was done by Claude Sabbah (cf. \cite{Sab05}, \cite{Sab08}) and Takuro Mochizuki (cf. \cite{Moc03}, \cite{Moc10}, \cite{Moc15}), building upon the notion of "variations of Twistor structure" introduced by Carlos Simpson (cf. \cite{Sim97}). One of the main results of this work is the following, which was first proved in a special case in \cite{Sab05}: 

\begin{theorem}[Mochizuki, cf. \cite{Moc10}, Theorem 19.4.1]\label{thm:mocintro}
There is a categorical equivalence between semi-simple holonomic \(\mathcal{D}_{\mathbb{P}^1}\)-modules and semi-simple polarizable purely imaginary Twistor \(\mathcal{D}\)-modules on \(\mathbb{P}^1\) of weight \(0\). 
\end{theorem}

This result is the starting point for this dissertation. 

\section{Outline of this work}

We shall give a brief summary of this dissertation, going chapter by chapter. 

\subsection{Outline of \Cref{chapter:dmodprep}}

We begin by recalling the definitions of algebraic and analytic \(\mathcal{D}\)-modules and several important functors (localisation, minimal extension, analytification etc.), objects (Bernstein polynomial, \(V\)-filtration, local system etc.) and results (Riemann-Hilbert correspondence, Levelt-Turrittin decomposition etc.). As a preparation for the proof of our main results, we make all of these constructions explicit for \(\mathcal{D}\)-modules on the complex plane of generic rank \(1\). 

\subsection{Outline of \Cref{chapter:rank1twistor}}

Our first main result makes the equivalence from \Cref{thm:mocintro} explicit for regular holonomic \(\mathcal{D}_{\mathbb{P}^1}\)-modules of generic rank \(1\): 

\begin{theorem}[\Cref{theorem:firstmainresult}]\label{theorem:firstmainresultfirst}
Let \(\mathcal{P} \in \mathcal{D} = \mathbb{C}\langle t, \partial_t \rangle\) be of order \(1\) and let \(\mathcal{M} \coloneqq \mathcal{D}/(\mathcal{P})\). Assume that the \(\mathcal{D}_{\mathbb{P}^1}\)-module \(\mathfrak{M} \coloneqq (p^{*, \operatorname{loc}}\mathcal{M})({!*}\infty)\) (minimal extension of \(\mathcal{M}\) at \(\infty\)) is irreducible and regular holonomic. Write \(\mathcal{P} = p(t)\partial_t + q(t)\) and 
\begin{align*}
q(t)/p(t) = -\sum_i \frac{m_i}{t - r_i}
\end{align*}
for \(m_i \in \mathbb{C} - \mathbb{Z}, r_i \in \mathbb{C}\). Set \(D = \operatorname{Sing}(\mathfrak{M})\), \(U = \mathbb{P}^1 - D\). Then, the unique object 
\begin{align*}
\mathfrak{T} = ((\mathfrak{M}_z, \mathfrak{M}_z, \mathfrak{C}), (\operatorname{id}, \operatorname{id})) \in \operatorname{MT}^{\operatorname{wild}}_{\mathbb{P}^1, \operatorname{s}}(\mathbb{P}^1, 0, \sqrt{-1}\mathbb{R})^{(p)}
\end{align*}
with \(\Xi_{\operatorname{DR}}(\mathfrak{T}) = \mathfrak{M}\) is given as follows: \(\mathfrak{M}_z = (p^{*, \operatorname{loc}}\mathcal{M}_{z})[{!*}\infty]\), where \(\mathcal{M}_{z} \coloneqq \mathscr{R}/(\mathcal{P}_{z})\) - we define \(\mathscr{R} = \mathcal{O}_{\mathbb{C}_z}\langle t, z\partial_t \rangle\) - with 
\begin{align*}
\mathcal{P}_{z} \coloneqq p(t)z\partial_t - p(t) \sum_{j = 1}^n \frac{m_j \star z}{t - r_j} \qquad \left(m_j \star z \coloneqq \Re\left(m_j\right)z + \Im\left(m_j\right)\frac{z^2 + 1}{2}i\right). 
\end{align*}
Furthermore, the pairing \(\mathfrak{C} : \overline{\sigma^*}(\mathfrak{M}_{z}\vert_{\boldsymbol{S}}) \otimes \mathfrak{M}_{z}\vert_{\boldsymbol{S}} \rightarrow \mathfrak{Db}_{\mathbb{P}^1 \times \boldsymbol{S}/\boldsymbol{S}}\) - it is a pairing between \(\mathfrak{M}_{z}\) and its Twistor-conjugate, both restricted to \(|z| = 1\), and it takes values in distributions - is a natural extension of the pairing 
\begin{align*}
\mathfrak{C}\vert_{U}: \overline{\sigma^*}(\mathfrak{M}_{z}\vert_{U \times \boldsymbol{S}}) \otimes \mathfrak{M}_{z}\vert_{U \times \boldsymbol{S}} & \rightarrow C^{\infty, \operatorname{an}}_{U \times \boldsymbol{S}} \\ \overline{\sigma^*}(u \cdot e) \otimes (v \cdot e) & \mapsto \overline{\sigma^*}u \cdot \prod_{j = 1}^n \vert t - r_j \vert^{2(m_j \star z)/z} \cdot v, 
\end{align*}
where \(u, v \in \mathcal{O}(U \times \boldsymbol{S})\) and where \(e\) is the \(\mathcal{O}_{U \times \boldsymbol{S}}\)-generator of \(\Gamma(U \times \boldsymbol{S}, \mathfrak{M}_{z})\) obtained as the image of the class of \(1\) under the \(t\)-analytification map 
\begin{align*}
\mathscr{R}(U^{\operatorname{alg}} \times \boldsymbol{S})/(\mathcal{P}_z) = \mathscr{R}(U^{\operatorname{alg}} \times \boldsymbol{S}) \otimes_{\mathscr{R}(\boldsymbol{S})} \Gamma(\boldsymbol{S}, \mathcal{M}_z) \mapsto \Gamma(U \times \boldsymbol{S}, \mathfrak{M}_{z}). 
\end{align*}
\end{theorem}

Using exponential twist, we can generalize this result to arbitrary holonomic \(\mathcal{D}_{\mathbb{P}^1}\)-modules of generic rank \(1\) which have irregular singularities. 

\subsection{Outline of \Cref{chapter:fouriertwistor}}

Our second main result makes the equivalence from \Cref{thm:mocintro} explicit for a family of holonomic \(\mathcal{D}_{\mathbb{P}^1}\)-modules of generic rank \(2\). This is done by calculating the Fourier transform of a family of Twistor \(\mathcal{D}\)-modules given by \Cref{theorem:firstmainresultfirst}: 

\begin{theorem}[\Cref{thm:FTofT}]\label{theorem:FTofTfirst}
Let \(\mathcal{P} \in \mathcal{D} = \mathbb{C}\langle t, \partial_t \rangle\) be the operator \(t\partial_t - \alpha - t(t\partial_t - \beta)\), where \(\alpha, \beta \in \mathbb{C}\) with \(\Re(\alpha) > -1, \Re(\beta - \alpha) > -1, \Re(\beta) < -1\). Let \(\mathcal{M} \coloneq \mathcal{D}/(\mathcal{P})\) and consider the \(\mathcal{D}_{\mathbb{P}^1}\)-module \(\mathfrak{M} = (p^{*, \operatorname{loc}}\mathcal{M})({!*}\infty)\). Denote the corresponding Twistor \(\mathcal{D}\)-module from \Cref{theorem:firstmainresultfirst} by \(\mathfrak{T}\). Then, we have 
\begin{align*}
\operatorname{FT}(\mathfrak{T}) \cong ((\mathfrak{N}_z, \mathfrak{N}_z, \mathfrak{D}), (\operatorname{id}, \operatorname{id})) \in \operatorname{MT}^{\operatorname{wild}}_{\mathbb{P}^1, \operatorname{s}}(\mathbb{P}^1, 0, \sqrt{-1}\mathbb{R})^{(p)}, 
\end{align*}

where \(\mathfrak{N}_z\) and \(\mathfrak{D}\) are given as follows: \(\mathfrak{N}_z\) is the \(\mathscr{R}_{\mathbb{P}^1 \times \mathbb{C}_z}\)-module \(\mathfrak{N}_z = (p^{*, \operatorname{loc}}\mathcal{N}_z)[{!*}\infty]\) where \(\mathcal{N}_z = \mathscr{R}/(\operatorname{FT}_z(\mathcal{P}_z))\) with \(\mathcal{P}_z \in \mathcal{R}\) being defined as in \Cref{theorem:firstmainresult}. The pairing \(\mathfrak{D}\) is a natural extension of the pairing 
\begin{align*}
\mathfrak{D}\vert_{\mathbb{C}_\tau^*}: \overline{\sigma^*}\mathfrak{N}_z\vert_{\mathbb{C}_\tau^* \times \boldsymbol{S}} & \otimes \mathfrak{N}_z\vert_{\mathbb{C}_\tau^* \times \boldsymbol{S}} & \rightarrow & \qquad C^{\infty, \operatorname{an}}_{\mathbb{C}_\tau^* \times \boldsymbol{S}} \\ \begin{pmatrix} \overline{\sigma^*}u_1 \\ \overline{\sigma^*}u_2\end{pmatrix} & \otimes \begin{pmatrix} v_1 \\ v_2\end{pmatrix} & \mapsto & \qquad \begin{pmatrix} \overline{\sigma^*}u_1 \\ \overline{\sigma^*}u_2\end{pmatrix}^T\cdot D_{\operatorname{mat}} \cdot \begin{pmatrix} v_1 \\ v_2\end{pmatrix}
\end{align*}

which is given in the basis of \(\Gamma(\mathbb{C}_\tau^* \times \boldsymbol{S}, \mathfrak{N}_{z})\) obtained as the image of the classes of \(1\) and \(z\partial_\tau\) under the \(\tau\)-analytification map 
\begin{align*}
\mathscr{R}(\mathbb{C}_\tau^{*,\operatorname{alg}} \times \boldsymbol{S})/(\operatorname{FT}_z(\mathcal{P}_z)) = \mathscr{R}(\mathbb{C}_\tau^{*,\operatorname{alg}} \times \boldsymbol{S}) \otimes_{\mathscr{R}(\boldsymbol{S})} \Gamma(\boldsymbol{S}, \mathcal{N}_z) \rightarrow \Gamma(\mathbb{C}_\tau^* \times \boldsymbol{S}, \mathfrak{N}_{z}). 
\end{align*}

Here, the matrix \(D_{\operatorname{mat}} \in (C^{\infty, \operatorname{an}}_{\mathbb{C}_\tau^* \times \boldsymbol{S}})^{2 \times 2}\) is given by \(D_{\text{mat}} = \overline{\sigma^*}B^T \cdot B\) with  
\begin{align*}
B(\tau, z) \coloneqq \operatorname{diag}(c_1(z), c_2(z)) \cdot \begin{pmatrix}
F_1(\tau, z) & z\partial_\tau(F_1(\tau, z)) \\ \tau^{-\beta \diamond z - 1}F_2(\tau, z) & z\partial_\tau(\tau^{-\beta \diamond z - 1}F_2(\tau, z)) 
\end{pmatrix}, 
\end{align*}

where the \(c_i\) are given as follows: Setting \(\widetilde\Gamma(s) \coloneqq \Gamma(1 + s)/\Gamma(-s)\) (\(s \in \mathbb{C}\)) and 
\begin{align*}
C_1(z) & \coloneqq \frac{\widetilde{\Gamma}(\alpha \diamond z)\widetilde{\Gamma}(\beta \diamond z - \alpha \diamond z)}{\widetilde{\Gamma}(1 + \beta \diamond z)}, \\  C_2(z) & \coloneqq \widetilde{\Gamma}(\beta \diamond z) 
\end{align*}

with \(\gamma \diamond z = (\gamma \star z)/z\) (\(\gamma \in \mathbb{C}\)), the \(c_i \in \mathcal{O}(\boldsymbol{S})\) \((i = 1,2)\) satisfy \(\overline{\sigma^*}c_i(z) \cdot c_i(z) = C_i(z)\) and are constructed via \Cref{thm:twistorroot}. The \(F_i \in \mathcal{O}(\mathbb{C}_\tau^* \times \boldsymbol{S})\) are given by 
\begin{alignat*}{4}
F_1(\tau, z) &\coloneqq {} & & \,_1F_1(1 + \alpha \diamond z, && \ 2 + \beta \diamond z, && \ -\tau/z), \\
F_2(\tau, z) &\coloneqq {} & & \,_1F_1(\alpha \diamond z - \beta \diamond z, && \ -\beta \diamond z, && \ -\tau/z).
\end{alignat*}

Here, \(_1F_1(a, b, \tau) \in \mathcal{O}_{\mathbb{C}_\tau}\) is the (confluent) hypergeometric Kummer function: 
\begin{align*}
_1F_1(a, b, \tau) = \sum_{n = 0}^\infty \frac{a^{(n)}\tau^n}{b^{(n)}n!}, 
\end{align*}

where \(a^{(n)}, b^{(n)}\) is the rising factorial. 
\end{theorem}

We note that the \(\mathcal{D}\)-modules considered in \Cref{theorem:firstmainresultfirst} and \Cref{theorem:FTofTfirst} are hypergeometric \(\mathcal{D}\)-modules. In general they are defined by an operator of the form 

\begin{align*}
\prod_{i = 1}^m (t\partial_t - \rho_i) - t \prod_{j = 1}^n (t\partial_t - \gamma_j), 
\end{align*}

where \(\rho_i, \gamma_j \in \mathbb{C}\) and \(m, n \in \mathbb{N}\). Hypergeometric operators are well-studied in the literature, see for example \cite{Fed17}, \cite{DomReiSev19} and \cite{QiXu25} for Hodge-theoretical aspects. In particular, \Cref{theorem:FTofTfirst} has slight overlap with \cite{DomReiSev19}, Proposition 5.4: There, the formula for \(\mathscr{N}_z\) is already obtained in the case for the operators we consider and also for a larger subfamily of the hypergeometric operators, but only for real parameters. Our methods are different, allowing us to consider complex \(\alpha, \beta\) and to obtain the formula for the matrix \(D_{\operatorname{mat}}\) as well, giving a complete description of the associated Twistor \(\mathcal{D}\)-modules. 

\subsection{Outline of \Cref{chapter:rigidalg}}

This chapter is about Twistor \(\mathcal{D}\)-modules on \(\mathbb{P}^1\) corresponding to rigid \(\mathcal{D}\)-modules, which we call rigid Twistor \(\mathcal{D}\)-modules. We give an equivalent definition for this rigidity without falling back to the case of \(\mathcal{D}\)-modules and we show that the purely imaginary Twistor \(\mathcal{D}\)-modules of this type have an algebraic substructure: 

\begin{theorem}[\Cref{lemma:zalgeb}]
Every object in \(\operatorname{MT}^{\operatorname{wild}}_{\mathbb{P}^1, \operatorname{s, rig}}(\mathbb{P}^1, \sqrt{-1}\mathbb{R})^{(p)}\) is \(z\)-algebraic. 
\end{theorem}

This result is a consequence of \Cref{theorem:firstmainresult} and the Katz-Arinkin algorithm (cf. \cite{Ari10}). We conjecture that this result also holds if one lifts the "purely imaginary" assumption: 

\begin{conjecture}[\Cref{conj:finalconj}]
Every object in \(\operatorname{MT}^{\operatorname{wild}}_{\mathbb{P}^1, \operatorname{s, rig}}(\mathbb{P}^1)^{(p)}\) is \(z\)-algebraic. 
\end{conjecture}

Lastly, we introduce an equivalence relation on simple holonomic \(\mathcal{D}_{\mathbb{P}^1}\)-modules, which we call \(\boldsymbol{S}\)-companionship, and we show that any \(\boldsymbol{S}\)-companion of any rigid object is also rigid. 

\subsection{Outline of \Cref{chapter:gagaapp}}

There are well-known categorical equivalences 
\[\begin{tikzcd}
	{\text{coherent }\mathbb{C}\langle t, \partial_t \rangle \text{-modules}} & {\text{coherent }\mathcal{D}_{\mathbb{A}^1} \text{-modules}} & {\text{coherent }\mathcal{D}_{\mathbb{P}^1}(*\infty) \text{-modules.}}
	\arrow[tail reversed, from=1-1, to=1-2]
	\arrow[tail reversed, from=1-2, to=1-3]
\end{tikzcd}\]

To prove our main results from the previous chapters, we will need the following analogue for \(\mathscr{R}\)-modules: 

\begin{theorem}[\Cref{lemma:partialGAGA}, \Cref{lemma:RaffineRmod}]
We have categorical equivalences 
\[\begin{tikzcd}
	{\text{coherent }\mathscr{R} \text{-modules}} & {\pi\text{-good }\mathscr{R}_{\mathbb{A}^1} \text{-modules}} & {p\text{-good }\mathscr{R}_{\mathbb{P}^1}(*\infty) \text{-modules.}}
	\arrow[tail reversed, from=1-1, to=1-2]
	\arrow[tail reversed, from=1-2, to=1-3]
\end{tikzcd}\]
\end{theorem}

Here, the notions of \(p\)-goodness and \(\pi\)-goodness are slightly stronger than coherence. 

\chapter{Preparations with \texorpdfstring{\(\mathcal{D}\)-modules}{D-modules}}\label{chapter:dmodprep}

In this chapter, we recall several relevant topics from the theory of \(\mathcal{D}\)-modules, starting with their definition. We will see how these concepts apply to the generic rank \(1\) case, which is a necessary preparation for the later results in Twistor \(\mathcal{D}\)-modules. Furthermore, most of these topics have an analogue for \(\mathcal{R}\)-modules and Twistor \(\mathcal{D}\)-modules, which will come up in the subsequent chapters. However, these analogues are more complicated than the concepts for \(\mathcal{D}\)-modules, which is why we start with this setting and try to be as concrete as possible. 

\section{Linear differential operators}

We begin by introducing the algebra of algebraic complex linear differential operators in \(n\) variables: 

\begin{definition}
For \(n \in \mathbb{N}_+\), the \(n\)-th Weyl algebra \(A_n\) over \(\mathbb{C}\) is defined as 

\begin{align*}
\mathbb{C}\langle t_1, \ldots, t_n, \partial_{t_1}, \ldots, \partial_{t_n} \rangle
\end{align*}

with the relations 
\begin{align*}
[\partial_{t_i}, t_j] = \delta_{ij}, \quad [\partial_{t_i}, \partial_{t_j}] = 0, \quad [t_i, t_j] = 0. 
\end{align*}

(The brackets \(\langle \cdot\rangle\) in this definition simply indicate that \(A_n\) is not commutative, because \([\partial_{t_i}, t_j] \neq 0\).) This \(\mathbb{C}\)-algebra also is a ring naturally, and we denote this ring as \(\mathcal{D}\) for \(n = 1\). 
\end{definition}

The relations in \(A_n\) are motivated as follows: 

\begin{remark}
Let \(f: \mathbb{C}^n \rightarrow \mathbb{C}\) be a holomorphic function in the variables \(t_i\). The relations in \(A_n\) imitate how multiplication by the coordinate functions \(t_i\) and differentiation with respect to them act on \(f\). For example, the Leibniz rule shows that 
\begin{align*}
\partial_{t_i}(t_if(t_1, \ldots, t_n)) - t_i\partial_{t_i}f(t_1, \ldots, t_n) = f(t_1, \ldots, t_n)
\end{align*}

This means that, if we view \(\partial_{t_i}\) and \(t_i\) as operators acting on the space of holomorphic functions \(\mathbb{C}^n \rightarrow \mathbb{C}\), their commutator \([\partial_{t_i}, t_i]\) acts as the identity, or in other words: \([\partial_{t_i}, t_i] = 1\). The other commutation relations are observed similarly. 
\end{remark}

For the remainder of this section, let \(X\) be a complex manifold or complex smooth variety. The next step is to generalize the Weyl-algebra to a sheaf of \(\mathbb{C}\)-algebras on \(X\), as done in \cite{Hot18}, Chapter 1.1 for the algebraic case and in \cite{Hot18}, Chapter 4.1 for the analytic case: First, one defines the sheaf of \(\mathbb{C}\)-linear derivations on \(X\): 
\begin{align*}
\mathcal{D}er_{\mathbb{C}_X}(\mathcal{O}_X) \coloneqq \{\theta \in \mathcal{E}nd_{\mathbb{C}_X}(\mathcal{O}_X) : \theta(fg) = \theta(f)g + f\theta(g) \quad (f, g \in \mathcal{O}_X)\}. 
\end{align*}
\(\mathcal{D}er_{\mathbb{C}_X}(\mathcal{O}_X)\) is \(\mathcal{O}_X\)-locally free of rank \(\dim X\) and \(\mathcal{O}_X\) can be seen as a subsheaf of \(\mathcal{E}nd_{\mathbb{C}_X}(\mathcal{O}_X)\), by sending a local section \(f \in \mathcal{O}_X\) to 
\begin{align*}
\mathcal{O}_X \rightarrow \mathcal{O}_X: g \mapsto fg. 
\end{align*}
The sheaf of differential operators on \(X\), denoted \(\mathcal{D}_X\), is defined as the \(\mathbb{C}\)-subalgebra of \(\mathcal{E}nd_{\mathbb{C}_X}(\mathcal{O}_X)\) generated by \(\mathcal{O}_X\) and \(\mathcal{D}er_{\mathbb{C}_X}(\mathcal{O}_X)\). We abbreviate \(\mathcal{D}(X) \coloneqq \Gamma(X, \mathcal{D}_X)\) and \(\mathcal{O}(X) \coloneqq \Gamma(X, \mathcal{O}_X)\). In a coordinate neighborhood \(U \subset X\) with coordinates \((t_1, \ldots, t_n) \in \mathcal{O}_X^{\oplus n}(U)\) of \(X\), we have the derivations \((\partial_{t_1}, \ldots, \partial_{t_n})\). For any point \(p\) in \(X\), a coordinate neighborhood \(U\) containing \(p\) exists: If \(X\) is a manifold, one chooses any neighborhood of \(p\) which is homeomorphic to \(\mathbb{C}^n\), and if \(X\) is a smooth variety, one uses Étale coordinates. An introduction to Étale coordinates is given in \cite{Mil13}, Section 1.2. We obtain 
\begin{align*}
\mathcal{D}er_{\mathbb{C}_U}(\mathcal{O}_U) = \oplus_{i = 1}^n \mathcal{O}_U\partial_{t_i}, \quad  [\partial_{t_i}, \partial_{t_j}] = 0, \quad [t_i, t_j] = 0, 
\end{align*}

where the commutation relations are the same as for the Weyl algebra defined earlier. \(\mathcal{D}_U\) is then free of infinite rank over \(\mathcal{O}_U\): 
\begin{align*}
\mathcal{D}_U = \oplus_{\alpha \in \mathbb{N}^{\alpha}} \mathcal{O}_U \partial_t^{\alpha} \quad (\partial_t \coloneqq \partial_{t_1}^{\alpha_1} \cdots \partial_{t_n}^{\alpha_n}). 
\end{align*}

We have the order filtration \(F_{\bullet}\) on \(\mathcal{D}_U\), which is indexed by \(\mathbb{Z}\): For \(l \in \mathbb{N}\), set 
\begin{align*}
F_{l}\mathcal{D}_U \coloneqq \oplus_{|\alpha| \leq l} \mathcal{O}_U \partial_t^{\alpha}, \quad \left(|\alpha| = \sum_i \alpha_i\right) 
\end{align*}

and for \(l \in \mathbb{Z}_{< 0}\), set \(F_{i}\mathcal{D}_U = 0\). This is naturally extended to a filtration \(F_{\bullet}\) on \(\mathcal{D}_X\). A local section \(\mathcal{P} \in \mathcal{D}_X\) is said to have order \(l\), if \(\mathcal{P} \in F_l\mathcal{D}_X - F_{l-1}\mathcal{D}_X\). We have the analogous filtration on \(A_n\) for \(n \in \mathbb{N}_+\). 

\begin{remark}
The order filtration does not give a grading on \(\mathcal{D}_X\), because the graded object 
\begin{align*}
\operatorname{gr}^F\mathcal{D}_X \coloneqq \oplus_{i = 0}^\infty \operatorname{gr}^F_i\mathcal{D}_X = \oplus_{i = 0}^\infty F_{i}\mathcal{D}_X/F_{i-1}\mathcal{D}_X 
\end{align*}

is not isomorphic to \(\mathcal{D}_X\). In fact, \(\operatorname{gr}^F\mathcal{D}_X\) is commutative (cf. \cite{Hot18}, Proposition 1.1.3(iii)), whereas \(\mathcal{D}_X\) is not. 
\end{remark}

\section{Algebraic and analytic \texorpdfstring{\(\mathcal{D}\)-modules}{D-modules}}

Let \(X\) be a complex manifold or complex smooth variety. A sheaf of \(\mathcal{D}_X\)-modules is called a \(\mathcal{D}_X\)-module or a \(\mathcal{D}\)-module on \(X\). There are however important distinctions between \(\mathcal{D}\)-modules on manifolds (analytic \(\mathcal{D}\)-modules) and \(\mathcal{D}\)-modules on varieties (algebraic \(\mathcal{D}\)-modules). For now, we let \(X\) be a smooth variety and let \(\operatorname{Mod}_{\operatorname{qc}}(\mathcal{O}_X)\) (resp. \(\operatorname{Mod}_{\operatorname{qc}}(\mathcal{D}_X)\)) be the category of \(\mathcal{O}_X\)-modules (resp. \(\mathcal{D}_X\)-modules), which are \(\mathcal{O}_X\)-quasi-coherent. The following is a classical fact: 

\begin{lemma}
Let \(X\) be affine. Then: 

\begin{itemize}
    \item The global sections functor \(\Gamma(X, \cdot): \operatorname{Mod}_{\operatorname{qc}}(\mathcal{O}_X) \rightarrow \operatorname{Mod}(\mathcal{O}(X))\) is exact. 
    \item If \(\Gamma(X, \mathcal{M}) = 0\) for \(\mathcal{M} \in \operatorname{Mod}_{\operatorname{qc}}(\mathcal{O}_X)\), we have \(\mathcal{M} = 0\). 
\end{itemize}
\end{lemma}

Inspired by this result, the following notion is introduced in \cite{Hot18}, Definition 1.4.2: 

\begin{definition}
\(X\) is called \(\mathcal{D}\)-affine, if the following holds: 

\begin{itemize}
    \item The global sections functor \(\Gamma(X, \cdot): \operatorname{Mod}_{\operatorname{qc}}(\mathcal{D}_X) \rightarrow \operatorname{Mod}(\mathcal{D}(X))\) is exact. 
    \item If \(\Gamma(X, \mathcal{M}) = 0\) for \(\mathcal{M} \in \operatorname{Mod}_{\operatorname{qc}}(\mathcal{D}_X)\), we have \(\mathcal{M} = 0\). 
\end{itemize}
\end{definition}

Any affine variety \(X\) is also \(\mathcal{D}\)-affine by definition. The following result (cf. \cite{Hot18}, Proposition 1.4.4), which justifies the name "\(\mathcal{D}\)-affine": 

\begin{lemma}\label{lemma:daffineequ}
Let \(X\) be \(\mathcal{D}\)-affine. Then, the functor 
\begin{align*}
\Gamma(X, \cdot): \operatorname{Mod}_{\operatorname{qc}}(\mathcal{D}_X) \rightarrow \operatorname{Mod}(\mathcal{D}(X)) 
\end{align*}

is an equivalence of categories. \(\Gamma(X, \cdot)\) has a quasi-inverse given by 
\begin{align*}
\mathcal{D}(X) \otimes_{\Gamma(X, \mathcal{D}_X))} (\cdot): \operatorname{Mod}(\mathcal{D}_X) \rightarrow \operatorname{Mod}_{\operatorname{qc}}(\mathcal{D}_X). 
\end{align*}
\end{lemma}

If \(\mathcal{M}\) is a \(\mathcal{D}_X\)-module, one can define the analytification \(\mathcal{M}^{\operatorname{an}}\) of \(\mathcal{M}\) (cf. \cite{Hot18}, Section 4.7). It is a \(\mathcal{D}_{X}^{\operatorname{an}}\)-module, where \(X^{\operatorname{an}}\) is the complex manifold obtained by equipping \(X\) with the analytic topology: 

\begin{definition}
Let \(X\) be a smooth variety and let \(\mathcal{M}\) be a \(\mathcal{D}_X\)-module. Consider the continuous map \(i: X^{\operatorname{an}} \rightarrow X\). Define 
\begin{align*}
\mathcal{M}^{\operatorname{an}} \coloneqq \mathcal{O}_{X^{\operatorname{an}}} \otimes_{i^{-1}\mathcal{O}_X} i^{-1}\mathcal{M}. 
\end{align*}

This gives an exact functor. 
\end{definition}

\begin{remark}
Let \(U \subset X\) open and let \(\mathcal{M}\) be a \(\mathcal{D}_X\)-module. We have the canonical injective map \(\Gamma(U, \mathcal{M}) \rightarrow \Gamma(U^{\operatorname{an}}, \mathcal{M}^{\operatorname{an}})\) given by \(m \mapsto m \otimes 1\). 
\end{remark}

\begin{example}
Recall that \(\mathcal{D}\) is the ring \(\mathbb{C}\langle t, \partial_t \rangle\) and let \(\mathcal{P} \in \mathcal{D}\). We obtain the \(\mathcal{D}\)-module \(\mathcal{D}/\mathcal{D}\mathcal{P}\). Consider the complex lines \(\mathbb{C}\) and \(\mathbb{A}^1\), equipped with the analytic and algebraic topology, respectively. We have the continuous map \(\mathbb{C} \rightarrow \mathbb{A}^1\). Because \(\mathcal{D}(\mathbb{A}^1) = \mathcal{D}\), \(\mathcal{D}/\mathcal{D}\mathcal{P}\) defines the quasi-coherent \(\mathcal{D}_{\mathbb{A}^1}\)-module 
\begin{align*}
\mathcal{M} \coloneqq \mathcal{D}/\mathcal{D}\mathcal{P} \otimes_{\mathcal{D}} \mathcal{D}_{\mathbb{A}^1} = \mathcal{D}_{\mathbb{A}^1}/\mathcal{D}_{\mathbb{A}^1}\mathcal{P}, 
\end{align*}

with \Cref{lemma:daffineequ}, where the last equality follows from the fact that \((\cdot) \otimes_{\mathcal{D}} \mathcal{D}_{\mathbb{A}^1}\) is an exact functor here, because it is an equivalence of categories. We can also consider the analytification 
\begin{align*}
\mathcal{M}^{\operatorname{an}} = i^*(\mathcal{D}_{\mathbb{A}^1}/\mathcal{D}_{\mathbb{A}^1}\mathcal{P}) \otimes_{i^*\mathcal{O}_{\mathbb{A}^1}} \mathcal{O}_{\mathbb{C}} = \mathcal{D}_{\mathbb{C}}/\mathcal{D}_{\mathbb{C}}\mathcal{P}, 
\end{align*}

where the last equality holds because the analytification functor is exact. 
\end{example}

\subsection{Sections of cyclic \(\mathcal{D}\)-modules}

The section of a \(\mathcal{D}_{\mathbb{A}^1}\)-module of the form \(\mathcal{D}_{\mathbb{A}^1}/\mathcal{D}_{\mathbb{A}^1}\mathcal{P}\) can be expressed in a convenient way: 

\begin{lemma}\label{lemma:sectalg}
Let \(\mathcal{P} \in \mathcal{D}\) and \(X \subset \mathbb{A}^1\) be (Zariski-)open. We have the isomorphism 
\begin{align*}
\Gamma(X, \mathcal{D}_{\mathbb{A}^1}/\mathcal{D}_{\mathbb{A}^1}\mathcal{P}) \cong \mathcal{D}(X)/\mathcal{D}(X)\mathcal{P}
\end{align*}
\end{lemma}

\begin{proof}
Consider the exact sequence of \(\mathcal{D}_{\mathbb{A}^1}\)-modules 
\begin{align*}
0 \to \mathcal{D}_{\mathbb{A}^1} \xrightarrow{\cdot \mathcal{P}} \mathcal{D}_{\mathbb{A}^1} \to \mathcal{D}_{\mathbb{A}^1}/\mathcal{D}_{\mathbb{A}^1}\mathcal{P} \to 0. 
\end{align*}

Because \(X\) is \(\mathcal{D}\)-affine, we get 
\begin{align*}
0 \to \Gamma(X, \mathcal{D}_{\mathbb{A}^1}) \xrightarrow{\cdot \mathcal{P}} \Gamma(X, \mathcal{D}_{\mathbb{A}^1}) \to \Gamma(X, \mathcal{D}_{\mathbb{A}^1}/\mathcal{D}_{\mathbb{A}^1}\mathcal{P}) \to 0, 
\end{align*}

which shows the assertion. 
\end{proof}

Note that this is not true in the analytic setting, as analytically open subsets of \(\mathbb{C}\) might not be \(\mathcal{D}\)-affine. However, we can at least get such a result on the level of stalks: 

\begin{lemma}
Let \(X \subset \mathbb{C}\) be open and \(\mathcal{P} \in \mathcal{D}(X)\). For any point \(p \in X\), we have the isomorphism 
\begin{align*}
(\mathcal{D}_X/\mathcal{D}_X\mathcal{P})_p \cong \mathcal{D}_{X, p}/\mathcal{D}_{X, p}\mathcal{P}
\end{align*}
\end{lemma}

\begin{proof}
We again consider the exact sequence 
\begin{align*}
0 \to \mathcal{D}_{X} \xrightarrow{\cdot \mathcal{P}} \mathcal{D}_{X} \to \mathcal{D}_{X}/\mathcal{D}_{X}\mathcal{P} \to 0 
\end{align*}

and conclude, because taking stalks is an exact functor. 
\end{proof}

\section{Singularities of \texorpdfstring{\(\mathcal{D}\)-modules}{D-modules}}

Let \(X\) be a complex manifold or a smooth complex variety. An important class of \(\mathcal{D}_X\)-modules is the class of holonomic \(\mathcal{D}_X\)-modules. We refer to \cite{Hot18}, Definition 2.3.6 for the general definition of holonomicity, and only recall the following: If \(X \subset \mathbb{C}\) or \(X \subset \mathbb{A}^1\) is an open subset, a coherent \(\mathcal{D}_X\)-module \(\mathcal{M}\) is holonomic if for every local section \(m \in \mathcal{M}\), there exists \(\mathcal{P} \in \mathcal{D}_X\) with \(\mathcal{P} \cdot m = 0\). For example, \(\mathcal{D}_X/\mathcal{D}_X\mathcal{P}\) is holonomic for every non-vanishing \(\mathcal{P} \in \mathcal{D}(X)\). 

\begin{definition}
We recall two notions of singular locus: 
\begin{itemize}
    \item Let \(X\) be a complex manifold or smooth complex variety and let \(\mathcal{M}\) be a holonomic \(\mathcal{D}_X\)-module with support \(X\). We define the singular locus \(\operatorname{Sing}(\mathcal{M})\) as the minimal subset of \(X\) such that \(\mathcal{M}|_{X - \operatorname{Sing}(\mathcal{M})}\) is \(\mathcal{O}_{X - \operatorname{Sing}(\mathcal{M})}\)-locally free of finite rank. (It is known that \(X - \operatorname{Sing}(\mathcal{M})\) is dense in \(X\)). If the support of \(\mathcal{M}\) is a closed subset \(Y\) of \(X\), one defines \(\operatorname{Sing}(\mathcal{M}) \coloneq \operatorname{Sing}(i^{*}\mathcal{M})\), where \(i^{*}\mathcal{M}\) is the pullback in the category of holonomic \(\mathcal{D}\)-modules along the closed inclusion \(i: Y \hookrightarrow X\). 
    \item Let \(X \subset \mathbb{C}\) or \(X \subset \mathbb{A}^1\) be an open set. Let \(\mathcal{P} \in \mathcal{D}(X)\) be written as \(\sum_{i = 0}^ma_i(t)\partial_t^i\), where \(a_i \in \mathcal{O}(X)\). The singular locus \(\operatorname{Sing}(\mathcal{P})\) of \(\mathcal{P}\) is defined as \((a_m)^{-1}(0) \subset X\).
\end{itemize}
\end{definition}

\begin{lemma}
Let \(X \subset \mathbb{C}\) or \(X \subset \mathbb{A}^1\) be an open set and let \(\mathcal{P} \in \mathcal{D}(X)\) be non-vanishing. We have 
\begin{align*}
\operatorname{Sing}(\mathcal{P}) = \operatorname{Sing}(\mathcal{D}_X/\mathcal{D}_X\mathcal{P}). 
\end{align*}

\end{lemma}

\begin{proof}
Cf. \cite{Hot18}, Exercise 2.2.6. 
\end{proof}

\subsection{Regular and irregular singularities}

As is well-known, singularities of both analytic and algebraic holonomic \(\mathcal{D}\)-modules can be classified as regular or irregular. We refer to Chapter 5 of \cite{Hot18} for the general definitions, and we only recall the following more classical notion: 

\begin{definition}
Let \(X \subset \mathbb{A}^1\) or \(X \subset \mathbb{C}\) be open. Let \(\mathcal{P} \in \mathcal{D}(X)\) of order \(n\), written as \(\mathcal{P} = \sum_i a_i(t)\partial_t^i\) for suitable \(a_i \in \mathcal{O}(X)\). \(\mathcal{P}\) is said to satisfy the Fuchs criterion at a point \(t_0 \in X\) if one of the following equivalent conditions is satisfied: 

\begin{itemize}
    \item For all \(k \in \{1, \ldots, n\}\), the quotient \(a_{n - k}(t)/a_n(t)\) has a pole of order at most \(k\) at \(t_0\). (Note that this condition is trivially fulfilled if \(t_0 \notin \operatorname{Sing}(\mathcal{P})\).)
    \item Every multi-valued solution \(f\) to the equation \(\mathcal{P}f = 0\) satisfies the equation
    \begin{align*}
    |f(t)| \leq C|t - t_0|^{-N}
    \end{align*}
    in any bounded sector starting at \(t_0\), for suitable \(C > 0, N \in \mathbb{N}\), as \(t \rightarrow t_0\). 
\end{itemize}
\end{definition}

Furthermore, we need the definition (cf. \cite{Hot18} Definition 5.1.21) of regular holonomic \(\mathcal{D}\)-modules on open subsets of \(\mathbb{A}^1\): 

\begin{definition}
Let \(X \subset \mathbb{A}^1\) be open. (Note that we are only in the algebraic setting here.) Consider the inclusion \(i: X \subset \mathbb{P}^1\). A holonomic \(\mathcal{D}_X\)-module \(\mathcal{M}\) is called regular holonomic if every singularity of \(i_*\mathcal{M}\) is regular. (We have \(\operatorname{Sing}(i_*\mathcal{M}) = \operatorname{Sing}(\mathcal{M}) \cup (\mathbb{P}^1 - X)\)). 
\end{definition}

\begin{remark}
We follow \cite{Hot18}, Example 5.1.22. Let \(\mathcal{P} \in \mathcal{D}\). The holonomic \(\mathcal{D}_{\mathbb{A}^1}\)-module \(\mathcal{D}_{\mathbb{A}^1}/\mathcal{D}_{\mathbb{A}^1}\mathcal{P}\) is regular holonomic if and only if \(\mathcal{P}\) satisfies the Fuchs criterion at every singular point and the same is true for the operator obtained from \(\mathcal{P}\) by the coordinate transform \(t' = 1/t\). (This transform is needed to ensure regularity of the singularity at \(\infty\).)
\end{remark}

\subsection{Regularity in the order \(1\) case}

An operator \(\mathcal{P} \in \mathcal{D}\) is called irreducible if it can not be written as \(\mathcal{A}\mathcal{B}\), where \(\mathcal{A}, \mathcal{B} \in \mathcal{D} - \mathbb{C}\). In the case where \(\mathcal{P}\) is irreducible and of order \(1\), the Fuchs criterion simplifies: 

\begin{theorem}\label{thm:checkregular}
Let \(\mathcal{P} \in \mathcal{D}\) be of order \(1\), written as \(\mathcal{P} = p(t)\partial_t + q(t)\) for suitable \(p, q \in \mathbb{C}[t]\). 

\begin{itemize}
    \item \(\mathcal{P}\) can be uniquely written as \(\partial_t p(t) + q^*(t)\) for suitable \(q^* \in \mathbb{C}[t]\). 
    \item \(\mathcal{P}\) is irreducible if and only if \(\gcd(p, q), \gcd(p, q^*) \in \mathbb{C}\). (That is, \(p\) shares no roots with \(q\) or \(q^*\)). 
    \item If \(\mathcal{P}\) is irreducible, a point \(t_0\) is a regular singularity of \(\mathcal{D}_{\mathbb{A}^1}/\mathcal{D}_{\mathbb{A}^1}\mathcal{P}\) if and only if \(p(t_0) = 0\) with multiplicity \(1\). 
    \item Let \(i : \mathbb{A}^1 \rightarrow \mathbb{P}^1\) be the inclusion. The singularity at \(\infty\) of \(i_*(\mathcal{D}_{\mathbb{A}^1}/\mathcal{D}_{\mathbb{A}^1}\mathcal{P})\) is regular if and only if \(\deg(p) > \deg(q)\). 
\end{itemize}
\end{theorem}

\begin{proof}
For the first point one verifies that 
\begin{align*}
\mathcal{P} = \partial_tp(t) - p'(t) + q(t), 
\end{align*}

where \(p'\) is the usual derivative of \(p\). For the second point, assume that \(\mathcal{P} = \mathcal{A}\mathcal{B}\) for some \(\mathcal{A}, \mathcal{B} \in \mathcal{D} - \mathbb{C}\). Then, either \(\operatorname{ord}(\mathcal{A}) = 1, \operatorname{ord}(\mathcal{B}) = 0\), or \(\operatorname{ord}(\mathcal{A}) = 0, \operatorname{ord}(\mathcal{B}) = 1\). For now, assume the latter. Then, \(\mathcal{A}\) would divide both \(p\) and \(q\), so that \(\gcd(p, q) \notin \mathbb{C}\). Similarly, \(\operatorname{ord}(\mathcal{A}) = 1, \operatorname{ord}(\mathcal{B}) = 0\) implies \(\gcd(p, q^*) \notin \mathbb{C}\). This shows the "if" part of the second point. The "only if" part is obvious. The last two points follow from the previous remark. 
\end{proof}

\section{\texorpdfstring{\(\mathcal{D}\)-modules}{D-modules} and holomorphic connections on \texorpdfstring{\(\mathbb{C}\)}{the complex plane}}

Let \(X\) be a complex manifold or a smooth complex variety. One calls vector bundles on \(X\) of finite rank, which are equipped with an integrable connection, simply integrable connections and denotes this category as \(\operatorname{Conn}(X\)), cf. \cite{Hot18}, Definition 1.2.2, Notation 4.1.1. 

We recall some well-known facts about \(\mathcal{D}\)-modules and integrable connections: Let \(X\) be a complex manifold (resp. a smooth complex variety). There is the well-known equivalence (cf. \cite{Hot18}, Lemma 1.2.1)

\[\begin{tikzcd}
	{\mathcal{D}_X \text{-modules which are locally free of finite rank over } \mathcal{O}_X} \\
	{\text{integrable connections on } X.}
	\arrow[tail reversed, from=1-1, to=2-1]
\end{tikzcd}\]

For a \(\mathcal{D}_X\)-module \(\mathcal{M}\), which is locally free over \(\mathcal{O}_X\) with finite rank, the corresponding integrable connection is the pair \((\mathcal{M}_{\mathcal{O}_X}, \nabla_{\mathcal{M}})\) given as follows: \(\mathcal{M}_{\mathcal{O}_X}\) is the underlying \(\mathcal{O}_X\)-module of \(\mathcal{M}\) and \(\nabla_{\mathcal{M}}\) is the connection induced by the action of \(\mathcal{D}\) on \(\mathcal{M}\). For example, if \(X\) is a curve with local coordinate \(t\), then \(\nabla_{\mathcal{M}}\) is directly induced by \(\partial_t\). 

\begin{remark}
Let \(X\) be a complex manifold and let \(\mathcal{M}\) be a holonomic \(\mathcal{D}_X\)-module with singular locus \(D\). By the definition of singular locus, the restriction \(\mathcal{M}|_{X-D}\) gives an integrable connection, whose rank we denote by \(n\). We can consider the rank \(n\) local system \(\ker(\nabla_{\mathcal{M}})\) on \(X - D\). (Recall that a rank \(n\) local system is a locally constant sheaf of \(n\)-dimensional \(\mathbb{C}\)-vector spaces.) On the other hand, if \(X\) is a smooth variety and \(\mathcal{M}\) is a holonomic \(\mathcal{D}_X\), then \(\ker(\nabla_{\mathcal{M}})\) does in general not form a rank \(n\) local system. However, we can instead consider \(\ker(\nabla_{\mathcal{M}^{\operatorname{an}}})\). 
\end{remark}

In the following proposition, we make these constructions explicit for a \(\mathcal{D}_{\mathbb{A}^1}\)-module of the form \(\mathcal{D}_{\mathbb{A}^1}/\mathcal{D}_{\mathbb{A}^1}\mathcal{P}\) for any \(\mathcal{P} \in \mathcal{D}\) of positive order. These results are known, but we collect the proofs for completeness. 

\begin{proposition}\label{prop:dmodsandconns}
Let \(\mathcal{P} \in \mathcal{D}\) be of order \(n \in \mathbb{N}_+\). Write \(D = \{p_i : i = 1, \ldots, m\} \coloneqq \operatorname{Sing}(\mathcal{P})\) and set \(U \coloneqq \mathbb{C}_t - D\).  We write \(\mathcal{P} = \sum_{i = 0}^n a_i(t)\partial_t^i\), where \(a_i(t) \in \mathbb{C}[t]\). Abbreviate  
\begin{align*}
\mathbb{C}[t, (t - p_i)^{-1}] & \coloneqq \mathbb{C}[t, (t - p_1)^{-1}, \ldots, (t - p_m)^{-1}], \\ \mathcal{D}[(t - p_i)^{-1}] & \coloneqq \mathcal{D}[(t - p_1)^{-1}, \ldots, (t - p_m)^{-1}]. 
\end{align*}

Then, the following statements hold: 

\begin{itemize}
    \item \(B \coloneqq \left\{[1], [\partial_t], \ldots, [\partial_t^{n-1}]\right\} \subset \mathcal{D} / \mathcal{D} \mathcal{P}\) is linearly independent over \(\mathbb{C}[t]\). 
    \item \(\mathcal{D}[(t - p_i)^{-1}]/\mathcal{D}[(t - p_i)^{-1}] \mathcal{P}\) is free over \(\mathbb{C}[t, (t - p_i)^{-1}]\) with basis 
    \begin{align*}
    B_U \coloneqq \left\{[1], [\partial_t], \ldots, [\partial_t^{n-1}]\right\} \subset \mathcal{D}[(t - p_i)^{-1}] / \mathcal{D}[(t - p_i)^{-1}] \mathcal{P}
    \end{align*}

    \item Consider the \(\mathcal{D}_{U^{\operatorname{an}}}\)-module 
    \begin{align*}
    \mathcal{M} \coloneqq \mathcal{O}_{U^{\operatorname{an}}} \otimes \mathcal{D}[(t - p_i)^{-1}] / \mathcal{D}[(t - p_i)^{-1}]\mathcal{P}. 
    \end{align*}
    
    The connection induced on \(\mathcal{M}\) by the action of \(\partial_t\) gives a rank \(n\) holomorphic vector bundle with integrable connection on \(U^{\operatorname{an}}\). (The bundle is automatically trivial.) In the \(\mathcal{O}_{U^{\operatorname{an}}}\)-basis induced by \(B_U\), this connection is given by \(d - Cdt\), where 
    \begin{align*}
    C = \begin{pmatrix}
    0 & 0 & \cdots & 0 & a_0(t)/a_n(t) \\
    -1 & 0 & \cdots & 0 & a_1(t)/a_n(t) \\
    0 & -1 & \cdots & 0 & a_2(t)/a_n(t) \\
    \vdots & \vdots & \ddots & \vdots & \vdots \\
    0 & 0 & \cdots & -1 & a_{n-1}(t)/a_n(t)
    \end{pmatrix}
    \end{align*}
    \item The system of differential equations equivalent to the equation \(\mathcal{P}f = 0\) is 
    \begin{align*}
    \frac{d}{dt}f(t) = -C^Tf(t)
    \end{align*}
    \item The \(\mathcal{O}_{U^{\operatorname{an}}}\)-bundle with connection \(\left(\mathcal{O}_{U^{\operatorname{an}}}^{\oplus n}, d + C^T dt\right)\) is dual to the \(\mathcal{O}_{U^{\operatorname{an}}}\)-bundle with connection induced by \(\mathcal{M}|_{U^{\operatorname{an}}}\). If \(F\) is a fundamental matrix of \(\left(\mathcal{O}_{U^{\operatorname{an}}}^{\oplus n}, d + C^T dt\right)\), then \((F^{-1})^T\) is one of \(\mathcal{M}|_{U^{\operatorname{an}}}\). 
\end{itemize}
\end{proposition}

\begin{proof}
For the first point, suppose there was a relation 
\begin{align*}
0 = \sum_{i = 0}^{n-1} c_i(t) [\partial_t^i]
\end{align*}

in \( \mathcal{D}/\mathcal{D} \mathcal{P}\) for some \(c_i \in \mathbb{C}[t]\), with \(c_i \neq 0\) for at least one \(i\). Then, we would have 
\begin{align*}
\sum_{i = 0}^{n-1} c_i(t) \partial_t^i \in \mathcal{D}\mathcal{P}, 
\end{align*}

so that \(\sum_{i = 0}^{n-1} c_i(t) \partial_t^i\) can be written as \(\mathcal{T}\mathcal{P}\) for some \(\mathcal{T} \in \mathcal{D}\). However, the order of \(\mathcal{T} \mathcal{P}\) has to be at least equal to the order of \(\mathcal{P}\), which is \(n\). This gives the desired contradiction, because \(\sum_{i = 0}^{n-1} c_i(t) \partial_t^i\), has order at most \(n-1\). 

We now show the second point. The linear independence of the elements of \(B_U\) over \(\mathbb{C}[t, (t - p_i)^{-1}]\) is proven as in the first point, which we will not repeat here. It remains to show that every element in \(\mathcal{D}[(t - p_i)^{-1}] / \mathcal{D}[(t - p_i)^{-1}] \mathcal{P}\) is expressed as a \(\mathbb{C}[t, (t - p_i)^{-1}]\)-linear combination of the elements of \(B_U\): Fix \(\mathcal{T} \in \mathcal{D}[(t - p_i)^{-1}]\). By the Euclidean algorithm for differential operators (cf. \cite{Sab90}, Proposition 2.1.1) one can write 
\begin{align*}
\mathcal{T} = \mathcal{S}\mathcal{P} + \mathcal{R}, 
\end{align*}

where \(\mathcal{S}, \mathcal{R} \in \mathcal{D}[(t - p_i)^{-1}]\) with the order of \(\mathcal{R}\) being less than \(n\). In other words, we obtain \([\mathcal{T}] = [\mathcal{R}]\) in \(\mathcal{D}[(t - p_i)^{-1}] / \mathcal{D}[(t - p_i)^{-1}] \mathcal{P}\). Because the order of \(\mathcal{R}\) is less than \(n\), \([\mathcal{R}]\) is generated over \(\mathbb{C}[t, (t - p_i)^{-1}]\) by the classes of \(1, \partial_t, \ldots, \partial_t^{n-1}\) in \(\mathcal{D}[(t - p_i)^{-1}] / \mathcal{D}[(t - p_i)^{-1}] \mathcal{P}\), and so the assertion of the second point follows. 

For the third assertion, consider a section \(s \in \mathcal{D}[(t - p_i)^{-1}] / \mathcal{D}[(t - p_i)^{-1}] \mathcal{P}\), written as \(\sum_{i = 0}^{n-1} c_i(t)[\partial_t^i]\) for \(c_i \in \mathbb{C}[t, (t - p_i)^{-1}]\). Because \(\partial_t\) satisfies the Leibniz rule, it follows that 
\begin{align*}
\partial_t s = \sum_{i = 0}^{n-1} c_i'(t) [\partial_t^i] + c_i(t) [\partial_t^{i+1}]. 
\end{align*}

in \(\mathcal{D}[(t - p_i)^{-1}] / \mathcal{D}[(t - p_i)^{-1}] \mathcal{P}\). We have the relation \([\mathcal{P}] = 0\), so we obtain 
\begin{align*}
[\partial_t^{n}] = - \sum_{i = 0}^{n - 1} \frac{a_i(t)}{a_n(t)}[\partial_t^i]. 
\end{align*}

Writing \(s = (s_0, \ldots, s_{n-1})^T\) in our fixed basis, we thus get \(\partial_t s = ds - Cdt\) with \(C\) defined as in the proposition, which is the claim of the third point. (Note that the integrability of this connection is clear, because \(U^{\operatorname{an}}\) is one-dimensional.) The fourth point follows directly from the usual translation of order \(n\) differential equations into systems of \(n\) first-order equations. For the fifth point, it is well-known that duality of connections transforms the connection matrix \(C\) into \(-C^T\). The remainder of the proof is an elementary calculation. 

\end{proof}

\subsection{Application to the regular order \(1\) case}

The solutions of regular order \(1\) operators take an easy form: 

\begin{lemma}\label{lemma:regsol}
Let \(\mathcal{P} \in \mathcal{D}\) be of order \(1\), written as \(p(t)\partial_t + q(t)\). Assume that \(\mathcal{P}\) only has regular singularities (including at \(\infty\)). Writing 
\begin{align*}
q(t)/p(t) = -\sum_i \frac{m_i}{t - r_i}
\end{align*}

for suitable \(m_i, r_i \in \mathbb{C}\), the solutions of \(\mathcal{P}\) are spanned (over \(\mathbb{C}\)) by \(\prod_i (t - r_i)^{m_i}\). 
\end{lemma}

\begin{proof}
First note that we can in fact write \(q(t)/p(t)\) as \(-\sum_i \frac{m_i}{t - r_i}\) for suitable \(m_i, r_i \in \mathbb{C}\). Indeed, because \(\mathcal{P}\) has a regular singularity at \(\infty\), we have \(\deg p > \deg q\) by \Cref{thm:checkregular}, so that \(q(t)/p(t)\) contains no polynomial part. Moreover, because the singularities at finite distance are also assumed to be regular, \(q(t)/p(t)\) only has poles of order \(1\) by \Cref{thm:checkregular}. The remainder of the proof is a standard calculation. 
\end{proof}

\begin{remark}
For \(\mathcal{P}\) as in \Cref{lemma:regsol}, the local system given by the flat sections of 
\begin{align*}
(\mathcal{D}_{\mathbb{C}}/\mathcal{D}_{\mathbb{C}}\mathcal{P})|_{\mathbb{C}-\operatorname{Sing}(\mathcal{P})}
\end{align*}

is spanned by \(\prod_i (t - r_i)^{-m_i}\), which follows from \Cref{prop:dmodsandconns}. 
\end{remark}

\section{Localisation and minimal extension of \texorpdfstring{\(\mathcal{D}\)-modules}{D-modules}}\label{sect:duallocmin}

Let \(X\) be a complex manifold or smooth complex variety. Recall that holonomic \(\mathcal{D}_X\)-modules, away from their singular locus, are just vector bundles of finite rank with an integrable connection. However, understanding the structure at the singular points requires more investigation: For example, consider the family of \(\mathcal{D}_{\mathbb{C}}\)-modules \(\mathcal{D}_{\mathbb{C}}/\mathcal{D}_{\mathbb{C}}(t\partial_t - k)\) (\(k \in \mathbb{Z}\)). We have \(\operatorname{Sing}(\mathcal{D}_{\mathbb{C}}/\mathcal{D}_{\mathbb{C}}(t\partial_t - k)) = \{0\}\) for all \(k\) and every \(\mathcal{D}_{\mathbb{C}}/\mathcal{D}_{\mathbb{C}}(t\partial_t - k)\) is regular holonomic, which follows from \Cref{thm:checkregular}. Moreover, \Cref{prop:dmodsandconns} implies that 
\begin{align*}
(\mathcal{D}_{\mathbb{C}}/\mathcal{D}_{\mathbb{C}}(t\partial_t - k))\vert_{\mathbb{C} - \{0\}} \cong (\mathcal{D}_{\mathbb{C}}/\mathcal{D}_{\mathbb{C}}(t\partial_t))\vert_{\mathbb{C} - \{0\}}
\end{align*}

However, as we will see later, we do not have 
\begin{align*}
(\mathcal{D}_{\mathbb{C}}/\mathcal{D}_{\mathbb{C}}(t\partial_t - k)) \cong (\mathcal{D}_{\mathbb{C}}/\mathcal{D}_{\mathbb{C}}(t\partial_t)). 
\end{align*}

for all \(k \in \mathbb{Z}\). In other words, regular holonomic \(\mathcal{D}_{\mathbb{C}}\)-modules with the same singular loci and isomorphic associated local systems might not be isomorphic. In this section, we will introduce the localisation functor \(\bullet(*D)\) for \(\mathcal{D}\)-modules, where \(D\) is any reduced divisor in \(X\). Localisation "forgets" this subtlety: For \(\mathcal{D}_X\)-modules \(\mathcal{M}\) and \(\mathcal{N}\) with singular locus \(D \subset X\), which are regular holonomic with isomorphic associated local systems, the Riemann-Hilbert correspondence (\Cref{lemma:locforgets}) shows that one has 
\begin{align*}
\mathcal{M}(*D) \cong \mathcal{N}(*D). 
\end{align*}

\subsection{Tensor product of \(\mathcal{D}\)-modules}

We will also introduce the dual localisation and minimal extension functors, which are defined using the localisation functor. As a preliminary, we recall the definition of the tensor product of \(\mathcal{D}\)-modules: 

\begin{definition}[\cite{Hot18}, Proposition 1.2.9]
Let \(X\) be a complex manifold or smooth complex variety and let \(\mathcal{M}, \mathcal{N}\) be \(\mathcal{D}_X\)-modules. The tensor product \(\mathcal{M} \otimes_{\mathcal{O}_X} \mathcal{N}\) is defined as the usual tensor product of the underlying \(\mathcal{O}_X\)-modules, equipped with the following \(\mathcal{D}_X\)-module structure: For any derivation \(\partial_t \in \mathcal{D}_X\) and local sections \(m \in \mathcal{M}, n \in \mathcal{N}\), define 
\begin{align*}
\partial_t (m \otimes n) \coloneqq (\partial_t m \otimes n) + (m \otimes \partial_t n). 
\end{align*}

This defines a \(\mathcal{D}_X\)-module structure, because (by definition) \(\mathcal{D}_X\) is generated by the sheaf of derivations as an \(\mathcal{O}_X\)-algebra. 
\end{definition}

\begin{example}\label{example:tensorhol}
We briefly recall a standard calculation: Let \(X\) be an algebraic or analytic curve and let \(\mathcal{M}, \mathcal{N}\) be integrable connections on \(X\). Then, we may restrict to an open subset \(Y\) on which \(\mathcal{M}\) and \(\mathcal{N}\) are \(\mathcal{O}_Y\)-free. After choosing a local coordinate \(t\), \(\mathcal{M}|_Y, \mathcal{N}|_Y\) are isomorphic to trivial holomorphic bundles with connection 
\begin{align*}
(\mathcal{O}_Y^{\oplus m}, d - A(t)dt), \quad (\mathcal{O}_Y^{\oplus n}, d - B(t)dt), 
\end{align*}

where \(m, n \in \mathbb{N}\) and \(A \in \mathcal{O}(Y)^{m \times m}, B \in\mathcal{O}(Y)^{n \times n}\). (Note that these isomorphisms are not canonical, but depend on the choice of bases.) Let \((e_1, \ldots, e_m)\) and \((f_1, \ldots, f_n)\) be the standard bases of \(\mathcal{O}_Y^{\oplus m}\) and \(\mathcal{O}_Y^{\oplus n}\), respectively. They induce the basis \((e_i \otimes f_j)\) on \(\mathcal{O}_Y^{\oplus m} \otimes \mathcal{O}_Y^{\oplus n}\), which gives an isomorphism of \(\mathcal{O}_Y\)-modules 
\begin{align*}
\mathcal{O}_Y^{\oplus m} \otimes \mathcal{O}_Y^{\oplus n} \cong \mathcal{M}|_Y \otimes_{\mathcal{O}_Y} \mathcal{N}|_Y. 
\end{align*}

One can now express the \(\mathcal{D}_Y\)-module structure of \(\mathcal{M}|_Y \otimes_{\mathcal{O}_Y} \mathcal{N}|_Y\), which is determined by the action of \(\partial_t\), along this isomorphism. If \(n = 1\), we have \(B \in \mathcal{O}_Y\) and 
\begin{align*}
\mathcal{M} \otimes_{\mathcal{O}_Y} \mathcal{N} \cong (\mathcal{O}_Y^{\oplus m}, d - Adt - B \cdot \mathbf{I}_{m}dt), 
\end{align*}

where \(\mathbf{I}_{m}\) denotes the identity matrix of size \(m\). 
\end{example}

\subsection{Meromorphic connections and localisation of \(\mathcal{D}\)-modules}

\begin{definition}
Let \(X\) be a complex manifold or smooth complex variety and let \(\mathcal{M}\) be a holonomic \(\mathcal{D}_X\)-module. Furthermore, let \(D \subset X\) be a reduced divisor (for example \(D = \operatorname{Sing}(\mathcal{M})\)). 
\begin{itemize}
    \item \(\mathcal{O}_X(*D)\) is the sheaf of meromorphic functions on \(X\), which are only allowed to have poles along \(D\). 
    \item The localisation \(\mathcal{M}(*D)\) of \(\mathcal{M}\) away from \(D\) is the \(\mathcal{D}_X\)-module defined as 
    \begin{align*}
    \mathcal{M}(*D) \coloneqq \mathcal{M} \otimes_{\mathcal{O}_X} \mathcal{O}_X(*D). 
    \end{align*}
    \item If \(D = \operatorname{Sing}(\mathcal{M})\) and \(\mathcal{M} = \mathcal{M}(*D)\), then \(\mathcal{M}\) is called an algebraic connection (in the algebraic case) or a meromorphic connection (in the analytic case) on \((X, D)\). One also says that \(\mathcal{M}\) is localised away from \(D\). 
\end{itemize}
\end{definition}

We keep the notation of the above. Let \(f \in \mathcal{O}_X\) be any function on an open set \(V \subset X\) such that the vanishing locus of \(f\) locally describes \(D\). \(\mathcal{M}\) is localized away from \(D\) if and only if the action of \(f\) on \(\mathcal{M}\) is invertible for all such \(f\): Indeed, \(\mathcal{O}_X(*D)\) is precisely defined to make all such \(f\) invertible. The localisation functor is idempotent in the sense that \((\mathcal{M}(*D))(*D) = \mathcal{M}(*D)\), which is easily verified. Furthermore, we have \(\mathcal{M}(*D)\vert_{X - D} = \mathcal{M}\vert_{X - D}\), which follows from \(\mathcal{O}_X(*D)\vert_{X - D} = \mathcal{O}_{X - D}\). 

A generalisation of holomorphic connections are meromorphic connections: 

\begin{definition}
Let \(X\) be a manifold with a reduced divisor \(D \subset X\). A meromorphic connection on \((X, D)\) is a pair \(\widetilde{V}, \nabla\), where \(\widetilde{V}\) is \(\mathcal{O}_X(*D)\)-locally free (of finite rank) and where \(\nabla\) is a map \(\tilde{V} \rightarrow \tilde{V} \otimes \Omega^1_X\) satisfying the Leibniz rule. If \(X\) is an algebraic variety, we call the analogous objects algebraic connections on \((X, D)\). 
\end{definition}

Every flat meromorphic or algebraic connection \(\widetilde{V}, \nabla\) on \((X, D)\) can be seen as a \(\mathcal{D}_{X}\)-module \(\mathcal{M}\), satisfying \(\mathcal{M}(*D) = \mathcal{M}\). 

\begin{example}\label{example:tensormer}
Similarly to \Cref{example:tensorhol}, one can show the following: Let \(X\) be a curve and let \(\mathcal{M}\) resp. \(\mathcal{N}\) be a meromorphic connection on \((X, D')\) resp. \((X, D'')\), where \(D', D''\) are finite subsets of \(X\). Let \(m\) be the rank of \(\mathcal{M}\) as a locally free \(\mathcal{O}(*D')\)-module and assume that the rank of \(\mathcal{N}\) is \(1\). Locally, in a chosen \(\mathcal{O}(*D')\)-basis of \(\mathcal{M}\) and a chosen \(\mathcal{O}(*D'')\)-basis of \(\mathcal{N}\), write the operators \(\nabla_{\mathcal{M}},\nabla_{\mathcal{N}}\) as \(d - Adt, d - Bdt\), where \(A \in \mathcal{O}(*D')^{m \times m}, B \in \mathcal{O}(*D'')\). Then, \(\mathcal{M} \otimes \mathcal{N}\) is a meromorphic connection on \((X, D' \cup D'')\) and, locally in the induced basis, its connection operator \(\nabla_{\mathcal{M} \otimes \mathcal{N}}\) is \(d - Adt - B \cdot \mathbf{I}_{m}dt\). 
\end{example}

\subsection{Algebraic connections of rank \(1\)}

We investigate the rank \(1\) case of algebraic connections on \(\mathbb{A}^1\) in more detail: In this case, we can make the correspondence between algebraic connections and \(\mathcal{D}\)-modules more explicit: 

\begin{theorem}\label{theorem:mero1}
Let \(\mathcal{P} \in \mathcal{D}\) be of order \(1\) with singular locus \(D\). Assume that \(\mathcal{M} \coloneqq \mathcal{D}_{\mathbb{A}^1}/\mathcal{D}_{\mathbb{A}^1}\mathcal{P}\) satisfies \(\mathcal{M}(*D) = \mathcal{M}\). Let \(f\) be a non-trivial multi-valued solution of \(\mathcal{P}f = 0\) on \(\mathbb{A}^1 - D\). Then, we have the isomorphism 
\begin{align*}
m: \mathcal{M} \rightarrow \left(\mathcal{O}_{\mathbb{A}^1}(*D), d + \frac{f'}{f}dt\right), \quad [s] \mapsto \frac{s(f)}{f}. 
\end{align*}
\end{theorem}

\begin{proof}
First, note that \(m\) is independent of the choice of \(f\): Indeed, since the order of \(\mathcal{P}\) is \(1\), another non-trivial solution \(g\) to \(\mathcal{P}g = 0\) has to be equal to \(cf\) for some \(c \in \mathbb{C}^*\), so that \(f'/f = g'/g\) and \(s(f)/f = s(g)/g\). Next, we need to show that \((\mathcal{O}_{\mathbb{A}^1}(*D), d + f'/f \cdot dt)\) is a well-defined object, which requires that \(f'/f\) is single-valued. This is verified as follows: Write \(\mathcal{P} = p(t)\partial_t + q(t)\). Because \(\mathcal{P}f = 0\), we have \(f'/f = -q(t)/p(t)\), which is single-valued. 

We also need to show that the map \(m\) is well-defined: Firstly, we need to show that the mapping is independent of the choice of representative of the class \([s]\). Indeed, if \([s] = 0\), then we have \(s \in \mathcal{D}_{\mathbb{A}^1}\mathcal{P}\), so that \(s(f) = 0\), which also gives \(s(f)/f = 0\). Secondly, we need to show that \(s(f)/f\) is a single-valued function for all \(s \in \mathcal{D}_{\mathbb{A}^1}\). This is deduced from the fact that \(s(f)\) has the same monodromy as \(f\) (which can only be multiplicative), so that \(s(f)/f\) has trivial monodromy, or in other words, this quotient is a single-valued global section of \(\mathcal{O}_{\mathbb{A}^1}(*D)\). This shows the well-definedness of \(m\). 

Clearly, \(m\) is \(\mathcal{O}_{\mathbb{A}^1}(*D)\)-linear. We will now show that \(m\) exchanges the connection \(d + f'/f \cdot dt\) with the action of \(\partial_t\) on \(\mathcal{M}\): For any \(s \in \mathcal{D}\), we have 
\begin{align*}
d(m([s])) + m([s]) \cdot \frac{f'}{f}dt & = d\left(\frac{s(f)}{f}\right) + \frac{s(f)}{f} \cdot \frac{f'}{f} dt\\ & = \frac{(\partial_ts)(f)}{f}dt \\ & = m(\partial_t s)dt. 
\end{align*}

It remains to show the bijectivity of \(m\). For any section \(h \in \mathcal{O}_{\mathbb{A}^1}(*D)\), we can consider \(h\) as a section of \(\mathcal{D}_{\mathbb{A}^1}(*D)\). Since \(\mathcal{M}\) is an algebraic connection by assumption, we have 
\begin{align*}
\mathcal{M} \cong \mathcal{D}_{\mathbb{A}^1}(*D)/\mathcal{D}_{\mathbb{A}^1}(*D)\mathcal{P}, 
\end{align*}

because localisation is an exact functor, and so \(h\) defines a class in \(\mathcal{M}\), which is mapped to \(h(f)/f = h\) under \(m\). This shows the surjectivity of \(m\). The injectivity is seen as follows: Let \(s \in \mathcal{D}_{\mathbb{A}^1}\) and suppose \(m([s]) = 0\), implying \(s(f) = 0\). Using the Euclidean algorithm for differential operators, we can deduce that this implies \(s \in \mathcal{D}_{\mathbb{A}^1}\mathcal{P}\), which implies that \([s] = 0 \in \mathcal{M}\), which shows the injectivity of \(m\). This completes the proof. 
\end{proof}

An important class of rank \(1\) algebraic or meromorphic connections is given in the following definition: 

\begin{definition}\label{def:exptwist}
Let \(f\) be a rational function \(\mathbb{A}^1\) with pole divisor \(D \subset \mathbb{A}^1\). The meromorphic connection \(\mathcal{E}^f\) on \((\mathbb{A}^1, D)\) is defined as \(\mathcal{O}_{\mathbb{A}^1}(*D)\) with the operator \(\nabla_{\mathcal{E}^f}\) given by \(d - df\). We make the analogous definition in the analytic case on \(\mathbb{C}\), where \(f\) is any meromorphic function. Tensoring a \(\mathcal{D}\)-module on \(\mathbb{A}^1\) or \(\mathbb{C}\) with \(\mathcal{E}^f\) is sometimes called "exponential twist". 
\end{definition}

\subsection{Distinctions between algebraic and analytic localisation}

Assume that \(X\) is a smooth complex variety. Then, we have \(\mathcal{M}(*D) = j_*(\mathcal{M}\vert_{X - D})\) as \(\mathcal{D}_X\)-modules, where \(j_*\) denotes pushforward of \(\mathcal{D}\)-modules along the inclusion \(X - D \hookrightarrow X\). This comes from the fact that the pushforward \(j_*(\mathcal{O}_{X - D})\) is isomorphic to \(\mathcal{O}_X(*D)\). In particular: For any \(\mathcal{D}_X\)-modules \(\mathcal{M}\), \(\mathcal{N}\), we have 
\begin{align*}
\mathcal{M}\vert_{X - D} = \mathcal{N}\vert_{X - D} \implies \mathcal{M}(*D) = \mathcal{N}(*D). 
\end{align*}

This is not true if \(X\) is a complex manifold. The reason is that \(j_*(\mathcal{O}_{X - D})\) then consists of all functions on \(X\) which are holomorphic away from \(D\), but might have essential singularities at \(D\). That is, sections of \(j_*(\mathcal{O}_{X - D})\) might not be meromorphic. 

This observation gives rise to the following distinction: Let \(X\) be a smooth variety and consider holonomic \(D_X\)-modules \(\mathcal{M}, \mathcal{N}\) with singular locus \(D\). We have the analytification \(X^{{\operatorname{an}}}\) and the \(\mathcal{D}_{X^{{\operatorname{an}}}}\)-modules \(\mathcal{M}^{{\operatorname{an}}}, \mathcal{N}^{{\operatorname{an}}}\). Consider the local systems \(\ker(\nabla_{\mathcal{M}^{\operatorname{an}}\vert_{X - D}})\) and \(\ker(\nabla_{\mathcal{N}^{\operatorname{an}}\vert_{X - D}})\) on \(X - D\) associated to \(\mathcal{M}\) and \(\mathcal{N}\). We have 
\begin{align*}
\ker(\nabla_{\mathcal{M}^{\operatorname{an}}\vert_{X - D}}) \cong \ker(\nabla_{\mathcal{N}^{\operatorname{an}}\vert_{X - D}}) \implies \mathcal{M}^{{\operatorname{an}}}\vert_{X^{{\operatorname{an}}} - D^{{\operatorname{an}}}} \cong  \mathcal{N}^{{\operatorname{an}}}\vert_{X^{{\operatorname{an}}} - D^{{\operatorname{an}}}}, 
\end{align*}

cf. \cite{Hot18}, Theorem 4.2.4. However, in the case \(\ker(\nabla_{\mathcal{M}^{\operatorname{an}}\vert_{X - D}}) \cong \ker(\nabla_{\mathcal{N}^{\operatorname{an}}\vert_{X - D}})\) it can still happen that \(\mathcal{M}\vert_{X - D} \not\cong \mathcal{N}\vert_{X - D}\), or equivalently \(\mathcal{M}(*D) \not\cong \mathcal{N}(*D)\). This is because restricting to \(X - D\) in the analytic category "forgets" the singular data of irregularity (and the Stokes phenomenon), and only "remembers" the monodromy data, whereas, in the algebraic category, restriction to \(X - D\) does not forget this information. We shall make this more explicit with the following example. 

\begin{example}
We work over \(X = \mathbb{A}^1\) with coordinate \(t\). Consider the algebraic connection \(\mathcal{M} = (\mathcal{O}_{X}(*0), d)\) on \((X, \{0\})\) and take the exponential twist 
\begin{align*}
\mathcal{N} \coloneqq \mathcal{M} \otimes \mathcal{E}^{-1/t} = \left(\mathcal{O}_{\mathbb{A}^1}(*0), d - \frac{1}{t^2} dt\right). 
\end{align*}

The local systems on \(\mathbb{C}^*\) associated to \(\mathcal{M}\) and \(\mathcal{N}\) are both trivial. In particular, they are isomorphic. \(\mathcal{M}\) is regular holonomic, whereas \(\mathcal{N}\) is holonomic with an irregular singularity at the origin. We have 
\begin{align*}
\mathcal{M}^{\operatorname{an}} = (\mathcal{O}_{\mathbb{C}}(*0), d), \quad \mathcal{N}^{\operatorname{an}} = \left(\mathcal{O}_{\mathbb{C}}(*0), d - \frac{1}{t^2} dt\right) 
\end{align*}

as meromorphic connections, cf. \Cref{example:tensormer}. Now, on \(\mathcal{M}^{\operatorname{an}}|_{\mathbb{C}^*}\), the base change \(1 \mapsto \exp(-1/t)\) gives \(\mathcal{N}^{\operatorname{an}}|_{\mathbb{C}^*}\), inducing an isomorphism \(\mathcal{M}^{\operatorname{an}}|_{\mathbb{C}^*} \cong \mathcal{N}^{\operatorname{an}}|_{\mathbb{C}^*}\). The fact that the function \(\exp{(-1/t)}\) is not meromorphic at \(0\) reflects the fact \(\mathcal{M}^{\operatorname{an}} \not\cong \mathcal{N}^{\operatorname{an}}\) as meromorphic connections and the fact that \(\exp{1/t}\) is not algebraic reflects the fact that \(\mathcal{M}\vert_{\mathbb{A}^1 - \{0\}} \not\cong \mathcal{N}\vert_{\mathbb{A}^1 - \{0\}}\). 
\end{example}

However, as mentioned in the beginning of \Cref{sect:duallocmin}, if we assume that \(\mathcal{M}\) and \(\mathcal{N}\) are regular holonomic, the situation simplifies both in the algebraic and in the analytic case: 

\begin{lemma}[Riemann-Hilbert correspondence]\label{lemma:locforgets}
We have the following equivalences: 

\begin{itemize}
    \item Let \(X\) be a complex manifold. Let \(\mathcal{M}\) and \(\mathcal{N}\) be regular holonomic \(\mathcal{D}_X\)-modules. Setting \(D \coloneqq \operatorname{Sing}(\mathcal{M}) \cup \operatorname{Sing}(\mathcal{N})\), we have: 
    \begin{align*}
    \ker(\nabla_{\mathcal{M}\vert_{X - D}}) \cong \ker(\nabla_{\mathcal{N}\vert_{X - D}}) \iff \mathcal{M}(*D) \cong \mathcal{N}(*D). 
    \end{align*}
    \item Let \(X\) be a smooth complex variety. Let \(\mathcal{M}\) and \(\mathcal{N}\) be regular holonomic \(\mathcal{D}_X\)-modules. Setting \(D \coloneqq \operatorname{Sing}(\mathcal{M}) \cup \operatorname{Sing}(\mathcal{N})\), we have: 
    \begin{align*}
    \ker(\nabla_{\mathcal{M}^{\operatorname{an}}\vert_{X - D}}) \cong \ker(\nabla_{\mathcal{N}^{\operatorname{an}}\vert_{X - D}}) \iff \mathcal{M}(*D) \cong \mathcal{N}(*D). 
    \end{align*}
\end{itemize}
\end{lemma}

\begin{proof}
The first point is the analytic version of Deligne's Riemann-Hilbert correspondence, cf. \cite{Hot18}, Theorem 5.2.20, and the second point is the algebraic version of this, cf. \cite{Hot18}, Theorem 5.3.10. 
\end{proof}

The Riemann-Hilbert correspondence has a generalisation to holonomic \(\mathcal{D}_X\)-modules with possibly irregular singularities, cf. \cite{Sab13}. The localisation functor gives rise to other functors - dual localisation and minimal extension -, which will also be important later. We describe them in the following. 

\subsection{Dual localisation}

Firstly, we recall the duality functor for \(\mathcal{D}\)-modules. For a detailed exposition, we refer to \cite{Hot18}, 2.6. Let \(X\) be a complex manifold or a smooth complex variety. Let \(\mathcal{M}\) be a holonomic \(\mathcal{D}_X\)-module. 

\begin{definition}
The dual of \(\mathcal{M}\) is defined as 
\begin{align*}
\mathbb{D}_X \mathcal{M} \coloneqq \mathrm{R} \mathcal{H}om_{\mathcal{D}_X}\left(\mathcal{M}, \mathcal{D}_X\right) \otimes_{\mathcal{O}_X} \Omega_X^{n, \otimes-1}[n]. 
\end{align*}

Here, \(\Omega_X^{n, \otimes-1}\) denotes the dual of \(\Omega_X^{n}\) as an \(\mathcal{O}_X\)-module and \([n]\) indicates a shift of a complex by \(n\) "to the left". If the context is clear, we write \(\mathbb{D}\) instead of \(\mathbb{D}_X\). \(\mathbb{D}_X\) gives a contravariant endofunctor of the derived category of \(\mathcal{D}_X\)-modules. 
\end{definition}

\begin{remark}
We briefly comment on this definition: Tensoring with \(\Omega_X^{\otimes-1}\) is the usual operation of turning a right \(\mathcal{D}_X\)-module into a left \(\mathcal{D}_X\)-module, cf. \cite{Hot18}, Proposition 1.2.12. If \(\mathcal{M}\) is holonomic, we have 
\begin{align*}
\mathbb{D}_X \mathcal{M} = \mathcal{E} x t_{\mathcal{D}_X}^n\left(\mathcal{M}, \mathcal{D}_X\right) \otimes_{\mathcal{O}_X} \Omega_X^{n, \otimes-1}, 
\end{align*}

because only the \(n\)-th cohomology group of \(\mathrm{R} \mathcal{H}om_{\mathcal{D}_X}\left(\mathcal{M}^{\cdot}, \mathcal{D}_X\right)\) might not vanish, cf. \cite{Hot18}, Corollary 2.6.8 (iii). In fact, \(\mathbb{D}_X\) restricts to an endofunctor of the (non-derived) category of holonomic \(\mathcal{D}_X\)-modules. 
\end{remark}

This notion of duality is compatible with the formal adjoint of a differential operator: 

\begin{definition}
Let \(X\) be a manifold or smooth variety. The formal adjoint map \((\cdot)^*\) is the \(\mathbb{C}\)-linear map \(\mathcal{D}_X \rightarrow \mathcal{D}_X\) characterized by: 

\begin{itemize}
    \item For any \(f \in \mathcal{O}_X\), we have \(f^* = f\), and if \(t\) is any local coordinate of \(X\), then \((\partial_t)^* = -\partial_t\). 
    \item For any \(\mathcal{P}, \mathcal{Q} \in \mathcal{D}_X\), we have \((\mathcal{P}\mathcal{Q})^* = \mathcal{Q}^*\mathcal{P}^*\). (That is, \((\cdot)^*\) is an anti-involution.) 
\end{itemize}
\end{definition}

For any \(\mathcal{P} \in \mathcal{D}_X\), we have \(\mathbb{D}(\mathcal{D}_X/\mathcal{D}_X\mathcal{P}) = \mathcal{D}_X/\mathcal{D}_X\mathcal{P}^*\) (cf. \cite{Hot18}, p. 71). We can now define the dual localisation functor: 

\begin{definition}
Let \(D \subset X\) be a reduced divisor. The dual localisation \(\mathcal{M}(!D)\) of \(\mathcal{M}\) away from \(D\) is defined as 
\begin{align*}
\mathcal{M}(!D) \coloneqq \mathbb{D}((\mathbb{D}\mathcal{M})(*D)). 
\end{align*}
\end{definition}

\subsection{Minimal extension}

There is the canonical inclusion \(\mathcal{M} \hookrightarrow \mathcal{M}(*D)\), denoted by \(\text{loc}_\mathcal{M}\). We also get a surjection \(\operatorname{dloc}_\mathcal{M} = \mathbb{D}\operatorname{loc}_{\mathbb{D}\mathcal{M}}: \mathcal{M}(!D) \twoheadrightarrow \mathcal{M}\) as the dual of the morphism \(\operatorname{loc}_{\mathbb{D}\mathcal{M}}: \mathbb{D}\mathcal{M} \hookrightarrow (\mathbb{D}\mathcal{M})(*D)\). We now define the morphism \(\operatorname{mid}_\mathcal{M}: \mathcal{M}(!D) \rightarrow \mathcal{M}(*D)\) as the composition of \(\operatorname{loc}_\mathcal{M}\) and \(\operatorname{dloc}_\mathcal{M}\). 

\begin{definition}
Let \(\mathcal{M}\) be a \(\mathcal{D}_X\)-module. The minimal extension \(\mathcal{M}({!*}D)\) of \(\mathcal{M}\) along \(D\) is defined as the image of the map \(\operatorname{mid}_\mathcal{M}: \mathcal{M}(!D) \rightarrow \mathcal{M}(*D)\). 
\end{definition}

The name \textit{minimal} extension is justified by the following fact: \(\mathcal{M}({!*}D)\) is uniquely determined by the condition \(\mathcal{M}({!*}D)(*D) = \mathcal{M}(*D)\) and the condition that \(\mathcal{M}({!*}D)\) has neither sub- nor quotient objects supported on \(D\), cf. \cite{Sab052}, 2.4.12. 

\begin{remark}\label{remark:irrimpliesmin}
If \(\mathcal{M}\) is irreducible and has support \(X\), then we clearly have \(\mathcal{M} = \mathcal{M}({!*}D)\): Indeed, if \(\mathcal{M}\) had a subobject supported on \(D\), this would directly contradict the irreducibility. Furthermore, if \(\mathcal{M}\) had a quotient object supported on \(D\), it would have a subobject supported on \(X - D\), again contradicting the irreducibility. 
\end{remark}

\begin{remark}\label{remark:MPirreducible}
Let \(\mathcal{P} \in \mathcal{D}\). If \(\mathcal{D}/\mathcal{D}\mathcal{P}\) is irreducible, so is \(\mathcal{P}\): Indeed, a factorisation \(\mathcal{P} = \mathcal{P}_1 \cdot \mathcal{P}_2\) (where both \(\mathcal{P}_1, \mathcal{P}_2\) have order at least \(1\)) would give the (left-)ideal \(\mathcal{D}\mathcal{P}_2\) containing \(\mathcal{D}\mathcal{P}\), so that we would get a surjection \(\mathcal{D}/\mathcal{D}\mathcal{P} \rightarrow \mathcal{D}/\mathcal{D}\mathcal{P}_2\) with non-trivial kernel, contradicting the irreducibility of \(\mathcal{D}/\mathcal{D}\mathcal{P}\). However, irreducibility of \(\mathcal{P}\) does not imply the irreducibility of \(\mathcal{D}/\mathcal{D}\mathcal{P}\). The reason is that not all ideals of \(\mathcal{D}\) are generated by one element, because \(\mathcal{D}\) is not a principal ideal domain. (All ideals of \(\mathcal{D}\) are however generated by at most two elements, see \cite{Sab90}, Proposition 2.3.1.)
\end{remark}

\section{Kashiwara-Malgrange filtration}

The conditions \(\mathcal{M} = \mathcal{M}(*D)\), \(\mathcal{M} = \mathcal{M}(!D)\), \(\mathcal{M} = \mathcal{M}({!*}D)\) are of local nature, and we need algebraic conditions to check them. This can be done with the notion of the Kashiwara-Malgrange filtration, which will be introduced in the following. We mainly follow \cite{Sab05}, 3.2 and add some own remarks. 

\begin{definition}
Let \(\mathcal{M}\) be a \(\mathcal{D}_X\)-module and let \(D \subset X\) be a codimension \(1\) sub-manifold/-variety, that is, a hypersurface. Let \(t\) be a local coordinate function, whose vanishing locus is \(D\). 

\begin{itemize}
    \item The \(V\)-filtration of \(\mathcal{D}_X\) with respect to \(t\) is ascending and indexed by \(\mathbb{Z}\). It is given by 
    \begin{align*}
    V^{(t)}_k\mathcal{D}_X = \begin{cases}
    \mathcal{O}_X\langle t \partial_t \rangle & \text{if } k = 0, \\ \sum_{i = 0}^k \partial_t^i V_0\mathcal{D}_X^{(t)} & \text{if } k > 0, \\ t^{-k}V^{(t)}_0\mathcal{D}_X & \text{if } k < 0. 
    \end{cases}
    \end{align*}
    \item For a local section \(m \in \mathcal{M}\), the Bernstein-(Sato-)polynomial \(b_m(v) \in \mathbb{C}[v]\) is defined as the minimal non-trivial polynomial such that 
\begin{align*}
b_m(-t\partial_t) \cdot m \in V_{-1}\mathcal{D}_X \cdot m. 
\end{align*}

(Here, \(b_m(-t\partial_t)\) is read as an element of \(\mathcal{D}\), and this definition is interesting because \(b_m(-t\partial_t) \in V_0\mathcal{D} - V_{-1}\mathcal{D}\). \(b_m\) is unique up to constant multiplication.)
    \item If every local section of \(\mathcal{M}\) has a Bernstein polynomial with respect to \(t\), \(\mathcal{M}\) is called specializable with respect to \(t\) (or along \(D\), which is independent of the choice of \(t\)). The Kashiwara-Malgrange filtration of \(\mathcal{M}\) with respect to \(t\) is then defined as 
\begin{align*}
V^{(t)}_k \mathcal{M} = \{m \in \mathcal{M} : \text{ord}(m)\ \leq k\}, 
\end{align*}

for \(k \in \mathbb{R}\), where \(\text{ord}(m) = \max\{\Re(r) : b_m(r) = 0\}\). (This filtration is also ascending.)
\end{itemize}
\end{definition}

\begin{remark}
This definition is only interesting at \(D\): Let \(m \in \mathcal{M}(U)\), where \(U \cap D = \emptyset\). We have \(V_{-1}\mathcal{D}\vert_U = \mathcal{D}\vert_U\), because \(t\) is invertible away from \(D\). It follows that \(b_m(-t\partial_t) \cdot m \in V_{-1}\mathcal{D} \cdot m\) is fulfilled for any choice of \(b_m\). This means that the minimal choice for \(b_m\) is \(1\), which means \(\text{ord}(m) = -\infty\) and \((V^{(t)}_k\mathcal{M})\vert_U = \mathcal{M}\vert_U\) for all \(k \in \mathbb{R}\). 
\end{remark}

We have the following fundamental result, cf. \cite{Sab01}, Section 1.2: 

\begin{theorem}
Every holonomic \(\mathcal{D}_X\)-module is specializable along any hypersurface. 
\end{theorem}

\begin{definition}
Let \(\mathcal{M}\) be a \(\mathcal{D}_X\)-module that is specializable with respect to \(D\). Abbreviating \(V^{(t)}_{\alpha}\mathcal{M}\) as \(V_{\alpha}\mathcal{M}\) we have the graded pieces \(\operatorname{gr}^{V}_\alpha \mathcal{M}\) for \(\alpha \in \mathbb{R}\). We obtain the graded object \(\operatorname{gr}^V \mathcal{M}\) as the direct sum of the graded pieces. For \(\alpha \in \mathbb{C}\), we then consider the generalized eigenspaces 
\begin{align*}
\psi_{t, \alpha}\mathcal{M} \coloneqq \bigcup_{n} \ker\left((\partial_t t + \alpha)^n : \operatorname{gr}^V \mathcal{M} \rightarrow \operatorname{gr}^V \mathcal{M}\right). 
\end{align*}
\end{definition}

Note that \(\psi_{t, \alpha}\mathcal{M} \neq 0\) if and only if there is a local section \(m \in \mathcal{M}\) such that \(b_m(\alpha) = 0\). 

\subsection{\(V\)-filtration in the order \(1\) case}

We will now give an explicit formula for the \(V\)-filtration of \(\mathcal{D}\)-modules of generic rank \(1\). To achieve this, we introduce the following notation: 

\begin{definition}
Let \(R\) be a ring and \(I \subset R\) be an ideal. For a subset \(S \subset R\), let \([S]_{R/I}\) be the image of \(S\) in \(R/I\) under the canonical projection \(R \rightarrow R/I\). We make the analogous definition for sheaves of rings and ideal sheaves. 
\end{definition}

From now on, we make the following convention: For \(p \in \mathbb{C}[t]\), we define \(p_k \in \mathbb{C}\) for \(k \in \mathbb{N}_0\) via \(p(t) = \sum_k p_kt^k\). 

\begin{definition}\label{def:alphaP}
Let \(X \subset \mathbb{C}\) open with \(0 \in X\). Let \(\mathcal{P} = p(t)\partial_t + q(t) \in \mathcal{D}(X)\), (where \(p, q \in \mathcal{O}(X)\)) have a regular singularity at \(0\), or be non-singular at \(0\). Define 
\begin{align*}
\alpha(\mathcal{P}) \coloneqq \operatorname{Res}_{t = 0}\left(\frac{q(t)}{p(t)}\right). 
\end{align*}
\end{definition}

\begin{remark}
Let \(\mathcal{P}\) be as above and set \(D \coloneqq \operatorname{Sing}(\mathcal{P})\). We make the following remarks on \Cref{def:alphaP}: 
\begin{itemize}
    \item Because \(\mathcal{P}\) is assumed to not have an irregular singularity at the origin, the holomorphic solutions to \(\mathcal{P}f = 0\) on \(X - D\) are spanned by \(h(t)t^{m}\) for some \(m \in \mathbb{C}\) and \(h \in \mathcal{O}(X - (D - \{0\}))\) with \(h(0) \neq 0\). We have \(m = -\alpha(\mathcal{P})\). 
    \item Write \(\mathcal{P} = t^n\mathcal{P}'\) for maximal \(n \in \mathbb{N}_0\) and \(\mathcal{P'} \in \mathcal{D}(X)\). Assume \(\mathcal{P}'\) is regular singular at \(0\). If \(\mathcal{P}' = p'(t)\partial_t + q'(t) \in \mathcal{D}(X)\) for suitable \(p', q' \in \mathcal{O}(X)\), we have \(\alpha(\mathcal{P}) = \alpha(\mathcal{P}') = q_0'/p_1'\), where we have written \(p'(t) = \sum_i p_i't^i\) and \(q'(t) = \sum_i q_i't^i\). 
\end{itemize}
\end{remark}

We can now give the formula for the \(V\)-filtration: 

\begin{proposition}\label{prop:bernsteinrank1}
Let \(X \subseteq \mathbb{C}\) be open with \(0 \in X\). Let \(\mathcal{P} = p(t)\partial_t + q(t) \in \mathcal{D}(X)\) be of order \(1\), where \(p, q \in \mathcal{O}(X)\) and assume that \(\mathcal{P}\) is irreducible. (In particular, \(\mathcal{P} \notin t\mathcal{D}(X)\)). We have the \(\mathcal{D}_X\)-module \(\mathcal{M} \coloneqq \mathcal{D}_X/\mathcal{D}_X\mathcal{P}\). Consider the class 
\begin{align*}
[1]_{\mathcal{M}_0} \in \mathcal{D}_{X, 0}/\mathcal{D}_{X, 0}\mathcal{P} = \mathcal{M}_0 
\end{align*}

\begin{itemize}
    \item If \(\mathcal{P}\) is non-singular or regular singular at \(0\), the Bernstein-polynomial of \([1]_{\mathcal{M}_0}\) with respect to the origin is given by \(b_{{[1]_{\mathcal{M}_0}}}(s) = s + \alpha + 1\), where \(\alpha = \alpha(\mathcal{P})\). The Kashiwara-Malgrange filtration of \(\mathcal{M}\) takes the form 
    \begin{align*}
    V^{(0)}_k\mathcal{M} = \left[V_{\lfloor k + \Re(\alpha) \rfloor + 1}  ^{(0)}\mathcal{D}_X\right]_{\mathcal{M}}
    \end{align*}
    for \(k \in \mathbb{R}\). In other words, the \(V\)-filtration of \(\mathcal{M}\) is directly induced by the \(V\)-filtration of \(\mathcal{D}\), with a suitable shift. 
    \item If \(P\) is irregular singular at \(0\), we have \(b_{{[1]_{\mathcal{M}_0}}}(s) = 1\). Furthermore, the \(V\)-filtration is constant. 
\end{itemize}
\end{proposition}

\begin{proof}
Abbreviate \(V^{(0)}\) as \(V\). We begin by showing the first point under the assumption that \(\mathcal{P}\) is regular singular at \(0\). We then have \(\alpha = q_0/p_1\). Because \(V_k\mathcal{D}_X\) coincides with \(\mathcal{D}_X\) away from \(0\), it suffices to show \(V_k\mathcal{M}_0 = \left[V_{\lfloor k + \Re(\alpha) \rfloor + 1}\mathcal{D}_{X, 0}\right]_{\mathcal{M}_0}\). Let \(s \in \mathcal{D}_{X, 0}\) and consider \([s]_{\mathcal{M}_0} \in \mathcal{M}_0\). The condition \(b_{[s]}(-\partial_t t) \cdot [s] \in V_{-1}(\mathcal{D}_{X, 0}) \cdot [s]\) is by definition 
\begin{align*}
b_{[s]}(-\partial_t t) \cdot s + V_{-1}(\mathcal{D}_{X, 0}) \cdot \mathcal{P} \subseteq V_{-1}(\mathcal{D}_{X, 0}) \cdot s + \mathcal{D}_{X, 0} \cdot \mathcal{P}, 
\end{align*}

where the inclusion is to be read inside of \(V_{-1}(\mathcal{D}_{X, 0})\). Because the ideal \(V_{-1}(\mathcal{D}_{X, 0})\mathcal{P}\) is closed under addition, this is equivalent to 
\begin{align*}
b_{[s]}(-\partial_t t) \cdot s \in V_{-1}(\mathcal{D}_{X, 0}) \cdot s + \mathcal{D}_{X, 0} \cdot \mathcal{P}. 
\end{align*}

Because \(\mathcal{P}\) is regular singular at \(0\), we have \(p_0 = 0, p_1 \neq 0\). It follows that, for \(s = [1]_{\mathcal{M}_0}\), the Bernstein condition is fulfilled with \(b_{[1]_{\mathcal{M}_0}}(s) = s + \alpha + 1\), as is seen with the following calculation: 
\begin{align*}
b_{[1]_{\mathcal{M}_0}}(-\partial_t t) \cdot 1 & = -(\partial_t t -(\alpha+1)) \cdot 1 = -(t\partial_t - \alpha) = -(t\partial_t + q_0/p_1) \\ & = -p_1^{-1} \cdot (p_1t\partial_t + q_0) = \underbrace{p_1^{-1}\sum_{i = 2}^{\infty} p_it^i\partial_t + p_1^{-1}\sum_{i = 1}^{\infty}q_it^i}_{\in V_{-1} \mathcal{D}_{X, 0}} - \underbrace{p_1^{-1}\mathcal{P}}_{\in \mathcal{D}_{X, 0}\mathcal{P}}, 
\end{align*}

so that indeed \(b_{[1]_{\mathcal{M}_0}}(s) = s + \alpha + 1\). In general, one obtains \(b_{[t^l\partial_t^k]_{\mathcal{M}_0}}(v) = v + \alpha - (k - l) + 1\) for \(k, l \in \mathbb{N}_0\), so that 
\begin{align*}
[t^l\partial_t^k]_{\mathcal{M}_0} \in V_{-\Re(\alpha) + k - l - 1}\mathcal{M}_0 - V_{<-\Re(\alpha) + k - l - 1}\mathcal{M}_0. 
\end{align*}

Since \(V_i\mathcal{D}_{X, 0} = \sum_{k - l \leq i}\mathcal{O}_{X, 0}t^l\partial_t^k\), we obtain \([V_i(\mathcal{D}_{X, 0})]_{\mathcal{M}_0} = V_{-\Re(\alpha) + i - 1}\mathcal{M}_0\). We have \(\max_{i \in \mathbb{Z}}\{-\Re(\alpha) + i - 1 \leq k\} = \lfloor k + \Re(\alpha) + 1 \rfloor\), finishing the proof of the first point under the assumption that \(\mathcal{P}\) is regular singular at \(0\). If \(\mathcal{P}\) is non-singular at \(0\), we proceed as follows: We have \(\alpha = 0\) and 
\begin{align*}
t\mathcal{P} = \underbrace{\sum_{i = 1}^{\infty} p_it^{i+1}\partial_t + \sum_{i = 0}^{\infty}q_it^{i+1}}_{\in V_{-1}(\mathcal{D}_{X, 0})} + p_0t\partial_t. 
\end{align*}

showing that \(b_{{[1]_{\mathcal{M}_0}}}(s) = s + 1\). One concludes analogously to the regular singular case. It remains to prove the second point: If \(\mathcal{P}\) is irregular singular at \(0\), we have \(p_0 = 0, p_1 = 0, q_0 \neq 0\). Then, 
\begin{align*}
\mathcal{P} = \underbrace{\sum_{i = 2}^{\infty} p_it^i\partial_t + \sum_{i = 1}^{\infty}q_it^i}_{\in V_{-1}(\mathcal{D}_{X, 0})} + q_0, 
\end{align*}

so that the Bernstein condition for any germ of \([s] \in \mathcal{M}_0\) is trivially fulfilled for any constant non-zero polynomial. It follows that \(\operatorname{ord}([s]) = -\infty\), so that the \(V\)-filtration is constant. 
\end{proof}

In the regular case, the \(V\)-filtration of such \(\mathcal{D}\)-modules can be expressed in terms of the order filtration of \(\mathcal{D}\), but only after possibly shrinking the underlying manifold: 

\begin{lemma}\label{lemma:VawayfromD}
Keep the assumptions from \Cref{prop:bernsteinrank1} and assume that \(\mathcal{P}\) is regular singular or non-singular at the origin. Moreover, assume that \(\mathcal{P}\) is non-singular on \(X - \{0\}\). We have 
\begin{align*}
    V^{(0)}_k\mathcal{M} = 
    \left\{\begin{array}{lr}
        \left[F_{\lfloor k + \Re(\alpha) \rfloor + 1}\mathcal{D}_X\right]_{\mathcal{M}}, & \text{if } \lfloor k + \Re(\alpha) \rfloor + 1 \geq 0\\
        \left[t^{-\lfloor k + \Re(\alpha) \rfloor - 1}\mathcal{O}_X\right]_{\mathcal{M}}, & \text{if } \lfloor k + \Re(\alpha) \rfloor + 1 < 0
    \end{array}\right. \quad. 
\end{align*}

\end{lemma}

\begin{proof}
Recall that \(F_0 \mathcal{D}_X = \mathcal{O}_X, V_0 \mathcal{D}_X = \mathcal{O}_X\langle t\partial_t \rangle\). We have \(F_0 \mathcal{D}_X \subset V_0 \mathcal{D}_X\), so \([F_0 \mathcal{D}_X]_{\mathcal{M}} \subset [V_0 \mathcal{D}_X]_{\mathcal{M}}\) in \(\mathcal{M}\). We will now show the reverse inclusion: It suffices to show it on the level of stalks, that is, \([t\partial_t]_{\mathcal{M}_{t_0}} \in [F_0 \mathcal{D}_{X, t_0}]_{\mathcal{M}_{t_0}}\) for every point \(t_0 \in X\). This is equivalent to the existence of \(f \in \mathcal{O}_{X, t_0}, \mathcal{A} \in \mathcal{D}_{X, t_0}\), such that 
\begin{align}\label{eq:awayfrom1}
t\partial_t = f + \mathcal{AP}. 
\end{align}

This equation is solved by (the germs at \(t_0\) of) \(f = tq(t)/p(t), \mathcal{A} = -1\). For \(t_0 \neq 0\), it is clear that this germ lies in \(\mathcal{O}_{X, t_0}\), because \(p\) has no roots away from \(0\). For \(t_0 = 0\), it also lies in \(\mathcal{O}_{X, t_0}\), because \(p\) has a root of multiplicity at most \(1\) at \(0\). Thus, we obtain \([F_0 \mathcal{D}_X]_{\mathcal{M}} \subset [V_0 \mathcal{D}_X]_{\mathcal{M}}\). For \(n \in \mathbb{N}\), we have 
\begin{align*}
F_n\mathcal{D}_X = \sum_{i=0}^n \partial_t^iF_0\mathcal{D}_X, \quad
V_n\mathcal{D}_X = \sum_{i=0}^n \partial_t^i V_0\mathcal{D}_X
\end{align*}

In other words, \(F_n\mathcal{D}_X\) is obtained from \(F_0\mathcal{D}_X\) in the same way as \(V_n\mathcal{D}_X\) is obtained from \(V_0\mathcal{D}_X\) for all \(n \in \mathbb{N}\). Combining this with \Cref{prop:bernsteinrank1}, this completes the proof in the case \(\lfloor k + \Re(\alpha) \rfloor + 1 \geq 0\). In the case \(\lfloor k + \Re(\alpha) \rfloor + 1 < 0\), the claim is also clear from \([F_0 \mathcal{D}_X]_{\mathcal{M}} \subset [V_0 \mathcal{D}_X]_{\mathcal{M}}\) with the definition of the \(V\)-filtration. 
\end{proof}

\subsection{An elementary order \(1\) example}

In the following, we apply the previous results to an elementary example. The following two results will only become relevant for \Cref{lemma:strictlyspecializable} and may be skipped until then. 

\begin{lemma}\label{lemma:strictspec1}
Let \(\mathcal{P} = t\partial_t + \alpha \in \mathcal{D}\), where \(\alpha \in \mathbb{C} - \mathbb{Z}\). Set \(\mathcal{M} \coloneqq \mathcal{D}_{\mathbb{C}}/\mathcal{D}_{\mathbb{C}}\mathcal{P}\). The morphisms 
\begin{align*}
t: \operatorname{gr}_{r+1}^V\mathcal{M} \rightarrow \operatorname{gr}_{r}^V\mathcal{M}, \quad \partial_t: \operatorname{gr}_{r}^V\mathcal{M} \rightarrow \operatorname{gr}_{r+1}^V\mathcal{M}
\end{align*}

are isomorphisms for all \(r \in \mathbb{R}\). 
\end{lemma}

\begin{proof}
It suffices to check this on the stalk at the origin. We argue with equivalence classes of equivalence classes, because we consider the graded pieces of the \(V\)-filtration of \(\mathcal{M}_0\), which is already a quotient. (In fact, there is a third type of equivalence class involved, because we are looking at germs and not sections.) In this proof, we abbreviate \(\operatorname{gr}^V\) as \(\operatorname{gr}\). 

The \(V\)-filtration of \(\mathcal{M}\) only jumps at \(-\Re(\alpha) + \mathbb{Z}\), so that we may assume \(r = -\Re(\alpha) + k\) for some \(k \in \mathbb{Z}\). For the moment, assume additionally that \(k \geq 0\) and consider the closed inclusion \(i: \{0\} \rightarrow \mathbb{C}\). \(\operatorname{gr}_{-\Re(\alpha) + k}\mathcal{M}\) is supported only at the origin, and so it can be seen as a \(\mathbb{C}\)-vector space. We then have 
\begin{align*}
\operatorname{gr}_{-\Re(\alpha) + k}\mathcal{M} = [V_{k+1}\mathcal{D}_X]_{\mathcal{M}}/[V_{k}\mathcal{D}_X]_{\mathcal{M}} = \mathbb{C}\cdot [[\partial_t^{k+1}]_{\mathcal{M}}]_{\operatorname{gr}_{-\Re(\alpha) + k}\mathcal{M}}, 
\end{align*}

in other words, \(i^*(\operatorname{gr}_{-\Re(\alpha) + k}\mathcal{M})\) is the one-dimensional \(\mathbb{C}\)-vector space spanned by the class of \(\partial_t^{k+1}\). Similarly, \(i^*(\operatorname{gr}_{-\Re(\alpha) + k + 1}\mathcal{M})\) is spanned by the class of \(\partial_t^{k + 2}\). The map 
\begin{align*}
\partial_t: \operatorname{gr}_{-\Re(\alpha) + k}\mathcal{M} \rightarrow \operatorname{gr}_{-\Re(\alpha)+ k + 1}\mathcal{M}
\end{align*}

sends the class of \(\partial_t^{k + 1}\) to the class of \(\partial_t^{k + 2}\), so it is an isomorphism. The map \(t: \operatorname{gr}_{-\Re(\alpha) + k + 1}\mathcal{M} \rightarrow \operatorname{gr}_{-\Re(\alpha) + k}\mathcal{M}\) sends the class of \(\partial_t^{k + 2}\) to the class of 
\begin{align*}
t \partial_t^{k + 2} = \partial_t^{k+1}(t \partial_t - k-1)
\end{align*}

We have \(t\partial_t = -\alpha\) in \(\mathcal{M}\), and so this relation also holds in the graded pieces of the \(V\)-filtration. It follows that the map \(t: \operatorname{gr}_{-\Re(\alpha) + k + 1}\mathcal{M} \rightarrow \operatorname{gr}_{-\Re(\alpha) + k}\mathcal{M}\), in the mentioned bases, is given by multiplication with \(-\alpha - k - 1\), which gives an isomorphism, because \(\alpha \notin \mathbb{Z}\). This shows the claim under the assumption \(k \geq 0\). The case \(k < 0\) is shown similarly. 
\end{proof}

\begin{lemma}\label{lemma:strictspec2}
Set \(\mathcal{M} \coloneqq \mathcal{D}_{\mathbb{A}^1}/\mathcal{D}_{\mathbb{A}^1}\partial_t\). The morphism \(t: \operatorname{gr}_{r+1}^V\mathcal{M} \rightarrow \operatorname{gr}_{r}^V\mathcal{M}\) is surjective for all \(r \in \mathbb{R} - \{-1\}\). The morphism \(\partial_t: \operatorname{gr}_{r}^V\mathcal{M} \rightarrow \operatorname{gr}_{r+1}^V\mathcal{M}\) is surjective for all \(r \in \mathbb{R}\). 
\end{lemma}

\begin{proof}
This proof is analogous to the proof of the previous lemma. 
\end{proof}

\subsection{Localisation and minimal extension for order \(1\): The "trivial" case}

We will now explain how we can determine the (dual) localisation and minimal extension of \(\mathcal{D}_{\mathbb{C}}/\mathcal{D}_{\mathbb{C}}\mathcal{P}\) at the origin, if the order of \(\mathcal{P}\) is \(1\). We begin with what might be viewed as the "trivial case": Namely, we assume that \(\mathcal{P}\) is regular singular at the origin with \(\alpha \coloneqq \alpha(\mathcal{P}) \notin \mathbb{Z}\), or that \(\mathcal{P}\) is irregular singular at the origin. We use the following algebraic result, characterizing when a \(\mathcal{D}\)-module is localised in terms of the \(V\)-filtration of \(\mathcal{D}\): 

\begin{lemma}\label{lemma:locandroots}
Let \(X \subset \mathbb{C}\) be open with \(0 \in X\). Let \(\mathcal{P} \in V_{0}\mathcal{D}(X) - V_{-1}\mathcal{D}(X)\), which allows us to write \(\mathcal{P} = b(t\partial_t) + t\mathcal{P}'\), where \(0 \neq b \in \mathbb{C}[s]\) and \(\mathcal{P}' \in V_{0}\mathcal{D}(X)\). Let \(\mathcal{M} \coloneqq \mathcal{D}_X/\mathcal{D}_X\mathcal{P}\). We have 
\begin{align*}
\mathcal{M}(*0) = \mathcal{M} 
\end{align*}

if and only if \(b(k) \neq 0\) for all \(k \in \mathbb{N}_0\) and we have 
\begin{align*}
\mathcal{M}(!0) = \mathcal{M} 
\end{align*}

if and only if \(b(k) \neq 0\) for all \(k \in \mathbb{Z}_{<0}\). (Note that \(0\) is allowed here as a zero, which is not in symmetry to the previous condition.) 
\end{lemma}

\begin{proof}
The condition for \(\mathcal{M}(*0) = \mathcal{M}\) is \cite{Sab90}, Lemma 4.2.7. The condition for \(\mathcal{M}(!0) = \mathcal{M}\) is derived from this in the following way: Note that 
\begin{align*}
\mathcal{M}(!0) = \mathbb{D}(\mathbb{D}(\mathcal{M})(*0)) = \mathbb{D}(\mathcal{D}_X/\mathcal{D}_X\mathcal{P}^*(*0)). 
\end{align*}

Via biduality, \(\mathcal{M}(!0) = \mathcal{M}\) is then equivalent to \(\mathcal{D}_X/\mathcal{D}_X\mathcal{P}^*(*0) = \mathcal{D}_X/\mathcal{D}_X\mathcal{P}^*\). Write \(P = b(t\partial_t) + tP'\) as in the statement of the lemma. Using the definition of the formal adjoint, we obtain 
\begin{align*}
\mathcal{P}^* = b(-t\partial_t - 1) + (\mathcal{P}')^*t. 
\end{align*}

Because \(\mathcal{P}' \in V_0 \mathcal{D}\), we have \((\mathcal{P}')^* \in V_0 \mathcal{D}\). Because of \(tV_0 \mathcal{D} = V_0\mathcal{D}t\), we can write \((\mathcal{P}')^*t = t\mathcal{P}''\) for some \(\mathcal{P}'' \in V_0\mathcal{D}t\). We are now in the situation of the condition for \(\mathcal{M}(*0) = \mathcal{M}\) and deduce that \(\mathcal{M}(!0) = \mathcal{M}\) if and only if \(b(k) \neq 0\) for all \(k \in \mathbb{Z}_{<0}\). 
\end{proof}

\begin{lemma}\label{lemma:MlocalphanotinZ}
Let \(X \subseteq \mathbb{C}\) be open and such that \(0 \in X\). Let \(\mathcal{P} \in \mathcal{D}(X)\) be irreducible of order \(1\). Let \(\mathcal{M} \coloneqq \mathcal{D}_X/\mathcal{D}_X\mathcal{P}\). If \(\mathcal{P}\) is regular singular at the origin with \(\alpha \coloneqq \alpha(\mathcal{P}) \notin \mathbb{Z}\), or if \(\mathcal{P}\) is irregular singular at the origin, we have 
\begin{align*}
\mathcal{M} = \mathcal{M}(*0) =\mathcal{M}(!0) = \mathcal{M}({!*}0). 
\end{align*}
\end{lemma}

\begin{proof}
First, assume that \(\mathcal{P}\) is regular singular at the origin with \(\alpha \coloneqq \alpha(\mathcal{P}) \notin \mathbb{Z}\). As in the proof of \Cref{prop:bernsteinrank1} we can write \(\mathcal{P} = t\partial_t - \alpha + t\mathcal{P}'\) for some \(\mathcal{P}' \in V_0\mathcal{D}(X)\). \Cref{lemma:locandroots} implies that \(\mathcal{M} \cong \mathcal{M}(*0) \cong \mathcal{M}(!0)\). This also means that the canonical morphism \(\mathcal{M}(!0) \rightarrow \mathcal{M}(*0)\) is the identity, and so \(\mathcal{M} \cong \mathcal{M}({!*}0)\). 

Now, assume that \(\mathcal{M}\) is irregular singular at the origin. In this case, we can write \(\mathcal{P} = q_0 + t\mathcal{P}'\) for some \(\mathcal{P}' \in V_0\mathcal{D}(X)\) as in the proof of \Cref{prop:bernsteinrank1}. We again conclude by \Cref{lemma:locandroots}. 
\end{proof}

We can use this result to obtain the following local picture: 

\begin{lemma}\label{lemma:localalphapicture}
Let \(X \subseteq \mathbb{C}\) be open and such that \(0 \in X\). Let \(\mathcal{P} \in \mathcal{D}(X)\) be of order \(1\) and such that \(\mathcal{M} \coloneqq \mathcal{D}_X/\mathcal{D}_X\mathcal{P}\) is irreducible. Furthermore, assume that \(\mathcal{P}\) is regular singular at the origin with \(\alpha \coloneqq \alpha(\mathcal{P}) \notin \mathbb{Z}\). Then, if \(Y\) is an open disk around the origin with 
\begin{align*}
Y \cap \operatorname{Sing}(\mathcal{M}) = \{0\}, 
\end{align*}

we have \(\mathcal{M}\vert_Y \cong \mathcal{D}_Y/\mathcal{D}_Y(t\partial_t - \alpha)\). 
\end{lemma}

\begin{proof}
First, note that \(\mathcal{D}_Y/\mathcal{D}_Y(t\partial_t - \alpha)\) has singular locus \(0\). So, the local systems associated to \(\mathcal{M}\vert_Y\) and \(\mathcal{D}_Y/\mathcal{D}_Y(t\partial_t - \alpha)\) are determined by their respective monodromy factor at the origin, which are both \(e^{2\pi i \alpha}\). So, these local systems are isomorphic. \Cref{lemma:MlocalphanotinZ} implies that both \(\mathcal{D}_Y/\mathcal{D}_Y(t\partial_t - \alpha)\) and \(\mathcal{M}\vert_Y\) are localized away from the origin. The claim follows with the Riemann-Hilbert correspondence (\Cref{lemma:locforgets}). 
\end{proof}

\subsection{Localisation and minimal extension for order \(1\): The "non-trivial" case}

It remains to determine the localisation, dual localisation and minimal extension if \(\mathcal{P}\) is regular singular at the origin with \(\alpha(\mathcal{P}) \in \mathbb{Z}\) or if \(\mathcal{P}\) is non-singular at the origin. To do so, we introduce the following useful "deformation" of the operator \(\mathcal{P}\): 

\begin{definition}
Let \(X \subset \mathbb{A}^1\) or \(X \subset \mathbb{C}\) open with \(0 \in X\). Let \(\mathcal{P} = p(t)\partial_t + q(t)\in \mathcal{D}(X)\) be of order \(1\) and in \(V_{0}\mathcal{D}(X) - V_{-1}\mathcal{D}(X)\), implying \(p_0 = 0\). For \(k \in \mathbb{Z}\), define 
\begin{align*}
\mathcal{P}_k \coloneqq p(t)\partial_t + q(t) + k \frac{p(t)}{t} \in \mathcal{D}(X). 
\end{align*}
\end{definition}

The following lemma shows that \(\mathcal{P}_k\) has similar properties to \(\mathcal{P}\) for any \(k \in \mathbb{Z}\): 

\begin{lemma}\label{lemma:Pkgood}
Let \(X \subset \mathbb{A}^1\) or \(X \subset \mathbb{C}\) be open with \(0 \in X\). Let \(\mathcal{P} = p(t)\partial_t + q(t)\in \mathcal{D}(X)\) be a section in \(V_{0}\mathcal{D}(X) - V_{-1}\mathcal{D}(X)\). Assume furthermore that \(\mathcal{P}\) is regular singular at \(0\). Set 
\begin{align*}
\mathcal{M} \coloneqq \mathcal{D}_X/\mathcal{D}_X\mathcal{P}, \quad \mathcal{M}_k \coloneqq \mathcal{D}_X/\mathcal{D}_X\mathcal{P}_k
\end{align*}

and recall that \(\alpha(\mathcal{P}) = q_0/p_1\). Let \(k \in \mathbb{Z}\). The following statements hold: 

\begin{itemize}
    \item A section \(f \in \mathcal{O}_{X - \operatorname{Sing}(\mathcal{P})}\) solves \(\mathcal{P}\) if and only if \(t^{-k}f\) solves \(\mathcal{P}_k\). In particular, \(\mathcal{P}_k\) is regular singular at \(0\). 
    \item \(\mathcal{M}(*0) \cong \mathcal{M}_k(*0)\). 
    \item \(\mathcal{P}_k \in V_{0}\mathcal{D}(X) - V_{-1}\mathcal{D}(X)\). 
    \item \(\alpha(\mathcal{P}_k) = \alpha(\mathcal{P}) + k\). 
\end{itemize}
\end{lemma}

\begin{proof}
The first point follows from \(\mathcal{P}_k(t^{-k}f) = t^{-k}\mathcal{P}f\). The second point follows directly from the first one, combined with the Riemann-Hilbert correspondence. The third point follows from \(\mathcal{P} \in V_{0}\mathcal{D} - V_{-1}\mathcal{D}\), where we distinguish two cases: If \(p_1 \neq 0\), then \(p(t)\partial_t \in V_{0} \mathcal{D} - V_{-1} \mathcal{D}\) and so we must also have \(\mathcal{P}_k \in V_{0} \mathcal{D} - V_{-1} \mathcal{D}\). If \(p_1 = 0\), then \(\mathcal{P} \in V_{0} \mathcal{D} - V_{-1} \mathcal{D}\) implies \(q(0) \neq 0\), so that \(\mathcal{P}\) is irregular singular at \(0\), which contradicts our assumptions. This shows the third point. The fourth point is obvious. 
\end{proof}

We finally obtain the formulas for the localisation and minimal extension of \(\mathcal{D}_X/\mathcal{D}_X\mathcal{P}\): 

\begin{lemma}\label{lemma:MlocalphainZ}
Let \(X \subset \mathbb{A}^1\) or \(X \subset \mathbb{C}\) be open with \(0 \in X\). Let \(\mathcal{P} \in \mathcal{D}(X)\) be irreducible, of order \(1\) and with a regular singularity at the origin. Let \(\mathcal{M} \coloneqq \mathcal{D}_X/\mathcal{D}_X\mathcal{P}\). Assume \(\alpha(\mathcal{P}) \in \mathbb{Z}\) and let \(k \in \mathbb{Z}\). Then, the following statements are true: 

\begin{itemize}
    \item \(\mathcal{M}(*0) = \mathcal{D}_X/\mathcal{D}_X\mathcal{P}_k\) if and only if \(\alpha(\mathcal{P}) + k < 0\). 
    \item \(\mathcal{M}(!0) = \mathcal{D}_X/\mathcal{D}_X\mathcal{P}_k\) if and only if \(\alpha(\mathcal{P}) + k \geq 0\). 
    \item \(\mathcal{M}({!*}0) = \mathcal{D}_X/\mathcal{D}_X(t^{-1}\mathcal{P}_{-\alpha(\mathcal{P})})\). 
\end{itemize}
\end{lemma}

\begin{proof}
Assume for contradiction that \(\mathcal{P} \notin V_0\mathcal{D}_X\). Because \(q(t) \in V_0\mathcal{D}(X)\), we must have \(p(t)\partial_t \notin V_0\mathcal{D}_X\), leading to \(p_0 \neq 0\), contradicting the fact that \(\mathcal{P}\) has a singularity at \(0\). Assume now for contradiction that \(\mathcal{P} \in V_{-1}\mathcal{D}(X)\). Then, we can write \(\mathcal{P} = t\mathcal{P}'\) for some \(\mathcal{P}' \in V_{0}\mathcal{D}(X)\), contradicting the irreducibility. This shows that \(\mathcal{P} \in V_0\mathcal{D}(X) - V_{-1}\mathcal{D}(X)\). 

The first two claims then follow from \Cref{lemma:locandroots} combined with the second point of \Cref{lemma:Pkgood}. For the third point, we begin by showing that \(t^{-1}\mathcal{P}_{-\alpha(\mathcal{P})}\) defines a section in \(\mathcal{D}(X)\), which is non-singular at the origin: Indeed, writing \(\mathcal{P}_{-\alpha(\mathcal{P})} = p(t)\partial_t + q^*(t)\) for a suitable \(q^* \in \mathcal{O}(X)\), we have \(q^*(0) = 0\). On the other hand, \(p(0) = 0\), because \(\mathcal{P} \in V_0\mathcal{D}(X) - V_{-1} \mathcal{D}(X)\), and so \(t^{-1}\mathcal{P}_{-\alpha(\mathcal{P})} \in \mathcal{D}(X)\). The multiplicity of the root \(0\) of \(p\) is \(1\), because \(0\) is a regular singularity of \(\mathcal{P}\) and \(\mathcal{P} \in V_0\mathcal{D}(X) - V_{-1} \mathcal{D}(X)\). This shows that \(t^{-1}\mathcal{P}_{-\alpha(\mathcal{P})}\) is non-singular at the origin, so that \(\mathcal{D}_X/\mathcal{D}_X(t^{-1}\mathcal{P}_{-\alpha(\mathcal{P})})\) is non-singular at the origin, which means that it is its own minimal extension there. We conclude with the calculation 
\begin{align*}
\mathcal{D}_X/\mathcal{D}_X(t^{-1}\mathcal{P}_{-\alpha(\mathcal{P})})(*0) & = \mathcal{D}_X(*0)/\mathcal{D}_X(*0)(t^{-1}\mathcal{P}_{-\alpha(\mathcal{P})}) \\ & = \mathcal{D}_X(*0)/\mathcal{D}_X(*0)\mathcal{P}_{-\alpha(\mathcal{P})} \\ & = \mathcal{D}_X(*0)/\mathcal{D}_X(*0)\mathcal{P} \\ & = (\mathcal{D}_X/\mathcal{D}_X\mathcal{P})(*0). 
\end{align*}

Here, the first equality holds because localisation is an exact functor. The second equality comes from \(t^{-1}\) being invertible in \(\mathcal{D}_X(*0)\) and the third equality follows from the second point of \Cref{lemma:Pkgood}. The last equality again holds because localisation is exact. This completes the proof. 
\end{proof}

\begin{lemma}\label{lemma:Mlocregular}
Let \(X \subset \mathbb{A}^1\) or \(X \subset \mathbb{C}\) be open with \(0 \in X\). Let \(\mathcal{P} \in \mathcal{D}(X)\) be irreducible, of order \(1\) and non-singular at the origin. Let \(\mathcal{M} \coloneqq \mathcal{D}_X/\mathcal{D}_X\mathcal{P}\). Let \(k \in \mathbb{Z}\). Then, the following statements are true: 

\begin{itemize}
    \item \(\mathcal{M}(*0) = \mathcal{D}_X/\mathcal{D}_X(t\mathcal{P})_k\) if and only if \(k < 0\). 
    \item \(\mathcal{M}(!0) = \mathcal{D}_X/\mathcal{D}_X(t\mathcal{P})_k\) if and only if \(k \geq 0\). 
    \item \(\mathcal{M}({!*}0) = \mathcal{M}\). 
\end{itemize}
\end{lemma}

\begin{proof}
Because \(\mathcal{P} = p(t)\partial_t + q(t)\) has no singularity at \(0\), we have \(p_0 \neq 0\). It follows that \(\mathcal{P} \in V_1\mathcal{D}_X - V_0 \mathcal{D}_X\), implying \(t\mathcal{P} \in V_0\mathcal{D}_X - V_{-1} \mathcal{D}_X\). Observing that \(\alpha(t\mathcal{P}) = 0\), the first two points follow directly from \Cref{lemma:MlocalphainZ}, if we can show that 
\begin{align*}
\mathcal{M}(*0) = \mathcal{D}_X/(\mathcal{D}_X \cdot t\mathcal{P})(*0) \quad \text{and} \quad \mathcal{M}(!0) = \mathcal{D}_X/(\mathcal{D}_X \cdot t\mathcal{P})(!0). 
\end{align*}

These equalities hold because \(t\) is invertible in \(\mathcal{D}_X(*0)\). The last point of the lemma is clear because \(\mathcal{P}\) is non-singular at \(0\). 
\end{proof}

\begin{remark}\label{remark:loccollection}(Summary)
Let \(X \subset \mathbb{A}^1\) or \(X \subset \mathbb{C}\) be open with \(0 \in X\). Let \(\mathcal{P} \in \mathcal{D}(X)\) be of order \(1\) and irreducible. We have now obtained simple formulas for \(\mathcal{D}_X/\mathcal{D}_X\mathcal{P}(*0)\), \(\mathcal{D}_X/\mathcal{D}_X\mathcal{P}(!0)\) and \(\mathcal{D}/\mathcal{D}\mathcal{P}({!*}0)\): 

\begin{itemize}
    \item If \(\mathcal{P}\) is non-singular at the origin, use \Cref{lemma:Mlocregular}. 
    \item If \(\mathcal{P}\) is regular singular at the origin with integer monodromy, use  \Cref{lemma:MlocalphainZ}. 
    \item If \(\mathcal{P}\) is regular singular at the origin with non-integer monodromy or if \(\mathcal{P}\) is irregular singular at the origin, use \Cref{lemma:MlocalphanotinZ}. 
\end{itemize}
\end{remark}

\begin{example}
We can now consider the example mentioned in the beginning of \Cref{sect:duallocmin}: Consider \(\mathcal{M}_k \coloneqq \mathcal{D}_{\mathbb{A}^1}/\mathcal{D}_{\mathbb{A}^1}(t\partial_t - k)\) for \(k \in \mathbb{Z}\). We have \(\mathcal{M}_k = \mathcal{M}_k(*0)\) if and only if \(k < 0\) and \(\mathcal{M}_k = \mathcal{M}_k(!0)\) if and only if \(k \geq 0\). 
\end{example}

If we assume that \(\mathcal{D}_X/\mathcal{D}_X\mathcal{P}\) and not just \(\mathcal{P}\) is irreducible, we obtain a strong result: 

\begin{lemma}\label{lemma:irrgivesloc}
Let \(X \subset \mathbb{A}^1\) or \(X \subset \mathbb{C}\) be open and let \(\mathcal{P} \in \mathcal{D}(X)\) be of order \(1\) with the property that \(\mathcal{M} \coloneqq \mathcal{D}_X/\mathcal{D}_X\mathcal{P}\) is irreducible. Then, 
\begin{align*}
\mathcal{M}({*}\operatorname{Sing(\mathcal{M})}) = \mathcal{M}. 
\end{align*}

In other words, \(\mathcal{M}\) is a meromorphic connection on \((X, \operatorname{Sing(\mathcal{M})})\). 
\end{lemma}

\begin{proof}
Let \(p \in \operatorname{Sing(\mathcal{M})}\), and without loss of generality assume \(p = 0\). If \(\mathcal{M}\) is irregular singular at \(0\) or regular singular at \(0\) with \(\alpha(\mathcal{P}) \notin \mathbb{Z}\), the claim is clear from \Cref{lemma:MlocalphanotinZ}. For contradiction, assume that \(\mathcal{M}\) is regular singular at \(0\) with \(\alpha(\mathcal{P}) \in \mathbb{Z}\). Because \(\mathcal{M}\) is irreducible, it is equal to its own minimal extension across the origin by \Cref{remark:irrimpliesmin}. However, \(\alpha(\mathcal{P}) \in \mathbb{Z}\) implies that the minimal extension is locally free in a neighborhood of \(0\), contradicting the assumption \(0 \in \operatorname{Sing}(\mathcal{M})\) by \Cref{lemma:MlocalphainZ}. 
\end{proof}

\section{Levelt-Turrittin decomposition} 

We will now recall the classical Levelt-Turrittin decomposition, cf. \cite{Lev75}. It allows to decompose the formalisation of a meromorphic connection (after possible ramification) into a direct sum of exponential twists of regular singular meromorphic connections. Here, formalization means the following: 

\begin{definition}
Let \(\mathcal{M}\) be a meromorphic connection on \(X\), where \(X\) is a small complex open disk around the origin with coordinate \(t\). Consider the ring of formal Laurent series \(\mathbb{C}((t))\) and the ring of formal differential operators \(\hat{\mathcal{D}} \coloneqq \mathbb{C}((t))\langle \partial_t \rangle\). The formalization of \(\mathcal{M}\) is the \(\hat{\mathcal{D}}\)-module 
\begin{align*}
\hat{\mathcal{M}} \coloneqq \mathbb{C}((t)) \otimes_{\mathcal{O}_{X, 0}} \mathcal{M} = \hat{\mathcal{D}} \otimes_{\mathcal{D}_{X, 0}} \mathcal{M}. 
\end{align*}
\end{definition}

We can now state the Levelt-Turrittin decomposition: 

\begin{theorem}[Levelt-Turrittin decomposition]
Let \(\mathcal{M}\) be a meromorphic connection on \(X\), where \(X\) is a small complex open disk around the origin with coordinate \(t\). There exists a positive integer \(r\) and a decomposition 
\begin{align*}
\rho^* \hat{\mathcal{M}} \cong \bigoplus_{p \in \mathcal{I}} (\mathcal{R}_p \otimes_{\mathbb{C}((t'))} \mathcal{E}^p), 
\end{align*}

where: 

\begin{itemize}
    \item \(\rho\) is the ramification map on \(X\) given by \(t = (t')^r\), where \(t'\) is the ramified coordinate. 
    \item \(\mathcal{I} \subset (t')^{-1} \mathbb{C}[(t')^{-1}]\) is a finite set, called the set of exponential factors. 
    \item Letting \(\hat{\mathcal{D}}_{t'} \coloneqq \rho^*\hat{\mathcal{D}} = \mathbb{C}((t'))\langle \partial_{t'} \rangle\), \(\mathcal{E}^p\) is the \(\hat{\mathcal{D}}_{t'}\)-module defined analogously to \Cref{def:exptwist}. 
    \item \(\mathcal{R}_p\) is a regular singular \(\hat{\mathcal{D}}_{t'}\)-module. 
\end{itemize}
\end{theorem}

We now ask how the Levelt-Turrittin decomposition simplifies if, instead of considering meromorphic connections on a small disk, we consider irreducible \(\mathcal{D}\)-modules of the form \(\mathcal{D}_{\mathbb{A}^1}/\mathcal{D}_{\mathbb{A}^1}\mathcal{P}\), where \(\mathcal{P} \in \mathcal{D}\) is of order \(1\). (These objects are algebraic connections by \Cref{lemma:irrgivesloc}.) It turns out that such objects can always be globally expressed as an exponential twist of a regular holonomic irreducible cyclic \(\mathcal{D}\)-module. As a preliminary, we point out some subtleties of the tensor product of \(\mathcal{D}\)-modules. 

\subsection{A note on tensor products and irreducibility} 
We point out some subtleties of the tensor product of \(\mathcal{D}\)-modules by considering the following situation. Let \(\mathcal{M}\) and \(\mathcal{N}\) be \(\mathcal{D}_X\)-modules with generic rank \(1\), and assume moreover that \(\mathcal{M}\) is irreducible. We ask if we have the following equivalence: 
\begin{align*}
\mathcal{N} \text{ is irreducible} \overset{?}{\iff} \mathcal{M} \otimes_{\mathcal{O}_{\mathbb{A}^1}} \mathcal{N} \text{ is irreducible}. 
\end{align*}

Neither of the implications holds: Indeed, let 
\begin{align*}
\mathcal{M} \coloneqq \mathcal{D}_{\mathbb{A}^1}/\mathcal{D}_{\mathbb{A}^1}(t\partial_t - \alpha), \quad \mathcal{N} \coloneqq \mathcal{D}_{\mathbb{A}^1}/\mathcal{D}_{\mathbb{A}^1}(t\partial_t + \alpha)
\end{align*}

for some \(\alpha \notin \mathbb{Z}\). Note that \(\operatorname{Sing}(\mathcal{M}) = \operatorname{Sing}(\mathcal{N}) = \{0\}\). \(\mathcal{M}\) and \(\mathcal{N}\) are irreducible, because they are meromorphic connections and they are equal to their own minimal extensions across the origin. Moreover, setting \(U \coloneqq \mathbb{A}^1 - \{0\}\), \(\mathcal{M}|_U\) and \(\mathcal{N}|_U\) are free of rank \(1\) over \(\mathcal{O}_U\), and 
\begin{align*}
\mathcal{M}|_U \otimes_{\mathcal{O}_{\mathbb{A}^1}} \mathcal{N}|_U \cong (\mathcal{O}_U, d). 
\end{align*}

However, \(\mathcal{M} \otimes_{\mathcal{O}_{\mathbb{A}^1}} \mathcal{N} \cong (\mathcal{O}_{\mathbb{A}^1}(*0), d)\). The minimal extension of \((\mathcal{O}_U, d)\) across the origin is \((\mathcal{O}_{\mathbb{A}^1}, d)\), so that \(\mathcal{M} \otimes_{\mathcal{O}_{\mathbb{A}^1}} \mathcal{N}\) is reducible. This shows that the implication from left to right does not hold. This phenomenon can be seen as "canceling out of monodromy". For the other direction, let 
\begin{align*}
\mathcal{M} \coloneqq \mathcal{D}_{\mathbb{A}^1}/\mathcal{D}_{\mathbb{A}^1}(t^2\partial_t + 1) \quad \mathcal{N} \coloneqq \mathcal{D}_{\mathbb{A}^1}/\mathcal{D}_{\mathbb{A}^1}(t\partial_t). 
\end{align*}

\(\mathcal{M}\) is irreducible, because it is of generic rank \(1\) and because its only singularity is irregular, whereas \(\mathcal{N}\) is reducible, because the operator \(t\partial_t\) is so. \(\mathcal{M} \otimes_{\mathcal{O}_{\mathbb{A}^1}} \mathcal{N}\) also has singular locus \(\{0\}\), where it is irregular singular, and it is also of generic rank \(1\), so it is irreducible. This shows that the implication from right to left does not hold either. 

\subsection{"Levelt-Turrittin" decomposition in the rank \(1\) case}

Keeping these subtleties in mind, we obtain the following rank \(1\) variant of the Levelt-Turrittin decomposition: 

\begin{lemma}\label{lemma:leveltrank1}
Let \(\mathcal{P} \in \mathcal{D}\) be of order \(1\) and such that \(\mathcal{M} \coloneqq \mathcal{D}_{\mathbb{A}^1}/\mathcal{D}_{\mathbb{A}^1}\mathcal{P}\) is irreducible. Then, there are \(\mathcal{P}_{\text{reg}} \in \mathcal{D}, g \in \mathbb{C}(t)\), such that 

\begin{itemize}
    \item \(\mathcal{D}_{\mathbb{A}^1}/\mathcal{D}_{\mathbb{A}^1}\mathcal{P}_{\text{reg}}\) is irreducible and regular holonomic, 
    \item \(\mathcal{D}_{\mathbb{A}^1}/\mathcal{D}_{\mathbb{A}^1}\mathcal{P}_{\text{reg}} \otimes_{\mathcal{O}_{\mathbb{A}^1}} \mathcal{E}^g = \mathcal{D}_{\mathbb{A}^1}/\mathcal{D}_{\mathbb{A}^1}\mathcal{P}\). 
\end{itemize}
\end{lemma}

\begin{proof}
Write \(\mathcal{P} = p(t)\partial_t + q(t)\) for \(p, q \in \mathbb{C}[t]\). We have \(\mathcal{M} = \mathcal{M}(*\operatorname{Sing}(\mathcal{P}))\) by \Cref{lemma:irrgivesloc}. So, \(\mathcal{M}\) can be viewed as the meromorphic connection 
\begin{align*}
\left(\mathcal{O}_{\mathbb{A}^1}(*\operatorname{Sing}(\mathcal{M})), d - \frac{q(t)}{p(t)}dt\right). 
\end{align*}

Write \(p(t)\) as \(\prod_{i = 1}^m (t - r_i)^{n_i}\). Via partial fraction decomposition, we can write 
\begin{align*}
\frac{q(t)}{p(t)} = E(t) + \sum_{i = 1}^m \frac{A_{i, 1}}{t - r_i} + \sum_{i = 1}^m \sum_{j = 2}^{n_i} \frac{A_{i, j}}{(t - r_i)^{j}}, 
\end{align*}

where \(E(t) \in \mathbb{C}[t], A_{i, j} \in \mathbb{C}\). Define \(f_{\text{reg}} = \sum_{i = 1}^m \frac{A_{i, 1}}{t - r_i}\) and 
\begin{align*}
g = \int E(t) + \sum_{i = 1}^m \sum_{j = 2}^{n_i} \frac{A_{i, j}}{(t - r_i)^{j}}
\end{align*}

with a chosen integration constant. Because the integrand does not contain any simple poles, the integral does not introduce any logarithms with branch cuts, and so we have \(g \in \Gamma(\mathbb{A}^1, \mathcal{O}_{\mathbb{A}^1}(*\operatorname{Sing}(\mathcal{M})))\). By construction, we have \((q(t)/p(t))dt = f_{\text{reg}}dt + dg\). It follows from \Cref{thm:checkregular} that \((\mathcal{O}_{\mathbb{A}^1}(\operatorname{Sing}(\mathcal{M}), d - f_{\text{reg}}dt)\) is regular holonomic. We have 
\begin{align*}
\mathcal{M} = \left(\mathcal{O}_{\mathbb{A}^1}(*\operatorname{Sing}(\mathcal{M})), d - f_{\text{reg}}dt\right) \otimes_{\mathcal{O}_{\mathbb{A}^1}} \left(\mathcal{O}_{\mathbb{A}^1}(*\operatorname{Sing}(\mathcal{M})), d - dg\right). 
\end{align*}

by \Cref{example:tensormer}. Because \(\mathcal{D}_{\mathbb{A}^1}/\mathcal{D}_{\mathbb{A}^1}\mathcal{P}\) is irreducible, it is equal to its minimal extension across \(\operatorname{Sing}(\mathcal{M})\). Let \(A\) and \(B\) be the sets of poles of \(f_{\text{reg}}dt\) and \(dg\), respectively. We have \(A \cup B = \operatorname{Sing}(\mathcal{M})\). Furthermore, let \(A' \subset A\) be the set of poles of \(f_{\text{reg}}dt\) with integer residues. We have \(A' \subset B\): Otherwise, \(\mathcal{M}\) would have a regular singularity with trivial monodromy, contradicting the irreducibility of \(\mathcal{M}\). Clearly, we have 
\begin{align*}
\mathcal{M} = \left(\mathcal{O}_{\mathbb{A}^1}(*A), d - f_{\text{reg}}dt\right) \otimes_{\mathcal{O}_{\mathbb{A}^1}} \mathcal{E}^g. 
\end{align*}

 Write \(f_{\text{reg}} = q_{\text{reg}}(t)/p_{\text{reg}}(t)\), where \(q_{\text{reg}}, p_{\text{reg}} \in \mathbb{C}[t]\) do not share any roots. Define \(\mathcal{P}_{\text{reg}}' \coloneqq p_{\text{reg}}(t)\partial_t + q_{\text{reg}}(t)\). We have 
\begin{align*}
\left(\mathcal{O}_{\mathbb{A}^1}(*A), d - f_{\text{reg}}dt\right) = \mathcal{D}_{\mathbb{A}^1}/\mathcal{D}_{\mathbb{A}^1}\mathcal{P}_{\text{reg}}'(*A) = \mathcal{D}_{\mathbb{A}^1}/\mathcal{D}_{\mathbb{A}^1}\mathcal{P}_{\text{reg}}'(*A'), 
\end{align*}

where the last equality follows from \(A' \subset A\) together with \Cref{lemma:MlocalphanotinZ}. Because we also have \(A' \subset B\), we get 
\begin{align*}
\mathcal{M} = \mathcal{D}_{\mathbb{A}^1}/\mathcal{D}_{\mathbb{A}^1}\mathcal{P}_{\text{reg}}' \otimes_{\mathcal{O}_{\mathbb{A}^1}} \mathcal{E}^g. 
\end{align*}

This is almost the claim of the lemma, except that \(\mathcal{D}_{\mathbb{A}^1}/\mathcal{D}_{\mathbb{A}^1}\mathcal{P}_{\text{reg}}'\) might be reducible: Indeed, \(\mathcal{D}_{\mathbb{A}^1}/\mathcal{D}_{\mathbb{A}^1}\mathcal{P}_{\text{reg}}'\) has singular locus \(A\), and it is not equal to its own minimal extension across \(A' \subset A\). If we can find a \(\mathcal{P}_{\text{reg}} \in \mathcal{D}\), such that 
\begin{align*}
\mathcal{D}_{\mathbb{A}^1}/\mathcal{D}_{\mathbb{A}^1}\mathcal{P}_{\text{reg}}'({!*}A') = \mathcal{D}_{\mathbb{A}^1}/\mathcal{D}_{\mathbb{A}^1}\mathcal{P}_{\text{reg}}, 
\end{align*}

the proof is complete, because then we would have 
\begin{align*}
\mathcal{D}_{\mathbb{A}^1}/\mathcal{D}_{\mathbb{A}^1}\mathcal{P}_{\text{reg}} \otimes_{\mathcal{O}_{\mathbb{A}^1}} \mathcal{E}^g = \mathcal{M}. 
\end{align*}

One constructs \(\mathcal{P}_{\text{reg}}\) by successively applying \Cref{lemma:MlocalphainZ} to every point in \(A'\), using suitable coordinate shifts. 
\end{proof}

\begin{example}
We will illustrate the procedure from the previous proof with the following calculation: Keep all notation from the previous proof. Let \(a \in \mathbb{C} - \mathbb{Z}\) and 
\begin{align*}
\mathcal{P} \coloneqq \underbrace{(t^5 - 4t^4 + 5t^3 - 2t^2)}_{\coloneqq p(t)}\partial_t - \underbrace{\left( (1+a)t^4 - (5+2a)t^3 + (8+a)t^2 - 5t + 2 \right)}_{\coloneqq q(t)}. 
\end{align*}

Via partial fraction decomposition, we obtain 
\begin{align*}
\frac{q(t)}{p(t)} = -\frac{1}{t^2} + \frac{1}{t-1} - \frac{1}{(t-1)^2} + \frac{a}{t-2}. 
\end{align*}

We see that \(\mathcal{P}\) has a regular singularity at \(t = 2\) with non-trivial monodromy. Furthermore, \(\mathcal{P}\) has irregular singularities at \(t = 0\) and \(t = 1\). These singularities are different from the perspective of the above procedure: Namely, the one at \(t = 0\) only has the pole order \(2\) contribution \(\frac{1}{t^2}\) and no pole order \(1\) contribution, whereas the one at \(t = 1\) has \(\frac{1}{t-1}\) as a pole order \(1\) contribution. We write 
\begin{align*}
\frac{q(t)}{p(t)}dt = \left(\frac{1}{t-1} + \frac{a}{t-2}\right)dt + d\left(\frac{1}{t} + \frac{1}{t-1}\right), 
\end{align*}

and we see \(A = \{1, 2\}\), \(B = \{0, 1\}\), \(A' = \{1\}\), \(g = \frac{1}{t} + \frac{1}{t-1}\). Furthermore, we have 
\begin{align*}
\mathcal{P}'_{\operatorname{reg}} = (t-1)(t-2)\partial_t + a(t-1) + t-2. 
\end{align*}

\(\mathcal{P}_{\operatorname{reg}}\) is seen to be \((t-2)\partial_t + a\). 
\end{example}

\section{Relations between algebraic and analytic \texorpdfstring{\(\mathcal{D}\)-modules}{D-modules} on \texorpdfstring{\(\mathbb{P}^1\)}{the Riemann sphere}}

Recall that, unless specified otherwise, we understand the complex projective line \(\mathbb{P}^1\) as being equipped with the analytic topology, whereas we understand the complex affine line \(\mathbb{A}^1\) as being equipped with the algebraic topology. Consider the projection \(p: \mathcal{P}^1 \rightarrow \{\operatorname{pt}\}\). Let \(\mathcal{M}\) be a coherent \(\mathcal{D}_{\mathbb{P}^1}(*\infty)\)-module. Since we have \(\Gamma(\mathbb{P}^1, \mathcal{D}_{\mathbb{P}^1}(*\infty)) = \mathcal{D}\), we get the \(\mathcal{D}\)-module \(p_*\mathcal{M} = \Gamma(\mathbb{P}^1, \mathcal{M})\). Furthermore, for a coherent \(\mathcal{D}\)-module \(\mathcal{N}\), we define 
\begin{align*}
p^{*, \operatorname{loc}}\mathcal{N} = \mathcal{D}_{\mathbb{P}^1}(*\infty) \otimes_{p^{-1}(\mathcal{D})} p^{-1}(\mathcal{N}). 
\end{align*}

This leads to the following well-known algebraization result (cf. \cite{DouSab03}, Appendix A.a): 

\begin{lemma}\label{lemma:DmodAP}
\(p_*\) induces an equivalence of categories 
\begin{align*}
\{\text{coherent } \mathcal{D}_{\mathbb{P}^1}(*\infty) \text{-modules}\} \longleftrightarrow \{\text{coherent } \mathcal{D} \text{-modules}\}
\end{align*}

with quasi-inverse \(p^{*, \operatorname{loc}}\). 
\end{lemma}

We will now make the last necessary preparations with \(\mathcal{D}\)-modules, before we look at Twistor \(\mathcal{D}\)-modules. We begin by recalling the definition of strictly supported \(\mathcal{D}\)-modules: 

\begin{definition}
Let \(X\) be a smooth complex variety or a complex manifold and let \(Y \subseteq X\) be a closed subset. Denote the category of irreducible (simple) holonomic \(\mathcal{D}_X\)-modules by \(\operatorname{Hol}_{\operatorname{s}}(X)\). A \(\mathcal{D}_X\)-module \(\mathcal{M}\) is said to be strictly supported on \(Y\), if \(\mathcal{M}\) has support \(Y\) and \(\mathcal{M}\) has no submodules or quotient modules supported on \(Z\), where \(Z \subsetneq Y\). We denote the category of simple holonomic \(\mathcal{D}_X\)-modules with strict support \(Y\) by \(\operatorname{Hol}_{Y, \operatorname{s}}(X)\), as in \cite{Moc10}, 19.4.1. 
\end{definition}

\Cref{lemma:DmodAP} has the following application: 

\begin{lemma}\label{lemma:projcanon}
Let \(\mathcal{P} \in \mathcal{D} = \mathbb{C}\langle t, \partial_t \rangle\) be of positive order and such that \(\mathcal{M} \coloneqq \mathcal{D}/\mathcal{D}\mathcal{P}\) is irreducible. Then, 
\begin{align*}
\mathfrak{M} \coloneqq (p^{*, \operatorname{loc}}\mathcal{M})({!*}\infty) \in \operatorname{Hol}_{\mathbb{P}^1, \operatorname{s}}(\mathbb{P}^1). 
\end{align*}
\end{lemma}

\begin{proof}
The \(\mathcal{D}_{\mathbb{A}^1}\)-module \(\mathcal{M}\) is clearly holonomic. It also has support \(\mathbb{A}^1\), because the order of \(\mathcal{P}\) is not \(0\). Because \(\mathcal{M}\) is irreducible, so is \(\mathfrak{M}\), cf. \cite{Hot18}, Theorem 3.4.2 (i). The support of \(\mathfrak{M}\) is \(\mathbb{P}^1\) (because support always has to be a closed subset), and because \(\mathfrak{M}\) is irreducible, this is automatically the strict support. The holonomicity of \(\mathcal{M}\) implies holonomicity of \(\mathfrak{M}\), cf. \cite{Hot18}, Proposition 3.1.2. 
\end{proof}

\begin{remark}\label{remark:projgaga}
Let \(X\) be a projective smooth complex algebraic variety and let \(Z \subseteq X\) be a closed subset. Because of the GAGA result for \(\mathcal{D}_X\)-modules on smooth projective varieties (cf. \cite{Mal04}), we have the equivalence of categories 
\begin{align*}
\operatorname{Hol}_{Z, \operatorname{s}}(X) \cong \operatorname{Hol}_{Z^{\operatorname{an}}, \operatorname{s}}(X^{\operatorname{an}}), 
\end{align*}

induced by the analytification functor. 
\end{remark}

\begin{remark}\label{remark:meroholon}
If \(X\) is a compact complex manifold (or a projective smooth complex variety) of dimension \(1\), there is the following well-known equivalence of categories: 

\[\begin{tikzcd}
	{\text{simple holonomic } \mathcal{D}_X\text{-modules with supp. }X} \\
	{\text{simple meromorphic connections on \(X\)}}
	\arrow[tail reversed, from=1-1, to=2-1]
\end{tikzcd}\]

(cf. \cite{Sab05}, section 5.2.a), preserving the singular loci. From top to bottom, it is given by localisation away from the singular locus and the quasi-inverse is given by minimal extension along the singular locus. 
\end{remark}

\chapter{Twistor \texorpdfstring{\(\mathcal{D}\)-modules}{D-modules} of generic rank \texorpdfstring{\(1\)}{1}}\label{chapter:rank1twistor}

\section{Statement of the main result}

Let \(X\) be a complex projective manifold and let \(Z \subset X\) be a submanifold. For any integer \(w\), let \(\operatorname{MT}^{\operatorname{wild}}_{Z, \operatorname{s}}(X, w, \sqrt{-1}\mathbb{R})^{(p)}\) be the category of simple polarizable \(\sqrt{-1}\mathbb{R}\)-wild pure Twistor \(\mathcal{D}\)-modules of weight \(w\) on \(X\) with strict support \(Z\). (This notion comes from \cite{Moc10}, 19.4.1.) For now, we will not explain the definition of \(\operatorname{MT}^{\operatorname{wild}}_{Z, \operatorname{s}}(X, w, \sqrt{-1}\mathbb{R})^{(p)}\). Instead, we will successively explain this category over the course of this section. The starting point for the remainder of this thesis is the following fundamental result: 

\begin{theorem}[Mochizuki, \cite{Moc10}, Theorem 19.4.1]\label{theorem:mocmainres}
The functor \(\Xi_{\operatorname{DR}}\) (taking the underlying \(\mathcal{D}\)-module) naturally induces an equivalence of categories 
\begin{align*}
\Xi_{\operatorname{DR}}: \operatorname{MT}^{\operatorname{wild}}_{Z, \operatorname{s}}(X, w, \sqrt{-1}\mathbb{R})^{(p)} \overset{\cong}{\rightarrow} \operatorname{Hol}_{Z, \operatorname{s}}(X) 
\end{align*}
\end{theorem}

Mochizuki proves this result by defining a quasi-inverse functor to \(\Xi_{\text{DR}}\). Starting from an object in \(\text{Hol}_{Z, \text{s}}(X)\), the corresponding object in \(\text{MT}^{\text{wild}}_{Z, \text{s}}(X, w, \sqrt{-1}\mathbb{R})^{(p)}\) is constructed in multiple steps. This spans hundreds of pages in \cite{Moc10}. 

The goal of this chapter is to make Mochizuki's equivalence more explicit in the case \(X = Z = \mathbb{P}^1\): We start with an object in \(\operatorname{Hol}_{Z, \operatorname{s}}(X)\) of generic rank \(1\) and we will determine the corresponding object of \(\operatorname{MT}^{\operatorname{wild}}_{Z, \operatorname{s}}(X, 0, \sqrt{-1}\mathbb{R})^{(p)}\). This gives the main result of this section, which will be stated now. We have used several notations which have not yet been explained in this thesis, but all of which are standard in the existing literature \cite{Sab05}, \cite{Moc10} etc. All of these notions will be explained over the course of this chapter. 

\begin{theorem}\label{theorem:firstmainresult}
Let \(\mathcal{P} \in \mathcal{D} = \mathbb{C}\langle t, \partial_t \rangle\) be of order \(1\) and let \(\mathcal{M} \coloneqq \mathcal{D}/(\mathcal{P})\). Assume that the \(\mathcal{D}_{\mathbb{P}^1}\)-module \(\mathfrak{M} \coloneqq (p^{*, \operatorname{loc}}\mathcal{M})({!*}\infty)\) is irreducible and regular holonomic. Write \(\mathcal{P} = p(t)\partial_t + q(t)\) and 
\begin{align*}
q(t)/p(t) = -\sum_i \frac{m_i}{t - r_i}
\end{align*}
for \(m_i \in \mathbb{C} - \mathbb{Z}, r_i \in \mathbb{C}\). Set \(D = \operatorname{Sing}(\mathfrak{M})\), \(U = \mathbb{P}^1 - D\). Then, the unique object 
\begin{align*}
\mathfrak{T} = ((\mathfrak{M}_z, \mathfrak{M}_z, \mathfrak{C}), (\operatorname{id}, \operatorname{id})) \in \operatorname{MT}^{\operatorname{wild}}_{\mathbb{P}^1, \operatorname{s}}(\mathbb{P}^1, 0, \sqrt{-1}\mathbb{R})^{(p)}
\end{align*}
with \(\Xi_{\operatorname{DR}}(\mathfrak{T}) = \mathfrak{M}\) is given as follows: \(\mathfrak{M}_z = (p^{*, \operatorname{loc}}\mathcal{M}_{z})[{!*}\infty]\), where \(\mathcal{M}_{z} \coloneqq \mathscr{R}/(\mathcal{P}_{z})\) - we define \(\mathscr{R} = \mathcal{O}_{\mathbb{C}_z}\langle t, z\partial_t \rangle\) - with 
\begin{align*}
\mathcal{P}_{z} \coloneqq p(t)z\partial_t - p(t) \sum_{j = 1}^n \frac{m_j \star z}{t - r_j} \qquad \left(m_j \star z \coloneqq \Re\left(m_j\right)z + \Im\left(m_j\right)\frac{z^2 + 1}{2}i\right). 
\end{align*}
Furthermore, the pairing \(\mathfrak{C} : \overline{\sigma^*}(\mathfrak{M}_{z}\vert_{\boldsymbol{S}}) \otimes \mathfrak{M}_{z}\vert_{\boldsymbol{S}} \rightarrow \mathfrak{Db}_{\mathbb{P}^1 \times \boldsymbol{S}/\boldsymbol{S}}\) - it is a pairing between \(\mathfrak{M}_{z}\) and its Twistor-conjugate, both restricted to \(|z| = 1\), and it takes values in distributions - is a natural extension of the pairing 
\begin{align*}
\mathfrak{C}\vert_{U}: \overline{\sigma^*}(\mathfrak{M}_{z}\vert_{U \times \boldsymbol{S}}) \otimes \mathfrak{M}_{z}\vert_{U \times \boldsymbol{S}} & \rightarrow C^{\infty, \operatorname{an}}_{U \times \boldsymbol{S}} \\ \overline{\sigma^*}(u \cdot e) \otimes (v \cdot e) & \mapsto \overline{\sigma^*}u \cdot \prod_{j = 1}^n \vert t - r_j \vert^{2(m_j \star z)/z} \cdot v, 
\end{align*}
where \(u, v \in \mathcal{O}(U \times \boldsymbol{S})\) and where \(e\) is the \(\mathcal{O}_{U \times \boldsymbol{S}}\)-generator of \(\Gamma(U \times \boldsymbol{S}, \mathfrak{M}_{z})\) obtained as the image of the class of \(1\) under the \(t\)-analytification map 
\begin{align*}
\mathscr{R}(U^{\operatorname{alg}} \times \boldsymbol{S})/(\mathcal{P}_z) = \mathscr{R}(U^{\operatorname{alg}} \times \boldsymbol{S}) \otimes_{\mathscr{R}(\boldsymbol{S})} \Gamma(\boldsymbol{S}, \mathcal{M}_z) \mapsto \Gamma(U \times \boldsymbol{S}, \mathfrak{M}_{z}). 
\end{align*}
\end{theorem}

\begin{remark}
We make some comments on \Cref{theorem:firstmainresult}: 
\begin{itemize}
    \item We have \(\mathfrak{M} \in \operatorname{Hol}_{\mathbb{P}^1, \operatorname{s}}(\mathbb{P}^1)\) by \Cref{lemma:projcanon}, so that \(\mathfrak{T}\) is guaranteed to exist and be uniquely defined. 
    \item We can really write \(q(t)/p(t) = -\sum_i m_i/(t - r_i)\), as explained in the proof of \Cref{lemma:regsol}. 
    \item We really have \(\mathcal{P}_{z} \in \mathcal{R}\): Even though there are terms like \(t^{-1}\) in the definition of \(\mathcal{P}_z\), these poles are canceled out, because they are multiplied by \(p(t)\). 
    \item If all the \(m_j\) are real, we have \(m_j \star z = m_j z\), leading to \(\mathcal{P}_z = z\mathcal{P}\). 
    \item \(\mathfrak{T}\vert_U\) is a polarized variation of Twistor structure in the sense of \Cref{def:vartwiststr}, because \(\mathfrak{M}\) is non-singular over \(U\). 
\end{itemize}
\end{remark}

Before going into the proof of \Cref{theorem:firstmainresult}, we will state the following consequence: 

\begin{remark}
Let \(\mathfrak{M}\) be any (possibly non-regular) object in \(\operatorname{Hol}_{\mathbb{P}^1, \operatorname{s}}(\mathbb{P}^1)\) of generic rank \(1\). One can show that \(\mathfrak{M}\) is of the form \((p^{*, \operatorname{loc}}(\mathcal{D}/(\mathcal{P})))({!*}\infty)\) for some \(\mathcal{P} \in \mathcal{D}\). (This only works because the generic rank of \(\mathfrak{M}\) is assumed to be \(1\).) By \Cref{lemma:leveltrank1}, we can find \(\mathcal{P}_{\operatorname{reg}} \in \mathcal{D}\) of order \(1\) and \(g \in \mathbb{C}(t)\) so that: 

\begin{itemize}
   \item \(\mathfrak{M}_{\operatorname{reg}} \coloneqq (p^{*, \operatorname{loc}}(\mathcal{D}/(\mathcal{P}_{\operatorname{reg}})))({!*}\infty)\) is irreducible and regular holonomic, 
   \item \((p^{*, \operatorname{loc}}(\mathcal{D}/(\mathcal{P}_{\operatorname{reg}}) \otimes \mathcal{E}^g))({!*}\infty) = \mathfrak{M}\). 
\end{itemize}

We can apply \Cref{theorem:firstmainresult} to \(\mathfrak{M}_{\operatorname{reg}}\), giving us a Twistor \(\mathcal{D}\)-module \(\mathfrak{T}_{\operatorname{reg}}\). There is a notion of exponential twist of Twistor \(\mathcal{D}\)-modules, which is known to be compatible with \(\Xi_{\operatorname{DR}}\). Namely, we have 
\begin{align*}
\Xi_{\operatorname{DR}}(\mathfrak{T}_{\operatorname{reg}} \otimes \mathcal{E}^{g/z}) = \Xi_{\operatorname{DR}}(\mathfrak{T}_{\operatorname{reg}}) \otimes \mathcal{E}^{g} = \mathfrak{M}. 
\end{align*}

This gives a generalisation of \Cref{theorem:firstmainresult} to any object in \(\operatorname{Hol}_{\mathbb{P}^1, \operatorname{s}}(\mathbb{P}^1)\) of generic rank \(1\). 
\end{remark}

\section{Proof strategy}

We give an outline of the proof of \Cref{theorem:firstmainresult}: In the language of \Cref{theorem:mocmainres}, we need to take \(X = Z = \mathbb{P}^1\). Because \(\mathbb{P}^1\) is a curve and because we assume \(\mathfrak{M}\) to be regular holonomic, we do not need to consider the full construction of the quasi-inverse to \(\Xi_{\operatorname{DR}}\) as constructed in \cite{Moc10}, but instead this construction reduces to the special case of Sabbah's preceding work \cite{Sab05}, Theorem 5.0.1. This result shows the following: If we only consider simple regular holonomic \(\mathcal{D}_X\)-modules with support \(X\), and if \(X\) is a curve, then the procedure of obtaining \(\mathfrak{T}\) from \(\mathfrak{M}\) takes the following form: 

\newpage

\[\begin{tikzcd}
	\begin{array}{c} \text{simple holonomic } \mathcal{D}_X\text{-module} \\ \text{ with singular locus } D \\ \text{and support }X  \end{array} && {\mathfrak{M}} \\
	{\text{simple meromorphic bundle on } (X, D)} && {\widetilde{\mathfrak{M}}} \\
	{\text{tame harmonic bundle on } X - D} && {(V, \nabla, h)} \\
	\begin{array}{c} \text{polarised variation of Twistor} \\ \text{structure of weight } 0 \text{ on }  X - D \end{array} && {(V_z, V_z, \pi^* h)} \\
	\begin{array}{c} \text{meromorphic bundle with } \\ z \text{-connection on } (X, D) \end{array} && {\widetilde{\mathfrak{M}_z}} \\
	{\mathscr{R}_{\mathscr{X}}\text{-module}} && {\mathfrak{M}_z} \\
	\begin{array}{c} \text{simple polarizable} \\ \text{tame purely imaginary} \\ \text{Twistor } \mathcal{D} \text{-module on} \\ X \text{ of weight 0} \end{array} && {\mathfrak{T} = ((\mathfrak{M}_z, \mathfrak{M}_z, \mathfrak{C}), (\text{id, \text{id}}))}
	\arrow["{\text{Step 1: Localisation}}", from=1-1, to=2-1]
	\arrow[from=1-3, to=2-3]
	\arrow["\begin{array}{c} \text{Step 2: Harmonic metric} \\ \text{of Deligne type} \end{array}", from=2-1, to=3-1]
	\arrow[maps to, from=2-3, to=3-3]
	\arrow["{\text{Step 3: Twistor-deformation}}", from=3-1, to=4-1]
	\arrow[maps to, from=3-3, to=4-3]
	\arrow["\begin{array}{c} \text{Step 4: Filtering by growth} \\ \text{with respect to } h  \text{ near } D \end{array}", from=4-1, to=5-1]
	\arrow[maps to, from=4-3, to=5-3]
	\arrow["{\text{Step 5: Minimal extension}}", from=5-1, to=6-1]
	\arrow[maps to, from=5-3, to=6-3]
	\arrow["\begin{array}{c} \text{Step 6: Pairing induced by } h \\ \text{and trivial polarization} \end{array}", from=6-1, to=7-1]
	\arrow[maps to, from=6-3, to=7-3]
\end{tikzcd}\]

We successively explain these steps in general and then apply them in order to prove \Cref{theorem:firstmainresult}. 

\section{Step 1: Localisation}

Let \(X\) be a complex curve and let \(\mathfrak{M} \in \operatorname{Hol}_{X, \operatorname{s}}(X)\). Setting \(D \coloneqq \operatorname{Sing}(\mathfrak{M})\), the first step consists of passing to the localisation \(\widetilde{\mathfrak{M}} \coloneqq \mathfrak{M}(*D)\). 

\begin{remark}
Recall the equivalence 

\[\begin{tikzcd}
	{\text{simple holonomic } \mathcal{D}_X\text{-modules with supp. }X} \\
	{\text{simple meromorphic connections on \(X\)}, }
	\arrow[tail reversed, from=1-1, to=2-1]
\end{tikzcd}\]

from \Cref{remark:meroholon}, which explains why no information about \(\mathfrak{M}\) is lost in Step \(1\). 
\end{remark}

\subsection{Step 1 in the order \(1\) case}

We make the assumptions of \Cref{theorem:firstmainresult}. Because we assume that \(\mathcal{P}\) is of order \(1\), we obtain \(\widetilde{\mathfrak{M}} = \mathfrak{M}\) by \Cref{lemma:irrgivesloc}. 

\section{Step 2: Harmonic metric of Deligne type}

Let \(X\) be a projective manifold and let \(\widetilde{\mathfrak{M}}\) be an irreducible meromorphic connection on \((X, D)\) that only has regular singularities. In the following, we will want to equip the holomorphic connection \(\widetilde{\mathfrak{M}}\vert_{X - D}\) with additional structure, namely with a harmonic metric. We begin by recalling the definition of harmonic metrics on holomorphic connections from \cite{SabSch}. 

\subsection{Introduction to harmonic metrics}

Let \(U\) be a complex manifold. For any holomorphic bundle \(V\), there is the associated smooth bundle \(H = C_U^{\infty} \otimes_{\mathcal{O}_U} V\). If \(V\) is equipped with a holomorphic flat connection \(\nabla\), then there is the induced flat connection \(D_V(f \otimes s) \coloneqq df \otimes s + f \otimes \nabla s\) on \(H\). There exists the decomposition 
\begin{align*}
D_V = \underbrace{D'_V}_{(1,0)\text{-part}} + \underbrace{D''_V}_{(0,1)\text{-part}}. 
\end{align*}

We recover \((V, \nabla)\) as \((\ker(D''), D'\vert_{\ker(D'')})\) by construction. 

\begin{example}
The prototypical example is the following: If \((V, \nabla) = (\mathcal{O}_\mathbb{C}, d)\), then \(H = C_\mathbb{C}^{\infty}\) and \(D_V = \partial + \overline{\partial}\). 
\end{example}

Suppose \(V\) is additionally equipped with a hermitian metric \(h\), that is, \(h\) is a sesquilinear pairing 
\begin{align*}
h: \overline{V} \otimes V \rightarrow C_U^\infty
\end{align*}

satisfying \(h(u, v) = \overline{h(v, u)}\) for any sections \(u, v\) of \(V\). Then, there exists a unique decomposition \(D_V = D_E + \theta\) descending to \((1,0)\)- and \((0,1)\)-parts, such that \(D_E\) is a connection (not necessarily flat) and \(\theta\) is \(C_U^\infty\)-linear satisfying the following conditions: 

\begin{itemize}
    \item \(h\) is \(D_E\)-flat 
    \item \(\theta = \theta' + \theta''\) is self-adjoint with respect to \(h\) 
\end{itemize} 

This is proved in \cite{SabSch}, 4.2.2. 

\begin{remark}
Going from \(D_V\) to \(D_E\) can be seen as the following trade-off: One sacrifices the flatness of \(D_V\) in order to obtain the flatness of \(h\) (with respect to \(D_E\)). 
\end{remark}

\begin{definition}
\((V, \nabla, h)\) is called a harmonic bundle and \(h\) is called a harmonic metric, if \(D''_E + \theta'\) has square \(0\). 
\end{definition}

\begin{remark}\label{remark:hshortcoming}
One can argue that, from the perspective of the meromorphic bundle \(\widetilde{\mathfrak{M}}\), the datum of a harmonic metric \(h\) on \(\widetilde{\mathfrak{M}}\vert_{X - D}\) has the following disadvantages: 

\begin{itemize}
    \item Such a harmonic metric might not be unique, so it can not be used to canonically define \(\mathfrak{T}\). 
    \item \(h\) is not related in any clear way to the singular information of \(\widetilde{\mathfrak{M}}\), because the conditions satisfied by \(h\) only depend on \(\widetilde{\mathfrak{M}}\) through its restriction to \(X - D\). For example, assume that \(X = \mathbb{C}\), \(0 \in D\) and consider \(\widetilde{\mathfrak{M}}' \coloneqq \widetilde{\mathfrak{M}} \otimes \mathcal{E}^{1/t}\), which is also a meromorphic bundle on \((X, D)\). While \(\widetilde{\mathfrak{M}}\) was assumed to be regular singular at \(0\), \(\widetilde{\mathfrak{M}}'\) has an irregular singularity at \(0\). Nevertheless, any harmonic metric on \(\widetilde{\mathfrak{M}}\vert_{X - D}\) is automatically also a harmonic metric on \(\widetilde{\mathfrak{M}}'\vert_{X - D}\), because \(\widetilde{\mathfrak{M}}'\vert_{X - D} \cong \widetilde{\mathfrak{M}}\vert_{X - D}\). 
    \item Harmonic metrics are not "elegant" in the sense that they are in general not flat with respect to the original holomorphic connection induced by \(\widetilde{\mathfrak{M}}\vert_{X - D}\). This leads to some lack of functoriality, for example with respect to Fourier transform, as we will discuss later. 
\end{itemize}
\end{remark}

All three "issues" are addressed by the theory of Twistor \(\mathcal{D}\)-modules. The third point in this list will be addressed at a later stage, and the first two points are solved at once by considering the following special type of harmonic metrics: 

\subsection{Harmonic metrics of Deligne type}

Denote the holomorphic connection obtained by restricting \(\widetilde{\mathfrak{M}}\) to \(X - D\) by \((V, \nabla)\). For a harmonic metric on \((V, \nabla)\), we introduce the following compatibility condition with the singular structure of \(\widetilde{M}\): 

\begin{definition}
Consider the inclusion \(j: X - D \hookrightarrow X\). A harmonic metric \(h\) on \((V, \nabla)\) is said to be of Deligne type (with respect to \(\widetilde{\mathfrak{M}}\)), if, for every \(p \in D\), the parabolic filtration of \(j_*V\) induced by \(h\) coincides with the image of the descending version of the Kashiwara-Malgrange filtration of \(\widetilde{\mathfrak{M}}\) under the canonical inclusion \(\widetilde{\mathfrak{M}} \hookrightarrow j_*V\). Here, if \(V_{\bullet}^{(p)}\widetilde{\mathfrak{M}}\) is the Kashiwara-Malgrange filtration of \(\widetilde{\mathfrak{M}}\) with respect to \(p\), then the descending version of this filtration is defined via 
\begin{align*}
V^{\bullet}\widetilde{\mathfrak{M}} \coloneqq V_{-\bullet - 1}^{(p)}\widetilde{\mathfrak{M}}, 
\end{align*}

cf. \cite{Sab05}, Remark 3.1.6(2). Furthermore, the parabolic filtration of \(\tilde{V} \coloneqq j_*V\) induced by \(h\) is defined as follows: Near a singularity \(t_0 \in D\) with local coordinate \(t'\) centered at \(t_0\), the parabolic filtration is given by 
\begin{align*}
P^b_{(t_0)}\tilde{V} \coloneqq \{s \in j_*V_{t_0} : \lim_{t' \rightarrow 0}|t'|^{-b+\epsilon}||s(t')||_h = 0 \text{ for some } \epsilon > 0 \text{ with } \epsilon \ll 1\} 
\end{align*}

for \(b \in \mathbb{R}\) (cf. \cite{Sab05}, 5.2.c). \(P^{\bullet}_{(t_0)}\tilde{V}\) simply filters by the growth of sections of \(j_*V\) with respect to \(h\) near \(t_0\). 
\end{definition}

\begin{remark}
We note that \(\widetilde{\mathfrak{M}}\) is really equipped with a Kashiwara-Malgrange filtration, because \(\widetilde{\mathfrak{M}}\) is assumed to be holonomic, and so it is specializable. 
\end{remark}

We need one more condition on harmonic metrics: 

\begin{definition}
Let \(X\) be a projective manifold. A harmonic metric \(h\) on a holomorphic bundle \((V, \nabla)\) (on \(X - D\)) is called tame if the \(h\)-norm of \(\nabla\)-flat sections of \(V\) grows at most polynomially near \(D\). \((V, \nabla, h)\) is then called a tame harmonic bundle. 
\end{definition}

We have the following crucial result: 

\begin{theorem}[\cite{Sab05}, Corollary 5.2.4]\label{thm:delunique}
Let \(X\) be a projective curve with a reduced divisor \(D \subset X\). For every irreducible meromorphic connection \(\widetilde{\mathfrak{M}}\) on \((X, D)\), which only has regular singularities, there is a unique harmonic metric \(h\) (up to multiplication by a positive constant) on \((V, \nabla) \coloneqq \widetilde{\mathfrak{M}}\vert_{X - D}\) such that \((V, \nabla, h)\) is a tame harmonic bundle of Deligne type with respect to \(\widetilde{\mathfrak{M}}\). 
\end{theorem}

\begin{definition}
Let \(X\) be a complex projective curve. Step 2 of constructing \(\mathfrak{T}\) from \(\mathfrak{M}\) is defined as equipping the irreducible connection \(\widetilde{\mathfrak{M}}\) on \((X, D)\) with its tame harmonic bundle of Deligne type, given by \Cref{thm:delunique}. 
\end{definition}

\subsection{Step 2 in the order \(1\) case} 

We use the notation from \Cref{theorem:firstmainresult}. We have already seen that the localisation \(\widetilde{\mathfrak{M}}\) of \(\mathfrak{M}\) at the singular locus is identical to \(\mathfrak{M}\) itself. We will now determine the associated harmonic metric of Deligne type. Set \((V, \nabla) \coloneqq \widetilde{\mathfrak{M}}\vert_{\mathbb{C} - D}\). Recall from \Cref{lemma:regsol} that the solutions of \(\mathcal{P}\) are spanned by \(g(t) \coloneqq \Pi_{i = 1}^n(t - r_i)^{m_i} 
\), that the local system of flat sections of \((V, \nabla)\) is \(\underline{\mathbb{C}} \cdot g(t)^{-1} 
\) and that we have 
\begin{align*}
(V, \nabla) \cong \left(\mathcal{O}_{\mathbb{C} - D}, d - \sum_{i = 1}^n \frac{-m_i}{t - r_i}dt\right). 
\end{align*}

We obtain a harmonic metric on \((V, \nabla)\) as follows: 

\begin{theorem}
Then, for any \(c \in \mathbb{R}_+\), a harmonic metric on \((V, \nabla)\) is given by 
\begin{align*}
h(u, \overline{v}) \coloneqq u \cdot c \prod_{j = 1}^n|t - r_j|^{2\Re(m_j)} \cdot \overline{v} 
\end{align*}

for any local sections \(u, v\) of \(C^\infty_{\mathbb{C} - D} \otimes V\). We have the following associated objects: 
\begin{align*}
\theta' & = \frac{1}{2}\sum_{j = 1}^n \frac{\Im(m_j)i}{t - r_j}dt & D_E' & = D_V^{(1,0)} - \theta' \\ \theta'' & = \frac{1}{2}\sum_{j = 1}^n \frac{-\Im(m_j)i}{\overline{t - r_j}} d\overline{t} & D_E'' & = \overline{\partial} - \theta''
\end{align*}
\end{theorem}

\begin{proof}
Denote the connection on \(C^\infty_{\mathbb{C} - D} \otimes V\) induced by \(\nabla\) as \(D_V\). We have \(D_V = \nabla + \overline{\partial}\). We define \(D_E', D_E'', \theta', \theta''\) as in the statement of the result and we will now show that all the conditions of harmonic bundles are satisfied. Firstly, note that we have the desired splitting \(D_V = D_E' + D_E'' + \theta' + \theta''\). Furthermore, \(D_E', \theta'\) have type \((1,0)\) and \(D_E'', \theta''\) have type \((0,1)\). Let \(\rho(t) \coloneqq \Pi_{j = 1}^n(t - r_j)^{\Re(m_j)}\). \(\theta'\) is adjoint to \(\theta''\) with respect to \(h\): Indeed, since \(\theta'' = \overline{\theta'}\) and since the rank of \(V\) is \(1\), we have
\begin{align*}
h(\theta' u, \overline{v}) = \theta' u |\rho|^2 \overline{v} = u |\rho|^2 \theta'\overline{v} = h(u, \overline{\overline{\theta''}v}). 
\end{align*}

We have the Chern-connection of \(h\) on \(H\), that is, the connection \(\text{Ch}_h = \text{Ch}_h^{(1,0)} + \text{Ch}_h^{(0,1)}\), uniquely characterized by the property of \(h\)-flatness and \(\text{Ch}_h^{(0,1)} = \overline{\partial}\). It is given by 
\begin{align*}
\text{Ch}_h^{(1,0)} = \partial - \sum_{j = 1}^n \frac{-\Re(m_j)}{t - r_j}dt, \quad \text{Ch}_h^{(0,1)} = \overline{\partial}. 
\end{align*}

Consider the difference 
\begin{align*}
D_V - \text{Ch}_h = \sum_{j = 1}^n \frac{\Im(m_j)i}{t - r_j}dt. 
\end{align*}

Splitting this expression into a self-adjoint part \(A^+\) and a skew-adjoint part \(A^-\) (with respect to \(h\)), we obtain 
\begin{align*}
A^+ = \underbrace{\frac{1}{2}\sum_{j = 1}^n \frac{\Im(m_j)i}{t - r_j} dt}_{= \theta'} + \underbrace{\frac{1}{2}\sum_{j = 1}^n \frac{-\Im(m_j)i}{\overline{t - r_j}} d\overline{t}}_{= \theta''}. 
\end{align*}

and 
\begin{align*}
A^- = \underbrace{\frac{1}{2}\sum_{j = 1}^n \frac{\Im(m_j)i}{t - r_j} dt}_{= \theta'} - \underbrace{\frac{1}{2}\sum_{j = 1}^n \frac{-\Im(m_j)i}{\overline{t - r_j}} d\overline{t}}_{= \theta''}. 
\end{align*}

It follows that \(D_E' + D_E'' = \text{Ch}_h + A^-\) is also compatible with \(h\), cf. \cite{SabSch}, proof of Lemma 4.2.2. We have seen that all harmonicity conditions are fulfilled and the claim is shown. 
\end{proof}

In the following, we make the choice \(c = 1\). 

\begin{remark}
Recall that \(h\) is called tame with respect to \(\nabla\) if and only if the \(h\)-norm of horizontal sections grows at most polynomially near the punctures. This is fulfilled in our case: We have 
\begin{align*}
||\Pi_{j = 1}^n(t - r_j)^{-m_j}||_h & = \left|\Pi_{j = 1}^n(t - r_j)^{-\Im(m_j)i}\right|. 
\end{align*}

This expression is bounded from above (in particular of moderate growth) in any bounded sector, because the exponents are purely imaginary. 
\end{remark}

\begin{proposition}
The harmonic metric \(h\) on \((V, \nabla)\) is of Deligne type with respect to \(\widetilde{\mathfrak{M}}\). 
\end{proposition}

\begin{proof}
Let \(j: \mathbb{P}^1 - D \hookrightarrow \mathbb{P}^1\) be the inclusion. Recall that, near a singularity \(t_0 \in D\) with local coordinate \(t'\) centered at \(t_0\), the parabolic filtration is defined as 
\begin{align*}
P^b_{(t_0)}\tilde{V} \coloneqq \{s \in j_*V_{t_0} : \lim_{t' \rightarrow 0}|t'|^{-b+\epsilon}||s(t')||_h = 0 \text{ for some } \epsilon > 0 \text{ with } \epsilon \ll 1\} 
\end{align*}

for \(b \in \mathbb{R}\). We consider the situation locally at one singularity, which we may assume to be the origin after a suitable coordinate change. Define \(g(t) = \prod_{j = 1}^n (t - r_j)^{m_j}\) and note that \(g\) has a root of order \(\alpha \coloneqq \alpha(\mathcal{P})\) at the origin. Comparing with the definition of \(h\), one obtains  
\begin{align}\label{eq:parabolicformula}
P^b_{(0)}\tilde{V} = t^{\lceil b-\alpha \rceil}\mathcal{O}_{\mathbb{C}, 0} \subset (j_*\mathcal{O}_{\mathbb{P}^1 - D})_{0} = (j_*V)_{0} 
\end{align}

for any \(b \in \mathbb{R}\). On the other hand, recall from \Cref{lemma:VawayfromD} that the (ascending) \(V\)-filtration of \(\widetilde{\mathfrak{M}}\) near the origin is 
\begin{align}\label{eq:FVfiltration}
V_{b}^{(0)}\widetilde{\mathfrak{M}}_0 = 
\left\{\begin{array}{lr}
    \left[F_{\lfloor b + \alpha \rfloor + 1}\mathcal{D}\right]_{\widetilde{\mathfrak{M}}_0}, & \text{if } \lfloor b + \alpha \rfloor + 1 \geq 0\\
    \left[t^{-\lfloor b + \alpha \rfloor - 1}\mathcal{O}_\mathbb{C}\right]_{\widetilde{\mathfrak{M}}_0}, & \text{if } \lfloor b + \alpha \rfloor + 1 < 0
    \end{array}\right.
\end{align}

for \(b \in \mathbb{R}\). By \Cref{theorem:mero1}, we have the identification 
\begin{align*}
m: \widetilde{\mathfrak{M}}_0= \mathcal{D}_{\mathbb{C}, 0}/(\mathcal{P}) \rightarrow \left(\mathcal{O}_{\mathbb{C}, 0}(*0), d - \sum_{i = 1}^n \frac{-m_i}{t - r_i}dt\right)
\end{align*}

of meromorphic connections given by 
\begin{align*}
[s]_{\widetilde{\mathfrak{M}}_0} \mapsto \frac{s(g)}{g} 
\end{align*}

for \(s \in \mathcal{D}_{\mathbb{C}, 0}\). For \(\lfloor b + \alpha \rfloor + 1 \geq 0\), one has \(s \in F_{\lfloor b + \alpha \rfloor + 1}\mathcal{D}\) if and only if \(m([s])\) has a pole of order at most \(\lfloor b + \alpha \rfloor + 1\) at \(0\), and for \(\lfloor b + \alpha \rfloor + 1 < 0\), one has \(s \in t^{-\lfloor b + \alpha \rfloor + 1}\mathcal{O}_\mathbb{C}\) if and only if \(m([s])\) has a zero of order at most \(-(\lfloor b + \alpha \rfloor + 1)\) at \(0\). With \Cref{eq:FVfiltration}, it follows that 
\begin{align*}
m\left(V_{b}^{(0)}\widetilde{\mathfrak{M}}_0\right) = t^{-\lfloor b + \alpha \rfloor - 1}\mathcal{O}_{\mathbb{C}, 0}
\end{align*}

Recall that the decreasing version of the \(V\)-filtration is obtained from the ascending one via \(V^\bullet = V_{-\bullet - 1}\), so that 
\begin{align}\label{eq:ascendingV}
m\left(V_{(0)}^{b}\widetilde{\mathfrak{M}}_0\right) = t^{\lceil b - \alpha \rceil}\mathcal{O}_{\mathbb{C}, 0}, 
\end{align}

where we have used that \(\lfloor -r \rfloor = -\lceil r \rceil\) for any \(r \in \mathbb{R}\). Comparing \Cref{eq:ascendingV} with \Cref{eq:parabolicformula}, the \(V\)-filtration and the parabolic filtration coincide at the origin. 
\end{proof}

This allows us to proceed to the next step, where we address the third "disadvantage" of harmonic metrics mentioned in \Cref{remark:hshortcoming}. 

\section{Step 3: Variation of Twistor structure}

In this step, we address the third point mentioned in \Cref{remark:hshortcoming}: For a harmonic bundle \((V, \nabla, h)\) on a complex manifold \(U\), the metric \(h\) will in general not be compatible with \(\nabla\). This shortcoming is solved by introducing a new coordinate \(z \in \mathbb{C}\), and using \(\nabla\) and \(h\) to construct a so-called bundle with \(z\)-connection on \(U\). \(h\) induces a pairing on this object, which will be compatible with the \(z\)-connection. The bundle with \(z\)-connection, equipped with this pairing, is then seen to form a so-called variation of Twistor structure. We begin by defining bundles with \(z\)-connection and variations of Twistor structure. 

\subsection{Definition of variations of Twistor structure}

Let \(U\) be a complex manifold. We set \(\mathscr{U} \coloneqq U \times \mathbb{C}_z\) and define the sheaf of differential forms 
\begin{align*}
\Omega_{\mathscr{U}}^1 \coloneqq z^{-1}\Omega^1_{U \times \mathbb{C}_z/\mathbb{C}_z}. 
\end{align*}

In local coordinates \((t_1, \ldots, t_n)\) of \(U\), a section of \(\Omega_{\mathscr{U}}^1\) is written as \(\sum_{i}f_i(t_1, \ldots, t_n, z) dt_i\), where \(f_i(t_1, \ldots, t_n, z) \in z^{-1}\mathcal{O}_{\mathscr{U}}\). 

\begin{definition}
A bundle with \(z\)-connection on \(U\) consists of the datum of a holomorphic bundle \(V_z\) on \(\mathscr{U}\) equipped with a connection 
\begin{align*}
\nabla: V_z \rightarrow V_z \otimes_{\mathcal{O}_{\mathscr{U}}} \Omega_{\mathscr{U}}^1, 
\end{align*}

satisfying the Leibniz rule \(\nabla(fs) = df \otimes s + f \nabla(s)\), where \(d = d_{U \times \mathbb{C}_z/\mathbb{C}_z}\). 
\end{definition}

\begin{example}
Any holomorphic bundle on \(\mathbb{C}_t \times \mathbb{C}_z\) is trivial. It follows that every rank \(1\) bundle with \(z\)-connection on \(\mathbb{C}_t\) can be written as 
\begin{align*}
(\mathcal{O}_{\mathbb{C}_t \times \mathbb{C}_z}, d - fdt), 
\end{align*}

where \(dg = (\partial g/\partial t) \cdot dt\) for any \(g \in \mathcal{O}_{\mathbb{C}_t \times \mathbb{C}_z}\) and where \(f \in z^{-1}\mathcal{O}_{\mathbb{C}_t \times \mathbb{C}_z}\). In particular, any usual connection on \(\mathcal{O}_{\mathbb{C}_t \times \mathbb{C}_z}\), which has no \(dz\)-contribution, also is a \(z\)-connection. 
\end{example}

We introduce the following notation: 

\begin{definition}
Let \(U\) be a complex manifold and let \(\mathbb{C}_z\) be the complex plane with the coordinate \(z\). We denote the sheaf of anti-holomorphic functions on \(U\) by \(\mathcal{O}_{\overline{U}}\). Let \(\boldsymbol{S}\) be the unit circle in \(\mathbb{C}_z\). Consider the inclusions \(i: \boldsymbol{S} \hookrightarrow \mathbb{C}_z\) and \(i_U: U \times \boldsymbol{S} \hookrightarrow U \times \mathbb{C}_z\). We denote the sheaf on \(U \times \mathbb{C}_z\) of smooth (in the real sense) functions on \(U \times \mathbb{C}_z\), which are analytic in the \(z\)-direction, by \(C_{U \times \mathbb{C}_z}^{\infty, \operatorname{an}}\). We set 
\begin{align*}
\mathcal{O}_{\boldsymbol{S}} & \coloneqq i^{-1}\mathcal{O}_{\mathbb{C}_z}, & \mathcal{O}(\boldsymbol{S}) & \coloneqq \Gamma(\boldsymbol{S, \mathcal{O}_{\boldsymbol{S}}}), \\ \mathcal{O}_{U \times \boldsymbol{S}} & \coloneqq i^{-1}\mathcal{O}_{U \times \mathbb{C}_z}, & \quad \mathcal{O}(U \times \boldsymbol{S}) & \coloneqq \Gamma(U \times \boldsymbol{S}, \mathcal{O}_{U \times \boldsymbol{S}}), \\ C_{U \times \boldsymbol{S}}^{\infty, \operatorname{an}} & \coloneqq i^{-1}C_{U \times \mathbb{C}_z}^{\infty, \operatorname{an}} & \quad C(U \times \boldsymbol{S})^{\infty, \operatorname{an}} & \coloneqq \Gamma(U \times \boldsymbol{S}, C_{U \times \boldsymbol{S}}^{\infty, \operatorname{an}}), 
\end{align*}

where all the inverse images are in the sheaf-theoretical sense. Similarly, we define \(\mathcal{O}_{\overline{U} \times \boldsymbol{S}}\) and so on. When we have a holomorphic \(z\)-connection \(V_z\) on \(U \times \mathbb{C}_z\), we define \(V_z\vert_{\boldsymbol{S}} = i^{-1}V_z\). 
\end{definition}

We recall the Twistor-conjugation \(\overline{\sigma^*}\): 

\begin{definition}
We have the anti-linear involution \(\sigma: \mathbb{P}^1 \rightarrow \overline{\mathbb{P}^1}: z \mapsto -1/\overline{z}\). 
\begin{itemize}
    \item We get the induced map 
    \begin{align*}
    \overline{\sigma^*}: C_{U \times \boldsymbol{S}}^{\infty, \operatorname{an}} \rightarrow C_{U \times \boldsymbol{S}}^{\infty, \operatorname{an}}, \quad f(t, z) \mapsto \overline{f(t, -1/\overline{z})}, 
    \end{align*}
    where \((t, z) \in U \times \boldsymbol{S}\). This definition is extended to square matrices with entries in \(C_{U \times \boldsymbol{S}}^{\infty, \operatorname{an}}\) by component-wise application. Let \(G\) be such a matrix. If \(\overline{\sigma^*}G^T = G\), one calls \(G\) Twistor-real or Twistor-hermitian. 
    \item If we have a bundle with \(z\)-connection \(V_z\) on \(U\), we get the induced bundle \(\overline{\sigma^*V_z}\), which we will denote as \(\overline{\sigma^*}V_z\). 
\end{itemize}
\end{definition}

\begin{remark}
For any \(f \in C_{U \times \boldsymbol{S}}^{\infty, \operatorname{an}}\), we have \(\overline{\sigma^*}\overline{\sigma^*}f = f\). \(\overline{\sigma^*}\) restricts to a morphism \(\mathcal{O}_{U \times \boldsymbol{S}} \rightarrow \mathcal{O}_{\overline{U} \times \boldsymbol{S}}\). If \(G\) is Twistor-hermitian and constant with respect to \(z\), it is clearly hermitian in the usual sense. 
\end{remark}

We can now recall the following fundamental definition from \cite{Sab05}, 2.2.a: 

\begin{definition}\label{def:vartwiststr}
A variation of pure Twistor-structure of weight \(k\) on a manifold \(U\) is a triple \((V_z', V_z'', C)\), where \(V_z'\) and \(V_z''\) are holomorphic bundles with flat \(z\)-connection on \(U\) and where 
\begin{align*}
C: \overline{\sigma^*}V_z'\vert_{\boldsymbol{S}} \otimes_{\mathcal{O}_{\boldsymbol{S}}} V_z''\vert_{\boldsymbol{S}} \rightarrow C_{U \times \boldsymbol{S}}^{\infty, {\operatorname{an}}}
\end{align*}

is \(\mathcal{O}_{U \times \mathbb{C}_z} \otimes_{\mathcal{O}_{\mathbb{C}_z}}\mathcal{O}_{\overline{U} \times \mathbb{C}_z}\)-linear, compatible with the \(z\)-connections on \(V_z', V_z''\) and \(C_{U \times \boldsymbol{S}}^{\infty, {\operatorname{an}}}\), and satisfies the following conditions: 

\begin{itemize}
    \item \(C\) is a non-degenerate pairing in the sense that the induced map (via the adjunction of tensor product and \(\mathscr{H}om\)) 
    \begin{align*}
    \overline{\sigma^*}V_z'\vert_{\boldsymbol{S}} \rightarrow \mathscr{H}om_{(\mathcal{O}_{U \times \boldsymbol{S}}, d)}(V_z''\vert_{\boldsymbol{S}}, (\mathcal{O}_{U \times \boldsymbol{S}}, d))
    \end{align*}
    is an isomorphism. 
    \item Gluing the underlying bundles of \(\overline{\sigma^*}V_z'\) and \(\mathscr{H}om_{(\mathcal{O}_{U \times \mathbb{C}_z}, d)}(V_z'', (\mathcal{O}_{U \times \mathbb{C}_z}, d))\) with \(C\) along \(\boldsymbol{S}\) to a bundle on \(U \times \mathbb{P}^1_z\) results in a bundle \(N\) such that, for every \(x \in U\), \(N\vert_{\{x\} \times \mathbb{P}_z} \cong \mathcal{O}_{\mathbb{P}_z}(k)^{\oplus r}\) for certain \(k \in \mathbb{Z}, r \in \mathbb{N}\) independent of \(x\). 
\end{itemize}
\end{definition}

Variations of Twistor structure can be equipped with an additional datum, namely with a polarization. Such a variation of Twistor structure is then called "polarized". We will not repeat the precise definition (cf. \cite{Sab05}, 2.2.a) here, as it is not relevant for our purposes. We only recall the following: A polarization on a variation of Twistor structure \((V_z', V_z'', C)\) of weight \(0\) on a complex manifold is a pair of two morphisms \(V_z'' \rightarrow V_z'\) satisfying a certain compatibility condition with \(C\). If \(V_z' = V_z''\), the pair \((\operatorname{id}, \operatorname{id})\) gives a polarization. 

\subsection{From harmonic bundles to Twistor structures}

We will now recall how one constructs a variation of Twistor structure from a flat harmonic bundle, where we follow \cite{Sab05}, Section 1.3. For a flat harmonic bundle \((V, \nabla, h)\) on a complex manifold \(U\), let \(H = V \otimes_{\mathcal{O}_U} C^\infty_U\) be the smooth bundle associated to the holomorphic bundle \(V\). As before, we have the data \(D_V, D_E, \theta\). Essentially, the idea is to produce an object interpolating between \(\nabla\) and \(\theta'\), for which we use an additional variable: The Twistor-variable \(z\). We have the projection \(\pi: \mathscr{U} = U \times \mathbb{C}_z \rightarrow U\).  Consider the \(C^{\infty, {\operatorname{an}}}_{\mathscr{U}}\)-module 
\begin{align*}
\mathscr{H} \coloneqq C^{\infty, {\operatorname{an}}}_{\mathscr{U}} \otimes_{\pi^{-1}C^\infty_U}\pi^*H. 
\end{align*}
We equip it with the \(d''\)-operator \(D_\mathscr{H}'' \coloneqq D_{E}'' +  z\theta_{E}''\) and we obtain a holomorphic subbundle of \(\mathscr{H}\) as \(\mathscr{H}' \coloneqq \text{Ker}(D''_{\mathscr{H}})\), which recovers \(\mathscr{H}\) via 
\begin{align*}
\mathscr{H} = C^{\infty, {\operatorname{an}}}_{\mathscr{U}} \otimes_{\mathcal{O}_{\mathscr{U}}}\mathscr{H}', 
\end{align*}
as in the case of usual non-\(z\)-connections. On \(\mathscr{H}\), we also have the flat \(z\)-connection \(D'_\mathscr{H} \coloneqq D'_{E} + z^{-1}\theta_{E}'\). It restricts to \(\mathscr{H'}\) as a flat holomorphic \(z\)-connection. In conclusion, we obtain the holomorphic bundle with flat \(z\)-connection \(V_z \coloneqq (\mathscr{H}', D'_\mathscr{H}\vert_{\mathscr{H}'})\). Additionally, we have the pairing 
\begin{align*}
\pi^*h: \overline{\sigma^*}V_z\vert_{\boldsymbol{S}} \otimes V_z\vert_{\boldsymbol{S}} \rightarrow C^{\infty, \operatorname{an}}_{U \times \boldsymbol{S}}. 
\end{align*}
Furthermore, we have the pair \((\operatorname{ker}(D''_E), \theta')\), which forms a so-called Higgs bundle. (This is not relevant for our purposes.) We can summarize the whole construction as follows: 

\[\begin{tikzcd}
	{(V, \nabla)} && {(H, D_V = \nabla + \overline{\partial})} && \begin{array}{c} (H, D_V = D_E' + \\ D_E'' + \theta' + \theta'', h) \end{array} \\
	\\
	&& {(\mathscr{H}, D_E'' + z\theta'')} \\
	\\
	\begin{array}{c} (\text{ker}(D_E'' + z\theta''), \\ D_E' + z^{-1}\theta')) \end{array} &&&& {(\text{ker}(D_E''), \theta')}
	\arrow["\begin{array}{c} \text{passing to} \\ \text{smooth bundle} \end{array}", maps to, from=1-1, to=1-3]
	\arrow["\begin{array}{c} \text{equipping with} \\ \text{harmonic metric} \end{array}", maps to, from=1-3, to=1-5]
	\arrow["\begin{array}{c} \text{introducing} \\ \text{Twistor-variable} \\ \text{and } d''\text{-operator} \end{array}"'{pos=1}, maps to, from=1-5, to=3-3]
	\arrow["\begin{array}{c} \text{induces holomorphic} \\ \text{Higgs bundle} \end{array}"'{pos=0.7}, maps to, from=1-5, to=5-5]
	\arrow["\begin{array}{c} \text{passing to} \\ \text{holomorphic} \\\text{subbundle} \end{array}"', maps to, from=3-3, to=5-1]
	\arrow["{\text{restriction to }z = 1}"'{pos=0.7}, maps to, from=5-1, to=1-1]
	\arrow["{\text{restriction to }z = 0}", maps to, from=5-1, to=5-5]
\end{tikzcd}\]

A fundamental result (essentially due to Simpson and revised by Sabbah in \cite{Sab05}, Lemma 2.2.2) is the following: 

\begin{theorem}
Let \(U\) be a complex manifold. We have an equivalence of categories 
\begin{align*}
\{\text{Flat harm. bundles on \(U\)}\} &  \rightarrow \{\text{Polarised weight \(0\) var. of Twistor structure on \(U\)}\} \\ (V, \nabla, h) & \mapsto (V_z, V_z, \pi^*h). 
\end{align*}
\end{theorem}

\begin{definition}
Step 3 of constructing \(\mathfrak{T}\) from \(\mathfrak{M}\) is defined as passing from the tame harmonic bundle of Deligne type \((\mathfrak{M}\vert_{X - D}, \nabla, h)\) to the associated variation of Twistor structure \((V_z, V_z, \pi^*h)\). 
\end{definition}

We will now carry out this step in the situation of \Cref{theorem:firstmainresult}. 

\subsection{Step 3 in the order \(1\) case} 

We use the notation from \Cref{theorem:firstmainresult}. Recall that we denote by \((V, \nabla)\) the non-singular part of \(\mathfrak{M}\). We have already equipped this flat holomorphic bundle with a harmonic metric of Deligne type. We will now determine the associated variation of Twistor structure. For \(\gamma \in \mathbb{C}\), set 
\begin{align*}
\gamma \diamond z \coloneqq (\gamma \star z)/z = \Re(\gamma) + \Im(\gamma)\left(\frac{z + z^{-1}}{2}\right)i. 
\end{align*}

This function should be viewed as a twisted version of \(\gamma\). We have \(\gamma = \gamma \diamond z\) for all \(z \in \mathbb{C}_z\) if and only if \(\gamma \in \mathbb{R}\) and we observe \(\overline{\sigma^*}(\gamma \diamond z) = \gamma \diamond z\). 

\begin{theorem}\label{theorem:vartwrank1}
Let \((V, \nabla, h)\) be as above. The variation of Twistor structure associated to \((V, \nabla, h)\) is given by \((V_z, V_z, C)\), where 
\begin{align*}
V_z = \left(\mathcal{O}_{U \times \mathbb{C}_z}, d - \sum_{i = 1}^n \frac{-m_i \diamond z}{t - r_i}dt\right) 
\end{align*}

and where the pairing \(C: \overline{\sigma^*}V_z\vert_{\boldsymbol{S}} \otimes V_z\vert_{\boldsymbol{S}} \rightarrow C^{\infty, {\operatorname{an}}}_{U \times \boldsymbol{S}}\) is given by 
\begin{align*}
\overline{\sigma^*}u \otimes v \mapsto \overline{\sigma^*}u \cdot \prod_{j = 1}^n \vert t - r_j \vert^{2m_j \diamond z} \cdot v. 
\end{align*}
\end{theorem}

Before we prove this result, we note the following. 

\begin{remark}\label{remark:twistorstr}
We make some observations on the previous result: 
\begin{itemize}
    \item \(\sum \frac{-m_i \diamond z}{t - r_i}\) has no pole with respect to \(z \in \mathbb{C}\) if \(m_i \in \mathbb{R}\) for all \(i\) and otherwise it has a pole of order \(1\) at \(z = 0\) and no other pole with respect to \(z\). In any case, we see that \(V_z\) is indeed a bundle with \(z\)-connection. 
    \item Note that \(\prod_{j = 1}^n \vert t - r_j \vert^{2m_j \diamond z}\) is single-valued, as expected from a pairing of a variation of Twistor structure. 
    \item One can directly verify that the pairing is indeed compatible with the \(z\)-connection, as required in the definition of a variation of Twistor structure. 
\end{itemize}
\end{remark}

We will now prove \Cref{theorem:vartwrank1}: 

\begin{proof}
We have the \((0,1)\)-connection
\begin{align*}
D_E'' + z\theta'' = \overline{\partial}_t + (z-1)\frac{1}{2}\sum_{j = 1}^n \frac{-\Im(m_j)i}{\overline{t - r_j}}d\overline{t}. 
\end{align*}

The holomorphic subbundle \(\mathscr{H}'\) of \(\mathscr{H}\) which is formed by the flat sections of this \((0,1)\)-connection is 
\begin{align*}
\mathcal{O}_{U \times \mathbb{C}_z} \cdot \prod_{j = 1}^n(\overline{t- r_j})^{f(m_j, z)} \subset \mathscr{H} \quad \text{with} \quad f(m_i, z) = \frac{i(z-1)}{2}\Im(m_i), 
\end{align*}

where \(\prod(\overline{t- r_j})^{f(m_j, z)}\) is to be interpreted in the multi-valued sense. Setting 
\begin{align*}
e(t, z) \coloneqq \prod_{j = 1}^n(t- r_j)^{f(m_j, z)} \in \mathcal{O}_{U \times \mathbb{C}_z}, \quad e' \coloneqq \prod_{j = 1}^n(\overline{t- r_j})^{f(m_j, z)} \in \mathcal{O}_{\overline{U} \times \mathbb{C}_z}, 
\end{align*}

\(\mathscr{H}'\) has the single-valued generator \(ee' \in C^{\infty, \operatorname{an}}_{U \times \mathbb{C}_z}\) over \(\mathcal{O}_{U \times \mathbb{C}_z}\). We also have the \((1,0)\)-connection 
\begin{equation} 
\begin{split}
D_E + z^{-1}\theta' & = \partial_t - \sum_{j = 1}^n \frac{-m_j}{t - r_j}dt -  \frac{1}{2} \sum_{j = 1}^n \frac{\Im(m_j)i}{t - r_j}dt + z^{-1} \frac{1}{2} \sum_{j = 1}^n \frac{\Im(m_j)i}{t - r_j}dt \\ & = \partial_t - \sum_{j = 1}^n \frac{-m_j}{t - r_j}dt + (1 - z^{-1})\frac{1}{2} \sum_{j = 1}^n \frac{\Im(m_j)i}{t - r_j}dt. 
\end{split}
\end{equation}

on \(\mathscr{H}\), which restricts to \(\mathscr{H}'\). We can express the connection matrix of \(D_E + z^{-1}\theta'\) in the basis \(ee'\): Via the usual base change formula for connections, this gives a twist by 
\begin{align*}
\frac{\partial_t (ee')}{ee'} = \frac{\partial_t e}{e} = \frac{i(z-1)}{2}\sum_{j = 1}^n \frac{\Im(m_j)}{t - r_j}dt
\end{align*}

where the first equality holds because \(e'\) is anti-holomorphic in \(t\). This allows us to write \((\mathscr{H}', D_E' + z^{-1}\theta')\) in the basis \(ee'\) as 
\begin{align*}
\left(\mathcal{O}_{U \times \mathbb{C}_z}, \partial_t - \sum_{j = 1}^n \frac{-m_j}{t - r_j}dt + \frac{z - z^{-1}}{2} \sum_{j = 1}^n \frac{\Im(m_j)i}{t - r_j}dt\right) 
\end{align*}

The sesquilinear pairing \(C\) on \(\mathscr{V}_{\vert \boldsymbol{S}}\) is the restriction of \(\pi^*h\), so in the basis \(ee'\) it is written as 
\begin{align*}
(u, v) \mapsto \overline{\sigma^*}u \cdot \overline{\sigma^*}(ee') \cdot \prod_{j = 1}^n \left|(t- r_j)^{\Re(m_j)}\right|^2 \cdot ee' \cdot v 
\end{align*}

for \(u, v \in \mathscr{H}'\). To prove the theorem, we are left to verify that 
\begin{align}\label{eq:rewritepairing}
\overline{\sigma^*}(ee') \cdot \prod_{j = 1}^n \left|(t- r_j)^{\Re(m_j)}\right|^2 \cdot ee' = \prod_{j = 1}^n \vert t - r_j \vert^{2m_j \diamond z}, 
\end{align}

which will be a direct calculation. To this end, let \(l(t, z) = \prod_{j = 1}^n(t - r_i)^{g(m_i, z)}\) with 
\begin{align*}
g(m_i, z) \coloneqq -m_i + \frac{(1-z^{-1})i}{2}\Im(m_i). 
\end{align*}

We have \begin{equation*} 
\begin{split}
t^{-g(m_i, z)} \overline{\sigma^*}\left(\overline{t}^{-f(m_i, z)}\right) & = t^{m_i} t^{-\frac{i}{2}(1 - z^{-1})\Im(m_i) + \frac{i}{2}(-z^{-1}-1)\Im(m_i)} \\ & = t^{\Re(m_i)}, 
\end{split}
\end{equation*}

which implies \((\overline{\sigma^*}e' \cdot l)^{-1} = \prod_{j = 1}^n(t - r_i)^{\Re(m_i)}\). Using this, we obtain 
\begin{align*}
\overline{\sigma^*}(ee') \cdot \prod_{j = 1}^n\vert t - r_i\vert^{2\Re(m_i)} \cdot ee' = \overline{\sigma^*}(el^{-1}) \cdot el^{-1}. 
\end{align*}

We will now simplify the right-hand side of this equation. We have 
\begin{align*}
el^{-1} = \prod_{j = 1}^n(t - r_i)^{-g(m_i, z) + f(m_i, z)}. 
\end{align*}

For \(|z| = 1\) we have 
\begin{align*}
-g(m_i, z) + f(m_i, z) & = m_i - \frac{i}{2}(1 - z^{-1})\Im(m_i) - \frac{i(1-z)}{2}\Im(m_i) \\ & = \Re(m_i) + \frac{i}{2}(z + z^{-1})\Im(m_i). 
\end{align*}

Observe that \(\overline{\sigma^*}(-g(m_i, z) + f(m_i, z)) = -g(m_i, z) + f(m_i, z)\), so that 
\begin{align*}
\overline{\sigma^*}(t^{-g(m_i, z) + f(m_i, z)})t^{-g(m_i, z) + f(m_i, z)} = |t|^{2m_i \diamond z}. 
\end{align*}

This implies \Cref{eq:rewritepairing}, completing the proof. 
\end{proof}

\section{Step 4: Filtering by growth with respect to the metric}

In the previous step, we have obtained the variation of Twistor structure \((V_z, V_z, C)\) on \(X - D\), which is of weight \(0\) and equipped with the polarization \((\operatorname{id}, \operatorname{id})\). Now, we want to extend the bundle with \(z\)-connection \(V_z\) over the singular locus \(D\), because the Twistor \(\mathcal{D}\)-module \(\mathfrak{T}\), that we want to obtain in the end, also lives on \(X\), just like the \(\mathcal{D}\)-module \(\mathfrak{M}\) we started with. We extend \(V_z\) over \(D\) in the following way: Let \(j: X - D \rightarrow X\) be the inclusion and define \(\widetilde{\mathfrak{M}}_z\) as the subsheaf of \(j_*V_z\) whose sections have moderate growth along \(D \times \mathbb{C}_z\) with respect to \(\pi^*h = C\). 

\begin{remark}
This construction is in fact similar to the definition of the Deligne extension of holomorphic bundles with connection to regular holonomic \(\mathcal{D}\)-modules. 
\end{remark}

Set \(\mathscr{X} = X \times \mathbb{C}_z\). One can show (cf. \cite{Sab05}, Corollary 5.3.1(1)) that \(\widetilde{\mathfrak{M}}_z\) is a locally free \(\mathcal{O}_{\mathscr{X}}(*D)\)-module, which is equipped with the action induced by the \(z\)-connection on \(V_z\). (Here, one needs to show that this \(z\)-connection preserves the property of moderate growth with respect to \(\pi^*h\).) 

\begin{definition}
Step 4 of constructing \(\mathfrak{T}\) from \(\mathfrak{M}\) is defined as building \(\widetilde{\mathfrak{M}}_z\) from the variation of Twistor structure \((V_z, V_z, C)\). 
\end{definition}

\subsection{Step 4 in the order \(1\) case} 

We go back to the setting of \Cref{theorem:firstmainresult}. We will now make explicit the meromorphic bundle with \(z\)-connection \(\widetilde{\mathfrak{M}}_z\) associated to the variation of Twistor structure determined in \Cref{theorem:vartwrank1}. 

\begin{theorem}\label{lemma:merozex}
We have 
\begin{align*}
\widetilde{\mathfrak{M}}_z \cong \left(\mathcal{O}_{\mathbb{P}^1 \times \mathbb{C}_z}(*D), d - \sum_{i = 1}^n \frac{-m_i \diamond z}{t - r_i}dt\right). 
\end{align*}
\end{theorem}

\begin{proof}
Let \(s\) be a section of \(j_*V_z\), which means that \(s\) can also be seen as a section of \(V_z\). Recall that 
\begin{align*}
V_z \subset \mathscr{H} = C^{\infty, {\operatorname{an}}}_{\mathscr{U}} \otimes_{\pi^{-1}C^\infty_U}\pi^*H \cong C^{\infty, {\operatorname{an}}}_{\mathscr{U}}. 
\end{align*}

We know from the proof of \Cref{theorem:vartwrank1} that 
\begin{align*}
s = g \cdot \prod_{j = 1}^n|t- r_j|^{2f(m_j, z)}, \quad f(m_j, z) = \frac{i(z-1)}{2}\Im(m_j), 
\end{align*}

for some \(g \in \mathcal{O}_{U \times \mathbb{C}_z}\). Recall that the harmonic metric on \((V, \nabla)\) is given by \(\prod|t - r_i|^{2\Re(m_i)}\). It follows that the \(\pi^*h\)-norm of \(s\) is given by 
\begin{align*}
||s||_{\pi^*h}(t, z) & = \overline{s(t, z)} \cdot \prod_{j = 1}^n|t - r_j|^{2\Re(m_j)} \cdot s(t, z). 
\end{align*}

Observe that \(\prod_{j = 1}^n|t- r_j|^{\Re(m_j) + f(m_j, z)}\) is of moderate growth along \(D \times \mathbb{C}_z\). Since \(g\) is of moderate growth along \(D \times \mathbb{C}_z\) if and only if \(g\) has no essential singularities at \(D \times \mathbb{C}_z\), it follows that the same is true for \(s\). This implies that the underlying \(\mathcal{O}_{\mathbb{P}^1 \times \mathbb{C}_z}\)-module of \(\widetilde{\mathfrak{M}}_z\) is 
\begin{align*}
\mathcal{O}_{\mathbb{P}^1 \times \mathbb{C}_z}(*D) \subset j_*\mathcal{O}_{(\mathbb{P}^1 - D) \times \mathbb{C}_z} \cong j_*V_z. 
\end{align*}

Because the \(z\)-connection on \(\widetilde{\mathfrak{M}}_z\) is defined as the restriction of the \(z\)-connection on \(j_*V_z\), the assertion follows. 
\end{proof}

\section{Step 5: Minimal extension}

In this step we leave the realm of meromorphic bundles with \(z\)-connection in order to ensure that the Twistor \(\mathcal{D}\)-module we obtain in the end is simple. To this end, we need to generalize the notion of bundle with \(z\)-connection: In the same way that flat holomorphic and flat meromorphic connections are special cases of \(\mathcal{D}\)-modules on a complex manifold \(X\), C. Sabbah introduced the notion of \(\mathcal{R}_X\)-modules in order to generalize the notion of flat holomorphic and flat meromorphic \(z\)-connections on \(X\), cf. \cite{Sab05}, Section 0.2. Here, \(\mathcal{R}_X\) is defined as the so-called Rees ring of \(\mathcal{D}_X\) equipped with the order filtration: 

\begin{definition}
Let \(X\) be a smooth complex variety or a complex manifold. 
\begin{itemize}
    \item Let \(S\) be a ring, equipped with an ascending filtration \(F_{\bullet}S\), indexed by \(\mathbb{N}_0\). The associated Rees ring is defined as 
    \begin{align*}
    RS \coloneqq \bigoplus_{k \in \mathbb{N}_0}F_kS \cdot z^k \subset S[z]. 
    \end{align*}
    \item Consider the ring \(\mathcal{D} = \mathbb{C}\langle t, \partial_t \rangle\) equipped with the order filtration \(F_{\bullet}\mathcal{D}\). We define the ring \(\mathcal{R} \coloneqq R\mathcal{D} = \mathbb{C}\langle t, z, z\partial_t\rangle\) with the natural commutation relations. 
    \item Consider the ring \(\mathcal{D}_X\) equipped with the order filtration \(F_{\bullet}\mathcal{D}_X\). We define \(\mathcal{R}_X \coloneqq R\mathcal{D}_X\). In local (in the analytic case) or Étale (in the algebraic case) coordinates \(t_1, \ldots, t_n\) of \(X\), we have 
    \begin{align*}
    \mathcal{R}_X = \mathcal{O}_{X}[z]\langle z\partial_{t_1}, \ldots, z\partial_{t_n} \rangle
    \end{align*}
    with the natural commutation relations. This is a sheaf of rings on \(X\). 
    \item We have the analytification of \(\mathcal{R}\) with respect to the \(z\)-variable: This gives the sheaf of rings on \(\mathbb{C}_z\) defined as \(\mathscr{R} \coloneqq \mathcal{O}_{\mathbb{C}_z}[t]\langle z\partial_t \rangle\). Furthermore, we have the analytification \(\mathscr{R}_{X \times \mathbb{C}_z}\) of \(\mathcal{R}_X\) with respect to the \(z\)-variable, which is locally expressed as 
    \begin{align*}
    \mathcal{O}_{X \times \mathbb{C}_z}\langle z\partial_{t_1}, \ldots, z\partial_{t_n} \rangle. 
    \end{align*}
    This is a sheaf of rings on the set \(X \times \mathbb{C}_z\) equipped with the product topology. One often abbreviates \(\mathscr{X} \coloneqq X \times \mathbb{C}_z\) in this context. (Note that, if \(X\) is a variety, then \(\mathscr{X}\) carries a partially analytic and partially algebraic topology.)
\end{itemize}
\end{definition}

Recall that, unless specified otherwise, we understand the complex projective line \(\mathbb{P}^1\) as being equipped with the analytic topology, whereas we understand the complex affine line \(\mathbb{A}^1\) as being equipped with the algebraic topology. Similar to \(\Gamma(\mathbb{A}^1, \mathcal{D}_{\mathbb{A}^1}) = \mathcal{D}\) we have 
\begin{align*}
\Gamma(\mathbb{A}^1, \mathcal{R}_{\mathbb{A}^1}) = \mathcal{R}, \quad \pi_*( \mathscr{R}_{\mathbb{A}^1 \times \mathbb{C}_z}) = \mathscr{R}, 
\end{align*}

where \(\pi\) denotes the projection \(\mathbb{A}^1 \times \mathbb{C}_z \rightarrow \mathbb{C}_z\). 

\begin{remark}
Let \(X\) be a complex manifold. We had already noted that any flat connection on \(\mathcal{O}_{\mathscr{X}}\), which has no \(dz\)-contribution, also is a flat \(z\)-connection. Analogously, any \(\mathcal{D}_{\mathscr{X}}\)-module is also an \(\mathscr{R}_{\mathscr{X}}\)-module, because \(\mathscr{R}_{\mathscr{X}} \subset \mathcal{D}_{\mathscr{X}}\). 
\end{remark}

\subsection{\(\mathcal{R}\)-affinity and analytification}

Let \(\operatorname{Mod}_{\operatorname{qc}}(\mathcal{R}_X)\) be the category of \(\mathcal{R}_X\)-modules, which are \(\mathcal{O}_X\)-quasi-coherent. As an analogue to \(\mathcal{D}\)-affinity, we introduce the following: 

\begin{definition}
A smooth variety \(X\) is called \(\mathcal{R}\)-affine, if the following holds: 

\begin{itemize}
    \item The global sections functor \(\Gamma(X, \cdot): \operatorname{Mod}_{\operatorname{qc}}(\mathcal{R}_X) \rightarrow \operatorname{Mod}(\mathcal{R}(X))\) is exact. 
    \item If \(\Gamma(X, \mathcal{M}) = 0\) for \(\mathcal{M} \in \operatorname{Mod}_{\operatorname{qc}}(\mathcal{R}_X)\), we have \(\mathcal{M} = 0\). 
\end{itemize}
\end{definition}

Any affine variety \(X\) is also \(\mathcal{R}\)-affine. Similarly to \Cref{lemma:daffineequ} (\cite{Hot18}, Proposition 1.4.4), one can however prove the following: 

\begin{lemma}
Let \(X\) be \(\mathcal{R}\)-affine. Then, the functor 
\begin{align*}
\Gamma(X, \cdot): \operatorname{Mod}_{\operatorname{qc}}(\mathcal{R}_X) \rightarrow \operatorname{Mod}(\mathcal{R}(X)) 
\end{align*}

is an equivalence of categories. \(\Gamma(X, \cdot)\) has a quasi-inverse given by 
\begin{align*}
\operatorname{Mod}(\mathcal{R}_X) \otimes_{\Gamma(X, \mathcal{R}_X))} (\cdot): \mathcal{R}(X) \rightarrow \operatorname{Mod}_{\operatorname{qc}}(\mathcal{R}_X). 
\end{align*}
\end{lemma}

\begin{proof}
This is proven in the same way as \Cref{lemma:daffineequ}. 
\end{proof}

\begin{remark}
Note that we do not introduce a notion of \(\mathscr{R}\)-affinity here. In this case, the situation is more complicated, because \(\mathscr{R}\) is analytic with respect to \(z\). See however \Cref{lemma:RaffineRmod}. 
\end{remark}

Let \(X\) be a smooth variety. Recall the analytification functor for \(\mathcal{D}_X\)-modules. For \(\mathcal{R}_X\)-modules, we can introduce an analogous notion. However, we can take analytification in two different directions, namely the \(X\)-direction and the \(z\)-direction. This leads to the following naturally defined exact functors, giving a commutative diagram: 

\[\begin{tikzcd}
	{\operatorname{Mod}(\mathcal{R}_X)} && {\operatorname{Mod}(\mathscr{R}_{X \times \mathbb{C}_z})} \\
	\\
	{\operatorname{Mod}(\mathcal{R}_{X^{\operatorname{an}}})} && {\operatorname{Mod}(\mathscr{R}_{X^{\operatorname{an}} \times \mathbb{C}_z}).}
	\arrow["{()^{z\operatorname{-an}}}"', from=1-1, to=1-3]
	\arrow["{()^{X\operatorname{-an}}}", from=1-1, to=3-1]
	\arrow["{()^{X\operatorname{-an}}}"', from=1-3, to=3-3]
	\arrow["{()^{z\operatorname{-an}}}", from=3-1, to=3-3]
\end{tikzcd}\]

If local coordinates of \(X\) are \(t\), we will write \(()^{t\operatorname{-an}} \coloneqq ()^{X\operatorname{-an}}\). 

\subsection{\(\mathscr{R}\)-modules and \(z\)-connections on \(\mathbb{C}\)}

Recall \Cref{prop:dmodsandconns}, where we made explicit the relationship between \(\mathcal{D}\)-modules of the form \(\mathcal{D}/(\mathcal{P})\) and integrable connections. We have a similar result for \(\mathcal{R}\)-modules: 

\begin{proposition}\label{prop:rmodsandzconns}
Let \(\mathcal{P}_z \in \mathcal{R} = \mathbb{C}\langle t, z, z\partial_t\rangle\) and write 
\begin{align*}
\mathcal{P}_z = \sum_{i = 0}^n a_i(t, z)(z\partial_t)^i, 
\end{align*}

where \(a_i(t, z) \in \mathbb{C}[t, z]\). Assume that \(n \in \mathbb{N}_+\) and that \(a_n\) does not depend on \(z\), so that \(a_n^{-1}(0) = D \times \mathbb{C}_z\) for a finite subset \(D \subset \mathbb{C}_t\). (The other \(a_i\) may depend on \(z\).) Write \(D = \{p_i : i = 1, \ldots, m\}\) and \(U = \mathbb{C}_t - D, \mathscr{U} \coloneqq U \times \mathbb{C}_z\). Abbreviate  
\begin{align*}
\mathbb{C}[t, (t - p_i)^{-1}, z] & \coloneqq \mathbb{C}[t, (t - p_1)^{-1}, \ldots, (t - p_m)^{-1}, z], \\ \mathcal{R}[(t - p_i)^{-1}] & \coloneqq \mathcal{R}[(t - p_1)^{-1}, \ldots, (t - p_m)^{-1}]. 
\end{align*}

Then, the following statements hold: 

\begin{itemize}
    \item \(B \coloneqq \{[1], [z\partial_t], \ldots, [(z\partial_t)^{n-1}]\} \subset \mathcal{R}/(\mathcal{P}_z)\) is linearly independent over \(\mathbb{C}[t, z]\). 
    \item \(\mathcal{R}[(t - p_i)^{-1}]/(\mathcal{P}_z)\) is free over \(\mathbb{C}[t, (t - p_i)^{-1}, z]\) with basis     
    \begin{align*}
    B_U \coloneqq \{[1], [z\partial_t], \ldots, [(z\partial_t)^{n-1}]\} \subset \mathcal{R}[(t - p_i)^{-1}] / (\mathcal{P}_z)
    \end{align*}
    \item Consider the \(\mathscr{R}_{U^{\operatorname{an}} \times \mathbb{C}_z}\)-module 
    \begin{align*}
    \mathscr{M} \coloneqq \mathcal{O}_{U^{\operatorname{an}} \times \mathbb{C}_z} \otimes_{\mathcal{R}[(t - p_i)^{-1}]} \mathcal{R}[(t - p_i)^{-1}] / (\mathcal{P}_z). 
    \end{align*}
    The \(z\)-connection induced on \(\mathscr{M}\) by the action of \(z\partial_t \in \mathscr{R}_{U^{\operatorname{an}}}\), in the \(\mathcal{O}_{U^{\operatorname{an}} \times \mathbb{C}_z}\)-basis obtained as the image of \(B_U\) under the map 
    \begin{align*}
    \mathcal{R}[(t - p_i)^{-1}] / (\mathcal{P}_z) \rightarrow \mathscr{M}, 
    \end{align*}
    is given by \(d - z^{-1}Cdt\), where 
    \begin{align*}
    C = \begin{pmatrix}
    0 & 0 & \cdots & 0 & a_0(t, z)/a_n(t) \\
    -1 & 0 & \cdots & 0 & a_1(t, z)/a_n(t) \\
    0 & -1 & \cdots & 0 & a_2(t, z)/a_n(t) \\
    \vdots & \vdots & \ddots & \vdots & \vdots \\
    0 & 0 & \cdots & -1 & a_{n-1}(t, z)/a_n(t)
    \end{pmatrix}
    \end{align*}
    \item Away from \(z = 0\), the system of differential equations equivalent to the equation \(\mathcal{P}_zf(t, z) = 0\) is 
    \begin{align*}
    \frac{d}{dt}f(t, z) = -z^{-1}C^Tf(t, z)
    \end{align*}
    \item If \(F\) is a fundamental matrix of \(\left(\mathcal{O}_{{U^{\operatorname{an}}} \times \mathbb{C}_z^*}^{\oplus n}, d + z^{-1}C^T dt\right)\), then \((F^{-1})^T\) is one of \(\mathscr{M}|_{{U^{\operatorname{an}}} \times \mathbb{C}_z^*}\). 
\end{itemize}
\end{proposition}

\begin{proof}
The proof of the first three points is analogous to the proof of \Cref{prop:dmodsandconns}. For the last two points, it suffices to observe that a \(z\)-connection gives a usual connection for every fixed \(z \neq 0\), which again allows to reduce to \Cref{prop:dmodsandconns}. 
\end{proof}






\subsection{Naive localisation of \(\mathscr{R}\)-modules}

Let \(X\) be a complex manifold with a reduced divisor \(D\). Recall that, for a \(\mathcal{D}_{\mathscr{X}}\)-module \(\mathcal{M}\), we had the localisation functor \(\mathcal{M}(*D) = \mathcal{M} \otimes \mathcal{O}_{\mathscr{X}}(*D)\). One has the analogous notion for an \(\mathscr{R}_{\mathscr{X}}\)-module \(\mathscr{M}\): 
\begin{align*}
\mathscr{M}(*D) \coloneqq \mathscr{M} \otimes \mathcal{O}_{\mathscr{X}}(*(D \times \mathbb{C}_z)), 
\end{align*}

which we call the \textit{naive} localisation functor for \(\mathscr{R}_{\mathscr{X}}\)-modules. In particular, we obtain \(\mathscr{R}_{\mathscr{X}}(*D)\), and \(\mathscr{M}(*D)\) is naturally equipped with an \(\mathscr{R}_{\mathscr{X}}(*D)\)-module structure. This functor is called \textit{naive} because it is not well-behaved: It does not send coherent \(\mathscr{R}_{\mathscr{X}}\)-modules to coherent \(\mathscr{R}_{\mathscr{X}}\)-modules. In the following, we recall C. Sabbah's definition of the non-naive localisation functor, which fixes this issue. 

\subsection{Kashiwara-Malgrange filtration for \(\mathscr{R}\)-modules}

Let \(X\) be a complex manifold. Recall that for \(\mathcal{D}_X\)-modules we introduced a localisation functor, and then we introduced the Kashiwara-Malgrange filtration as a tool to analyse whether a \(\mathcal{D}\)-module is localised. For the (non-naive) localisation of \(\mathscr{R}_{\mathscr{X}}\)-modules, C. Sabbah proceeds in the reverse order, cf. \cite{Sab05}: One introduces a notion of Kashiwara-Malgrange filtration for \(\mathscr{R}_{\mathscr{X}}\)-modules, and then one uses it to define the localisation functor. In the following, we give a brief summary of this construction. 

\begin{definition}
Let \(\mathscr{M}\) be a \(\mathscr{R}_{\mathscr{X}}\)-module and let \(D \subset X\) be a codimension \(1\) sub-manifold/-variety, that is, a hypersurface. Let \(t\) be a coordinate function, whose vanishing locus is \(D\) and consider a local coordinate system \((t, x_2, \ldots, x_n)\) of \(X\). 

\begin{itemize}
    \item The \(V\)-filtration of \(\mathscr{R}_{\mathscr{X}}\) with respect to \(t\) is ascending and indexed by \(\mathbb{Z}\). It is given by 
    \begin{align*}
    V^{(t)}_k\mathscr{R}_{\mathscr{X}} = \begin{cases}
    \mathcal{O}_{\mathscr{X}}\langle t z\partial_t, z\partial_{x_2}, \ldots z\partial_{x_n} \rangle & \text{if } k = 0, \\ \sum_{i = 0}^k (z\partial_t)^i V^{(t)}_0\mathscr{R}_{\mathscr{X}} & \text{if } k > 0, \\ t^{-k}V^{(t)}_0\mathscr{R}_{\mathscr{X}} & \text{if } k < 0. 
    \end{cases}
    \end{align*}
    \item For a local section \(m \in \mathscr{M}\), a Bernstein-(Sato-)polynomial \(b_m(v)\) is defined as the minimal non-trivial polynomial of the form \(b_m(s) = \Pi_{\alpha \in A} \Pi_{l \in \mathbb{Z}} (s - (\alpha + l) \star z)^{i_{\alpha, l}}\) (for some finite set \(A \in \mathbb{C}\) and of finite degree) such that 
    \begin{align*}
    b_m(-tz\partial_t) \cdot m \in V_{-1}\mathscr{R}_{\mathscr{X}} \cdot m. 
    \end{align*}
    We remark that the coefficients of \(b_m(s)\) depend on \(z\), but only through the \(\star\) operation. 
    \item If every local section of \(\mathscr{M}\) has a Bernstein polynomial with respect to \(t\) with the same set \(A\), \(\mathscr{M}\) is called specializable with respect to \(t\) (or along \(D\), which is independent of the choice of \(t\)). For \(z' \in \mathbb{C}_z, \alpha = \alpha' + i\alpha''\in \mathbb{C}\), define 
    \begin{align*}
    l_{z'}(\alpha) \coloneqq \Re(\alpha' + iz'\alpha''). 
    \end{align*}    
    For a Bernstein-polynomial \(b_m\) of a section \(m \in \mathscr{M}\), the \(l_{z'}\)-order of \(b_m\) is defined as \(\operatorname{ord}_{l_{z'}}(b_m) \coloneqq \max(l_{z'}(r) : b_m(r) = 0)\). Near \(X \times \{z'\} \subset \mathscr{X}\), define the filtration by the \(l_{z'}\)-order as 
    \begin{align*}
    V_k^{(z')}\mathscr{M} \coloneqq \{m \in \mathscr{M} : \operatorname{ord}_{l_{z'}}(b_m) \leq k\}
    \end{align*}
    for \(k \in \mathbb{R}\). Note that this filtration only jumps at the discrete set \(l_{z'}(A + \mathbb{Z})\). If, for any \(z'\), this filtration has no \(z\)-torsion, one says that \(V^{(z')}\mathscr{M}\) is the Kashiwara-Malgrange filtration with respect to \(t\) near \(X \times \mathbb{C}_z\). 
\end{itemize}
\end{definition}

\begin{remark}
In \cite{Moc15} 2.1.2, T. Mochizuki defines the specializability in a more general way: He allows a more general form of Bernstein-polynomial, and what we call "specializable" in the sense of \cite{Sab05} actually corresponds to "\(\sqrt{-1}\mathbb{R}\)-specializable" in the language of \cite{Moc15}. 
\end{remark}

Since \(l_{z'}\) only depends on \(z'\) through \(\Im(z')\), the filtrations by the \(z'\)-order glue for different \(z' \in \mathbb{R} + i\zeta\) for any \(\zeta \in \mathbb{R}\). However, the \(V\)-filtrations do not glue globally with respect to \(z'\) in general, because the filtrations may jump at different indices, namely at \(l_{z'}(A +\mathbb{Z})\). See \cite{Sab05}, Remarks 3.3.6(2) for more details. In the case that these filtrations have no \(z\)-torsion (so that they are Kashiwara-Malgrange filtrations), one gets a decomposition 
\begin{align*}
\operatorname{gr}_a^{V^{(z')}}\mathscr{M} = \bigoplus_{l_{z'}(\alpha) = a} \psi_{t, \alpha}^{(z')}\mathscr{M}, 
\end{align*}

with 
\begin{align*}
\psi_{t, \alpha}^{(z')}\mathscr{M} \coloneqq \bigcup_n \ker[(z\partial_t t + \alpha \star z)^n : \operatorname{gr}_{l_{z'}(\alpha)}^{V^{(z')}}\mathscr{M} \rightarrow \operatorname{gr}_{l_{z'}(\alpha)}^{V^{(z')}}\mathscr{M}], 
\end{align*}

cf. \cite{Sab05}, Lemma 3.3.4(6). Even though the filtrations \(V^{(z')}\mathscr{M}\) may not glue, the \(\psi_{t, \alpha}^{(z')}\mathscr{M}\) glue to give an \(\mathscr{R}_{D \times \mathbb{C}_z}\)-module \(\psi_{t, \alpha}\mathscr{M}\), cf. \cite{Sab05}, Lemma 3.3.4(7). We emphasize that this is only possible if the graded pieces of the filtration \(V^{(z')}_\bullet\mathscr{M}\) have no \(z\)-torsion. 

\subsection{Strict specializability}

To get a well-behaved localisation functor, we need a condition on \(\mathscr{R}_{\mathscr{X}}\)-modules, which is even stronger than specializability: 

\begin{definition}\label{def:strictspec}
Let \(\mathscr{M}\) be an \(\mathscr{R}_{\mathscr{X}}\)-module which is specializable along \(t\). \(\mathscr{M}\) is called strictly specializable along \(t\), if it satisfies the following conditions: 
\begin{itemize}
    \item For every \(z' \in \mathbb{C}_z\), the graded pieces of \(V^{(z')}\mathscr{M}\) have no \(z\)-torsion, 
    \item \(t: \psi_{t, \alpha}\mathscr{M} \rightarrow \psi_{t, \alpha - 1}\mathscr{M}\) is surjective for all \(\alpha \in A + \mathbb{Z}\) with \(l_{z'}(\alpha) < 0\), 
    \item \(z\partial_t: \psi_{t, \alpha}\mathscr{M} \rightarrow \psi_{t, \alpha + 1}\mathscr{M}\) is surjective for all \(\alpha \in A + \mathbb{Z}\) with \(l_{z'}(\alpha) \geq - 1\) and \(\alpha \neq -1\). 
\end{itemize}
\end{definition}

We write \(\mathbb{C}_{t, z} \coloneqq \mathbb{C}_t \times \mathbb{C}_z\). 

\begin{definition}
Consider the map \(i: \{0\} \times \mathbb{C}_z \hookrightarrow \mathbb{C}_t \times \mathbb{C}_z\). Define \(\mathcal{O}_{0 \times \mathbb{C}_z}\) as the sheaf-theoretical preimage \(i^{-1}\mathcal{O}_{\mathbb{C}_t \times \mathbb{C}_z}\) (taking stalks in the \(t\)-direction) and \(\mathscr{R}_{0 \times \mathbb{C}_z} \coloneqq  \mathcal{O}_{0 \times \mathbb{C}_z}\langle z\partial_t \rangle\), which has the \(V\)-filtration induced by the one on \(\mathscr{R}_{\mathbb{C}_{t, z}}\). An operator \(\mathscr{P} \in \mathscr{R}_{0 \times \mathbb{C}_z}\) shall be called \(V\)-homogeneous of degree \(k\), if one can write 
\begin{align*}
\mathscr{P} = \sum_{j = \max(0, k)}^n f_j(z)t^{-k + j} \cdot (z\partial_t)^j, \quad \operatorname{ where } n \in \mathbb{N}_0, f_j \in \mathcal{O}_{0 \times \mathbb{C}_z}. 
\end{align*}
\end{definition}

Because sections of \(\mathcal{O}_{0 \times \mathbb{C}_z}\) have unique Taylor expansions, every \(\mathscr{P} \in \mathscr{R}_{0 \times \mathbb{C}_z}\) can be uniquely written as a sum
\begin{align}\label{eq:Vhomdecomp}
\mathscr{P} = \sum_{k = -\infty}^{\infty} v_k(\mathscr{P}), 
\end{align}
where \(v_k(\mathscr{P}) \in \mathscr{R}_{0 \times \mathbb{C}_z}\) is \(V\)-homogeneous of degree \(k\) and \(n \in \mathbb{Z}\), and where \(v_k(\mathscr{P}) = 0\) for all \(k\) greater than some \(k' \in \mathbb{Z}\). The minimal such \(k'\) is denoted as \(\deg_V(\mathscr{P})\) and is called the \(V\)-degree of \(\mathscr{P}\). Furthermore, set \(v_{\geq k} (\mathscr{P}) \coloneqq \sum_{i \geq k} v_i(\mathscr{P})\). 

\begin{remark}
We emphasize that the notion of \(V\)-homogeneity depends on the chosen coordinates. Furthermore, the decomposition in \Cref{eq:Vhomdecomp} is induced by the canonical isomorphism 
\begin{align*}
\mathscr{R}_{0 \times \mathbb{C}_z} \cong \mathcal{O}_{0 \times \mathbb{C}_z} \otimes_{\mathbb{O}_{\mathbb{C}_z}[t]} \operatorname{gr}^V \mathscr{R}_{0 \times \mathbb{C}_z}, 
\end{align*}

which equips \(\mathscr{R}_{0 \times \mathbb{C}_z}\) with an "analytic" grading (which is not an actual grading). 
\end{remark}

We have the following technical, but unsurprising result: 

\begin{lemma}\label{lemma:symboldivision}
Let \(\mathscr{P}, \mathscr{Q} \in \mathscr{R}_{0 \times \mathbb{C}_z}\) and set \(a\coloneqq \deg_V(\mathscr{P})\), \(b \coloneqq \deg_V(\mathscr{Q})\). Let \(z' \in \mathbb{C}\) and assume \(v_a(\mathscr{P}) \notin (z-z')\mathscr{R}_{0 \times \mathbb{C}_z}\). If \(v_{a+b}(\mathscr{P}\mathscr{Q}) \in (z-z')\mathscr{R}_{0 \times \mathbb{C}_z}\), then \(v_{b}(\mathscr{Q}) \in (z-z')\mathscr{R}_{0 \times \mathbb{C}_z}\). 
\end{lemma}

\begin{proof}
Set \(\mathcal{I} \coloneqq (z-z')\mathscr{R}_{0 \times \mathbb{C}_z}\) and \(\mathscr{R}_{0 \times z'} \coloneqq \mathscr{R}_{0 \times \mathbb{C}_z}/\mathcal{I}\). \(z-z'\) is \(V\)-homogeneous of degree \(0\), so that there is a notion of \(V\)-homogeneity for sections of \(\mathscr{R}_{0 \times z'}\) and every section \(\mathscr{P} \in \mathscr{R}_{0 \times z'}\) decomposes uniquely as a sum - analogous to \Cref{eq:Vhomdecomp} - of \(V\)-homogeneous elements: 
\begin{align}
\mathscr{P} = \sum_{k = -\infty}^{\deg_V(\mathscr{P})} v_k(\mathscr{P}). 
\end{align}

We have the canonical projection \(\pi: \mathscr{R}_{0 \times \mathbb{C}_z} \rightarrow \mathscr{R}_{0 \times z'}\), which is compatible with homogeneous summands. Assume for the sake of contradiction that \(v_b(\mathscr{Q}) \notin \mathcal{I}\). We have 
\begin{align*}
v_{a + b}(\pi(\mathscr{P}\mathscr{Q})) = \sum_{i + j = a+b}v_{i}(\pi(\mathscr{P})) \cdot v_{j}(\pi(\mathscr{Q})). 
\end{align*}

Because \(a = \deg_V(\mathscr{P})\), \(b = \deg_V(\mathscr{Q})\), we get 
\begin{align*}
v_{a+b}(\pi(\mathscr{P}\mathscr{Q})) = v_{a}(\pi(\mathscr{P})) \cdot v_{b}(\pi(\mathscr{Q})). 
\end{align*}

We have \(v_{a}(\pi(\mathscr{P})), v_{b}(\pi(\mathscr{Q})) \neq 0\), and so \(v_{a+b}(\pi(\mathscr{P}\mathscr{Q})) \neq 0\), because \(\mathscr{R}_{0 \times z'}\) has no zero divisors. On the other hand, we had assumed that \(v_{a + b}(\mathscr{P}\mathscr{Q}) \in (z-z')\mathscr{R}_{0 \times \mathbb{C}_z}\). This gives the desired contradiction. 
\end{proof}

\subsection{Localisation, dual localisation and minimal extension of \(\mathscr{R}\)-modules}

Let \(\mathscr{M}\) be a coherent \(\mathscr{R}_{\mathscr{X}}\)-module which is strictly specializable along \(t\). After \cite{Sab05}, Lemma 3.4.1, the naive localisation \(\mathscr{M}(*D)\) is coherent as an \(\mathscr{R}_{\mathscr{X}}(*D)\)-module (but not necessarily as an \(\mathscr{R}_{\mathscr{X}}\)-module, as already mentioned) and also strictly specializable along \(t\). This means that one has the Kashiwara-Malgrange filtration of \(\mathscr{M}(*D)\), which allows to make the following definitions: 

\begin{definition}
Let \(\mathscr{M}\) be a coherent \(\mathscr{R}_{\mathscr{X}}\)-module which is strictly specializable along \(t\). 
\begin{itemize}
    \item Let \(z' \in \mathbb{C}_z\) and consider a small open neighborhood \(U\) of \(z'\). The (non-naive) localisation of \(\mathscr{M}\vert_{X \times U}\) along \(D\) is defined as 
    \begin{align*}
    \mathscr{M}[*D]\vert_{X \times U} \coloneqq \mathscr{R}_{\mathscr{X}}\vert_{X \times U} \cdot \left.V_0^{(z')}(\mathscr{M}(*D))\right\vert_{X \times U} \subset \mathscr{M}(*D)\vert_{X \times U}, 
    \end{align*}
    Varying \(z' \in \mathbb{C}_z\), the \(\mathscr{M}[*D]\vert_{X \times U}\) glue to give an \(\mathscr{R}_{\mathscr{X}}\)-module \(\mathscr{M}[*D]\), cf. \cite{Moc15}, Lemma 3.1.4. 
    \item In contrast to the case of \(\mathcal{D}_X\)-modules, the dual localisation of \(\mathscr{M}\) is not defined using a duality functor for \(\mathscr{R}_{\mathscr{X}}\)-modules. Instead, let \(z' \in \mathbb{C}_z\) and consider a small open neighborhood \(U\) of \(z'\). The dual localisation of \(\mathscr{M}\vert_{X \times U}\) along \(D\) is defined as 
    \begin{align*}
    \mathscr{M}[!D]\vert_{X \times U} \coloneqq \mathscr{R}_{\mathscr{X}}\vert_{X \times U} \otimes_{V_0\mathscr{R}_{\mathscr{X}}\vert_{X \times U}} \left.V_{<0}^{(z')}(\mathscr{M}(*D))\right\vert_{X \times U}
    \end{align*}
    Varying \(z' \in \mathbb{C}_z\), the \(\mathscr{M}[!D]\vert_{X \times U}\) glue to give an \(\mathscr{R}_{\mathscr{X}}\)-module \(\mathscr{M}[!D]\), cf. \cite{Moc15}, Lemma 3.1.10. 
    \item As in the case of \(\mathcal{D}_X\)-modules, there are natural morphisms \(\mathscr{M}[!D] \rightarrow \mathscr{M} \rightarrow \mathscr{M}[*D]\) (cf. \cite{Moc15}, Lemma 3.2.8) and one defines the minimal extension \(\mathscr{M}[{!*}D]\) as the image of the composed map \(\mathscr{M}[!D] \rightarrow \mathscr{M}[*D]\). 
\end{itemize}
\end{definition}

\begin{remark}
We make some comments on this definition: 
\begin{itemize}
    \item As desired, \(\mathscr{M}[*D]\), \(\mathscr{M}[!D]\) and \(\mathscr{M}[{!*}D]\) are \(\mathscr{R}_{\mathscr{X}}\)-coherent and strictly specializable along \(t\), cf. \cite{Moc15}, Lemma 3.1.1, Lemma 3.1.10. 
    \item One can define a non-naive localisation functor for \(\mathcal{D}_X\)-modules in the same way. However, one simply recovers the usual localisation of \(\mathcal{D}_X\)-modules. 
    \item In the case of \(\mathscr{R}_{\mathscr{X}}\)-modules belonging to Twistor \(\mathcal{D}\)-modules, one does get a well-behaved duality functor \(\mathbb{D}\) and one can write \(\mathscr{M}[!D] = \mathbb{D}((\mathbb{D}\mathscr{M})[*D])\), cf. \cite{Moc15}, Proposition 13.2.1. 
\end{itemize}
\end{remark}

\begin{definition}
Step 5 of constructing \(\mathfrak{T}\) from \(\mathfrak{M}\) is defined as taking the minimal extension over the singular locus \(D\) of the meromorphic bundle with \(z\)-connection \(\widetilde{\mathfrak{M}}\): 
\begin{align*}
\mathfrak{M}_z \coloneqq \widetilde{\mathfrak{M}}_z[{!*}D]. 
\end{align*}
\end{definition}

\subsection{Step 5 in the order \(1\) case} 

We continue to use the notation of \Cref{theorem:firstmainresult}. We have already determined the meromorphic bundle with \(z\)-connection \(\widetilde{\mathfrak{M}}_z\) in \Cref{lemma:merozex}. In the following, we will determine the \(\mathscr{R}_{\mathbb{P}^1 \times \mathbb{C}_z}\)-module obtained as the minimal extension of \(\widetilde{\mathfrak{M}}_z\) at the points of \(D\). To do so, we will consider the following Twistor-version of the operator \(\mathcal{P}\): 

\begin{definition}\label{def:twistedzoperator}
Set 
\begin{align*}
\mathcal{P}_z \coloneqq p(t)z\partial_t - p(t) \sum_{j = 1}^n \frac{m_j \star z}{t - r_j} \in \mathcal{R}
\end{align*}

and \(\mathcal{P}_{z, k} \coloneqq \mathcal{P}_z + kp(t)z/t \in \mathcal{R}\) for \(k \in \mathbb{Z}\). (Note the analogy to the already defined operator \(\mathcal{P}_k\).) Note that setting \(z = 1\) in \(\mathcal{P}_z\) gives \(\mathcal{P}\). Furthermore, \(\mathcal{P}_z\) and \(\mathcal{P}_{z, k}\) can be seen as global sections of \(\mathscr{R}_{\mathbb{C}_{t, z}}\). 
\end{definition}

The following result shows that the naive localisation functor is not "well-behaved", because it does not take into account the \(\star z\) operation: 

\begin{lemma}
Let \(\alpha \coloneqq \alpha(\mathcal{P})\) and let \(U\) be an open subset of \(\mathbb{C}_z^*\) with the property that the function \(z \mapsto \Re(\alpha \diamond z)\) is strictly bounded from above by a natural number \(k(U)\). Then, 
\begin{align*}
(\mathscr{R}_{\mathbb{C}_{t, z}}/(\mathcal{P}_z))(*0)\vert_U = \mathscr{R}_{\mathbb{C}_{t, z}}/(\mathcal{P}_{z, k})\vert_U
\end{align*}

for all \(k \geq k(U)\). 
\end{lemma}

\begin{proof}
So, let \(z' \neq 0\). We have 
\begin{align*}
(\mathscr{R}_{\mathbb{C}_{t, z}}/(\mathcal{P}_z))(*0)\vert_{z = z'} = (\mathscr{R}_{\mathbb{C}_{t, z}}/(\mathcal{P}_{z, k}))\vert_{z = z'}(*0), 
\end{align*}

where the left-hand side (resp. the right-hand side) denotes localisation of \(\mathscr{R}\)-modules (resp. \(\mathcal{D}\)-modules). This is clear because \(\mathcal{O}_{\mathbb{C}_t \times \mathbb{C}_z}(*(0 \times \mathbb{C}_z))\vert_{z = z'} \cong \mathcal{O}_{\mathbb{C}_t}(*0)\). We have \((\mathscr{R}_{\mathbb{C}_{t, z}}/(\mathcal{P}_z))\vert_{z = z'} = \mathcal{D}_{\mathbb{C}_t}/(\mathcal{P}_z\vert_{z = z'})\). Noting that \(\alpha(\mathcal{P}_z\vert_{z = z'}) = \alpha \diamond z'\), the claim follows with \Cref{lemma:locandroots}. 
\end{proof}

On the other hand, the \(V\)-filtration of \(\mathscr{R}_{\mathbb{C}_{t, z}}/(\mathcal{P}_z)\) takes a simple form: 

\begin{lemma}\label{prop:zbernsteinrank1}
The following statements are true: 
\begin{itemize}
    \item \(\mathscr{R}_{\mathbb{C}_{t, z}}/(\mathcal{P}_z)\) is specializable with respect to \(t = 0\). The Bernstein-polynomial of \([1] \in \mathscr{R}_{\mathbb{C}_{t, z}}/(\mathcal{P}_z)\) with respect to \(t = 0\) is given by 
    \begin{align*}
    b_{{[1]}}(s) = s + \alpha(\mathcal{P}) \star z + 1. 
    \end{align*}
    \item The \(V\)-filtration of \(\mathscr{R}_{\mathbb{C}_{t, z}}/(\mathcal{P}_z)\) near \(z' \in \mathbb{C}_z\) (with respect to \(t = 0\)) takes the form 
    \begin{align*}
    V^{(z')}_k\mathscr{R}_{\mathbb{C}_{t, z}}/(\mathcal{P}_z) = \left[V_{\lfloor k + l_{z'}(\alpha) \rfloor + 1}\mathscr{R}_{\mathbb{C}_{t, z}}\right]
    \end{align*}
    for \(k \in \mathbb{R}\). 
\end{itemize}
\end{lemma}

\begin{proof}
The proof is analogous to the proof of the first point of \Cref{prop:bernsteinrank1}. 
\end{proof}

We want to show that \(\mathscr{R}_{\mathbb{C}_{t, z}}/(\mathcal{P}_z)\) is strictly specializable along \(t\). In preparation, we need the following result: 

\begin{lemma}\label{lemma:noztorsion}
For every \(z' \in \mathbb{C}_z\), the graded pieces of \(V^{(z')}_\bullet(\mathscr{R}_{\mathbb{C}_{t, z}}/(\mathcal{P}_z))\) have no \(\mathcal{O}_z\)-torsion. 
\end{lemma}

\begin{proof}
It suffices to check the claim for \(V_\bullet(\mathscr{R}_{0 \times \mathbb{C}_z}/(\mathcal{P}_z))\), since \(V_\bullet(\mathscr{R}_{\mathbb{C}_{t, z}}/(\mathcal{P}_z))\) is constant away from \(t = 0\). In this proof, all equivalence classes are taken in \(\mathscr{R}_{0 \times \mathbb{C}_z}/(\mathcal{P}_z)\). Set \(d \coloneqq \deg_V(\mathcal{P}_z)\) and \(k' \coloneqq \lfloor k + l_{z'}(\alpha) + 1 \rfloor\). For the sake of contradiction, assume there was 
\begin{align}\label{eq:contr1}
q \in V^{(z')}_k(\mathscr{R}_{0 \times \mathbb{C}_z}/(\mathcal{P}_z)) - V^{(z')}_{k-1}(\mathscr{R}_{0 \times \mathbb{C}_z}/(\mathcal{P}_z)) 
\end{align}

with \(sq \in V^{(z')}_{k-1}\mathscr{R}_{0 \times \mathbb{C}_z}/(\mathcal{P}_z)\) for some \(s \in \mathcal{O}_{\mathbb{C}_z}, k \in \mathbb{Z}\). Without loss of generality we may assume that \(s(z) = z - r\) for some \(r \in \mathbb{C}_z\). By \Cref{prop:zbernsteinrank1}, we have \(q = [\mathscr{Q}]\) for some \(\mathscr{Q} \in V_{k'}\mathscr{R}_{0 \times \mathbb{C}_z} - V_{k' - 1}\mathscr{R}_{0 \times \mathbb{C}_z}\). By the definition of equivalence classes, this implies 
\begin{align}\label{eq:contr2}
\mathscr{Q} \neq \mathscr{U}'\mathcal{P}_z + \mathscr{U}'' \text{ for all } \mathscr{U}' \in \mathscr{R}_{0 \times \mathbb{C}_z}, \mathscr{U}'' \in V_{k' - 1}\mathscr{R}_{0 \times \mathbb{C}_z}, 
\end{align}

\Cref{eq:contr2} implies 
\begin{align}\label{eq:nosuchU}
v_{k'}(\mathscr{Q}) \neq v_{k'}(\mathscr{U} \mathcal{P}_z) \text{ for all } \mathscr{U} \in V_{k'-d}\mathscr{R}_{0 \times \mathbb{C}_z}. 
\end{align}

Indeed, if we had such \(\mathscr{U}\) giving equality, we would have 
\begin{align*}
\mathscr{Q} = \mathscr{U}\mathcal{P}_z - \underbrace{v_{\leq k' - 1}(\mathscr{U}\mathcal{P}_z) + v_{\leq k' - 1}(\mathscr{Q})}_{\in V_{k' - 1}\mathscr{R}_{0 \times \mathbb{C}_z}}, 
\end{align*}

which would contradict \Cref{eq:contr2}. On the other hand, we assumed \((z-r)q \in V^{(z')}_{k-1}\mathscr{R}_{0 \times \mathbb{C}_z}/(\mathcal{P}_z)\), so that \((z-r)q = [\mathscr{T}]\) for some \(\mathscr{T} \in V_{k' - 1}\mathscr{R}_{0 \times \mathbb{C}_z}\). We have \([\mathscr{T}] = [(z-r)\mathscr{Q}]\), so 
\begin{align*}
(z-r)\mathscr{Q} = \mathscr{S}\mathcal{P}_z + \mathscr{T}
\end{align*}

for some \(\mathscr{S} \in \mathscr{R}_{0 \times \mathbb{C}_z}\). Since \(\deg_V(\mathscr{Q}) = k'\), we must have \(\deg_V(\mathscr{S}) = k'-d\). The condition \(\mathscr{T} \in V_{k' - 1}\mathscr{R}_{0 \times \mathbb{C}_z}\) implies 
\begin{align}\label{eq:strictfinal}
(z-r)v_{k'}(\mathscr{Q}) = v_{k'}(\mathscr{S}\mathcal{P}_z), 
\end{align}

in particular: \(v_{k'}(\mathscr{S}\mathcal{P}_z) \in (z-r)\mathscr{R}_{0 \times \mathbb{C}_z}\). Note that \(v_d(\mathcal{P}_z) \notin (z-r)\mathscr{R}_{0 \times \mathbb{C}_z}\). With \Cref{lemma:symboldivision} we deduce that \(v_{k'-d}(\mathscr{S}) \in (z-r)\mathscr{R}_{0 \times \mathbb{C}_z}\). Write \(v_{k'-d}(\mathscr{S}) = (z-r)\mathscr{U}\), where \(\mathscr{U}\) is \(V\)-homogeneous of degree \(k' - d\). We derive from \Cref{eq:strictfinal}: 
\begin{align*}
(z - r)v_{k'}(\mathscr{Q}) & = v_{k'-d}(\mathscr{S})v_d(\mathcal{P}_z) \\ & = (z-r)\mathscr{U} \cdot v_d(\mathcal{P}_z) \\ &  = (z - r)v_{k'}(\mathscr{U}\mathcal{P}_z). 
\end{align*}

Because \(\mathscr{R}_{0 \times \mathbb{C}_z}\) has no zero divisors, this implies \(v_{k'}(\mathscr{Q}) = v_{k'}(\mathscr{U}\mathcal{P}_z)\). However, we had already deduced in \Cref{eq:nosuchU} that such an equality can not hold, giving the desired contradiction. 
\end{proof}

\begin{lemma}\label{lemma:strictlyspecializable}
The \(\mathscr{R}_{\mathbb{C}_{t, z}}\)-module \(\mathscr{R}_{\mathbb{C}_{t, z}}/(\mathcal{P}_z)\) is strictly specializable at every point \(t' \in \mathbb{C}_{t}\). 
\end{lemma}

\begin{proof}
We work near a point \((t', z') \in \mathbb{C}_{t, z}\), and after a suitable coordinate shift in the \(t\)-direction, we may assume that \(t' = 0\). We already know from \Cref{lemma:noztorsion} that the graded pieces of \(V^{(z')}_\bullet(\mathscr{R}_{\mathbb{C}_{t, z}}/(\mathcal{P}_z))\) have no \(\mathcal{O}_z\)-torsion. It remains to check the surjectivity conditions for \(t\) and \(z\partial_t\) acting on the graded pieces of \(V^{(z')}_\bullet(\mathscr{R}_{\mathbb{C}_{t, z}}/(\mathcal{P}_z))\), as given in \Cref{def:strictspec}. We distinguish two cases: 
\begin{itemize}
    \item If \(\mathcal{P}\) is regular singular at \(t = 0\), then we know from \Cref{lemma:localalphapicture} that 
    \begin{align*}
    (\mathcal{D}_{\mathbb{C}_t}/(\mathcal{P}))\vert_U = (\mathcal{D}_{\mathbb{C}_t}/(t\partial_t - \alpha(\mathcal{P}))\vert_U
    \end{align*}
    for an open neighborhood \(U\) of \(0\). This means that 
    \begin{align*}
    (\mathscr{R}_{\mathbb{C}_{t, z}}/(\mathcal{P}_z))\vert_U = (\mathscr{R}_{\mathbb{C}_{t, z}}/(tz\partial_t - \alpha(\mathcal{P}) \star z))\vert_U
    \end{align*}
    (as is seen by applying the construction laid out above) and we can conclude with an argument analogous to \Cref{lemma:strictspec1}. 
    \item If \(\mathcal{P}\) is non-singular at \(t = 0\), we have \((\mathcal{D}_{\mathbb{C}_t}/(\mathcal{P}))\vert_U = (\mathcal{D}_{\mathbb{C}_t}/(\partial_t))\vert_U\) for an open neighborhood \(U\) of \(0\). This means that \((\mathscr{R}_{\mathbb{C}_{t, z}}/(\mathcal{P}_z))\vert_U = (\mathscr{R}_{\mathbb{C}_{t, z}}/(z\partial_t))\vert_U\) and we can conclude with an argument analogous to \Cref{lemma:strictspec2}. 
\end{itemize}
This completes the proof. 
\end{proof}

Recall that we denoted the singular locus of \(\mathcal{P}\) by \(D\). 

\begin{lemma}\label{lemma:minextexample}
We have \(\mathscr{R}_{\mathbb{C}_{t, z}}/(\mathcal{P}_z) = \mathscr{R}_{\mathbb{C}_{t, z}}/(\mathcal{P}_z)[{!*}D]\). 
\end{lemma}

\begin{proof}
This is a local assertion and it suffices to show that 
\begin{align*}
\mathscr{R}_{\mathbb{C}_{t, z}}/(\mathcal{P}_z) = \mathscr{R}_{\mathbb{C}_{t, z}}/(\mathcal{P}_z)[{!*}\{p\}] 
\end{align*}

for every \(p \in D\). After a suitable coordinate shift in the \(t\)-direction, we may assume that \(p = 0\). We know from \Cref{lemma:strictlyspecializable} that \(\mathscr{R}_{\mathbb{C}_{t, z}}/(\mathcal{P}_z)\) is strictly specializable along \(t\). Because \(\mathfrak{M}\) is irreducible, we have \(\mathfrak{M}({!*}0) = \mathfrak{M}\), and so we have \(\alpha(\mathcal{P}) \notin \mathbb{Z}\). It follows from \Cref{prop:zbernsteinrank1} that \(\psi_{t, k}\mathscr{M} = 0\) for \(k \in \mathbb{Z}\). The claim then follows from \cite{Sab05}, Proposition 3.4.6. 
\end{proof}

We have 
\begin{align*}
(\mathscr{R}_{\mathbb{C}_{t, z}}/(\mathcal{P}_z))(*D) = \left(\mathcal{O}_{\mathbb{C}_t \times \mathbb{C}_z}(*D), d - \sum_{i = 1}^n \frac{-m_i \diamond z}{t - r_i}dt\right). 
\end{align*}

After \Cref{lemma:minextexample}, we have \(\mathscr{R}_{\mathbb{C}_{t, z}}/(\mathcal{P}_z) = \mathscr{R}_{\mathbb{C}_{t, z}}/(\mathcal{P}_z)[{!*}D]\). Consider the \(\mathscr{R}_{\mathbb{A}^1 \times \mathbb{C}_z}\)-module \(\mathscr{R}_{\mathbb{A}^1 \times \mathbb{C}_z}/(\mathcal{P}_z)\). After \Cref{lemma:partialGAGA}, it corresponds to a \(\mathscr{R}_{\mathbb{P}^1 \times \mathbb{C}_z}(*\infty)\)-module, which is \(\mathfrak{M}_z(*\infty)\). \Cref{lemma:merozex} implies that \(\mathfrak{M}_z\) is the minimal extension of this \(\mathscr{R}_{\mathbb{P}^1 \times \mathbb{C}_z}(*\infty)\)-module at \(\infty\). 

\section{Step 6: Extending the pairing and the polarization}

The final step is to extend the variation of Twistor structure \((\mathscr{H}', \mathscr{H}', C)\) on \(X - D\), which was obtained in Step 3, to a so-called \(\mathscr{R}_{\mathscr{X}}\)-triple: 

\begin{definition}
Let \(X\) be a complex manifold. As usual, we set \(\mathscr{X} \coloneqq X \times \mathbb{C}_z\). An \(\mathscr{R}_{\mathscr{X}}\)-triple is a triple \((\mathfrak{M}', \mathfrak{M}'', \mathfrak{C})\), where \(\mathfrak{M}'\) and \(\mathfrak{M}''\) are \(\mathscr{R}_{\mathscr{X}}\) modules and \(\mathfrak{C}\) is a Twistor-sesquilinear pairing 
\begin{align*}
\mathfrak{C}: \overline{\sigma^*}\mathfrak{M}'\vert_{\boldsymbol{S}} \otimes_{\mathcal{O}_{\boldsymbol{S}}} \mathfrak{M}''\vert_{\boldsymbol{S}} \rightarrow \mathfrak{Db}_{X_{\mathbb{R}} \times \boldsymbol{S}/\boldsymbol{S}}. 
\end{align*}
\end{definition}

The sheaf on \(X \times \boldsymbol{S}\) of distributions \(\mathfrak{Db}_{X \times \boldsymbol{S}/\boldsymbol{S}}\) is defined in \cite{Sab05}, Section 0.5. We will not repeat the precise definition here, but we note that there is a natural inclusion 
\begin{align*}
C^{\infty, \operatorname{an}}_{X \times \boldsymbol{S}} \hookrightarrow \mathfrak{Db}_{X_{\mathbb{R}} \times \boldsymbol{S}/\boldsymbol{S}}. 
\end{align*}

This means that every variation of Twistor structure on \(X\) can naturally be seen as an \(\mathscr{R}_{\mathscr{X}}\)-triple. In fact, one can show that for every \(\mathscr{R}_{\mathscr{X}}\)-triple \((\mathfrak{M}', \mathfrak{M}'', \mathfrak{C})\), where \(\mathfrak{M}'\) and \(\mathfrak{M}''\) are \(\mathcal{O}_{\mathscr{X}}\)-locally free, the pairing \(\mathfrak{C}\) takes values in \(C^{\infty, \operatorname{an}}_{X \times \boldsymbol{S}}\), cf. \cite{Sab05}, Lemma 1.5.3. 

Recall that there is the notion of polarization on variations of Twistor structure. This notion can be generalized to polarizations of \(\mathscr{R}_{\mathscr{X}}\)-triples. We will not repeat this definition here. 

\begin{definition}
Step 6 is defined by first extending the variation of Twistor structure \((\mathscr{H}', \mathscr{H}', C)\) on \(X - D\) to the \(\mathscr{R}_{\mathscr{X}}\)-triple \((\mathfrak{M}, \mathfrak{M}, \mathfrak{C})\), where \(\mathfrak{M}\) was constructed in Step 5 and \(\mathfrak{C}\) is a natural extension of \(C\) defined as in \cite{Sab05}, 5.3.c. \((\mathfrak{M}, \mathfrak{M}, \mathfrak{C})\) is equipped with the polarization \((\operatorname{id}, \operatorname{id})\), where \(\operatorname{id}\) denotes the identity on \(\mathfrak{M}\). The result is a simple polarizable tame purely imaginary Twistor \(\mathcal{D}\)-module on \(X\) of weight \(0\), as defined in \cite{Sab05}, Definition 4.1.2. 
\end{definition}

\subsection{Step 6 in the order \(1\) case}

We go back to the situation of \Cref{theorem:firstmainresult}, whose proof is complete. Indeed, our obtained formulas for \(\mathfrak{M}\) and \(\mathfrak{C}\) coincide with the ones given in the statement of \Cref{theorem:firstmainresult}. 

\section{Integrability}

Let \(X\) be a complex manifold. One can sometimes equip Twistor \(\mathcal{D}\)-modules on \(X\) with an additional datum, called an integrable structure. This notion was introduced by C. Sabbah in \cite{Sab05}, Section 7.1: 

\begin{definition}
Let \(\mathfrak{M}\) be an \(\mathscr{R}_{\mathscr{X}}\)-module, which we view as an \(\mathcal{O}_{\mathscr{X}}\)-module with a flat \(z\)-connection \(\nabla\). Let \((\mathfrak{M}', \mathfrak{M}'', \mathfrak{C})\) be an \(\mathscr{R}_{\mathscr{X}}\)-triple. 
\begin{itemize}
    \item \(\mathfrak{M}\) is called integrable if \(\nabla\) is induced by a flat non-\(z\)-connection (with differentiation both in the \(\mathscr{X}\) and in the \(z\)-direction), whose connection matrix, in local coordinates \(t_1, \ldots, t_n\) of \(\mathscr{X}\), has holomorphic entries when expressed in the local basis \(dt_1, \ldots, dt_n, dz/z\). 
    \item \((\mathfrak{M}', \mathfrak{M}'', \mathfrak{C})\) is called integrable if \(\mathfrak{M}'\) and \(\mathfrak{M}''\) are so, and the pairing \(\mathfrak{C}\) is compatible with the integrable structures of \(\mathfrak{M}'\) and \(\mathfrak{M}''\) in the following sense: 
    \begin{align*}
    z\partial_z\mathfrak{C}(\overline{\sigma^*}u, v) = \mathfrak{C}(\overline{\sigma^*}u, z\partial_zv) - \mathfrak{C}(\overline{\sigma^*}(z\partial_zu), v). 
    \end{align*}
    \item A Twistor \(\mathcal{D}\)-module on \(X\) is called integrable if it is so as an \(\mathscr{R}_{\mathscr{X}}\)-triple. 
\end{itemize}
\end{definition}

Integrable structures on Twistor \(\mathcal{D}\)-modules are used to define - under certain conditions - an irregular Hodge filtration of the underlying \(\mathcal{D}_X\)-module, cf. \cite{SabYu18}, \cite{Moc22}. We now ask if the Twistor \(\mathcal{D}\)-modules on \(\mathbb{P}^1\) given by \Cref{theorem:firstmainresult} can be equipped with an integrable structure. By \cite{SabYu18}, Theorem 0.7, we expect this to only be the case if the underlying \(\mathcal{D}_{\mathbb{P}^1}\)-module \(\mathfrak{M}\) is locally unitary, that is, if the monodromy at every singularity has absolute value \(1\). In other words, we must have \(m_i \in \mathbb{R}\) for all \(i\) in the language of \Cref{theorem:firstmainresult}. For illustrative purposes, we can check directly that we do in fact get an integrable structure if and only if all \(m_i\) are real: Indeed, the integrable structure \(\nabla_{\operatorname{int}}\) on the non-singular part of \(\mathfrak{M}_z\) has to be given by 
\begin{align*}
\nabla_{\operatorname{int}} = \frac{d}{dt} + \frac{d}{dz} - \left(\sum_{i = 1}^n \frac{-m_i \diamond z}{t - r_i}dt + A(t, z)dz\right) 
\end{align*}

for some \(A \in z^{-1}\mathcal{O}_{\mathbb{C}_t \times \mathbb{C}_z}\). The flatness condition \(\nabla_{\operatorname{int}}^2  = 0\) is equivalent to 
\begin{align*}
d\left(\sum_{i = 1}^n \frac{-m_i \diamond z}{t - r_i}dt + A(t, z)dz\right) = 0, 
\end{align*}

which is rewritten as 
\begin{align*}
\frac{\partial}{\partial z} \left(\sum_{i = 1}^n \frac{-m_i \diamond z}{t - r_i}\right) dt \wedge dz + \frac{\partial A(t, z)}{\partial t} dt \wedge dz. 
\end{align*}

This is equivalent to 
\begin{align*}
\frac{\partial A(t, z)}{\partial t} = \frac{\partial}{\partial z} \left(\sum_{i = 1}^n \frac{m_i \diamond z}{t - r_i}\right). 
\end{align*}

Because \(A\) is a single-valued function, this is equivalent to \(\frac{\partial}{\partial z}(m_i \diamond z) = 0\) for all \(i\), which means \(\Re(m_i) = 0\). Under this condition, the above equation implies \(A \in z^{-1}\mathcal{O}_{\mathbb{C}_z}\). One verifies that the compatibility condition with \(\mathfrak{C}\) is satisfied with the choice \(A = 0\), turning \(\mathfrak{T}\) from \Cref{theorem:firstmainresult} into an integrable Twistor \(\mathcal{D}\)-module. 

\chapter{Fourier Transform of a Twistor \texorpdfstring{\(\mathcal{D}\)-module}{D-module}}\label{chapter:fouriertwistor}

\Cref{theorem:firstmainresult} gives us an explicit description of the Twistor \(\mathcal{D}\)-modules associated to \(\mathcal{D}_{\mathbb{P}^1}\)-modules of generic rank \(1\). But what about \(\mathcal{D}_{\mathbb{P}^1}\)-modules of higher generic rank? Our proof strategy breaks down: If the generic rank is higher than \(1\), it becomes very difficult to calculate the harmonic metric of Deligne type, which was crucial in constructing the Twistor \(\mathcal{D}\)-module. However, there is a different way to obtain Twistor \(\mathcal{D}\)-modules associated to \(\mathcal{D}_{\mathbb{P}^1}\)-modules of higher rank: There is a notion of Fourier transform in the category of Twistor \(\mathcal{D}\)-modules, which we can apply to Twistor \(\mathcal{D}\)-modules obtained with \Cref{theorem:firstmainresult}. In this way, we obtain a Twistor \(\mathcal{D}\)-module, whose underlying \(\mathcal{D}_{\mathbb{P}^1}\)-module has rank \(2\). Since the Fourier transform of regular holonomic \(\mathcal{D}\)-modules (except for trivial cases) has singular locus \(\{0, \infty\}\), where the singularity at the origin is regular and the other one is irregular, the obtained Twistor \(\mathcal{D}\)-module is wild and not tame. The following will be the main result of this section: 

\begin{theorem}\label{thm:FTofT}
Let \(\mathcal{P} \in \mathcal{D} = \mathbb{C}\langle t, \partial_t \rangle\) be the operator \(t\partial_t - \alpha - t(t\partial_t - \beta)\), where \(\alpha, \beta \in \mathbb{C}\) with \(\Re(\alpha) > -1, \Re(\beta - \alpha) > -1, \Re(\beta) < -1\). Let \(\mathcal{M} \coloneq \mathcal{D}/(\mathcal{P})\) and consider the \(\mathcal{D}_{\mathbb{P}^1}\)-module \(\mathfrak{M} = (p^{*, \operatorname{loc}}\mathcal{M})({!*}\infty)\). Then, we have 
\begin{align*}
\operatorname{FT}(\mathfrak{T}) \cong ((\mathfrak{N}_z, \mathfrak{N}_z, \mathfrak{D}), (\operatorname{id}, \operatorname{id})) \in \operatorname{MT}^{\operatorname{wild}}_{\mathbb{P}^1, \operatorname{s}}(\mathbb{P}^1, 0, \sqrt{-1}\mathbb{R})^{(p)}, 
\end{align*}

where \(\mathfrak{N}_z\) and \(\mathfrak{D}\) are given as follows: \(\mathfrak{N}_z\) is the \(\mathscr{R}_{\mathbb{P}^1 \times \mathbb{C}_z}\)-module \(\mathfrak{N}_z = (p^{*, \operatorname{loc}}\mathcal{N}_z)[{!*}\infty]\) where \(\mathcal{N}_z = \mathscr{R}/(\operatorname{FT}_z(\mathcal{P}_z))\) with \(\mathcal{P}_z \in \mathcal{R}\) being defined as in \Cref{theorem:firstmainresult}. The pairing \(\mathfrak{D}\) is a natural extension of the pairing 
\begin{align*}
\mathfrak{D}\vert_{\mathbb{C}_\tau^*}: \overline{\sigma^*}\mathfrak{N}_z\vert_{\mathbb{C}_\tau^* \times \boldsymbol{S}} & \otimes \mathfrak{N}_z\vert_{\mathbb{C}_\tau^* \times \boldsymbol{S}} & \rightarrow & \qquad C^{\infty, \operatorname{an}}_{\mathbb{C}_\tau^* \times \boldsymbol{S}} \\ \begin{pmatrix} \overline{\sigma^*}u_1 \\ \overline{\sigma^*}u_2\end{pmatrix} & \otimes \begin{pmatrix} v_1 \\ v_2\end{pmatrix} & \mapsto & \qquad \begin{pmatrix} \overline{\sigma^*}u_1 \\ \overline{\sigma^*}u_2\end{pmatrix}^T\cdot D_{\operatorname{mat}} \cdot \begin{pmatrix} v_1 \\ v_2\end{pmatrix}
\end{align*}

which is given in the basis of \(\Gamma(\mathbb{C}_\tau^* \times \boldsymbol{S}, \mathfrak{N}_{z})\) obtained as the image of the classes of \(1\) and \(z\partial_t\) under the \(\tau\)-analytification map 
\begin{align*}
\mathscr{R}(\mathbb{C}_\tau^{*,\operatorname{alg}} \times \boldsymbol{S})/(\operatorname{FT}_z(\mathcal{P}_z)) = \mathscr{R}(\mathbb{C}_\tau^{*,\operatorname{alg}} \times \boldsymbol{S}) \otimes_{\mathscr{R}(\boldsymbol{S})} \Gamma(\boldsymbol{S}, \mathcal{N}_z) \rightarrow \Gamma(\mathbb{C}_\tau^* \times \boldsymbol{S}, \mathfrak{N}_{z}). 
\end{align*}

Here, the matrix \(D_{\operatorname{mat}} \in (C^{\infty, \operatorname{an}}_{\mathbb{C}_\tau^* \times \boldsymbol{S}})^{2 \times 2}\) is given by \(D_{\text{mat}} = \overline{\sigma^*}B^T \cdot B\) with  
\begin{align*}
B(\tau, z) \coloneqq \operatorname{diag}(c_1(z), c_2(z)) \cdot \begin{pmatrix}
F_1(\tau, z) & z\partial_\tau(F_1(\tau, z)) \\ \tau^{-\beta \diamond z - 1}F_2(\tau, z) & z\partial_\tau(\tau^{-\beta \diamond z - 1}F_2(\tau, z)) 
\end{pmatrix}, 
\end{align*}

where the \(c_i\) are given as follows: Setting \(\widetilde\Gamma(s) \coloneqq \Gamma(1 + s)/\Gamma(-s)\) (\(s \in \mathbb{C}\)) and 
\begin{align*}
C_1(z) & \coloneqq \frac{\widetilde{\Gamma}(\alpha \diamond z)\widetilde{\Gamma}(\beta \diamond z - \alpha \diamond z)}{\widetilde{\Gamma}(1 + \beta \diamond z)}, \\  C_2(z) & \coloneqq \widetilde{\Gamma}(\beta \diamond z), 
\end{align*}

the \(c_i \in \mathcal{O}(\boldsymbol{S})\) \((i = 1,2)\) satisfy \(|c_i(z)|_z^2 = C_i(z)\) and are constructed via \Cref{thm:twistorroot}. The \(F_i \in \mathcal{O}(\mathbb{C}_\tau^* \times \boldsymbol{S})\) are given by 
\begin{alignat*}{4}
F_1(\tau, z) &\coloneqq {} & & \,_1F_1(1 + \alpha \diamond z, && \ 2 + \beta \diamond z, && \ -\tau/z), \\
F_2(\tau, z) &\coloneqq {} & & \,_1F_1(\alpha \diamond z - \beta \diamond z, && \ -\beta \diamond z, && \ -\tau/z).
\end{alignat*}

Here, \(_1F_1(a, b, \tau) \in \mathcal{O}_{\mathbb{C}_\tau}\) is the (confluent) hypergeometric Kummer function: 
\begin{align*}
_1F_1(a, b, \tau) = \sum_{n = 0}^\infty \frac{a^{(n)}\tau^n}{b^{(n)}n!}, 
\end{align*}

where \(a^{(n)}, b^{(n)}\) is the rising factorial. 
\end{theorem}

We note that the \(\mathcal{D}\)-modules considered in this theorem are hypergeometric \(\mathcal{D}\)-modules. In general they are defined by an operator of the form 

\begin{align*}
\prod_{i = 1}^m (t\partial_t - \rho_i) - t \prod_{j = 1}^n (t\partial_t - \gamma_j), 
\end{align*}

where \(\rho_i, \gamma_j \in \mathbb{C}\) and \(m, n \in \mathbb{N}\). As already mentioned in the first chapter, \Cref{thm:FTofT} has slight overlap with \cite{DomReiSev19}, Proposition 5.4: There, the formula for \(\mathscr{N}_z\) is already obtained in the case \(\alpha, \beta \in \mathbb{R}\). Our methods are different, allowing us to consider complex \(\alpha, \beta\) and to obtain the formula for the matrix \(D_{\operatorname{mat}}\). Before we begin with the proof of \Cref{thm:FTofT}, we recall the definition of the Fourier transform for \(\mathcal{D}\)-modules and Twistor \(\mathcal{D}\)-modules: 

\section{Definition of Fourier Transform for (Twistor) \texorpdfstring{\(\mathcal{D}\)-modules}{D-modules}}

Recall that, unless specified otherwise, we understand the complex projective line \(\mathbb{P}^1\) as being equipped with the analytic topology, whereas we understand the complex affine line \(\mathbb{A}^1\) as being equipped with the algebraic topology. One of the most important functors on the category of \(\mathcal{D}\)-modules is the Fourier-Transform, which we recall as follows: 

\begin{definition}
Let \(\mathcal{D}_t = \mathbb{C}\langle t, \partial_t \rangle\) and \(\mathcal{D}_\tau = \mathbb{C}\langle \tau, \partial_\tau \rangle\). Define the \(\mathbb{C}\)-linear Fourier automorphism 
\begin{align*}
\text{FT}: \mathcal{D}_t \rightarrow \mathcal{D}_\tau, \quad \quad t \mapsto -\partial_\tau, \quad \partial_t \mapsto \tau. 
\end{align*}

This induces the Fourier transform on \(\mathcal{D}\)-modules and \(\mathcal{D}_{\mathbb{A}^1}\)-modules. 
\end{definition}

Recall that \(\operatorname{Hol}_{\mathbb{P}^1, s}(\mathbb{P}^1)\) denotes the category of holonomic irreducible \(\mathbb{P}^1\)-modules with strict support \(\mathbb{P}^1\). If we lift the assumption of strict support, we get the category \(\operatorname{Hol}_{s}(\mathbb{P}^1)\). 

\begin{lemma}\label{lemma:canonFT}
The Fourier transform of \(\mathcal{D}\)-modules defines a canonical map 
\begin{align*}
\operatorname{FT}: \operatorname{Hol}_{s}(\mathbb{P}^1_t) \rightarrow \operatorname{Hol}_{s}(\mathbb{P}^1_\tau). 
\end{align*}
\end{lemma}

\begin{proof}
We use the equivalence given by \Cref{lemma:DmodAP}. Let \(\mathfrak{M} \in \operatorname{Hol}_{s}(\mathbb{P}^1_t)\) and consider the \(\mathcal{D}\)-module \(\mathcal{M} \coloneqq \Gamma(\mathbb{P}^1, \mathfrak{M}(*\infty))\). Let \(\mathcal{N}\) be the \(\mathcal{D}_{\mathbb{P^1}}(*\infty)\)-module corresponding to \(\operatorname{FT}(\mathcal{M})\). We set 
\begin{align*}
\operatorname{FT}(\mathfrak{M}) \coloneqq \mathcal{N}({!*}\infty). 
\end{align*}

The result lies in \(\operatorname{Hol}_{s}(\mathbb{P}^1_\tau)\), because Fourier transform of \(\mathcal{D}\)-modules irreducibility (because it is an equivalence of categories) holonomicity (cf. \cite{Hot18}, Proposition 3.2.7). 
\end{proof}

Let \(\mathfrak{M} \in \operatorname{Hol}_{s}(\mathbb{P}^1_t)\) and denote the corresponding Twistor \(\mathcal{D}\)-module by \(\mathfrak{T}\). \(\operatorname{FT}(\mathfrak{M})\) also corresponds to a unique Twistor \(\mathcal{D}\)-module. One is naturally led to the following question: How can one define a Fourier transform 
\begin{align*}
\operatorname{FT}_z: \operatorname{MT}^{\operatorname{wild}}_{\operatorname{s}}(\mathbb{P}^1_t, 0, \sqrt{-1}\mathbb{R})^{(p)} \rightarrow \operatorname{MT}^{\operatorname{wild}}_{\operatorname{s}}(\mathbb{P}^1_\tau, 0, \sqrt{-1}\mathbb{R})^{(p)}, 
\end{align*}
such that it is compatible with \(\Xi_{\operatorname{DR}}\)? This question has been answered by the works of C. Sabbah (cf. \cite{Sab04}, \cite{Sab07}) and T. Mochizuki (cf. \cite{Moc22}). In the following, we will recall the Fourier transform of Twistor \(\mathcal{D}\)-modules, after which we will prove \Cref{thm:FTofT}. We begin with the Fourier transform of \(\mathcal{R}\)- and \(\mathscr{R}\)-modules: 

\begin{definition}\label{def:RmodFTz}
We write \(\mathcal{R}_t = \mathbb{C}\langle t, z, z\partial_t\rangle\) (this is a ring) and \(\mathscr{R}_t = \mathcal{O}_{\mathbb{C}_z}[t]\langle z\partial_t \rangle\) (this is a sheaf on \(\mathbb{C}_z\)). We have the Fourier automorphisms 
\begin{align*}
\text{FT}_z: \mathcal{R}_t \rightarrow \mathcal{R}_\tau, \quad \quad t \mapsto -z\partial_\tau, \quad z\partial_t \mapsto \tau \\ \text{FT}_z: \mathscr{R}_t \rightarrow \mathscr{R}_\tau, \quad \quad t \mapsto -z\partial_\tau, \quad z\partial_t \mapsto \tau
\end{align*}
where the first one is defined to be \(\mathbb{C}[z]\)-linear and the second one is defined to be \(\mathcal{O}_{\mathbb{C}_z}\)-linear. These automorphisms induce Fourier transforms on \(\mathcal{R}_t\)-modules and \(\mathscr{R}_t\)-modules. 
\end{definition}

We recall the definition of the Fourier transform of Schwartz functions: 

\begin{definition}\label{def:SchwartzFT}
\begin{itemize}
    \item Let \(\mathscr{S}(\mathbb{C}_t)\) be the space of Schwartz functions on \(\mathbb{C}_t\). For \(\psi \in \mathscr{S}(\mathbb{C}_t)\), one defines the Fourier transform \(\text{FT}(\psi) \in \mathscr{S}(\mathbb{C}_t)\) via 
    \begin{align*}
    \text{FT}(\psi)(\tau) \coloneqq \frac{1}{\pi} \int_{\mathbb{C}} \psi(t) e^{\overline{\tau t} - \tau t} dt_1dt_2 
    \end{align*}
    \item Let \(\mathscr{S}(\mathbb{C}_t \times \boldsymbol{S})\) be the space of functions on \(\mathbb{C}_t \times \boldsymbol{S}\), which are Schwartz in \(t\) and continuous in \(z\).  For \(\psi_z \in \mathscr{S}(\mathbb{C}_t \times \boldsymbol{S})\), one defines the Fourier transform \(\text{FT}_z(\psi_z) \in \mathscr{S}(\mathbb{C}_t \times \boldsymbol{S})\) via 
    \begin{align*}
    \text{FT}_z(\psi_z)(\tau, z) \coloneqq \frac{1}{\pi} \int_{\mathbb{C}} \psi_z(t, z) e^{\overline{\tau t}z - \tau t/z} dt_1dt_2 
    \end{align*}
\end{itemize}
\end{definition}

We also need to recall the Fourier transform of distributions, cf. \cite{Sab05}, p.196: 

\begin{definition}\label{def:fourierdist}
\begin{itemize}
    \item Let \(\mathscr{S}'(\mathbb{C}_t)\) be the space of tempered distributions on \(\mathbb{C}_t\). (This is not a sheaf.) We define the Fourier-Transform of \(u \in \mathscr{S}'(\mathbb{C}_t)\) as 
    \begin{align*}
    \text{FT}(u): \psi \mapsto u(\text{FT}(\psi))
    \end{align*}
    for all \(\psi \in \mathscr{S}(\mathbb{C}_t)\), where \(\text{FT}(\psi)\) is defined by \Cref{def:SchwartzFT}. 
    \item Let \(\mathscr{S}'(\mathbb{C}_t \times \boldsymbol{S})\) be the sheaf on \(\boldsymbol{S}\) of distributions on \(\mathbb{C}_t \times \boldsymbol{S}\) which are tempered in \(t\) and continuous in \(z\). The Twistor Fourier transform of \(u_z \in \mathscr{S}'(\mathbb{C}_t \times \boldsymbol{S})\) is 
    \begin{align*}
    \text{FT}_z(u_z): \psi \mapsto u_z(\text{FT}_z(\psi_z))
    \end{align*}
    for all \(\psi_z \in \mathscr{S}(\mathbb{C}_t \times \boldsymbol{S})\). 
\end{itemize}
\end{definition}

Now, we can define the Fourier transform of simple Twistor \(\mathcal{D}\)-modules on \(\mathbb{P}^1\). 

\begin{definition}\label{def:FTTwistorD}
Let \(\mathfrak{T} \coloneqq (\mathfrak{M}'_z, \mathfrak{M}''_z, \mathfrak{C})\) be a simple Twistor \(\mathcal{D}\)-module on \(\mathbb{P}^1\). Denote the \(\mathscr{R}_{\mathbb{A}^1 \times \mathbb{C}_z}\)-modules obtained from the localisation of \(\mathfrak{M}'_z, \mathfrak{M}''_z\) at \(\infty\) by \(\mathcal{M}'_z, \mathcal{M}''_z\). One defines the Fourier transform of \(\mathfrak{T}\) as \(\operatorname{FT}_z(\mathfrak{T}) \coloneqq (\operatorname{FT}_z(\mathfrak{M}'_z), \operatorname{FT}_z(\mathfrak{M}''_z), \operatorname{FT}_z(\mathfrak{C}))\), where: 
\begin{itemize}
    \item \(\mathfrak{M}'_z\) and \(\mathfrak{M}''_z\) are good \(\mathscr{R}_{\mathbb{P}^1}\)-modules by the definition of Twistor \(\mathcal{D}\)-modules. This means that \(\mathfrak{M}'_z(*\infty)\) and \(\mathfrak{M}''_z(*\infty)\) are \(p\)-good by \Cref{remark:goodpgood}. Using the equivalence from \Cref{lemma:partialGAGA}, we can define 
    \begin{align*}
    \operatorname{FT}_z(\mathfrak{M}'_z) = p^{*, \operatorname{loc}}(\operatorname{FT}_z(p_*(\mathfrak{M}'_z(*\infty))))[{!*}\infty]
    \end{align*}
    by \cite{Sab05}, Remark 8.2.8. Here, \(\operatorname{FT}_z\) on the right-hand side is the Fourier transform of \(\mathscr{R}\)-modules from \Cref{def:RmodFTz} and \([{!*}\infty]\) denotes minimal extension at \(\infty\). \(\operatorname{FT}(\mathfrak{M}''_z)\) is defined in the same way. 
    \item The pairing \(\operatorname{FT}_z(\mathfrak{C})\) is defined as a natural extension across \(\tau = \infty\) of 
    \begin{align*}
    \operatorname{FT}_z(\mathfrak{C}\vert_{\mathbb{C}_t}): \overline{\sigma^*}(\operatorname{FT}_z(\mathcal{M}'_z))\vert_{\boldsymbol{S}} \otimes \operatorname{FT}_z(\mathcal{M}''_z)\vert_{\boldsymbol{S}} & \rightarrow \mathfrak{Db}_{\mathbb{C}_\tau\times\boldsymbol{S}/\boldsymbol{S}} \\ \overline{\sigma^*} u \otimes v & \mapsto \operatorname{FT}_z(\mathfrak{C}\vert_{\mathbb{C}_t}(\overline{\sigma^*}u, v)), 
    \end{align*}
    where the distributional Fourier transform \(\operatorname{FT}_z(\mathfrak{C}\vert_{\mathbb{C}_\tau}(\overline{\sigma^*}u, v))\) is given by \Cref{def:fourierdist}. 
\end{itemize}
\end{definition}

As already mentioned, this gives a functor
\begin{align*}
\operatorname{MT}^{\operatorname{wild}}_{\operatorname{s}}(\mathbb{P}^1_t, 0, \sqrt{-1}\mathbb{R})^{(p)} \rightarrow \operatorname{MT}^{\operatorname{wild}}_{\operatorname{s}}(\mathbb{P}^1_\tau, 0, \sqrt{-1}\mathbb{R})^{(p)}. 
\end{align*}

Before we continue, we make the following elementary observation: 

\begin{lemma}\label{lemma:fttranslate}
Let \(\mathcal{P} \in \mathcal{D}\) and \(\mathcal{P}_z \in \mathcal{R}\). Then, the following statements hold: 

\begin{itemize}
    \item We have the canonical isomorphism of \(\mathcal{D}\)-modules 
    \begin{align*}
    m: \operatorname{FT}(\mathcal{D}/(\mathcal{P})) \rightarrow \mathcal{D}/(\operatorname{FT}(\mathcal{P})), \quad [s] \mapsto [\operatorname{FT}(s)]. 
    \end{align*} 
    \item We have the canonical isomorphism of \(\mathcal{R}\)-modules 
    \begin{align*}
    m_z: \operatorname{FT}_z(\mathcal{R}/(\mathcal{P}_z)) \rightarrow \mathcal{R}/(\operatorname{FT}_z(\mathcal{P}_z)), \quad [s] \mapsto [\operatorname{FT}_z(s)]. 
    \end{align*}
\end{itemize}
\end{lemma}

\begin{proof}
We only show the first point, the second one being analogous. Begin by showing the well-definedness of \(m\): Suppose \(s \in \mathcal{D} \cdot \mathcal{P}\), say \(s = s' \cdot \mathcal{P}\). Then, \(\text{FT}(s) = \text{FT}(s') \cdot \text{FT}(\mathcal{P}) \in \mathcal{D} \cdot \text{FT}(\mathcal{P})\). The \(\mathcal{D}\)-linearity is seen as follows: 
\begin{align*}
m(\tau[s]) = m([\partial_t s]) = \text{FT}([\partial_ts]) = \tau m([s]) 
\end{align*}

and similarly \(m([\partial_\tau s]) = \partial_\tau m([s])\). An inverse to \(m\) is given by \([s] \mapsto [\operatorname{FT}^{-1}s]\), which completes the proof. 
\end{proof}

\section{Twistor-square roots}
As another preparation for the proof of \Cref{thm:FTofT}, we need to further investigate the Twistor-conjugation \(\overline{\sigma^*}\), which will be done in the following. Recall that we denote by \(\mathcal{O}(\boldsymbol{S})\) the space of functions on the unit circle \(\boldsymbol{S} \subset \mathbb{C}_z\) which have an analytic extension onto some annulus containing \(\boldsymbol{S}\). We have the space of nowhere-vanishing such functions \(\mathcal{O}(\boldsymbol{S})^\times\). We remark that the results of this section are similar (but not identical) to the ones obtained in \cite{Sab05}, 8.4.a, however, our arguments are different. 

\begin{lemma}\label{lemma:holroot}
Let \(r \in \mathcal{O}(\boldsymbol{S})^\times\). The following conditions are equivalent: 

\begin{itemize}
    \item There exists a function \(s \in \mathcal{O}(\boldsymbol{S})^\times\) with \(s(z)^2 = r(z)\) for all \(z \in \boldsymbol{S}\). 
    \item The winding number of \(r\), as \(z\) moves once around \(\boldsymbol{S}\) counter-clockwise, is even: 
    \begin{align*}
    \frac{1}{2\pi i}\oint_{\boldsymbol{S}}\frac{r'(z)}{r(z)}dz \in 2\mathbb{Z}
    \end{align*}
\end{itemize}
\end{lemma}

\begin{proof}
Assume there exists \(s^2 = r\) as in the statement. Then, we have 
\begin{align*}
\frac{1}{2\pi i}\oint_{\boldsymbol{S}}\frac{r'(z)}{r(z)}dz = 2 \cdot \frac{1}{2\pi i}\oint_{\boldsymbol{S}}\frac{s'(z)}{s(z)}dz. 
\end{align*}

Since the winding number of \(s\) is an integer, this expression is an even number. 
Conversely, assume that the winding number of \(r\) is even, say, \(2k\). It follows that the winding number of \(q(z) \coloneqq z^{-2k}r(z)\) vanishes. Let \(A\) be an annulus containing \(\boldsymbol{S}\) such that \(q\) has a non-vanishing extension onto \(A\). \(q\) has holomorphic natural logarithms on \(A\) given by 
\begin{align*}
p(z) \coloneqq 2\pi i n + \ln q(1) + \oint_{\gamma(z, 1)} \frac{q'(\zeta)}{q(\zeta)}d\zeta, 
\end{align*}

for any integer \(n\). Here, \(\gamma(z, 1)\) is any path in \(A\) from \(1\) to \(z\) and \(p\) is independent of this choice. Now, setting \(s(z) \coloneqq z^{k}e^{p(z)/2}\), it is clear that \(s^2 = r\). (Changing the parity of \(n\) corresponds to multiplying \(s\) by \(-1\).) 
\end{proof}

\begin{theorem}\label{thm:twistorroot}
Let \(r \in \mathcal{O}(\boldsymbol{S})^\times\). Assume that \(r\) satisfies \(\overline{\sigma^*}r = r\) (that is, \(r\) is Twistor-real) and 
\begin{align*}
\frac{1}{2\pi i}\oint_{\boldsymbol{S}}\frac{r'(z)}{r(z)}dz = 0. 
\end{align*}

Then there exists a function \(s: \boldsymbol{S} \in \mathcal{O}(\boldsymbol{S})^\times\) with \(|s(z)|_z^2 = r(z)\) for all \(z \in \boldsymbol{S}\). 
\end{theorem}

\begin{proof}
Because of 
\begin{align*}
\frac{1}{2\pi i}\oint_{\boldsymbol{S}}\frac{r'(z)}{r(z)}dz = 0, 
\end{align*}

\Cref{lemma:holroot} implies that there is a function \(s \in \mathcal{O}(\boldsymbol{S})\) such that \(s^2 = r\). \(\overline{\sigma^*}r = r\) implies that \((\overline{\sigma^*}s(z))^2 = s(z)^2\), giving 
\begin{align*}
\overline{\sigma^*}s(z) = c(z)s(z)
\end{align*}

with \(c(z) \in \{-1, 1\}\) for every \(z \in \boldsymbol{S}\). Because \(c\) is clearly continuous, we either have \(c \equiv 1\) or \(c \equiv -1\). We distinguish two cases: 

\begin{itemize}
    \item \(c \equiv 1\): We have 
    \begin{align*}
    \overline{\sigma^*}s(z) \cdot s(z) = s(z)^2 = r(z), 
    \end{align*}
    concluding the proof in this case. 
    \item \(c \equiv -1\): Here, we get \(\overline{\sigma^*}s(z) \cdot s(z) = -r(z)\). We correct for the factor of \(-1\) by replacing \(s(z)\) with \(zs(z)\): Indeed, we have 
    \begin{align*}
    \overline{\sigma^*}(zs(z)) \cdot zs(z) = z (-1/z)(-r(z)) = r(z), 
    \end{align*}
    since \(z \in \boldsymbol{S}\). This completes the proof. 
\end{itemize}
\end{proof}

We give a function for which the conditions of \Cref{thm:twistorroot} are fulfilled: 

\begin{lemma}\label{lemma:gammaroot}
Let \(\gamma \in \mathbb{C}\) with \(\Re(\gamma) \notin \mathbb{Z}\). The function \(r(z) \coloneqq \widetilde{\Gamma}(\gamma \diamond z)\) satisfies \(\overline{\sigma^*}r = r\) and 
\begin{align*}
\frac{1}{2\pi i}\oint_{\boldsymbol{S}}\frac{r'(z)}{r(z)}dz = 0, 
\end{align*}
\end{lemma}

\begin{proof}
\(\overline{\sigma^*}r = r\) directly follows from \(\overline{\sigma^*}(\gamma \diamond z) = \gamma \diamond z\). As \(z\) moves around \(\boldsymbol{S}\) in the counterclockwise direction (starting from \(1\)) the point \(\gamma \diamond z\) starts at \(\gamma\), travels to \(\overline{\gamma}\) along a straight line and then travels back to \(\gamma\) along the same line. If \(\Re(\gamma) \notin \mathbb{Z}\), the resulting loop is null-homotopic in \(\mathbb{C} - \mathbb{Z}\). Since \(\widetilde{\Gamma}(t)\) has no roots or poles except for \(t \in \mathbb{Z}\), the claim follows. 
\end{proof}

\section{Holomorphic functions and linear disjointness}

We require some further preparation for the calculation of the Fourier transform of the Twistor pairing as in \Cref{thm:FTofT}. Let \(G\) be a field and let \(F\) be a \(G\)-algebra. In the following, \(F\) is always assumed to be unital, that is, there is a neutral element \(1_F\) of the multiplication in \(F\). This gives a canonical inclusion

\begin{align*}
G \rightarrow F, \quad g \mapsto g \cdot 1_F. 
\end{align*}

\begin{definition}
Let \(G\) be a field and let \(F\) be a \(G\)-algebra. Furthermore, let \(A, B\) be two sub-\(G\)-algebras of \(F\). \(A\) and \(B\) are called linearly disjoint if one of the following equivalent conditions is satisfied: 

\begin{itemize}
    \item The map 
    \begin{align*}
    A \otimes_G B \rightarrow AB, \quad (a \otimes b) \rightarrow ab 
    \end{align*}
    (which is surjective a priori) is an isomorphism. Here, 
    \begin{align*}
    AB \coloneqq \left\{\sum_{i = 1}^n a_i b_i : a_i \in A, b_i \in B\right\}. 
    \end{align*}
    \item Any basis of \(A\) as a \(G\)-vector space remains linearly independent over \(B\). 
\end{itemize}
\end{definition}

\begin{remark}
If \(A\) and \(B\) are linearly disjoint as \(G\)-algebras, one has \(A \cap B = G\): Indeed, we have \(G \subset A \cap B\), because \(A\) and \(B\) are \(G\)-algebras. For the reverse inclusion, assume that there was \(f \in A \cap B\) with \(f \notin G\). The set \(\{1, f\}\) is linearly independent over \(G\), and so it is also linearly independent over \(A\). However, we can write \(f \cdot 1 - 1 \cdot f = 0\), which gives the desired contradiction. 
\end{remark}

\begin{theorem}\label{theorem:lindisC}
Let \(G\) be a field and let \(F\) be a \(G\)-algebra. Furthermore, let \(A, B\) be two sub-\(G\)-algebras of \(F\). Assume that \(A\) and \(B\) are linearly disjoint as \(G\)-algebras. Let \(f \in F\) be such that there are \(a_i \in A\) and \(b_i \in B\) (\(i \in \{1, \ldots, n\}\)) with 
\begin{align*}
f \in \sum_{i = 1}^n b_i \cdot A, \quad f \in \sum_{i = 1}^n a_i \cdot B. 
\end{align*}

\begin{itemize}
    \item Then, there are \(k_{ij} \in G\) such that 
    \begin{align*}
    f = \sum_{i, j = 1}^n k_{ij} a_i b_j. 
    \end{align*} 
    \item If we assume furthermore that the set \(\{a_i\}\) is linearly independent over \(G\), and if the same is true for the set \(\{b_i\}\), then the \(k_{ij}\) are uniquely defined. Assume additionally that there exists an automorphism \(\rho\) of \(F\) as a \(G\)-algebra such that 

\begin{itemize}
    \item \(\rho(f) = f\)
    \item \(\rho(a_i) = \lambda_ia_i\) and \(\rho(b_i) = \mu_ib_i\) for \(\lambda_i, \mu_i \in G\) (\(i \in \{1, \ldots, n\}\))
    \item \(\lambda_i \mu_j = 1 \implies i = j\). 
\end{itemize}

Then, the unique \(k_{ij}\) satisfy \(k_{ij} = 0\) if \(i \neq j\). 
\end{itemize}
\end{theorem}

\begin{proof}
By the definition of linear disjointness we have the isomorphism \(i: A \otimes_{G}B \rightarrow AB\). Set 
\begin{align*}
U \coloneqq \text{span}_G\{a_1, \ldots, a_n\} \subset A, \quad V \coloneqq \text{span}_G\{b_1, \ldots, b_n\} \subset B 
\end{align*}

and choose \(G\)-vector spaces \(U', V'\) such that \(U \oplus U' = A\) and \(V \oplus V' = B\). Fix \(a^i \in A, b^i \in B\) such that 
\begin{align*}
f = \sum_{i = 1}^n b_i \cdot a^i, \quad f = \sum_{i = 1}^n a_i \cdot b^i. 
\end{align*}

We have the tensors 
\begin{align*}
\sum_{i = 1}^n a^i \otimes b_i \in A \otimes_G V \subset A \otimes_G B \supset U \otimes_G B \ni \sum_{i = 1}^n a_i \otimes b^i. 
\end{align*}

It follows that 
\begin{equation*}
\begin{split}
i^{-1}(f) & \in (A \otimes_G V) \cap (U \otimes_G B) \\ & = ((U \oplus U') \otimes_G V) \cap (U \otimes_G (V \oplus V')) \\ & = ((U \otimes_G V) \oplus (U' \otimes_G V)) \cap ((U \otimes_G V) \oplus (U \otimes_G V')) \\ & = U \otimes_G V. 
\end{split}
\end{equation*}

We have \(i(U \otimes_G V) = UV = \text{span}_{A \cap B}\left\{a_i b_j: (i, j) \in \{1, \ldots, n\}^2\right\}\), which shows the existence of the \(k_{ij}\) and completes the proof of the first point. 

Now, assume additionally that the set \(\{a_i\}\) and the set \(\{b_i\}\) is linearly independent over \(G\). The uniqueness of the \(k_{ij}\) is seen as follows: Suppose for contradiction that we had 
\begin{align*}
f = \sum_{i, j = 1}^n k_{ij} \cdot a_i \cdot b_j = \sum_{i, j = 1}^n k'_{ij} \cdot a_i \cdot b_j, 
\end{align*}

\(k_{ij} \neq k_{ij}'\) for at least one pair \((i, j)\). Then, 
\begin{align*}
0 = \sum_{i, j = 1}^n (k_{ij} - k'_{ij})a_{i} \cdot b_{j} = \sum_{i = 1}^n a_i \sum_{j = 1}^n (k_{ij} - k'_{ij})b_{j}. 
\end{align*}

Since \(\{b_i\}\) is linearly independent over \(G\), there exists \(i \in \{1, \ldots, n\}\) such that \(\sum_{j = 1}^n (k_{ij} - k'_{ij})b_{j} \neq 0\). This implies that the \(a_i\) are not linearly independent over \(B\). On the other hand, they are linearly independent over \(G\) by assumption. We obtain the desired contradiction by the linear disjointness of \(A\) and \(B\) in \(F/G\), proving the uniqueness of the \(k_{ij}\). In other words, \(\{a_ib_j\}\) is linearly independent over \(G\). 

Now, assume that there also exists an automorphism \(\rho\) of \(F\) as in the claim. Then, \(f = \rho(f)\) implies 
\begin{align*}
\sum_{i, j = 1}^n k_{ij} \cdot a_i \cdot b_j = \sum_{i, j = 1}^n \rho(k_{ij}) \cdot \rho(a_i) \cdot \rho(b_j). 
\end{align*}

With \(\rho(a_i) = \lambda_ia_i\) and \(\rho(b_i) = \mu_ib_i\), and because \(G\) is invariant under \(\rho\), we can write this equation as 
\begin{align*}
\sum_{i, j = 1}^n k_{ij} \cdot a_i \cdot b_j = \sum_{i, j = 1}^n k_{ij} \cdot  \lambda_ia_i \cdot \mu_jb_j, 
\end{align*}

giving 
\begin{align*}
\sum_{i, j = 1}^n k_{ij} (1 - \lambda_i \mu_j)a_ib_j = 0. 
\end{align*}

Because \(\{a_ib_j\}\) is linearly independent over \(G\), we must have \(k_{ij} (1 - \lambda_i \mu_j) = 0\) for all \(i\) and \(j\). We had assumed \(\lambda_i \mu_j = 1 \implies i = j\), giving \(k_{ij} = 0\) for \(i \neq j\). This completes the proof. 
\end{proof}

Let \(\Omega\) be a Riemann surface. Consider the ring of smooth functions (in the real sense) \(C^{\infty}_{\Omega}\) as a \(\mathbb{C}\)-algebra. We have the canonical injections 
\begin{align*}
\mathcal{O}(\Omega) \hookrightarrow C^{\infty}(\Omega) \hookleftarrow \mathcal{O}(\overline{\Omega}), 
\end{align*} 

of \(\mathbb{C}\)-algebras. We make an analogous construction with a parameter \(z\): Let \(\mathcal{M}(\mathbb{C}^*_z)\) be the field of meromorphic functions on \(\mathbb{C}_z^*\) and let \(C^{\infty, \operatorname{mer}}(\Omega \times \mathbb{C}_z^*)\) be the smooth functions on \(\Omega \times \mathbb{C}_z^*\) which are smooth in the \(t\)-direction and meromorphic in the \(z\)-direction. Furthermore, let \(\mathcal{M}^{\operatorname{an}, \operatorname{mer}}(\Omega \times \mathbb{C}_z^*)\) be the smooth functions on \(\Omega \times \mathbb{C}_z^*\) which are analytic in the \(t\)-direction and meromorphic in the \(z\)-direction. \(\mathcal{M}^{\operatorname{an}, \operatorname{mer}}(\overline{\Omega} \times \mathbb{C}_z^*)\) has the analogous meaning. We have the \(\mathcal{M}(\mathbb{C}_z^*)\)-algebra \(C^{\infty, \operatorname{mer}}(\Omega \times \mathbb{C}_z^*)\) and the injections 
\begin{align*}
\mathcal{M}^{\operatorname{an}, \operatorname{mer}}(\Omega \times \mathbb{C}_z^*) \hookrightarrow C^{\infty, \operatorname{mer}}(\Omega \times \mathbb{C}_z^*) \hookleftarrow \mathcal{M}^{\operatorname{an}, \operatorname{mer}}(\overline{\Omega} \times \mathbb{C}_z^*) 
\end{align*} 

of \(\mathcal{M}(\mathbb{C}_z^*)\)-algebras. 

\begin{proposition}\label{prop:Olindisj}
Let \(\Omega\) be a Riemann surface. The following statements are true: 

\begin{itemize}
    \item \(\mathcal{O}(\Omega)\) and \(\mathcal{O}(\overline{\Omega})\) are linearly disjoint as \(\mathbb{C}\)-algebras. 
    \item \(\mathcal{M}^{\operatorname{an}, \operatorname{mer}}(\Omega \times \mathbb{C}_z^*)\) and \(\mathcal{M}^{\operatorname{an}, \operatorname{mer}}(\overline{\Omega} \times \mathbb{C}_z^*) \) are linearly disjoint as \(\mathcal{M}(\mathbb{C}_z^*)\)-algebras. 
\end{itemize}
\end{proposition}

\begin{proof}
We only show the first point, the second one is proved analogously. We argue via contradiction. To this end, assume that there exists a \(\mathbb{C}\)-basis \(E\) (it is uncountable) of \(\mathcal{O}(\Omega)\) which is not linearly independent over \(\mathcal{O}(\overline{\Omega})\). This means that there are functions \(g_e \in \mathcal{O}(\overline{\Omega})\) (\(e \in E\)), only finitely many non-zero, such that 
\begin{align*}
\sum_{e \in E} g_e \cdot e = 0. 
\end{align*}

Let \(E' \coloneqq \{e \in E: g_e \neq 0\}\), write \(E' = \{f_1, \ldots, f_n\}\) and set \(g_i \coloneqq g_{f_i}\), leading to \(\sum_{i = 1}^n g_if_i = 0\). Let \(U \subset \Omega\) be an open set with coordinate \(t\). Differentiating this relation \(k\) times with respect to \(t\), where \(k \in \{0, \ldots, n - 1\}\), gives 
\begin{align*}
\sum_{i = 1}^n g_i\partial_t^{(k)}f_i = 0, 
\end{align*}

over \(U\), because the \(g_i\) are anti-holomorphic. These \(n\) equations are rewritten as the linear system \(W(t) \cdot g(t) = 0\), where \(g = (g_1, \ldots, g_n)^T\) and \(W\) is the Wronskian matrix of the \(f_i\vert_U\): 
\begin{align*}
W(t) = \begin{pmatrix}
    f_1(t) & \cdots & f_n(t) \\
    f'_1(t) & \cdots & f'_n(t) \\
    \vdots & \ddots & \vdots \\
    f_1^{(n-1)}(t) & \cdots & f_n^{(n-1)}(t)
\end{pmatrix}. 
\end{align*}

Since the \(f_i\) are linearly independent over \(\mathbb{C}\), the same is true for the \(f_i\vert_{U}\) by the identity theorem. Because the \(f_i\vert_{U}\) are holomorphic, this means that the Wronskian \(\det(W(t))\) does not vanish on all of \(U\). Since \(\det(W(t))\) is also holomorphic, it is invertible on an open and dense subset \(U' \subset U\). We obtain \(g(t)\vert_{U'} = 0\), so that \(g(t) = 0\) on \(\Omega\) because of the identity theorem (for anti-holomorphic functions). We had assumed that none of the \(g_i\) vanish on all of \(\Omega\), and so they also cannot vanish on all of \(U\). This gives the desired contradiction. 
\end{proof}

We obtain the following result: 

\begin{theorem}\label{thm:smoothsymmetry} 
Let \(\widetilde{\mathbb{C}}_t^*\) be the universal covering of the punctured plane with coordinate \(t\). Let \(f \in C^{\infty}(\mathbb{C}_t^*)\) (single-valued function) with the following property: Denoting the pullback of \(f\) onto \(\widetilde{\mathbb{C}}_t^*\) by \(\widetilde{f}\), assume that  
\begin{align*}
\widetilde{f} \in \sum_{i = 1}^n \mathcal{O}(\widetilde{\mathbb{C}}_t^*) \cdot q_i, \quad \widetilde{f} \in \sum_{i = 1}^n \mathcal{O}(\overline{\widetilde{\mathbb{C}}_t^*}) \cdot p_i, 
\end{align*}

for some \(p_i \in \mathcal{O}(\widetilde{\mathbb{C}}_t^*)\) and \(q_i \in \mathcal{O}(\overline{\widetilde{\mathbb{C}}_t^*})\). (Note that these are allowed to be multi-valued as functions on \(\mathbb{C}_t^*\).) 
\begin{itemize}
    \item Then, 
    \begin{align*}
    f = \sum_{i, j = 1}^n C_{ij} \cdot p_i \cdot q_j
    \end{align*}
    for some \(C_{ij} \in \mathbb{C}\). 
    \item Assume moreover that 
    \begin{align*}
    p_i(t) = t^{m_i} p_i'(t), \quad q_i(t) = \overline{t}^{m_i} q_i'(t) 
    \end{align*}
    for some \(m_i \in \mathbb{R}\) which are pairwise different modulo \(\mathbb{Z}\) and for some (single-valued) \(p_i' \in \mathcal{O}(\mathbb{C}_t^*)\), \(q_i' \in \mathcal{O}(\overline{\mathbb{C}_t^*})\). Then, we have \(C_{ij} = 0\) if \(i \neq j\). 
\end{itemize}
\end{theorem}

\begin{proof}
The existence and uniqueness of the \(C_{ij}\) follow from the combination of the first point of \Cref{theorem:lindisC} and the first point of \Cref{prop:Olindisj}. It remains to show that \(C_{ij} = 0\) for \(i \neq j\), if we have \(p_i', q_i', m_i\) as in the statement. To this end, consider the automorphism \(\rho\) of the \(\mathbb{C}\)-algebra \(\mathcal{O}(\mathbb{C}_t^*)\mathcal{O}(\overline{\mathbb{C}_t^*})\) given by analytic and anti-analytic continuation in the counter-clockwise direction around the origin. We have 
\begin{align*}
\rho(p_i(t)) = e^{2\pi i m_i}p_i(t), \quad \rho(q_i(t)) = e^{-2\pi i m_i}q_i(t). 
\end{align*}

Since the \(m_i\) are assumed to be pairwise different modulo \(\mathbb{Z}\), we have \(e^{2\pi i m_i} e^{-2\pi i m_j} = 1 \implies i = j\). This implies \(C_{ij} = 0\) for \(i \neq j\) by the second point of \Cref{theorem:lindisC}, finishing the proof. 
\end{proof}

Recall that for \(\gamma \in \mathbb{C}\), we have set 
\begin{align*}
\gamma \diamond z = \Re(\gamma) + \frac{z + z^{-1}}{2}\Im(\gamma)i. 
\end{align*}

We have the following Twistor-variant of \Cref{thm:smoothsymmetry}, which is proven analogously: 

\begin{theorem}\label{thm:smoothZsymmetry}
Let \(\widetilde{\mathbb{C}}_t^*\) be the universal covering of the punctured plane with coordinate \(t\). Let \(f \in C^{\infty, \operatorname{an}}(\mathbb{C}_t^* \times \mathbb{C}_z^*)\) (single-valued function) with the following property: Denoting the pullback of \(f\) onto \(\widetilde{\mathbb{C}}_t^* \times \mathbb{C}_z^*\) by \(\widetilde{f}\), assume that 
\begin{align*}
\widetilde{f} \in \sum_{i = 1}^n \mathcal{M}^{\operatorname{an}, \operatorname{mer}}(\widetilde{\mathbb{C}}_t^* \times \mathbb{C}_z^*) \cdot q_i, \quad \widetilde{f} \in \sum_{i = 1}^n \mathcal{M}^{\operatorname{an}, \operatorname{mer}}(\overline{\widetilde{\mathbb{C}}_t^*} \times \mathbb{C}_z^*) \cdot p_i, 
\end{align*}

for some \(p_i \in \mathcal{M}^{\operatorname{an}, \operatorname{mer}}(\widetilde{\mathbb{C}}_t^* \times \mathbb{C}_z^*)\) and \(q_i \in \mathcal{M}^{\operatorname{an}, \operatorname{mer}}(\overline{\widetilde{\mathbb{C}}_t^*} \times \mathbb{C}_z^*)\). (Note that these are allowed to be multi-valued in \(t\) as functions on \(\mathbb{C}_t^* \times \mathbb{C}_z^*\).) 

\begin{itemize}
    \item Then, 
    \begin{align*}
    f = \sum_{i, j = 1}^n C_{ij} \cdot p_i \cdot q_j
    \end{align*}
    for some \(C_{ij} \in \mathcal{M}(\mathbb{C}_z^*)\). 
    \item Assume moreover that 
    \begin{align*}
    p_i(t) = t^{m_i \diamond z} p_i'(t), \quad q_i(t) = \overline{t}^{m_i \diamond z} q_i'(t) 
    \end{align*}
    for some \(m_i \in \mathbb{C}\) (note that the \(m_i\) do not have to be real) which are pairwise different modulo \(\mathbb{Z}\) and for some (single-valued) \(p_i' \in \mathcal{M}^{\operatorname{an}, \operatorname{mer}}(\mathbb{C}_t^* \times \mathbb{C}_z^*)\), \(q_i' \in \mathcal{M}^{\operatorname{an}, \operatorname{mer}}(\overline{\mathbb{C}_t^*} \times \mathbb{C}_z^*)\). Then, we have \(C_{ij} = 0\) if \(i \neq j\). 
\end{itemize}
\end{theorem}

\section{Fourier Transform and conjugation on \texorpdfstring{\(\mathcal{D}\)}{D} and \texorpdfstring{\(\mathcal{R}\)}{R}}

In the following, we will investigate the relationship between Fourier transform and conjugation, both in the context of \(\mathcal{D}\)- and \(\mathcal{R}\)-modules. This requires us to introduce a notion of conjugation on \(\mathcal{D}\) and \(\mathcal{R}\): 

\begin{definition} 
We define (Twistor-)conjugates of \(\mathcal{D}\) and \(\mathcal{R}\): 

\begin{itemize}
    \item We have the conjugate of \(\mathcal{D}\): Set \(\overline{\mathcal{D}} \coloneqq \mathbb{C}\langle \overline{t}, \partial_{\overline{t}}\rangle\), where \(\overline{t}\) and \(\partial_{\overline{t}}\) satisfy the natural relation. There is the natural \(\mathbb{R}\)-linear conjugation map given by 

    \begin{align*}
    \overline{(\cdot)}: \mathcal{D} \rightarrow \overline{\mathcal{D}}, \qquad i \mapsto -i, \quad t \mapsto \overline{t}, \quad \partial_t \mapsto \partial_{\overline{t}}. 
    \end{align*}
    \(\overline{\mathcal{D}}\) acts on \(\mathcal{O}_{\overline{\mathbb{C}_t}}\) in the obvious way. 
    \item Similarly, we define the Twistor-conjugate of \(\mathcal{R} = \mathbb{C}\langle t, z, z\partial_t\rangle\) as 
    \begin{align*}
    \overline{\sigma^*}\mathcal{R} \coloneqq \mathbb{C}\langle \overline{t}, -z^{-1}, -z^{-1}\partial_{\overline{t}}\rangle, 
    \end{align*}
    again with the natural commutation relations. The \(\mathbb{R}\)-linear Twistor-conjugation map is 
    \begin{align*}
    \overline{\sigma^*}(\cdot): \mathcal{R} \rightarrow \overline{\sigma^*}\mathcal{R}, \qquad i \mapsto -i, \quad t \mapsto \overline{t}, \quad z \mapsto -z^{-1}, \quad z\partial_t \mapsto -z^{-1}\partial_{\overline{t}}. 
    \end{align*}
    Again, we get a natural action of \(\overline{\sigma^*}\mathcal{R}\) on \(\mathcal{O}_{\overline{\mathbb{C}_t} \times \mathbb{C}_z^*}\). 
\end{itemize}
\end{definition}

These conjugation maps are compatible with conjugation of functions: 

\begin{lemma}\label{lemma:conjcomp}
The following statements are true: 
\begin{itemize}
    \item Let \(\mathcal{P} \in \mathcal{D}\) and \(f \in \mathcal{O}_{\mathbb{C}_t}\). Then, \(\overline{\mathcal{P} \cdot f} = \overline{\mathcal{P}} \cdot \overline{f}\). (Here, the conjugates of \(Pf\) and \(f\) are conjugates of functions, whereas the conjugate of \(\mathcal{P}\) is the conjugate of an element of \(\mathcal{D}\).) 
    \item Let \(\mathcal{P}_z \in \mathcal{R}\) and \(g \in \mathcal{O}_{\mathbb{C}_t \times \mathbb{C}_z^*}\). Then, \(\overline{\sigma^*}(\mathcal{P}_z \cdot g) = \overline{\sigma^*}\mathcal{P}_z \cdot \overline{\sigma^*}g\). 
\end{itemize}
\end{lemma}

\begin{proof}
We have \(\overline{\partial_tf} = \partial_{\overline{t}}\overline{f}\) by the definition of complex differentiation, so that the assertion of the first point follows. The second point follows from \(\overline{\sigma^*}(z\partial_tf) = -z^{-1}\partial_{\overline{t}}(\overline{\sigma^*}f)\). 
\end{proof}

\begin{example}
Let \(\mathcal{P} \in \mathcal{D}\) and consider a function \(f \in \mathcal{O}_{\mathbb{C}}\) with \(\mathcal{P}f = 0\). Then, \Cref{lemma:conjcomp} implies that \(\overline{\mathcal{P}} \cdot \overline{f} = 0\). We have the analogous observation in the Twistor-setting. 
\end{example}

\begin{definition}\label{def:fourieronDandR}
We define (Twistor-)Fourier transforms on the conjugates of \(\mathcal{D}\) and \(\mathcal{R}\). We recall the (Twistor-)Fourier transforms of \(\mathcal{D}\) and \(\mathcal{R}\) for comparison. 

\begin{itemize}
    \item The Fourier transforms on \(\mathcal{D}_t = \mathbb{C}\langle t, \partial_t \rangle\) and \(\overline{\mathcal{D}_t}\) are the \(\mathbb{C}\)-linear maps given by 
    \begin{align*}
    \text{FT}: \mathcal{D}_t \rightarrow \mathcal{D}_\tau, \quad \quad t \mapsto -\partial_\tau, \quad \partial_t \mapsto \tau \\ \text{FT}: \overline{\mathcal{D}_t} \rightarrow \overline{\mathcal{D}_{\tau}}, \quad \quad \overline{t} \mapsto \partial_{\overline{\tau}}, \quad \partial_{\overline{t}} \mapsto -\overline{\tau}
    \end{align*}
    \item The Fourier transforms on \(\mathcal{R}_t = \mathbb{C}\langle t, z, z\partial_t \rangle\) and \(\overline{\sigma^*}\mathcal{R}_t\) are the maps given by 
    \begin{align*}
    \text{FT}_z: \mathcal{R}_t \rightarrow \mathcal{R}_\tau, & \quad \quad t \mapsto -z\partial_\tau, \quad z\partial_t \mapsto \tau \\ \text{FT}_z: \overline{\sigma^*}\mathcal{R}_t \rightarrow \overline{\sigma^*}\mathcal{R}_\tau, & \quad \quad \overline{t} \mapsto z^{-1}\partial_{\overline{\tau}}, \quad z^{-1}\partial_{\overline{t}} \mapsto -\overline{\tau}, 
    \end{align*}
    where the first one is \(\mathbb{C}[z]\)-linear and the second one is \(\mathbb{C}[z^{-1}]\)-linear. 
\end{itemize}
\end{definition}

We can see that conjugation and Fourier-Transform do not exactly commute on \(\mathcal{D}\), whereas they commute on \(\mathcal{R}\): 

\begin{lemma}\label{lemma:ftconj}
Let \(\mathcal{P} \in \mathcal{D}\) and \(\mathcal{P}_z \in \mathcal{R}\). Define the \(\mathbb{C}\)-linear involution 
\begin{align*}
\iota: \mathcal{D}_t \rightarrow \mathcal{D}_t, \quad  t \mapsto -t, \quad \partial_t \mapsto -\partial_t. 
\end{align*}

We have 
\begin{align*}
\overline{\operatorname{FT}(\mathcal{P})} = \operatorname{FT}(\overline{\iota(\mathcal{P})}), \quad \overline{\sigma^*}\operatorname{FT}_z(\mathcal{P}_z) = \operatorname{FT}_z(\overline{\sigma^*}\mathcal{P}_z). 
\end{align*}

(Note that there is no involution in the second equality.) 
\end{lemma}

\begin{proof}
We have \(\overline{\operatorname{FT}(t)} = -\partial_{\overline{\tau}} = -\operatorname{FT}(\overline{t})\) and \(\overline{\operatorname{FT}(\partial_t)} = \overline{\tau} = - \operatorname{FT}(\overline{\partial_t})\), showing the first equality. The second equality follows from \(\overline{\sigma^*}\operatorname{FT}_z(t) = z^{-1}\partial_{\overline{\tau}} = \operatorname{FT}_z(\overline{\sigma^*}t)\) and \(\overline{\sigma^*}\operatorname{FT}_z(z\partial_t) = \overline{\tau} = \operatorname{FT}_z(\overline{\sigma^*}(z\partial_t))\). 
\end{proof}

\begin{lemma}\label{lemma:FTcomp}
(Twistor-)Fourier transform of distributions satisfies the following properties: 

\begin{itemize}
    \item Let \(u \in \mathscr{S}'(\mathbb{C}_t)\). We have 
    \begin{align*}
    \operatorname{FT}(\partial_t u) & = \tau \operatorname{FT}(u), &\operatorname{FT}(tu) & = -\partial_{\tau}\operatorname{FT}(u), \\ \operatorname{FT}(\partial_{\overline{t}} u) & = -\overline{\tau} \operatorname{FT}(u), &\operatorname{FT}(\overline{t}u) & = \partial_{\overline{\tau}}\operatorname{FT}(u). 
    \end{align*}
    \item Let \(u \in \mathscr{S}'(\mathbb{C}_t \times \boldsymbol{S})\). We have 
    \begin{align*}
    \operatorname{FT}_z(z\partial_t u) & = \tau \operatorname{FT}_z(u), &\operatorname{FT}_z(tu) & = -z\partial_{\tau}\operatorname{FT}_z(u), \\ \operatorname{FT}_z(z^{-1}\partial_{\overline{t}} u) & = -\overline{\tau} \operatorname{FT}_z(u), &\operatorname{FT}_z(\overline{t}u) & = z^{-1}\partial_{\overline{\tau}}\operatorname{FT}_z(u). 
    \end{align*}
\end{itemize}
\end{lemma}

\begin{proof}
These are elementary calculations. 
\end{proof}

Note that these "exchange rules" are the same as the rules used to define the Fourier transforms on \(\mathcal{D}\) and \(\mathcal{R}\) in \Cref{def:fourieronDandR}. 

\section{Computing the Fourier integral}

We recall the standard Riemann-Lebesgue lemma and introduce a Twistor-version: 

\begin{lemma}\label{lemma:riemannlebesgue}
The following statements are true: 

\begin{itemize}
    \item Let \(f(t)\) be a function on \(\mathbb{C}_t\) which is globally integrable, that is, it lies in \(L^1(\mathbb{C}_t)\). Then, its Fourier-Transform \(\text{FT}(f)\) is well-defined as a function on \(\mathbb{C}_\tau\) and vanishes as \(|\tau| \rightarrow \infty\). 
    \item Let \(f(t, z)\) be a function on \(\mathbb{C}_t \times \boldsymbol{S}\) such that for each \(z \in \boldsymbol{S}\), the function \(t \mapsto f(t, z)\) is globally integrable. Then, the Fourier-Transform \(\text{FT}_z(f)\) is well-defined as a function on \(\mathbb{C}_\tau \times \boldsymbol{S}\) with the property that, for every \(z \in \boldsymbol{S}\), the function \(\tau \mapsto \text{FT}_z(f)(\tau, z)\) vanishes as \(|\tau| \rightarrow \infty\). 
\end{itemize}
\end{lemma}

\begin{proof}
The first point follows from the usual Riemann-Lebesgue lemma for the classical \(2\)-dimensional real Fourier-Transform. The second point is seen as follows: Assume \(|z| = 1\). We have 
\begin{align*}
\text{FT}_z(f)(\tau, \overline{z}) = \frac{-1}{4 \pi} \int_{\mathbb{C}_t} f(t, \overline{z}) e^{\overline{z\tau t} - \tau t z} dt_1 \wedge dt_2 = \text{FT}(t \mapsto f(t, \overline{z}))(\tau z). 
\end{align*}

Since \(t \mapsto f(t, \overline{z})\) is globally integrable by assumption and since the property of vanishing at \(\infty\) is invariant under scaling the variable by a non-zero complex number, the claim follows from the first point. 
\end{proof}

From now on, we use the notation from \Cref{thm:FTofT}. 

\begin{definition}\label{def:F1F2}
Let \(_1F_1(a, b, t)\) (\(a, t \in \mathbb{C}, b \in \mathbb{C} - \mathbb{N}_0\)) be Kummer's hypergeometric function. It is an entire function in \(a, t\) and meromorphic in \(b\). We define the following two functions in \(\mathcal{O}(\mathbb{C}_t \times \mathbb{C}_z^*)\): 
\begin{align*}
F_1(\tau, z) & = \,_1F_1(1 + \alpha \diamond z, 2 + \beta \diamond z, -\tau/z), \\ F_2(\tau, z) & = \,_1F_1(\alpha \diamond z - \beta \diamond z, -\beta \diamond z, -\tau/z). 
\end{align*}
\end{definition}

Recall the operator \(\mathcal{P}_z \in \mathcal{R}\) from \Cref{def:twistedzoperator}: 
\begin{align*}
\mathcal{P}_z = p(t)z\partial_t - p(t) \sum_{j = 1}^n \frac{m_j \star z}{t - r_j} \in \mathcal{R}, 
\end{align*}

where \(p(t)\partial_t + q(t) = t\partial_t - \alpha - t(t\partial_t - \beta) = \mathcal{P}\). 

\begin{proposition}\label{prop:FTzPsol}
The subsheaf of \(C^{\infty, {\operatorname{an}}}_{\mathbb{C}_\tau^* \times \mathbb{C}_z^*}\), whose sections \(h\) satisfy \((\text{FT}_z\mathcal{P}_z)h = 0\), is given by 
\begin{align*}
\mathcal{O}_{\overline{\mathbb{C_\tau}^*} \times \mathbb{C}_z^*}F_1(\tau, z) + \mathcal{O}_{\overline{\mathbb{C_\tau}^*} \times \mathbb{C}_z^*} \tau^{-1-(\beta \diamond z)}F_2(\tau, z). 
\end{align*}
\end{proposition}

\begin{proof}
Recall the operator \(\mathcal{P} = t\partial_t - \alpha - t(t\partial_t - \beta) \in \mathcal{D}\). We have 
\begin{align}\label{eq:FTP}
\text{FT}(\mathcal{P}) = -\partial_\tau \tau - \alpha + \partial_\tau(-\partial_\tau \tau - \beta) 
\end{align}

by the definition of the Fourier transform on \(\mathcal{D}\). The equation \(\text{FT}(\mathcal{P})f = 0\) is rewritten as 
\begin{align*}
-\tau \frac{d^2f}{d\tau^2} - (\tau + 2 + \beta)\frac{df}{d\tau} - (1 + \alpha)f = 0, 
\end{align*}

Via the ansatz \(g = \exp{\tau} \cdot f\), this equation is rewritten as 
\begin{align*}
\tau \frac{d^2g}{d\tau^2} + (2 + \beta - \tau )\frac{dg}{d\tau} - (1 - \alpha + \beta)g = 0, 
\end{align*}

which is Kummer's confluent hypergeometric equation (with suitably shifted parameters). As is well-known, it has the local system of holomorphic solutions 
\begin{align*}
g(\tau) \in \underline{\mathbb{C}}_{\mathbb{C}_\tau^*}\,_1F_1(1 - \alpha + \beta, 2 + \beta, \tau) + \underline{\mathbb{C}}_{\mathbb{C}_\tau^*}\tau^{-\beta - 1} \,_1F_1(-\alpha, -\beta, \tau), 
\end{align*}

where \(_1F_1\) is the hypergeometric Kummer function (it is an entire function). Because we had set \(g = \exp{\tau} \cdot f\), this implies  
\begin{align*}
f \in \underline{\mathbb{C}}_{\mathbb{C}_\tau^*}\,_1F_1(1 +  \alpha, 2 + \beta, -\tau) + \underline{\mathbb{C}}_{\mathbb{C}_\tau^*}\tau^{-\beta - 1} \,_1F_1(\alpha - \beta, -\beta, -\tau), 
\end{align*}

where we have also used \cite{DLMF}, Formula 13.2.39, which states that 
\begin{align}\label{eq:kummerminusz}
\,_1F_1(a, b, \tau) = e^\tau\,_1F_1(b-a, b, -\tau) 
\end{align}

for any parameters \(a, b \in \mathbb{C}\). The solutions of \(\text{FT}(\mathcal{P})f = 0\) (interpreting \(\partial_\tau\) as the Wirtinger derivative) in \(C_{\mathbb{C}_\tau^*}^\infty\) are then given by 
\begin{align*}
\mathcal{O}_{\overline{\mathbb{C}_\tau}}\,_1F_1(1 +  \alpha, 2 + \beta, -\tau) + \mathcal{O}_{\overline{\mathbb{C}_\tau}}\tau^{-\beta - 1}\,_1F_1(\alpha - \beta, -\beta, -\tau). 
\end{align*}

The equation \((\text{FT}_z\mathcal{P}_z)f = 0\) arises from \(\text{FT}(\mathcal{P})f = 0\) by replacing \(\partial_\tau\) with \(z\partial_\tau\), \(\tau\) with \(\tau/z\), \(\alpha\) with \((\alpha \diamond z)\) and \(\beta\) with \((\beta \diamond z)\). One deduces that the solutions of \((\text{FT}_z\mathcal{P}_z)h = 0\) in \(C^{\infty, {\operatorname{an}}}_{\mathbb{C}_\tau^* \times \mathbb{C}_z^*}\) are given by 
\begin{align*}
\mathcal{O}_{\overline{\mathbb{C}_\tau^*} \times \mathbb{C}_z^*}F_1(\tau, z) + \mathcal{O}_{\overline{\mathbb{C}_\tau^*} \times \mathbb{C}_z^*} (-\tau/z)^{-1-(\beta \diamond z)}F_2(\tau, z), 
\end{align*}

where \(F_1, F_2\) are defined as in the statement. We have 
\begin{align*}
(-\tau/z)^{-1-(\beta \diamond z)} = (-1/z)^{-1-(\beta \diamond z)}\tau^{-1-(\beta \diamond z)} 
\end{align*}

in the multi-valued sense. Since \((-1/z)^{-1-(\beta \diamond z)}\) does not depend on \(\tau\) and since \(\FT_z(\mathscr{P})\) does not involve differentiation with respect to \(z\), the solutions of \((\text{FT}_z\mathcal{P}_z)h = 0\) in \(C^{\infty, {\operatorname{an}}}_{\mathbb{C}_\tau^* \times \mathbb{C}_z^*}\) are 
\begin{align*}
\mathcal{O}_{\overline{\mathbb{C}_\tau^*} \times \mathbb{C}_z^*}F_1(\tau, z) + \mathcal{O}_{\overline{\mathbb{C}_\tau^*} \times \mathbb{C}_z^*}\tau^{-1-\beta \diamond z}F_2(\tau, z), 
\end{align*}

which completes the proof. 
\end{proof}

The following result is the reason for the restriction we make on \(\alpha\) and \(\beta\) in the statement of \Cref{thm:FTofT}: 

\begin{proposition}\label{prop:alphabetarange}
Let \(\alpha, \beta \in \mathbb{C}\). The following statements are equivalent: 

\begin{itemize}
    \item \(\Re(\alpha) > -1, \Re(\beta - \alpha) > -1\) and \(\Re(\beta) < -1\). 
    \item For each fixed \(z' \in \mathbf{S}\), the function 
    \begin{align*}
    \mathbb{C}_t - \{0, 1\} \rightarrow \mathbb{C}, \quad t \mapsto |t|^{2(\alpha \diamond z')}|1-t|^{2((\beta \diamond z') - (\alpha \diamond z'))}
    \end{align*}
    is globally integrable. 
\end{itemize}
\end{proposition}

\begin{proof}
Let \(z' \in \mathbf{S}\). Recall that 
\begin{align*}
\alpha \diamond z' = \Re(\alpha) + \Im(\alpha)\Re(z')i, \quad \beta \diamond z' = \Re(\beta) + \Im(\beta)\Re(z')i, 
\end{align*}

so that \(\Re(\alpha \diamond z') = \Re(\alpha)\) and  \(\Re(\beta \diamond z') = \Re(\beta)\). It follows that 
\begin{align*}
\left\vert|t|^{2(\alpha \diamond z')}|1-t|^{2((\beta \diamond z') - (\alpha \diamond z'))}\right\vert = |t|^{2\Re(\alpha)}|1-t|^{2(\Re(\beta) - \Re(\alpha))}. 
\end{align*}

It is easy to check that this expression possesses a finite integral over \(\mathbb{C}_t - \{0, 1\}\) if and only if \(\alpha\) and \(\beta\) satisfy the conditions specified in the claim. 
\end{proof}

\begin{figure}[h]
    \centering
    \includegraphics[width=0.5\linewidth]{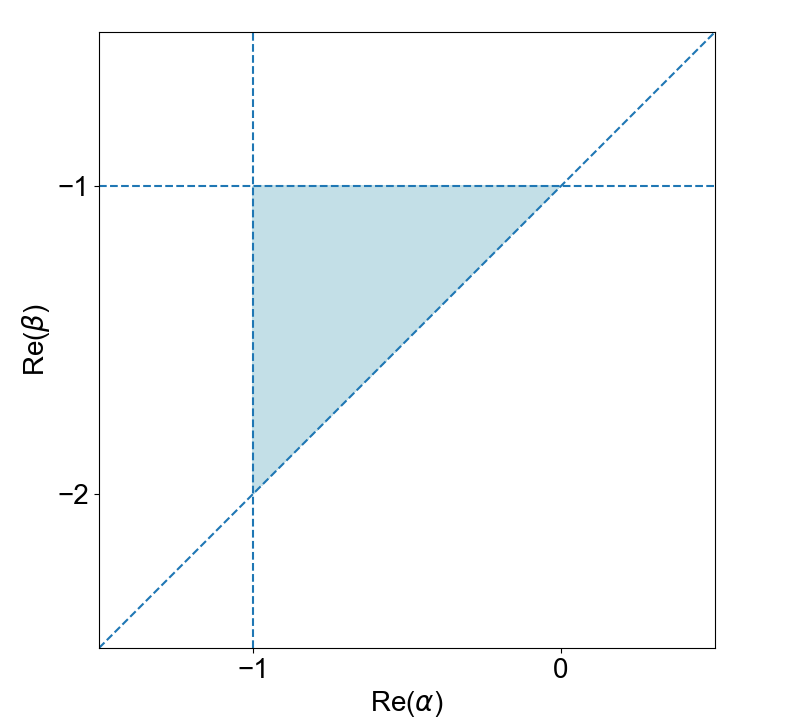}
    \caption{The admissible region for \(\Re(\alpha)\) and \(\Re(\beta)\) from \Cref{prop:alphabetarange}}
\end{figure}

\begin{definition}
We call \(\alpha, \beta \in \mathbb{C}\) admissible if the conditions of \Cref{prop:alphabetarange} are satisfied. 
\end{definition}

\begin{theorem}\label{thm:C1C2form}
Let \(\alpha, \beta \in \mathbb{C}\) be admissible and let \(g(t, z) = t^{(\alpha \diamond z)}(1-t)^{(\beta \diamond z) - (\alpha \diamond z)}\). Then, for \((\tau, z) \in \mathbb{C}_\tau \times \boldsymbol{S}\) we have 
\begin{align*}
\operatorname{FT}_z\left(|g|^2_z\right) & = C_1(z) \cdot |F_1(\tau, z)|_z^2 + C_2(z) \cdot \left|\tau^{-1-\beta \diamond z}F_2(\tau, z)\right|_z^2, 
\end{align*}

where \(C_1(z), C_2(z)\) are unique elements of \(\mathcal{O}(\boldsymbol{S})\). 
\end{theorem}

\begin{proof}

Let \(g(t, z) = t^{(\alpha \diamond z)}(1-t)^{(\beta \diamond z) - (\alpha \diamond z)}\). We have \(\mathcal{P}_zg = 0\). We use \Cref{lemma:conjcomp} to deduce from \(\mathcal{P}_zg = 0\) that \((\overline{\sigma^*}\mathcal{P}_z)(\overline{\sigma^*}g) = 0\). Recall that \(|g|^2_z = \overline{\sigma^*} g \cdot g\), where \(g\) resp. \(\overline{\sigma^*}g\) is holomorphic resp. anti-holomorphic in \(t\). Because \(\mathcal{P}_z \in \mathcal{R}\) and \(\overline{\sigma^*}\mathcal{P}_z \in \overline{\sigma^*}\mathcal{R}\), it follows that 
\begin{align*}
\mathcal{P}_z(|g|^2_z) = 0 \quad \text{and} \quad (\overline{\sigma^*}\mathcal{P}_z)(|g|^2_z) = 0. 
\end{align*}

Next, we use \Cref{lemma:FTcomp} to deduce that 
\begin{align}\label{eq:FTcrucial1}
\text{FT}_z(\mathcal{P}_z)\text{FT}_z(|g|^2_z(t, z)) = 0 \quad \text{and} \quad \text{FT}_z(\overline{\sigma^*}\mathcal{P}_z)\text{FT}_z(|g|^2_z(t, z)) = 0. 
\end{align}

Consider the universal covering \(\widetilde{\mathbb{C}_\tau^*}\) of \(\mathbb{C}_\tau^*\). Recall from \Cref{prop:FTzPsol} that the solutions of \((\text{FT}_z\mathcal{P}_z)s = 0\) in \(C^{\infty, {\operatorname{an}}}(\widetilde{\mathbb{C}_\tau^*} \times \mathbb{C}_z^*)\) are given by 
\begin{align}\label{eq:FTcrucial2}
\mathcal{O}\left(\overline{\widetilde{\mathbb{C}_\tau^*}} \times \mathbb{C}_z^*\right) \cdot \underbrace{F_1(\tau, z)}_{=: p_1} + \mathcal{O}\left(\overline{\widetilde{\mathbb{C}_\tau^*}} \times \mathbb{C}_z^*\right) \cdot \underbrace{\tau^{-1-(\beta \diamond z)}F_2(\tau, z)}_{=: p_2}. 
\end{align}

(Here, \(F_1, F_2\) actually denote the pullback of \(F_1, F_2\) along the projection \(\widetilde{\mathbb{C}_\tau^*} \times \mathbb{C}_z^* \rightarrow \mathbb{C}_\tau^* \times \mathbb{C}_z^*\), whereas \(\tau^{-1-(\beta \diamond z)}\) is multi-valued on \(\mathbb{C}_\tau^* \times \mathbb{C}_z^*\), but defines a single-valued function on \(\widetilde{\mathbb{C}_\tau^*} \times \mathbb{C}_z^*\).) We have \(\overline{\sigma^*}\text{FT}_z\mathcal{P}_z = \text{FT}_z \overline{\sigma^*}\mathcal{P}_z\) by \Cref{lemma:ftconj}. By \Cref{lemma:conjcomp}, this implies that the solutions of  \((\text{FT}_z(\overline{\sigma^*}\mathcal{P}_z))s = 0\) in \(C^{\infty, {\operatorname{an}}}(\widetilde{\mathbb{C}_\tau^*} \times \mathbb{C}_z^*)\) are given by 
\begin{align}\label{eq:FTcrucial3}
\mathcal{O}\left(\widetilde{\mathbb{C}_\tau^*} \times \mathbb{C}_z^*\right) \cdot \underbrace{(\overline{\sigma^*}(F_1(\tau, z)))}_{=: q_1} + \mathcal{O}\left(\widetilde{\mathbb{C}_\tau^*} \times \mathbb{C}_z^*\right) \cdot \underbrace{\overline{\sigma^*}(\tau^{-1-(\beta \diamond z)}F_2(\tau, z))}_{=: q_2}. 
\end{align}

\Cref{eq:FTcrucial1} implies that \(\text{FT}_z(|g|^2_z)\) lies in the intersection of the spaces given by \Cref{eq:FTcrucial2} and \Cref{eq:FTcrucial3}. Hence, the second point of \Cref{thm:smoothZsymmetry} implies 
\begin{align*}
\text{FT}_z(|g|^2_z(t, z)) & = \sum_{i}^2C_{i}(z) \cdot p_i(\tau, z) \cdot q_i(\tau, z) \\ & =  C_1(z) \cdot |F_1(\tau, z)|_z^2 + C_2(z) \cdot \left|\tau^{-1-(\beta \diamond z)}F_2(\tau, z)\right|_z^2. 
\end{align*}

for unique \(C_{i} \in \mathcal{M}(\mathbb{C}_z^*)\). (Here, the current coordinate \(\tau\) is \(t\) in the language of \Cref{thm:smoothZsymmetry}.) Because \(\text{FT}_z(|g|^2_z(t, z))\) is continuous by \Cref{lemma:riemannlebesgue} (Riemann-Lebesgue lemma), the \(C_i\) lie in \(\mathcal{O}(\boldsymbol{S})\). Recalling that \(|g|^2_z = |t|^{2(\alpha \diamond z)}|1-t|^{2((\beta \diamond z) - (\alpha \diamond z))}\), the proof is complete. 
\end{proof}

Under the admissibility assumption, we can actually evaluate the integral of \(|t|^{2(\alpha \diamond z')}|1-t|^{2((\beta \diamond z') - (\alpha \diamond z'))}\): 

\begin{lemma}\label{lemma:FTat0formula}
Let \(\alpha, \beta \in \mathbb{C}\) be admissible and let \(z' \in \boldsymbol{S}\). Then, we have 
\begin{align*}
\int_{\mathbb{C}} |t|^{2(\alpha \diamond z')}|1-t|^{2((\beta \diamond z') - (\alpha \diamond z'))} d^2t  = \pi\frac{\widetilde{\Gamma}((\alpha \diamond z'))\widetilde{\Gamma}((\beta \diamond z') - (\alpha \diamond z'))}{\widetilde{\Gamma}(1 + (\beta \diamond z'))}, 
\end{align*}

where we have set \(\widetilde{\Gamma}(s) \coloneqq \Gamma(1+s)/\Gamma(-s)\) \((s \in \mathbb{C})\). 
\end{lemma}

\begin{proof}
This proof just consists of integral manipulations and was obtained with the use of the Google Gemini LLM. Set \(a = \alpha \diamond z'\) and \(b = \beta \diamond z' - \alpha \diamond z'\). Abbreviate 
\begin{align*}
I = \int_{\mathbb{C}} |t|^{2a}|1-t|^{2b} d^2t. 
\end{align*}

Our first goal is to show that, for any \(z, c \in \mathbb{C}\) with \(-1 < \Re(c) < 0\), one has 
\begin{align}\label{eq:gammamellin}
|z|^{2c} = \frac{1}{\Gamma(-c)}\int_0^\infty x^{-c - 1}e^{-x|z|^2}dx. 
\end{align}

To this end, we have by definition of the \(\Gamma\)-function: 
\begin{align*}
\Gamma(-c) = \int_0^\infty v^{-c - 1}e^{-v}dv. 
\end{align*}

The substitution \(v = x|z|^2\) gives 
\begin{align*}
\Gamma(-c) & = \int_0^\infty (x|z|^2)^{-c-1} e^{-x|z|^2} |z|^2 dx \\ & = |z|^{-2c} \int_0^\infty x^{-c - 1}e^{-x|z|^2} dx, 
\end{align*}

which gives \Cref{eq:gammamellin}. This implies 
\begin{align*}
I = \frac{1}{\Gamma(-a)\Gamma(-b)} \int_{\mathbb{C}} \int_0^\infty \int_0^\infty x^{-a-1}y^{-b-1}e^{-x|t|^2 - y|1-t|^2} dx\, dy\, d^2t
\end{align*}

One has 
\begin{align*}
\int_\mathbb{C} e^{-x|t|^2 - y|1-t|^2} d^2t = e^{-\frac{xy}{x + y}} \frac{\pi}{x + y}
\end{align*}

as a variation of the Gaussian integral. So, we can rewrite 
\begin{align*}
I = \frac{\pi}{\Gamma(-a)\Gamma(-b)} \int_0^\infty \int_0^\infty \frac{ x^{-a-1}y^{-b-1}}{x+y} e^{-\frac{xy}{x + y}} dx\, dy. 
\end{align*}

We perform the polar substitution \(x = r\xi, y = r(1 - \xi)\), where \(r \in [0, \infty), \xi \in [0, 1]\). This gives \(dx \, dy = r \, d\xi \, dr\) and 
\begin{align*}
I = \frac{\pi}{\Gamma(-a)\Gamma(-b)} \int_0^\infty \int_0^1 r^{-a - b - 2} \xi^{-a - 1} (1 - \xi)^{-b-1} e^{-r \xi(1 - \xi)} d\xi\,dr. 
\end{align*}

Using the substitution \(s = r\xi(1-\xi)\) and the definition of the \(\Gamma\)-function, one obtains 
\begin{align*}
\int_0^\infty r^{-a - b - 2} e^{-r \xi(1 - \xi)} \, dr = (\xi(1 - \xi))^{a + b + 1}\Gamma(-a-b-1), 
\end{align*}

leading to 
\begin{align*}
I = \frac{\pi\Gamma(-a-b-1)}{\Gamma(-a)\Gamma(-b)} \int_0^1 \xi^b(1 - \xi)^a d\xi. 
\end{align*}

The remaining integral is the Euler Beta function and the result follows. 
\end{proof}

\begin{proposition}\label{prop:Cirelation}
Let \(\alpha, \beta \in \mathbb{C}\) be admissible. We have \(C_1(z)D_1(z) + C_2(z)D_2(z) = 0\) for \(z \in \boldsymbol{S}\), where \(D_1, D_2 \in \mathcal{O}(\boldsymbol{S})\) are given by 
\begin{equation*} 
\begin{split}
D_1(z) & = \frac{\Gamma(2 + (\beta \diamond z))^2}{\Gamma(1 + (\alpha \diamond z))\Gamma(1 + (\beta \diamond z) - (\alpha \diamond z))}, \\ D_2(z) & = \frac{\Gamma(-(\beta \diamond z))^2}{\Gamma(-(\beta \diamond z) + (\alpha \diamond z))\Gamma(-(\alpha \diamond z))}. 
\end{split}
\end{equation*}
\end{proposition}

\begin{proof}
Using \Cref{thm:C1C2form} in combination with the identities \(\overline{\sigma^*}(\alpha \diamond z) = \alpha \diamond z\) and \(\overline{\sigma^*}(\beta \diamond z) = \beta \diamond z\), we obtain 
\begin{equation}\label{eq:FTstokes}
\begin{split}
\text{FT}_z\left(|g|^2_z\right) & = C_1(z) \cdot F_1(\overline{\tau}, -1/z) \cdot F_1(\tau, z) \\ & + C_2(z) \cdot (-\overline{\tau}z)^{-1-\beta \diamond z}F_2(\overline{\tau}, -1/z) \cdot (-\tau/z)^{-1-\beta \diamond z}F_2(\tau, z) 
\end{split}
\end{equation}

Fix \(z' \in \boldsymbol{S}\). We will now investigate the asymptotic behavior of \(\tau \mapsto \text{FT}_z(|f|^2_z)(\tau, z')\) as \(\tau \rightarrow -z' \cdot \infty\). To do so, we need the following asymptotic expansion of the confluent hypergeometric function \(_1F_1\) from \cite{AbrSte}, Chapter 13: For \(a, b \in \mathbb{C}\) with \(a, b, b - a \notin \mathbb{Z}\), we have the following asymptotic behavior, as the complex variable \(\upsilon\) takes the limit \(|\upsilon| \rightarrow \infty\): 
\begin{align*}
_1F_1(a, b, \upsilon) \sim \Gamma(b)\left(\frac{e^{\upsilon} \upsilon^{a-b}}{\Gamma(a)} + \frac{(-\upsilon)^{-a}}{\Gamma(b-a)}\right), 
\end{align*}

where the powers of \(\upsilon\) are taken using \(-3\pi/2 < \arg \tau < \pi/2\). In the following, we will use \(\upsilon = -\tau/z'\). One easily verifies the following: If \(\upsilon\) goes to \(\infty\) on the positive (negative) real half-axis, the second (first) term of this expansion can be omitted. We have 
\begin{align*}
\lim_{\tau \rightarrow -z' \cdot \infty} -\tau/z' = 1 \cdot \infty, \quad \lim_{\tau \rightarrow -z' \cdot \infty} \overline{\tau}z' = -1 \cdot \infty. 
\end{align*}

Defining \(D_1(z'), D_2(z')\) as in the statement, we obtain from \Cref{eq:FTstokes} the following asymptotic expansion for \(\text{FT}_z(|f|^2_z)(\tau, z')\) as \(\tau \rightarrow -z' \cdot \infty\): 
\begin{equation} 
\begin{split}
 & \, C_1(z')D_1(z') \cdot e^{-\tau/z'}(-\tau/z')^{-1 + (\alpha \diamond z') - (\beta \diamond z')}(-\overline{\tau}z')^{-1-(\alpha \diamond z')} \\ + &\, C_2(z')D_2(z') \cdot (-\tau/z')^{-1-(\beta \diamond z')}(-\overline{\tau}z')^{-1-(\beta \diamond z')}e^{-\tau/z'}(-\tau/z')^{(\alpha \diamond z')}(-\overline{\tau}z')^{(\beta \diamond z') - (\alpha \diamond z')}
\end{split}
\end{equation}

This expression is factorized as 
\begin{equation} 
\begin{split}
\left(C_1(z')D_1(z') + C_2(z')D_2(z')\right) \cdot \underbrace{(-\tau/z')^{-1-(\beta \diamond z')+(\alpha \diamond z')}(-\overline{\tau}z')^{-1-(\alpha \diamond z')}e^{-\tau/z'}}_{=: E(\tau, z')}. 
\end{split}
\end{equation}

Note that the absolute value of \(E(\tau, z')\) tends to \(\infty\) as \(\tau \rightarrow -z' \cdot \infty\), because of the factor \(e^{-\tau/z'}\). On the other hand, \(\tau \mapsto \text{FT}_z(|f|^2_z)(\tau, z')\) vanishes at \(\infty\) by the Riemann-Lebesgue lemma (\Cref{lemma:riemannlebesgue}). It follows that 
\begin{align*}
C_1(z')D_1(z') + C_2(z')D_2(z') = 0, 
\end{align*}

completing the proof. 
\end{proof}

\begin{remark}\label{hiddenstokes}
This argument can naturally be formulated with the Stokes phenomenon: Fix \(z' \in \boldsymbol{S}\). The Stokes directions at \(\infty\) for \(\operatorname{FT}_z\mathcal{P}_z\) at \(z = z'\) are given by the angles \(\arg(z'), \arg(z') + \pi\). Comparing with \Cref{eq:FTstokes}, this means that \(D_1, D_2\) can be obtained from the asymptotic behavior of \(F_1, F_2\) along the single ray \(c \cdot (-z')\) (\(c \in \mathbb{R}_+\)) in combination with the Stokes matrix. 
\end{remark}

\begin{proposition}\label{lemma:Cifinal}
Let \(\alpha, \beta\) be admissible. Setting \(\widetilde\Gamma(z) \coloneqq \Gamma(1 + z)/\Gamma(-z)\), we have  
\begin{align*}
C_1(z) & = \frac{\widetilde{\Gamma}((\alpha \diamond z))\widetilde{\Gamma}((\beta \diamond z) - (\alpha \diamond z))}{\widetilde{\Gamma}(1 + (\beta \diamond z))}, \\  C_2(z) & = \widetilde{\Gamma}((\beta \diamond z)). 
\end{align*}
\end{proposition}

\begin{proof}
We have \(_1F_1(a, b, 0) = 1\) for any non-integers \(a\) and \(b\), so that 
\begin{align*}
\text{FT}_z(|g|_z^2)(0, z) = C_1(z)
\end{align*}

On the other hand, we have 
\begin{align*}
\text{FT}_z(|g|_z^2)(0, z) & = \frac{1}{\pi}\int_{\mathbb{C}}|g|_z^2(t, z)d^2t \\ & = \frac{\widetilde{\Gamma}((\alpha \diamond z))\widetilde{\Gamma}((\beta \diamond z) - (\alpha \diamond z))}{\widetilde{\Gamma}(1 + (\beta \diamond z))}, 
\end{align*}

by \Cref{lemma:FTat0formula}. Combining this with \Cref{prop:Cirelation} completes the proof. 
\end{proof}

\section{The Fourier Transform of \texorpdfstring{\(\mathfrak{T}\)}{T} in the admissible case} 

All of the preparation is complete and we can prove \Cref{thm:FTofT}. We begin by recalling some notation: We have set \(\mathfrak{T} = (\mathfrak{M}_z, \mathfrak{M}_z, \mathfrak{C})\) and we denote the \(\mathscr{R}_{\mathbb{A}^1 \times \mathbb{C}_z}\)-module corresponding to the \(\mathscr{R}_{\mathbb{P}^1 \times \mathbb{C}_z}(*\infty)\)-module \(\mathfrak{M}_z(*\infty)\) (in the sense of \Cref{lemma:partialGAGA}) by \(\mathcal{M}_z\). We have \(\mathcal{M}_z = \mathscr{R}_{\mathbb{A}^1_t \times \mathbb{C}_z}/(\mathcal{P}_z)\) and \((\mathcal{M}_z)^{t\operatorname{-an}} = \mathfrak{M}_z\vert_{\mathbb{C}_t}\). Moreover, we define the \(\mathscr{R}_{\mathbb{A}^1_\tau \times \mathbb{C}_z}\)-module 
\begin{align*}
\mathcal{N}_z \coloneqq \mathscr{R}_{\mathbb{A}^1_\tau \times \mathbb{C}_z}/(\text{FT}_z(\mathcal{P}_z)). 
\end{align*}

and the \(\mathscr{R}_{\mathbb{P}^1 \times \mathbb{C}_z}\)-module \(\mathfrak{N}_z = p^{*, \operatorname{loc}}(\mathcal{N}_z)[{!*}\infty]\) as in \Cref{lemma:partialGAGA}. We have \((\mathcal{N}_z)^{\tau\operatorname{-an}} = \mathfrak{N}_z\vert_{\mathbb{C}_\tau}\). In \Cref{def:FTTwistorD}, we have defined \(\operatorname{FT}_z(\mathfrak{T})\) as \((\operatorname{FT}_z(\mathfrak{M}_z), \operatorname{FT}_z(\mathfrak{M}_z), \operatorname{FT}_z(\mathfrak{C}))\). By \Cref{lemma:fttranslate}, we have the canonical isomorphism 
\begin{align*}
m: \text{FT}_z(\mathcal{M}_z) \rightarrow \mathcal{N}_z, \quad [s] \mapsto [\text{FT}_z(s)], 
\end{align*}

inducing a canonical isomorphism \(\operatorname{FT}_z(\mathfrak{M}_z) \cong \mathfrak{N}_z\). Using the tensor product, analytification (with respect to \(\tau\)), restriction to \(\mathbb{C}_\tau^* \times \boldsymbol{S}\) and the Twistor conjugation functor \(\overline{\sigma^*}\), we get the induced isomorphism \(M \coloneq (\overline{\sigma^*}m^{\tau\operatorname{-an}} \otimes m^{\tau\operatorname{-an}})\vert_{\mathbb{C}_\tau^* \times \boldsymbol{S}}\): 
\begin{align*}
M: \overline{\sigma^*}(\text{FT}_z(\mathfrak{M}_z))\vert_{\mathbb{C}_\tau^* \times \boldsymbol{S}} \otimes \text{FT}_z(\mathfrak{M}_z)\vert_{\mathbb{C}_\tau^* \times \boldsymbol{S}} \rightarrow \overline{\sigma^*}(\mathfrak{N}_z)\vert_{\mathbb{C}_\tau^* \times \boldsymbol{S}} \otimes \mathfrak{N}_z\vert_{\mathbb{C}_\tau^* \times \boldsymbol{S}}. 
\end{align*}

Observe that the domain of \(M\) is the domain of \(\operatorname{FT}_z(\mathfrak{C})\vert_{\mathbb{C}_\tau^*}\). Consider the restricted Twistor-sesquilinear pairing 
\begin{align*}
D \coloneqq \operatorname{FT}_z(\mathfrak{C})\vert_{\mathbb{C}_\tau^*} \circ M^{-1}: \overline{\sigma^*}(\mathfrak{N}_z)\vert_{\mathbb{C}_\tau^* \times \boldsymbol{S}} \otimes \mathfrak{N}_z\vert_{\mathbb{C}_\tau^* \times \boldsymbol{S}} \rightarrow C^{\infty, \operatorname{an}}_{\mathbb{C}_\tau^* \times \boldsymbol{S}}. 
\end{align*}

\(D\) takes values in functions and not distributions, because we have restricted away from the singular locus \(\{\tau = 0\}\). We have the isomorphism of variations of Twistor structure 
\begin{align*}
(\operatorname{FT}_z(\mathfrak{T}))\vert_{\mathbb{C}_\tau^*} \cong (\mathfrak{N}_z\vert_{\mathbb{C}_\tau^*}, \mathfrak{N}_z\vert_{\mathbb{C}_\tau^*}, D). 
\end{align*}

We will now construct an \(\mathcal{O}_{\mathbb{C}_\tau^* \times \boldsymbol{S}}\)-basis of \(\mathfrak{N}_z\vert_{\mathbb{C}_\tau^* \times \boldsymbol{S}}\): First, note that we have 
\begin{align}\label{eq:RA1andS}
\Gamma(\mathbb{A}_\tau^{1, *} \times \boldsymbol{S}, \mathcal{N}_z) = \mathscr{R}(\mathbb{A}_\tau^{1, *} \times \boldsymbol{S})/(\text{FT}_z(\mathcal{P}_z))
\end{align}

where we have abbreviated the global sections of \(\mathscr{R}_{\mathbb{A}_\tau^{1, *} \times \mathbb{C}_z}\vert_{\boldsymbol{S}}\) as \(\mathscr{R}(\mathbb{A}_\tau^{1, *} \times \boldsymbol{S})\). (As 
usual, the restriction to \(\boldsymbol{S}\) is in the sense of sheaves.) \Cref{eq:RA1andS} is seen as follows: First, note that 
\begin{align*}
\Gamma(\mathbb{A}_\tau^{1, *} \times \mathbb{C}_z, \mathcal{N}_z) = \mathscr{R}(\mathbb{A}_\tau^{1, *} \times \mathbb{C}_z)/(\text{FT}_z(\mathcal{P}_z))
\end{align*}

by an argument analogous to \Cref{lemma:sectalg}. \Cref{eq:RA1andS} then follows from the fact that the sheaf-theoretical restriction to \(|z| = 1\) is an exact functor. By \Cref{prop:rmodsandzconns}, the right-hand side of \Cref{eq:RA1andS} is generated over \(\mathcal{O}(\mathbb{A}_\tau^{1, *} \times \boldsymbol{S})\) by the classes \(b_1 \coloneqq [1]\) and \(b_2 \coloneqq [z\partial_\tau]\). Denote the image of \((b_1, b_2)\) under the analytification map 

\begin{align*}
\Gamma(\mathbb{A}_\tau^{1, *} \times \boldsymbol{S}, \mathcal{N}_z) \rightarrow \Gamma(\mathbb{A}_\tau^{1, *} \times \boldsymbol{S}, \mathfrak{N}_z\vert_{\mathbb{C}_\tau^{*} \times \boldsymbol{S}})
\end{align*}

by \((e_1, e_2)\). Because \(\mathfrak{N}_z\vert_{\mathbb{A}_\tau^{1, *} \times \boldsymbol{S}}\) is \(\mathcal{O}_{\mathbb{A}_\tau^{1, *} \times \boldsymbol{S}}\)-free, \((e_1, e_2)\) generate \(\mathfrak{N}_z\vert_{\mathbb{C}_\tau^{*} \times \boldsymbol{S}}\). We will now describe \(D\) in terms of its \((2 \times 2)\)-matrix \(D_{\operatorname{mat}}\) with entries in \(C^{\infty, \operatorname{an}}_{\mathbb{C}_\tau^* \times \boldsymbol{S}}\) in the basis \((e_1, e_2)\) of \((\mathfrak{N}_z)\vert_{\mathbb{C}_\tau^* \times \boldsymbol{S}}\): 

\begin{proposition}\label{prop:dmatformula}
Let \(\alpha, \beta \in \mathbb{C}\) be admissible and set 
\begin{align*}
G(t, z) & = |t|^{2(\alpha \diamond z)}|1-t|^{2(\beta - \alpha) \diamond z}\vert_{\boldsymbol{S}} \\ F(\tau, z) & = \FT_z(G)(\tau, z)\vert_{\mathbb{C}_\tau^* \times \boldsymbol{S}}. 
\end{align*}

Then,  
\begin{align*}
D_{\operatorname{mat}}(\tau, z) = \left.\begin{pmatrix}
F(\tau, z) & z\partial_\tau F(\tau, z)\\
-\frac{1}{z}\partial_{\overline{\tau}} F(\tau, z) & -\partial_\tau \partial_{\overline{\tau}} F(\tau, z)
\end{pmatrix}\right. 
\end{align*}
\end{proposition}

\begin{proof}
Via the natural isomorphism 
\begin{align*}
(m^{-1})\vert_{\mathbb{A}_\tau^{1, *} \times \boldsymbol{S}}: \mathcal{N}_z\vert_{\mathbb{A}_\tau^{1, *} \times \boldsymbol{S}} \rightarrow \operatorname{FT}_z(\mathcal{M}_z)\vert_{\mathbb{A}_\tau^{1, *} \times \boldsymbol{S}}, 
\end{align*}

the global section \(b_1\) resp. \(b_2\) of the domain is mapped to the class of \(1\) resp. \(-t\) in the global sections of the co-domain. Denote these classes by \(c_1, c_2\) and their analytifications in \(\operatorname{FT}_z(\mathfrak{M})\vert_{\mathbb{C}^*_\tau \times \boldsymbol{S}}\) by \(f_1, f_2\). By \Cref{prop:rmodsandzconns}, \(\operatorname{FT}_z(\mathfrak{M})\vert_{\mathbb{C}^*_\tau \times \boldsymbol{S}}\) is generated over \(\mathcal{O}_{\mathbb{C}_\tau^* \times \boldsymbol{S}}\) by \((f_1, f_2)\). By the definition of the Fourier transform of Twistor-sesquilinear pairings (\Cref{def:FTTwistorD}), we have 
\begin{align*}
D(\overline{\sigma^*}e_i, e_j) = (\text{FT}_z(\mathfrak{C}))\vert_{\mathbb{C}_\tau^*}(\overline{\sigma^*}f_i, f_j) = \text{FT}_{z}(\mathfrak{C}(\overline{\sigma^*}f_i, f_j)). 
\end{align*}

(The first \(\text{FT}_z\) is the Fourier-Transform of pairings and the second one is the Fourier-Transform of functions.) After \Cref{theorem:firstmainresult}, we have  
\begin{align*}
\mathfrak{C}(\overline{\sigma^*}f_i, f_j) = \overline{\sigma^*}f_i \cdot G \cdot f_j
\end{align*}

and so we have 
\begin{align*}
D_{\text{mat}} = \left.\begin{pmatrix}
\text{FT}_z\left(G\right) & \text{FT}_z\left(-tG\right)\\
\text{FT}_z\left(-\overline{t}G\right) & \text{FT}_z\left(|t|^2G\right)
\end{pmatrix}\right\vert_{\mathbb{C}_\tau^*} = \begin{pmatrix}
F & z\partial_\tau F\\
-\frac{1}{z}\partial_{\overline{\tau}} F & -\partial_\tau \partial_{\overline{\tau}} F
\end{pmatrix}, 
\end{align*}

where the first equality follows from the definition of \(D_{\text{mat}}\) and the second one follows from \Cref{lemma:FTcomp}. 
\end{proof}

Recall the functions \(C_1, C_2 \in \mathcal{O}(\boldsymbol{S})\) from \Cref{lemma:Cifinal} and the functions \(F_1, F_2 \in \mathcal{O}(\boldsymbol{S})\) from \Cref{def:F1F2}. We will now obtain our final form of \(D_{\text{mat}}\): 

\begin{theorem}
We have \(D_{\text{mat}} = \overline{\sigma^*}B^T \cdot B\) for 
\begin{align*}
B(\tau, z) \coloneqq \operatorname{diag}(c_1(z), c_2(z)) \cdot \begin{pmatrix}
F_1(\tau, z) & z\partial_\tau(F_1(\tau, z)) \\ \tau^{-1-(\beta \diamond z)}F_2(\tau, z) & z\partial_\tau(\tau^{-1-(\beta \diamond z)}F_2(\tau, z)) 
\end{pmatrix}. 
\end{align*}

Here, the \(c_i \in \mathcal{O}(\boldsymbol{S})\) satisfy \(|c_i(z)|_z^2 = C_i(z)\). 
\end{theorem}

\begin{proof}
Recall from \Cref{thm:C1C2form} and \Cref{lemma:Cifinal} that 
\begin{equation} 
\begin{split}
F(\tau, z) & = C_1(z) \cdot ||F_1(\tau, z)||_z^2 + C_2(z) \cdot \left|\left|\tau^{-1-\beta \diamond z}F_2(\tau, z)\right|\right|_z^2, 
\end{split}
\end{equation}

One easily verifies 
\begin{align*}
\partial_\tau F\left(\tau, z\right) & = C_1(z) \overline{\sigma^*}F_1(\tau, z) \cdot \partial_\tau(F_1(\tau, z)) \\ & + C_2(z)\overline{\tau}^{-1 - \beta \diamond z}\overline{\sigma^*}F_2(\tau, z) \cdot \partial_\tau\left(\tau^{-1 - \beta \diamond z}F_2(\tau, z)\right) \\ \partial_{\overline{\tau}} F\left(\tau, z\right) & = C_1(z) \partial_{\overline{\tau}}(\overline{\sigma^*}F_1(\tau, z)) \cdot F_1(\tau, z) \\ & + C_2(z)\partial_{\overline{\tau}}\left(\overline{\tau}^{-1 - \beta \diamond z}\overline{\sigma^*}F_2(\tau, z)\right) \cdot \tau^{-1 - \beta \diamond z}F_2(\tau, z) \\ \partial_{\overline{\tau}}\partial_\tau F\left(\tau, z\right) & = C_1(z) \partial_{\overline{\tau}}(\overline{\sigma^*}F_1(\tau, z)) \cdot \partial_\tau(F_1(\tau, z)) \\ & + C_2(z)\partial_{\overline{\tau}}\left(\overline{\tau}^{-1 - \beta \diamond z}\overline{\sigma^*}F_2(\tau, z)\right) \cdot \partial_\tau\left(\tau^{-1 - \beta \diamond z}F_2(\tau, z)\right). 
\end{align*}

Recall from \Cref{prop:dmatformula} that 
\begin{align}\label{eq:Dmatdef}
D_{\text{mat}} = \begin{pmatrix}
F & z\partial_\tau F\\
-\frac{1}{z}\partial_{\overline{\tau}} F & -\partial_\tau \partial_{\overline{\tau}} F
\end{pmatrix}. 
\end{align}

For a \(2 \times 2\)-matrix 
\begin{align*}
B = \begin{pmatrix}
a(\tau, z) & b(\tau, z) \\ c(\tau, z) & d(\tau, z)
\end{pmatrix}. 
\end{align*}

with entries in \(\mathcal{O}_{\mathbb{C}_\tau^* \times \boldsymbol{S}}\), we have 
\begin{align*}
\overline{\sigma^*}B^T \cdot B = \begin{pmatrix}
||a||_z^2 + ||c||_z^2 & \overline{\sigma^*}a \cdot b + \overline{\sigma^*}c \cdot d \\ \overline{\sigma^*} b \cdot a + \overline{\sigma^*} d \cdot c & ||b||_z^2 + ||d||_z^2 
\end{pmatrix}, 
\end{align*}

as is confirmed by a simple calculation. We would like to write the matrix \(D_{\text{mat}}\) from \Cref{eq:Dmatdef} as \(D_{\text{mat}} = \overline{\sigma^*}B^T \cdot B\) for some suitable \(B\) with entries in \(\mathcal{O}_{\mathbb{C}_\tau^* \times \boldsymbol{S}}\). Our first attempt to find such \(B\) is 
\begin{align*}
B_1(\tau, z) \coloneqq \begin{pmatrix}
F_1(\tau, z) & z\partial_\tau(F_1(\tau, z)) \\ \tau^{-1 - \beta \diamond z}F_2(\tau, z) & z\partial_\tau(\tau^{-1 - \beta \diamond z}F_2(\tau, z))
\end{pmatrix}. 
\end{align*}

Noting that for any smooth function \(f\) on \(\mathbb{C}_\tau\), one has \(\overline{\partial_\tau f} = \partial_{\overline{\tau}}\overline{f}\), we obtain 
\begin{align*}
\overline{\sigma^*}B_1(\tau, z)^T = \begin{pmatrix}
\overline{\sigma^*}F_1(\tau, z) & \overline{\tau}^{-1 - \beta \diamond z}\overline{\sigma^*}F_2(\tau, z)\\ (-1/z)\partial_{\overline{\tau}}(\overline{\sigma^*}F_1(\tau, z)) & (-1/z)\partial_{\overline{\tau}}(\overline{\tau}^{-1 - \beta \diamond z}\overline{\sigma^*}F_2(\tau, z))
\end{pmatrix}, 
\end{align*}

so that \(\overline{\sigma^*}B_1(\tau, z)^T \cdot B_1(\tau, z)\) is evaluated to \(D_\text{mat}\), except that the functions \(C_1, C_2 \in \mathcal{O}(\boldsymbol{S})\) are missing. To account for them, we modify \(B_1\) to get our correct choice for \(B\): To this end, note that \(\overline{\sigma^*}C_i(z) = C_i(z)\) and that the \(C_i\) have winding number \(0\) (\(i \in \{1,2\}\)) after \Cref{lemma:gammaroot}. \Cref{thm:twistorroot} implies that there are functions \(c_i \in \mathcal{O}(\boldsymbol{S})\) with \(\overline{\sigma^*}c_i(z)c_i(z) = C_i(z)\). We set 
\begin{align*}
B(\tau, z) \coloneqq \operatorname{diag}(c_1(z), c_2(z)) \cdot \begin{pmatrix}
F_1(\tau, z) & z\partial_\tau(F_1(\tau, z)) \\ \tau^{-1-\beta \diamond z}F_2(\tau, z) & z\partial_\tau(\tau^{-1-\beta \diamond z}F_2(\tau, z)) 
\end{pmatrix}. 
\end{align*}

which fulfills \(D_{\text{mat}} = \overline{\sigma^*}B^T \cdot B\). Thus, we have proven \Cref{thm:FTofT}. 
\end{proof}

\begin{remark}
We make the following remarks on \Cref{thm:FTofT}: 
\begin{itemize}
    \item Let \(X\) be a projective manifold. Recall that constructing a Twistor \(\mathcal{D}\)-module from a \(\mathcal{D}_X\)-module \(\mathfrak{M} \in \operatorname{Hol_{\operatorname{s}}}(X)\) requires knowing the harmonic metric of Deligne type on the localisation of \(\mathfrak{M}\) at the singular points. It is an interesting question to ask the following: If we know the Twistor \(\mathcal{D}\)-module, can we recover the harmonic metric? Although the construction in the proof of \cite{Sab05}, Lemma 2.2.2 answers this in principle, it does not give an explicit solution. In particular, it is unclear to the author how to find the harmonic metric associated to \(\operatorname{FT}(\mathfrak{T})\) in the context of \Cref{thm:FTofT}. 
    \item It should be possible to generalise \Cref{thm:FTofT} to find the Twistor \(\mathcal{D}\)-module associated to the Fourier transform of any regular holonomic \(\mathcal{D}\)-module of generic rank \(1\). 
\end{itemize}
\end{remark}

\chapter{Rigid Twistor \texorpdfstring{\(\mathcal{D}\)-modules}{D-modules} on \texorpdfstring{\(\mathbb{P}^1\)}{the Riemann sphere} and algebraicity}\label{chapter:rigidalg}

The goal of this section is to show an algebraicity result in the \(z\)-direction for a certain class of Twistor \(\mathcal{D}\)-modules: Namely, the class of Twistor \(\mathcal{D}\)-modules on \(\mathbb{P}^1\) corresponding (in the sense of \Cref{theorem:mocmainres}) to rigid irreducible \(\mathcal{D}\)-modules. 

\section{A variant of the Katz-Arinkin algorithm}

Recall the equivalence between \(\operatorname{Hol}_{\mathbb{P}^1, \operatorname{s}}(\mathbb{P}^1)\) and simple meromorphic connections on \(\mathbb{P}^1\) (which we will denote as \(\operatorname{Mer}_{\operatorname{s}}(\mathbb{P}^1)\)), given by localisation at the singularities in one direction and by minimal extension in the other direction.

We now recall the Fourier transform of meromorphic connections on \(\mathbb{P}^1\), as defined in \cite{Ari10}, Section 2.2: 

\begin{definition}
Let \(\mathfrak{M} \in \operatorname{Mer}(\mathbb{P}^{1, \operatorname{alg}})\) and let \(D\) be the singular divisor of \(\mathfrak{M}\). Viewing \(\mathfrak{M}\) as a \(\mathcal{D}_{\mathbb{P}^1}\)-module, one sets 
\begin{align*}
\operatorname{FT}_{\operatorname{mer}}(\mathfrak{M}) \coloneqq j_*(\operatorname{FT}(\mathfrak{M}'({!*}D'))), 
\end{align*}

where \(\mathfrak{M}' = \mathfrak{M}\vert_{\mathbb{A}^1}\), \(D' = D \cap \mathbb{A}^1\), \(j: \mathbb{P}^1 - \operatorname{Sing}(\operatorname{FT}(\mathfrak{M}'({!*}D'))) \hookrightarrow \mathbb{P}^1\) and \(\operatorname{FT}\) is the usual Fourier transform of \(\mathcal{D}_{\mathbb{A}^1}\)-modules. Now, let \(\mathfrak{M} \in \operatorname{Mer}(\mathbb{P}^1)\). The Fourier transform \(\operatorname{FT}_{\operatorname{mer}}(\mathfrak{M})\) is defined in the same way via GAGA for objects of \(\operatorname{Mer}(\mathbb{P}^{1})\). 
\end{definition}

If \(\mathfrak{M} \in \operatorname{Mer}_{\operatorname{s}}(\mathbb{P}^1)\), we have \(\operatorname{FT}_{\operatorname{mer}}(\mathfrak{M}) \in \operatorname{Mer}_{\operatorname{s}}(\mathbb{P}^1)\), because \(\operatorname{FT}\) preserves irreducibility of \(\mathcal{D}_{\mathbb{A}^1}\)-modules. We recall the definition of rigidity as given in \cite{Ari10}: 

\begin{definition}
We define rigidity in the following ways: 
\begin{itemize}
    \item An object \(\mathfrak{M} \in \operatorname{Mer}_{\operatorname{s}}(\mathbb{P}^{1})\) with singular locus \(D\) is said to be rigid if 
    \begin{align*}
    \operatorname{rig}(\mathfrak{M}) \coloneqq \sum_{i = 0}^2 (-1)^i \dim \,H^i_{\operatorname{dR}}(\mathcal{E}nd(\mathfrak{M})({!*}D)) = 2. 
    \end{align*}
    \(\operatorname{rig}(\mathfrak{M})\) is called the index of rigidity of \(\mathfrak{M}\). 
    \item An object \(\mathfrak{M}\) of \(\operatorname{Hol}_{\mathbb{P}^1, \operatorname{s}}(\mathbb{P}^1)\) with singular locus \(D\) is said to be rigid if \(\mathfrak{M}(*D)\) is so as an object of \(\operatorname{Mer}_{\operatorname{s}}(\mathbb{P}^{1})\). 
    \item The full subcategory of rigid objects of \(\operatorname{Mer}_{\operatorname{s}}(\mathbb{P}^1)\) is denoted as \(\operatorname{Mer}_{\operatorname{s, rig}}(\mathbb{P}^1)\) and similarly for the rigid objects of \(\operatorname{Hol}_{\operatorname{s}, \mathbb{P}^1}(\mathbb{P}^1)\). 
\end{itemize}
\end{definition}

\begin{remark}
We note that, for an object \(\mathfrak{M} \in \operatorname{Hol}_{\mathbb{P}^1, \operatorname{s}}(\mathbb{P}^1)\) with singular locus \(D\), taking \(\mathcal{E}nd\) and minimal extension across \(D\) do in general not commute. So, in many cases we have 
\begin{align*}
\mathcal{E}nd(\mathfrak{M}(*D))({!*}D) \neq \mathcal{E}nd(\mathfrak{M}). 
\end{align*}
\end{remark}

We can now recall the Katz-Arinkin algorithm as stated in \cite{Ari10}, Corollary 2.5: 

\begin{theorem}[Katz-Arinkin algorithm]\label{thm:KatzAri}
Let \(\mathfrak{M} \in \operatorname{Mer}_{\operatorname{s, rig}}(\mathbb{P}^1)\). Then, \(\mathfrak{M}\) is obtained from \((\mathcal{O}_{\mathbb{P}^1}, d)\) by a finite sequence of the following operations: 

\begin{itemize}
    \item Tensoring with an irreducible meromorphic connection on \(\mathbb{P}^1\) of rank \(1\). 
    \item Applying the Fourier transform of meromorphic connections. 
    \item Pullback along a Möbius transformation. 
\end{itemize}
\end{theorem}

We note that every connection \(\mathfrak{N}\) obtained from \((\mathcal{O}_{\mathbb{P}^1}, d)\) by these operations lies in \(\operatorname{Mer}_{\operatorname{s, rig}}(\mathbb{P}^1)\). We will need the following variant of the Katz-Arinkin algorithm: 

\begin{lemma}\label{lemma:KatzAriVar}
Let \(\mathfrak{M} \in \operatorname{Hol}_{\operatorname{s}, \mathbb{P}^1, \operatorname{rig}}(\mathbb{P}^1)\). Then, \(\mathfrak{M}\) is obtained from \((\mathcal{O}_{\mathbb{P}^1}, d)\) by a finite sequence of the following operations: 

\begin{itemize}
    \item Localisation along the singular locus, tensoring with an irreducible meromorphic connection on \(\mathbb{P}^1\) of rank \(1\) and then taking the minimal extension along the new singular locus. 
    \item Applying the Fourier transform \(\operatorname{FT}: \operatorname{Hol}_{\operatorname{s}}(\mathbb{P}^1) \rightarrow \operatorname{Hol}_{\operatorname{s}}(\mathbb{P}^1)\) from \Cref{lemma:canonFT}. 
    \item Pullback along a Möbius transformation. 
\end{itemize}
\end{lemma}

\begin{proof}
This follows directly from \Cref{thm:KatzAri} with the correspondence \(\operatorname{Hol}_{\operatorname{s}, \mathbb{P}^1}(\mathbb{P}^1) \cong \operatorname{Mer}_{\operatorname{s}}(\mathbb{P}^1)\) given by localisation along the singular locus in one direction and minimal extension along the singular locus in the other direction. 
\end{proof}

\begin{example}
An example of an object in \(\operatorname{Hol}_{\operatorname{s}, \mathbb{P}^1, \operatorname{rig}}(\mathbb{P}^1)\) is given by the Fourier transform of \(p^{*, \operatorname{loc}}(\mathcal{D}/(\mathcal{P}))({!*}\infty)\), as defined in the statement of \Cref{thm:FTofT}. 
\end{example}

\section{Rigid Twistor \texorpdfstring{\(\mathcal{D}\)-modules}{D-modules}}

Recall from \Cref{theorem:mocmainres} that there is the equivalence  
\begin{align}\label{eq:mocmaincorr}
\Xi_{\operatorname{DR}}: \operatorname{MT}^{\operatorname{wild}}_{\mathbb{P}^1, \operatorname{s}}(\mathbb{P}^1, k, \sqrt{-1}\mathbb{R})^{(p)} \overset{\cong}{\rightarrow} \operatorname{Hol}_{\mathbb{P}^1, \operatorname{s}}(\mathbb{P}^1) 
\end{align}

for any \(k \in \mathbb{Z}\), where the left-hand category is the category of simple purely imaginary polarizable Twistor \(\mathcal{D}\)-modules on \(\mathbb{P}^1\) of weight \(k\). When omitting the argument \(k\) in this notation, this means that we do not impose a weight restriction. 

\begin{definition}\label{def:rig-1R}
We denote the subcategory of \(\operatorname{MT}^{\operatorname{wild}}_{\mathbb{P}^1, \operatorname{s}}(\mathbb{P}^1, k, \sqrt{-1}\mathbb{R})^{(p)}\) mapped to \(\operatorname{Hol}_{\operatorname{s}, \mathbb{P}^1, \operatorname{rig}}(\mathbb{P}^1)\) under this correspondence as \(\operatorname{MT}^{\operatorname{wild}}_{\mathbb{P}^1, \operatorname{s, rig}}(\mathbb{P}^1, k, \sqrt{-1}\mathbb{R})^{(p)}\). We call the objects of this category rigid simple purely imaginary polarizable Twistor \(\mathcal{D}\)-modules on \(\mathbb{P}^1\) of weight \(k\). If we do not impose a weight restriction, we write \(\operatorname{MT}^{\operatorname{wild}}_{\mathbb{P}^1, \operatorname{s, rig}}(\mathbb{P}^1, \sqrt{-1}\mathbb{R})^{(p)}\). 
\end{definition}

We can generalise this definition as follows: Denote the category of simple rigid polarizable Twistor \(\mathcal{D}\)-modules on \(\mathbb{P}^1\) of weight \(k\) by \(\operatorname{MT}^{\operatorname{wild}}_{\mathbb{P}^1, \operatorname{s, rig}}(\mathbb{P}^1, k)^{(p)}\). (In comparison to before, we lose the "purely imaginary" assumption.) If we do not impose a weight restriction, we write \(\operatorname{MT}^{\operatorname{wild}}_{\mathbb{P}^1, \operatorname{s, rig}}(\mathbb{P}^1)^{(p)}\). We give the following notion of rigidity for objects of this category: 

\begin{definition}
An object \(\mathfrak{T}\) of \(\operatorname{MT}^{\operatorname{wild}}_{\mathbb{P}^1, \operatorname{s}}(\mathbb{P}^1)^{(p)}\) is called rigid if 
\begin{align*}
\operatorname{rig}(\mathfrak{T}) \coloneqq \sum_{i = 0}^2 (-1)^i \operatorname{rk} \,R^ip_*(\mathscr{E}nd(\mathfrak{T}(*D))[{!*}D]) = 2, 
\end{align*}

where \(p\) is the constant map \(\mathbb{P}^1 \rightarrow \{\operatorname{pt}\}\) and \(D \times \mathbb{C}_z\) is the singular locus of \(\mathfrak{T}\). 
\end{definition}

\begin{remark}\label{remark:rigidvalid}
We make the following remarks on this definition: 
\begin{itemize}
    \item We note that \(R^ip_*(\mathscr{E}nd(\mathfrak{T}(*D))[{!*}D])\) is a Twistor structure, so that there is no \(z\)-torsion and the rank of this object is well-defined (for \(i = 0, 1, 2\)). This is because the property of being a good wild harmonic bundle is preserved under duality and tensor product. 
    \item For any \(\mathfrak{T}\) of \(\operatorname{MT}^{\operatorname{wild}}_{\mathbb{P}^1, \operatorname{s}}(\mathbb{P}^1)^{(p)}\) and \(z' \in \boldsymbol{S}\), we have 
    \begin{align*}
    \operatorname{rig}(\mathfrak{T}) = \operatorname{rig}(\mathfrak{P}/(z - z')\mathfrak{P}), 
    \end{align*}    
    where \(\mathfrak{P}\) is the underlying \(\mathscr{R}_{\mathbb{P}^1 \times \mathbb{C}_z}\)-module of \(\mathfrak{T}\). This is seen by understanding the variation of the harmonic metric and the Deligne-Malgrange lattice as \(z\) moves in \(\boldsymbol{S}\). In particular, we have \(\operatorname{rig}(\mathfrak{T}) = \operatorname{rig}(\Xi_{\operatorname{DR}}\mathfrak{T})\), which shows that this definition of rigidity is indeed compatible with \Cref{def:rig-1R}, where we had already defined rigidity for the purely imaginary objects of \(\operatorname{MT}^{\operatorname{wild}}_{\mathbb{P}^1, \operatorname{s}}(\mathbb{P}^1)^{(p)}\). 
\end{itemize}
\end{remark}

This definition generalizes \Cref{def:rig-1R} to the non-purely imaginary case. 

\section{Algebraicity of certain rigid Twistor \texorpdfstring{\(\mathcal{D}\)-modules}{D-modules}}

Consider the constant map \(p: \mathbb{P}^1 \rightarrow \{\operatorname{pt}\}\). Moreover, recall that \(\mathscr{R} = \mathcal{O}_{\mathbb{C}_z}\langle t, z\partial_t \rangle\) (a sheaf on \(\mathbb{C}_z\)) and \(\mathcal{R} = \mathbb{C}\langle t, z, z\partial_t\rangle\) (a ring). By definition, for every Twistor \(\mathcal{D}\)-module \((\mathfrak{M}', \mathfrak{M}'', \mathfrak{C})\) on \(\mathbb{P}^1\), \(\mathfrak{M}'\) and \(\mathfrak{M}''\) are good \(\mathscr{R}_{\mathbb{P}^1 \times \mathbb{C}_z}\)-modules. This means that \(\mathfrak{M}'(\infty)\) and \(\mathfrak{M}''(\infty)\) are \(p\)-good by \Cref{remark:goodpgood} and so \(p_*(\mathfrak{M}'(*\infty)),p_*(\mathfrak{M}''(*\infty)) \in \operatorname{Mod}_{\operatorname{coh}}(\mathscr{R})\) by \Cref{lemma:partialGAGA}. 

\begin{definition}
Let \(\mathcal{M} \in \operatorname{Mod}_{\operatorname{coh}}(\mathscr{R})\). We say that \(\mathcal{M}\) is \(z\)-algebraic if there exists \(\mathcal{M}^{z\operatorname{-alg}} \in \operatorname{Mod}_{\operatorname{coh}}(\mathcal{R})\) such that 
\begin{align*}
\mathcal{M} = \mathscr{R}\otimes_{q^*\mathcal{R}} q^*\mathcal{M}^{z\operatorname{-alg}}, 
\end{align*}

where \(q: \mathbb{C}_z \rightarrow \{\operatorname{pt}\}\) is the projection. We call a Twistor \(\mathcal{D}\)-module \((\mathfrak{M}', \mathfrak{M}'', \mathfrak{C})\) on \(\mathbb{P}^1\) \(z\)-algebraic if \(p_*(\mathfrak{M}'(*\infty))\) and \(p_*(\mathfrak{M}''(*\infty))\) are so. 
\end{definition}

\begin{theorem}\label{lemma:zalgeb}
Every object in \(\operatorname{MT}^{\operatorname{wild}}_{\mathbb{P}^1, \operatorname{s, rig}}(\mathbb{P}^1, \sqrt{-1}\mathbb{R})^{(p)}\) is \(z\)-algebraic. 
\end{theorem}

\begin{proof}
We first show the result for objects of weight \(0\). First we observe that \Cref{theorem:firstmainresult} implies the \(z\)-algebraicness of any object of \(\operatorname{MT}^{\operatorname{wild}}_{\mathbb{P}^1, \operatorname{s}}(\mathbb{P}^1, 0, \sqrt{-1}\mathbb{R})^{(p)}\) corresponding to an object of \(\operatorname{Hol}_{\mathbb{P}^1, \operatorname{s}}(\mathbb{P}^1)\) with generic rank \(1\). This is because the operator \(\mathcal{P}_z\) as given in \Cref{theorem:firstmainresult} lies in \(\mathcal{R}\). In particular, the Twistor \(\mathcal{D}\)-module corresponding to \((\mathcal{O}_{\mathbb{P}^1}, d)\) is \(z\)-algebraic. To conclude in the weight \(0\) case, it remains to show that the three steps given in \Cref{lemma:KatzAriVar} - translated via the correspondence from \Cref{eq:mocmaincorr} for weight \(0\) - preserve \(z\)-algebraicness. This is clear, because all three steps can be defined on the level of \(\mathcal{R}_{\mathbb{P}^1}\)-modules. We generalise to arbitrary weight with the Tate twist. 
\end{proof}

We make the following conjectures: 

\begin{conjecture}\label{conj:finalconj}
We conjecture that there exists an analogue of the Katz-Arinkin algorithm for the category \(\operatorname{MT}^{\operatorname{wild}}_{\mathbb{P}^1, \operatorname{s, rig}}(\mathbb{P}^1)^{(p)}\). We also conjecture that this statement implies the following: Every object in \(\operatorname{MT}^{\operatorname{wild}}_{\mathbb{P}^1, \operatorname{s, rig}}(\mathbb{P}^1)^{(p)}\) is \(z\)-algebraic. 
\end{conjecture}

\section{Companions of \texorpdfstring{\(\mathcal{D}_{\mathbb{P}^1}\)-modules}{D-modules}}

Recall that, for an \(\mathscr{R}_{\mathbb{P}^1 \times \mathbb{C}_z}\)-module \(\mathfrak{M}\), the restriction \(\mathfrak{M}/(z - z')\mathfrak{M}\) to any \(z' \in \mathbb{C}_z^*\) is a \(\mathcal{D}_{\mathbb{P}^1}\)-module. 

\begin{definition}
We make the following definitions: 
\begin{itemize}
    \item We call two objects \(\mathfrak{M}, \mathfrak{N} \in \operatorname{Hol}_{\mathbb{P}^1}(\mathbb{P}^1)\) (they might not be simple) \(\mathbb{C}_z^*\)-companions if there exists \(\mathfrak{T} \in \operatorname{MT}^{\operatorname{wild}}_{\mathbb{P}^1, \operatorname{s}}(\mathbb{P}^1, \sqrt{-1}\mathbb{R})^{(p)}\) and \(z', z'' \in \mathbb{C}_z^*\), with the following property: Denoting by \(\mathfrak{P}\) the underlying \(\mathscr{R}_{\mathbb{P}^1}\)-module of \(\mathfrak{T}\), we have 
    \begin{align*}
    \mathfrak{P}/(z - z')\mathfrak{P} = \mathfrak{M}, \quad \mathfrak{P}/(z - z'')\mathfrak{P} = \mathfrak{N}. 
    \end{align*}
    \item We call two objects \(\mathfrak{M}, \mathfrak{N} \in \operatorname{Hol}_{\mathbb{P}^1}(\mathbb{P}^1)\) \(\boldsymbol{S}\)-companions (possibly not simple) if there exists \(\mathfrak{T} \in \operatorname{MT}^{\operatorname{wild}}_{\mathbb{P}^1, \operatorname{s}}(\mathbb{P}^1, \sqrt{-1}\mathbb{R})^{(p)}\) and \(z', z'' \in \boldsymbol{S}\), with the following property: Denoting by \(\mathfrak{P}\) the underlying \(\mathscr{R}_{\mathbb{P}^1}\)-module of \(\mathfrak{T}\), we have 
    \begin{align*}
    \mathfrak{P}/(z - z')\mathfrak{P} = \mathfrak{M}, \quad \mathfrak{P}/(z - z'')\mathfrak{P} = \mathfrak{N}. 
    \end{align*}
\end{itemize}
In both cases, \(\mathfrak{T}\) is called the connector of \(\mathfrak{M}\) and \(\mathfrak{N}\). 
\end{definition}

\begin{remark}
In some ways, \(\boldsymbol{S}\)-companionship is better behaved than \(\mathbb{C}_z^*\)-companionship: 

\begin{itemize}
    \item If \(\mathfrak{M}, \mathfrak{N} \in \operatorname{Hol}_{\mathbb{P}^1}(\mathbb{P}^1)\) are \(\boldsymbol{S}\)-companions, they are both simple: Indeed, let \(\mathfrak{T}\) be the connector of \(\mathfrak{M}\) and \(\mathfrak{N}\). After rotating \(\mathfrak{T}\) in the \(z\)-plane, we obtain \(\mathfrak{T}'\) with \(\operatorname{\Xi}_{\operatorname{DR}}(\mathfrak{T}') = \mathfrak{M}\) and \(\mathfrak{T}''\) with \(\operatorname{\Xi}_{\operatorname{DR}}(\mathfrak{T}'') = \mathfrak{N}\), where \(\mathfrak{T}', \mathfrak{T}'' \in \operatorname{MT}^{\operatorname{wild}}_{\mathbb{P}^1, \operatorname{s}}(\mathbb{P}^1, \sqrt{-1}\mathbb{R})^{(p)}\). The simplicity of \(\mathfrak{M}\) and \(\mathfrak{N}\) then follows from \Cref{theorem:mocmainres}. We emphasize that we can only obtain \(\mathfrak{T}'\) and \(\mathfrak{T}''\) in this way, because \(z'\) and \(z''\) are assumed to be in \(\boldsymbol{S}\), and so we cannot use the same argument for \(\mathbb{C}_z^*\)-companions. 
    \item \(\boldsymbol{S}\)-companionship is a symmetric relation by definition. If we restrict this relation to \(\operatorname{Hol}_{\mathbb{P}^1, \operatorname{s}}(\mathbb{P}^1)\), it is reflexive by \Cref{theorem:mocmainres}, where we use the surjectivity of \(\operatorname{\Xi}_{\operatorname{DR}}\). On the other hand, the injectivity of \(\operatorname{\Xi}_{\operatorname{DR}}\) shows that this relation is transitive (if we restrict again to simple objects), where we use a rotation argument as before. In other words, \(\boldsymbol{S}\)-companionship is an equivalence relation on \(\operatorname{Hol}_{\mathbb{P}^1, \operatorname{s}}(\mathbb{P}^1)\). 
    \item It follows directly from the second point of \Cref{remark:rigidvalid} that any \(\boldsymbol{S}\)-companion of an object \(\mathfrak{M}\) in \(\operatorname{Hol}_{\mathbb{P}^1, \operatorname{s}}(\mathbb{P}^1)\) is rigid if and only if \(\mathfrak{M}\) is so. 
\end{itemize}
\end{remark}

The second point of this remark naturally leads us to the question of how to characterize 
\begin{align*}
\operatorname{Hol}_{\mathbb{P}^1, \operatorname{s}}(\mathbb{P}^1)/(\boldsymbol{S}\operatorname{-companionship}), 
\end{align*}

for which we make the following conjecture: 

\begin{conjecture}
Denote the full subcategory of objects of \(\operatorname{Hol}_{\mathbb{P}^1, \operatorname{s}}(\mathbb{P}^1)\) which are locally formally unitary by \(\operatorname{Hol}^{\operatorname{lfu}}_{\mathbb{P}^1, \operatorname{s}}(\mathbb{P}^1)\). We have 
\begin{align*}
\operatorname{Hol}_{\mathbb{P}^1, \operatorname{s}}(\mathbb{P}^1)/(\boldsymbol{S}\operatorname{-companionship}) \cong \operatorname{Hol}^{\operatorname{lfu}}_{\mathbb{P}^1, \operatorname{s}}(\mathbb{P}^1)/\boldsymbol{S}, 
\end{align*}

where the action of \(\boldsymbol{S}\) on the right-hand side is a Stokes deformation similar to the notion defined by Mochizuki in \cite{Moc10}, Section 4.5.2, which acts trivially on objects without irregular singularities. 
\end{conjecture}

We can also make the following observation, which is concerned with restricting an \(\mathscr{R}_{\mathbb{P}^1 \times \mathbb{C}_z}\)-module to any point \(z'\) in \(\mathbb{C}_z^*\) and not just \(\boldsymbol{S}\). 

\begin{lemma}
Let \(\mathfrak{M} \in \operatorname{Hol}_{\mathbb{P}^1, \operatorname{s, \operatorname{rig}}}(\mathbb{P}^1)\) and denote the corresponding Twistor \(\mathcal{D}\)-module (in the sense of \Cref{eq:mocmaincorr}) by \(\mathfrak{T} = (\mathfrak{P}, \mathfrak{P}, \mathfrak{C})\). Then, we can find a countable set \(Y \subset \mathbb{C}_z\) with \(Y \cap \boldsymbol{S} = \emptyset\), such that the restriction of \(\mathfrak{P}\) to any \(z' \notin Y\) lies in \(\operatorname{Hol}_{\mathbb{P}^1, \operatorname{s, \operatorname{rig}}}(\mathbb{P}^1)\) as well. If \(\mathfrak{M}\) is locally formally unitary, we can choose \(Y = \emptyset\). 
\end{lemma}

\begin{proof}
We begin by assuming that the generic rank of \(\mathfrak{M}\) is \(1\): By \Cref{theorem:firstmainresult}, we have \(\mathfrak{P}/(z - z')\mathfrak{P} \in \operatorname{Hol}_{\mathbb{P}^1, \operatorname{s}, \operatorname{rig}}(\mathbb{P}^1)\) for any \(z' \notin Y\), where \(Y\) is countable and intersects trivially with \(\boldsymbol{S}\). Using the Katz-Arinkin algorithm, we deduce that the same is true if \(\mathfrak{T}\) is of arbitrary generic rank: For each twisting step in the algorithm, \(Y\) might have to grow, but it can stay the same for the steps involving Fourier transform and pullback along Möbius maps. 

If \(\mathfrak{M}\) is locally formally unitary, we can however directly show that \(\mathfrak{P}/(z - z')\mathfrak{P} \in \operatorname{Hol}_{\mathbb{P}^1, \operatorname{s}, \operatorname{rig}}(\mathbb{P}^1)\) for any \(z' \in \mathbb{C}_z^*\), which follows again from \Cref{theorem:firstmainresult} in combination with the Katz-Arinkin algorithm. 
\end{proof}

\appendix

\chapter{Algebraic and analytic \texorpdfstring{\(\mathscr{R}\)-modules}{R-modules}}\label{chapter:gagaapp}

\section{Partial algebraization for \texorpdfstring{\(\mathscr{R}_{\mathbb{P}^1}(*\infty)\)-modules}{R-modules}}

Recall that we consider \(\mathbb{C}, \mathbb{P}^1\) with the analytic topologies and \(\mathbb{A}^1, \mathbb{P}^{1, \operatorname{alg}}\) with the algebraic topologies. Consider the projection \(p: \mathbb{P}^1 \rightarrow \{\operatorname{pt}\}\). Recall from \Cref{lemma:DmodAP} that \(p_*\) gives an equivalence 
\begin{align*}
\operatorname{Mod}_{\operatorname{coh}}(\mathcal{D}_{\mathbb{P}^1}(*\infty)) \rightarrow \operatorname{Mod}_{\operatorname{coh}}(\mathcal{D}), 
\end{align*}

with quasi-inverse \(p^{*, \operatorname{loc}}\). The goal of this subsection is to show a similar result for \(\mathscr{R}_{\mathbb{P}^1 \times \mathbb{C}_z}(*\infty)\)-modules on one side and \(\mathscr{R}\)-modules on the other side. 

\begin{definition}
Let \(\mathscr{M}\) be an \(\mathscr{R}_{\mathbb{P}^1 \times \mathbb{C}_z}(*\infty)\)-module. 
\begin{itemize}
    \item (This is in analogy to \cite{DouSab03}, Appendix A.a.) \(\mathscr{M}\) is called \(p\)-good, if it is \(\mathscr{R}_{\mathbb{P}^1 \times \mathbb{C}_z}(*\infty)\)-coherent and there exists - locally in the \(z\)-direction - a coherent \(\mathcal{O}_{\mathbb{P}^1 \times \mathbb{C}_z}\)-module \(\mathscr{F}\) with a surjective map 
    \begin{align*}
    \mathscr{R}_{\mathbb{P}^1 \times \mathbb{C}_z}(*\infty) \otimes_{\mathcal{O}_{\mathbb{P}^1 \times \mathbb{C}_z}} \mathscr{F} \rightarrow \mathscr{M} \rightarrow 0. 
    \end{align*}
    \item \(\mathscr{M}\) is called weakly \(p\)-good, if it admits - locally in the \(z\)-direction - a finite filtration of \(\mathscr{R}_{\mathbb{P}^1 \times \mathbb{C}_z}(*\infty)\) such that each graded piece is \(p\)-good. 
\end{itemize}
\end{definition}

Recall that \(\mathscr{R} = \mathcal{O}_{\mathbb{C}_z}\langle t, z\partial_t \rangle\). For a coherent \(\mathscr{R}\)-module \(\mathscr{N}\), we define 
\begin{align*}
p^{*, \operatorname{loc}}(\mathscr{N}) \coloneqq \mathscr{R}_{\mathbb{P}^1 \times \mathbb{C}_z}(*\infty) \otimes_{p^{-1}\mathscr{R}} p^{-1}\mathscr{N}. 
\end{align*}

Recall the following result from \cite{DouSab03}, Appendix A.a, A.2. It is a consequence of the Grauert-Remmert Theorems A and B. 

\begin{lemma}[Grauert-Remmert]\label{lemma:grauertremmert}
Let \(\mathscr{F}\) be a coherent \(\mathcal{O}_{\mathbb{P}^1 \times \mathbb{C}_z}\)-module. Then \(R^kp_*\mathscr{F}\) is \(\mathcal{O}_{\mathbb{C}_z}\)-coherent for any \(k\) and for any relatively compact open subset \(W\) of \(\mathbb{C}_z\), there exists \(q_0 \in \mathbb{N}\) such that, for any \(q \geq q_0\) and restricting to \(\mathbb{P}^1 \times W\): 

\begin{itemize}
    \item the natural morphism \(p^*p_*\mathscr{F}(q) \rightarrow \mathscr{F}(q)\) is onto, 
    \item \(R^kp_*\mathscr{F}(q) = 0\) for \(k \geq 1\). 
\end{itemize}
\end{lemma}

We can now show the desired analogue of \Cref{lemma:DmodAP}: 

\begin{lemma}[Partial algebraization for \(\mathscr{R}\)-modules]\label{lemma:partialGAGA}
\(p_*\) gives an equivalence 
\begin{align*}
\operatorname{Mod}_{p \operatorname{-good}}(\mathscr{R}_{\mathbb{P}^1 \times \mathbb{C}_z}(*\infty)) \longrightarrow \operatorname{Mod}_{\operatorname{coh}}(\mathscr{R}) 
\end{align*}

with quasi-inverse \(p^{*, \operatorname{loc}}\). 
\end{lemma}

\begin{proof}
We adapt the proof in \cite{DouSab03}, Appendix A.a to our situation, but we will not repeat every step. We work locally in the \(z\)-direction. Let \(\mathscr{M} \in \operatorname{Mod}_{p \operatorname{-good}}(\mathscr{R}_{\mathbb{P}^1 \times \mathbb{C}_z}(*\infty))\). By the definition of \(p\)-goodness, there exists a coherent \(\mathcal{O}_{\mathbb{P}^1 \times \mathbb{C}_z}\)-module \(\mathscr{F}\) with a surjective map 
\begin{align*}
\pi: \mathscr{R}_{\mathbb{P}^1 \times \mathbb{C}_z}(*\infty) \otimes_{\mathcal{O}_{\mathbb{P}^1 \times \mathbb{C}_z}} \mathscr{F} \rightarrow \mathscr{M} \rightarrow 0. 
\end{align*}

Define \(\mathscr{N}\) as the image of \(\mathscr{R}_{\mathbb{P}^1 \times \mathbb{C}_z} \otimes_{\mathcal{O}_{\mathbb{P}^1 \times \mathbb{C}_z}} \mathscr{F}\) under \(\pi\). We clearly have \(\mathscr{N}(*\infty) = \mathscr{M}\). We have a coherent \(\mathcal{O}_{\mathbb{P}^1 \times \mathbb{C}_z}\)-module \(\mathscr{F}\) with a surjective map 
\begin{align*}
\mathscr{R}_{\mathbb{P}^1 \times \mathbb{C}_z} \otimes_{\mathcal{O}_{\mathbb{P}^1 \times \mathbb{C}_z}} \mathscr{F} \rightarrow \mathscr{N} \rightarrow 0. 
\end{align*}

By \Cref{lemma:grauertremmert}, one gets a surjection 
\begin{align}\label{eq:surjontoN}
\mathscr{R}_{\mathbb{P}^1 \times \mathbb{C}_z} \otimes_{\mathcal{O}_{\mathbb{P}^1 \times \mathbb{C}_z}} \mathcal{O}_{\mathbb{P}^1 \times \mathbb{C}_z}(-q_0)^{l_0} \rightarrow \mathscr{N} \rightarrow 0 
\end{align}

for suitable \(l_0, q_0 \in \mathbb{N}\), the kernel \(\mathscr{N}'\) of which is \(\mathscr{R}_{\mathbb{P}^1 \times \mathbb{C}_z}\)-coherent and equipped with a good filtration: Indeed, the order filtration on \(\mathscr{R}_{\mathbb{P}^1 \times \mathbb{C}_z}\) induces a good filtration on \(\mathscr{R}_{\mathbb{P}^1 \times \mathbb{C}_z} \otimes \mathcal{O}_{\mathbb{P}^1 \times \mathbb{C}_z}(-q_0)^{l_0}\). Intersecting this filtration with \(\mathscr{N}'\) gives a good filtration on \(\mathscr{N}'\) (Artin-Rees lemma, cf. \cite{Hot18}, p. 367), because \(\operatorname{gr}^F\mathscr{R}_{\mathbb{P}^1 \times \mathbb{C}_z}\) is a coherent sheaf of rings. So, we obtain a surjection \(\mathscr{R}_{\mathbb{P}^1 \times \mathbb{C}_z} \otimes \mathscr{G} \rightarrow \mathscr{N}' \rightarrow 0\) from the definition of good filtrations. Repeating these steps with \(\mathscr{G}\) instead of \(\mathscr{F}\), we get another surjective map 
\begin{align*}
\mathscr{R}_{\mathbb{P}^1 \times \mathbb{C}_z} \otimes_{\mathcal{O}_{\mathbb{P}^1 \times \mathbb{C}_z}} \mathcal{O}_{\mathbb{P}^1 \times \mathbb{C}_z}(-q_1)^{l_1} \rightarrow \mathscr{N}' \rightarrow 0. 
\end{align*}

Combining this with \Cref{eq:surjontoN} gives a presentation 
\begin{align*}
\mathscr{R}_{\mathbb{P}^1 \times \mathbb{C}_z} \otimes_{\mathcal{O}_{\mathbb{P}^1 \times \mathbb{C}_z}} \mathcal{O}_{\mathbb{P}^1 \times \mathbb{C}_z}(-q_1)^{l_1} \rightarrow\mathscr{R}_{\mathbb{P}^1 \times \mathbb{C}_z} \otimes_{\mathcal{O}_{\mathbb{P}^1 \times \mathbb{C}_z}} \mathcal{O}_{\mathbb{P}^1 \times \mathbb{C}_z}(-q_0)^{l_0} \rightarrow \mathscr{N} \rightarrow 0. 
\end{align*}

From this, one gets a free presentation 
\begin{align*}
\mathscr{R}_{\mathbb{P}^1 \times \mathbb{C}_z}(*\infty)^{l_1} \rightarrow \mathscr{R}_{\mathbb{P}^1 \times \mathbb{C}_z}(*\infty)^{l_0} \rightarrow \mathscr{N}(*\infty) \rightarrow 0. 
\end{align*}

Applying \(p_*\) to this presentation gives a presentation of \(p_*\mathscr{M}\) by free \(\mathscr{R}\)-modules. This implies the \(\mathscr{R}\)-coherence of \(p_*\mathscr{M}\), and one also deduces that the natural transformations \(\operatorname{id} \rightarrow p_*p^{*, \operatorname{loc}}\) and \(p^{*, \operatorname{loc}}p_* \rightarrow \operatorname{id}\) are isomorphisms, which finishes the proof. 
\end{proof}

As a corollary of \Cref{lemma:partialGAGA}, we can show that \(p\)-goodness is preserved under extensions: 

\begin{lemma}\label{lemma:extgoodcoh}
Consider a short exact sequence 
\begin{align}\label{eq:startingsequence}
0 \rightarrow \mathscr{M}' \rightarrow \mathscr{M} \rightarrow \mathscr{M}'' \rightarrow 0 
\end{align}

of \(\mathscr{R}_{\mathbb{P}^1 \times \mathbb{C}_z}(*\infty)\)-modules, where \(\mathscr{M}'\) and \(\mathscr{M}''\) are \(p\)-good. Then, \(\mathscr{M}\) is so as well. 
\end{lemma}

\begin{proof}
The idea for this argument was given to the author by C. Sabbah in a discussion. Our first goal is to show \(R^1p_*\mathscr{M}' = 0\). As in the proof of \Cref{lemma:partialGAGA}, we obtain a good \(\mathscr{R}_{\mathbb{P}^1 \times \mathbb{C}_z}\)-module \(\mathscr{N}\) with \(\mathscr{N}(*\infty) = \mathscr{M}'\) and with a good filtration \(F_k\mathscr{N}\). Because \(p\) is proper, \Cref{lemma:grauertremmert} gives 
\begin{align*}
R^1p_*\mathscr{M}' & = R^1p_*\mathscr{N}(*\infty) \\ & = \mathop{\lim_{\longrightarrow}}_{i, k}R^1p_*((F_k\mathscr{N})(i)) = 0. 
\end{align*}

So, we get the exact sequence 
\begin{align*}
0 \rightarrow p_*\mathscr{M}' \rightarrow p_*\mathscr{M} \rightarrow p_*\mathscr{M}'' \rightarrow 0 
\end{align*}

of \(\mathscr{R}\)-modules. Furthermore, since \(p_*\mathscr{M}'\) and \(p_*\mathscr{M}''\) are coherent, \(p_*\mathscr{M}\) is so as well. (This is because \(\mathscr{R}\) is coherent.) By \Cref{lemma:partialGAGA}, we can apply \(p^{*, \operatorname{loc}}\) to get an exact sequence 
\begin{align*}
0 \rightarrow \mathscr{M}' \rightarrow p^{*, \operatorname{loc}}p_*\mathscr{M} \rightarrow \mathscr{M}'' \rightarrow 0 
\end{align*}

of \(p\)-good \(\mathscr{R}_{\mathbb{P}^1 \times \mathbb{C}_z}(*\infty)\)-modules. Comparing this sequence with the sequence from \Cref{eq:startingsequence}, we conclude that \(p^{*, \operatorname{loc}}p_*\mathscr{M} \cong \mathscr{M}\) by the five lemma, so that \(\mathscr{M}\) is \(p\)-good, showing the assertion. 
\end{proof}

Finally, we deduce that weak \(p\)-goodness is in fact equivalent to \(p\)-goodness: 

\begin{lemma}\label{lemma:weakpstrong}
We have \(\operatorname{Mod}_{p \operatorname{-good}}(\mathscr{R}_{\mathbb{P}^1 \times \mathbb{C}_z}(*\infty)) = \operatorname{Mod}_{\operatorname{weakly} p \operatorname{-good}}(\mathscr{R}_{\mathbb{P}^1 \times \mathbb{C}_z}(*\infty))\). 
\end{lemma}

\begin{proof}
We work locally in the \(z\)-direction. Let \(\mathscr{M} \in \operatorname{Mod}_{\operatorname{weakly} p \operatorname{-good}}(\mathscr{R}_{\mathbb{P}^1 \times \mathbb{C}_z}(*\infty))\). By definition of weak \(p\)-goodness, there is a finite filtration \(F_{\bullet}\mathscr{M}\) of \(\mathscr{M}\) such that each graded piece is \(p\)-good. Consider the short exact sequence 
\begin{align*}
0 \rightarrow F_0\mathscr{M} \rightarrow F_1\mathscr{M} \rightarrow \operatorname{gr}_1^F\mathscr{M} \rightarrow 0. 
\end{align*}

\(\operatorname{gr}_0^F\mathscr{M} = F_0\mathscr{M}\) and \(\operatorname{gr}_1^F\mathscr{M}\) are \(p\)-good by assumption, and so \(F_1\mathscr{M}\) is \(p\)-good by \Cref{lemma:extgoodcoh}. Continuing iteratively with the finitely many graded pieces, we obtain that \(\mathscr{M} \in \operatorname{Mod}_{p \operatorname{-good}}(\mathscr{R}_{\mathbb{P}^1 \times \mathbb{C}_z}(*\infty))\), showing the assertion. 
\end{proof}

\section{Good \texorpdfstring{\(\mathscr{R}\)-modules}{R-modules}}

The sheaf of rings \(\mathscr{R}_{\mathbb{P}^1 \times \mathbb{C}_z}\) has the order filtration \((F_i\mathscr{R}_{\mathbb{P}^1 \times \mathbb{C}_z})\). We have the notion of a good filtration of \(\mathscr{R}_{\mathbb{P}^1 \times \mathbb{C}_z}\)-modules analogously to good filtrations of \(\mathcal{D}\)-modules as in \cite{Hot18}, Section 2.1. 

\begin{definition}
Let \(X\) be a complex manifold. After \cite{Sab05}, 1.1.c, an \(\mathscr{R}_{\mathscr{X}}\)-module \(\mathscr{M}\) is called good if there exists, for every compact \(\mathscr{K} \subset \mathscr{X}\), an open set \(\mathscr{U}\) containing \(\mathscr{K}\) such that \(\mathscr{M}\vert_{\mathscr{U}}\) has a filtration by \(\mathscr{R}_{\mathscr{X}}\vert_{\mathscr{U}}\)-modules with the property that each graded piece of this filtration has a good filtration. 
\end{definition}

\begin{example}
A good \(\mathscr{R}_{\mathbb{P}^1 \times \mathbb{C}_z}\)-module \(\mathscr{M}\) has - locally in the \(z\)-direction - a filtration by \(\mathscr{R}_{\mathbb{P}^1 \times \mathbb{C}_z}\)-modules with the property that each graded piece of this filtration has a good filtration. This is seen as follows: Fix an open bounded set \(U \subset \mathbb{C}_z\). \(\mathbb{P}^1 \times \overline{U}\) is compact in \(\mathbb{P}^1 \times \mathbb{C}_z\) and so the claim follows from the goodness of \(\mathscr{M}\). 
\end{example}

Good \(\mathscr{R}_{\mathbb{P}^1 \times \mathbb{C}_z}\)-modules are automatically \(\mathscr{R}_{\mathbb{P}^1 \times \mathbb{C}_z}\)-coherent, cf. \cite{Sab05}, 1.1.c. Furthermore, goodness is clearly preserved under restriction to open subsets. 

\begin{lemma}\label{lemma:goodweaklypgood}
Let \(\mathscr{M}\) be a good \(\mathscr{R}_{\mathbb{P}^1 \times \mathbb{C}_z}\)-module. Then, \(\mathscr{M}(*\infty)\) is weakly \(p\)-good. 
\end{lemma}

\begin{proof}
We work locally in the \(z\)-direction. Let \(F_\bullet\mathscr{M}\) be a finite filtration of \(\mathscr{M}\) such that \(\operatorname{gr}_k^F\mathscr{M}\) has a good filtration for every index \(k\). By the definition of good filtrations, we have  coherent \(\mathcal{O}_{\mathbb{P}^1 \times \mathbb{C}_z}\)-modules \(\mathscr{F}_k\) with surjections 
\begin{align*}
\mathscr{R}_{\mathbb{P}^1 \times \mathbb{C}_z} \otimes_{\mathcal{O}_{\mathbb{P}^1 \times \mathbb{C}_z}} \mathscr{F}_k \rightarrow \operatorname{gr}_k^F\mathscr{M} \rightarrow 0. 
\end{align*}

Let \(F'_\bullet(\mathscr{M}(*\infty))\) be the filtration of \(\mathscr{M}(*\infty)\) defined by 
\begin{align*}
F'_k(\mathscr{M}(*\infty)) \coloneqq (F_k\mathscr{M})(*\infty). 
\end{align*}

Because \(\mathcal{O}_{\mathbb{P}^1 \times \mathbb{C}_z}(*\infty)\) is flat over \(\mathcal{O}_{\mathbb{P}^1 \times \mathbb{C}_z}\), we have 
\begin{align*}
\operatorname{gr}_k^{F'}(\mathscr{M}(*\infty)) = (\operatorname{gr}_k^{F}\mathscr{M})(*\infty). 
\end{align*}

In other words, we obtain surjections 
\begin{align*}
\mathscr{R}_{\mathbb{P}^1 \times \mathbb{C}_z}(*\infty) \otimes_{\mathcal{O}_{\mathbb{P}^1 \times \mathbb{C}_z}} \mathscr{F}_k \rightarrow \operatorname{gr}_k^{F'}(\mathscr{M}(*\infty)) \rightarrow 0, 
\end{align*}

which shows that \(\mathscr{M}(*\infty)\) is weakly \(p\)-good. 
\end{proof}

\begin{remark}\label{remark:goodpgood}
The combination of \Cref{lemma:goodweaklypgood} with \Cref{lemma:weakpstrong} implies that, for every good \(\mathscr{R}_{\mathbb{P}^1 \times \mathbb{C}_z}\)-module \(\mathscr{M}\), the localisation \(\mathscr{M}(*\infty)\) is \(p\)-good. 
\end{remark}

\section{\texorpdfstring{\(\mathscr{R}\)-modules}{R-modules} on the affine line}

Consider the projection \(\pi: \mathbb{A}^1_t \times \mathbb{C}_z \rightarrow \mathbb{C}_z\). 

\begin{lemma}\label{lemma:RaffineOmod}
\(\pi_*\) induces an equivalence of categories 
\begin{align*}
\operatorname{Mod}_{\operatorname{coh}}(\mathcal{O}_{\mathbb{A}^1_t \times \mathbb{C}_z}) \rightarrow \operatorname{Mod}_{\operatorname{coh}}(\mathcal{O}_{\mathbb{C}_z}[t]) 
\end{align*}

with quasi-inverse given by \(\pi^*\) defined as \(\mathcal{O}_{\mathbb{A}^1_t \times \mathbb{C}_z} \otimes_{\pi^{-1}\mathcal{O}_{\mathbb{C}_z}[t]} \pi^{-1}(\cdot)\). 
\end{lemma}

\begin{proof}
We emphasize that \(\mathbb{A}^1_t \times \mathbb{C}_z\) is neither a scheme nor a manifold, which means that many of the classical results do not apply here. However, there is the notion of relative schemes over an arbitrary ringed space (cf. \cite{Hak72}), and \((\mathbb{A}^1_t \times \mathbb{C}_z, \mathcal{O}_{\mathbb{A}^1_t \times \mathbb{C}_z})\) is an affine scheme over the ringed space \((\mathbb{C}_z, \mathcal{O}_{\mathbb{C}_z})\). The result follows from \cite{Hak72}, Chapter VII, Proposition 1.2.7. (There, this result is stated for finitely presented modules, which is equivalent to coherence in our case.) 
\end{proof}

Recall that \(\mathscr{R}_{\mathbb{A}^1_t \times \mathbb{C}_z} = \mathcal{O}_{\mathbb{A}^1_t \times \mathbb{C}_z}\langle z\partial_t \rangle\) and \(\mathscr{R} = \mathcal{O}_{\mathbb{C}_z}\langle t, z\partial_t \rangle\). We call an \(\mathscr{R}_{\mathbb{A}^1_t \times \mathbb{C}_z}\)-module \(\mathscr{M}\) \(\pi\)-good if \(\mathscr{M}\) has - locally in the \(z\)-direction - an \(\mathcal{O}_{\mathbb{A}^1_t \times \mathbb{C}_z}\)-coherent generator over \(\mathscr{R}_{\mathbb{A}^1_t \times \mathbb{C}_z}\). 

\begin{lemma}\label{lemma:RaffineRmod}
\(\pi_*\) induces an equivalence of categories 
\begin{align*}
\operatorname{Mod}_{\pi\operatorname{-good}}(\mathscr{R}_{\mathbb{A}^1_t \times \mathbb{C}_z}) \rightarrow \operatorname{Mod}_{\operatorname{coh}}(\mathscr{R}) 
\end{align*}

with quasi-inverse given by \(\mathscr{R}_{\mathbb{A}^1_t \times \mathbb{C}_z} \otimes_{\pi^{-1}\mathscr{R}} \pi^{-1}(\cdot)\). 
\end{lemma}

\begin{proof}
We work locally in the \(z\)-direction. Let \(\mathscr{M} \in \operatorname{Mod}_{\pi \operatorname{-good}}(\mathscr{R}_{\mathbb{A}^1_t \times \mathbb{C}_z})\) with \(\mathcal{O}_{\mathbb{A}^1_t \times \mathbb{C}_z}\)-coherent generator \(\mathscr{F}\). By \Cref{lemma:RaffineOmod}, \(\pi_*\mathscr{F}\) is in \(\operatorname{Mod}_{\operatorname{coh}}(\mathcal{O}_{\mathbb{C}_z}[t])\). \(\pi_*\mathscr{M}\) is generated over \(\mathscr{R}_{\mathbb{C}_z}\) by \(\pi_*\mathscr{F}\), because \(\pi_*\) is an exact functor on quasi-coherent \(\mathcal{O}_{\mathbb{A}^1_t \times \mathbb{C}_z}\)-modules. This exactness of \(\pi_*\) follows from \cite{Hak72}, Chapter VI, Corollaire 2.2, because \(\pi\) is an affine morphism of \((\mathbb{C}_z, \mathcal{O}_{\mathbb{C}_z})\)-schemes. It follows that \(\pi_*\mathscr{M}\) lies in \(\operatorname{Mod}_{\operatorname{coh}}(\mathscr{R})\). 

On the other hand, let \(\mathscr{N} \in \operatorname{Mod}_{\operatorname{coh}}(\mathscr{R})\). Then, it is easy to show that \(\mathscr{N}\) is generated over \(\mathscr{R}\) by some \(\mathcal{O}_{\mathbb{C}_z}[t]\)-coherent module \(\mathscr{G}\). By \Cref{lemma:RaffineOmod}, we get 
\begin{align*}
\mathcal{O}_{\mathbb{A}^1_t \times \mathbb{C}_z} \otimes_{\pi^{-1}\mathcal{O}_{\mathbb{C}_z}[t]} \pi^{-1}\mathscr{G} \in \operatorname{Mod}_{\operatorname{coh}}(\mathcal{O}_{\mathbb{A}^1_t \times \mathbb{C}_z}), 
\end{align*}

which generates \(\mathscr{R}_{\mathbb{A}^1_t \times \mathbb{C}_z} \otimes_{\pi^{-1}\mathscr{R}} \pi^{-1}\mathscr{N}\): This is clear because \(\mathscr{R}_{\mathbb{A}^1_t \times \mathbb{C}_z} \otimes_{\pi^{-1}\mathscr{R}} \pi^{-1}(\cdot)\) is right exact. One verifies that \(\pi_*\) and \(\pi^*\) are quasi-inverse to each other, which completes the proof. 
\end{proof}

We can summarize our findings with the following diagram, which commutes by construction. We have shown that all functors in the second row are equivalences 

\[\begin{tikzcd}
	{\operatorname{Mod}_{\operatorname{good}}(\mathscr{R}_{\mathbb{P}_t^1 \times \mathbb{C}_z})} & {\operatorname{Mod}_{\operatorname{good}}(\mathscr{R}_{\mathbb{C}_t \times \mathbb{C}_z})} & {\operatorname{Mod}_{\operatorname{coh}}(\mathscr{R}_{\mathbb{C}_t \times \mathbb{C}_z})} \\
	\\
	{\operatorname{Mod}_{p \operatorname{-good}}(\mathscr{R}_{\mathbb{P}_t^1 \times \mathbb{C}_z}(*\infty))} & {\operatorname{Mod}_{\operatorname{coh}}(\mathscr{R}) } & {\operatorname{Mod}_{\pi\operatorname{-good}}(\mathscr{R}_{\mathbb{A}^1_t \times \mathbb{C}_z})}
	\arrow["{(\cdot)\vert_{\mathbb{C}_t}}"', from=1-1, to=1-2]
	\arrow["{(\cdot)(*\infty)}", from=1-1, to=3-1]
	\arrow[hook, from=1-2, to=1-3]
	\arrow["{p_*}", shift left, from=3-1, to=3-2]
	\arrow["{p^{*, \operatorname{loc}}}", shift left, from=3-2, to=3-1]
	\arrow["{\pi^*}", shift left, from=3-2, to=3-3]
	\arrow["{(\cdot)^{t\operatorname{-an}}}"', from=3-3, to=1-3]
	\arrow["{\pi_*}", shift left, from=3-3, to=3-2]
\end{tikzcd}\]

\printbibliography

\end{document}